\documentclass[12pt]{article}
\usepackage{amsmath,amsfonts,amsbsy,amssymb,amsthm} 
\usepackage{graphicx,psfrag,epsf}
\usepackage{natbib}
\usepackage{bbm}
\usepackage{rotating}  % for rotating figures 
\usepackage{adjustbox} % for scaling tables
\usepackage{booktabs}  % for nice tables
\usepackage{url}       % used for URLs
\usepackage{multibib}  % separate references for suppl paper
\usepackage{enumitem}  % resume enumerate 
\newcites{appendix}{References}% use \citeappendix{} and \citepappendix[][]{}
\usepackage{todonotes}
\usepackage{soul}
\usepackage{color}

\newtheorem{theorem}{Theorem}[section]
\newtheorem{proposition}{Proposition}[section]
\newtheorem{prop}{Proposition}[section]
\newtheorem{lemma}{Lemma}[section]
\newtheorem{corollary}{Corollary}[section] 

\newtheorem{assumption}{Assumption}[section]
\newtheorem{algo}{Algorithm}[section]
\DeclareMathOperator{\E}{\mathbb{E}}

\DeclareMathOperator{\V}{\mathbb{V}}

\DeclareMathOperator{\bbeta}{\boldsymbol{\beta}}
\DeclareMathOperator{\bD}{\mathbf{D}}
\DeclareMathOperator{\bV}{\mathbf{V}}
\DeclareMathOperator{\bmu}{\boldsymbol{\mu}}
\DeclareMathOperator{\bSigma}{\boldsymbol{\Sigma}}
\DeclareMathOperator{\bY}{\mathbf{Y}}

\DeclareMathOperator{\bF}{\mathbf{F}}
\DeclareMathOperator{\bU}{\mathbf{U}}
\DeclareMathOperator{\bI}{\mathbf{I}}
\DeclareMathOperator{\bX}{\mathbf{X}}

\DeclareMathOperator{\bH}{\mathbf{H}}
\DeclareMathOperator{\bR}{\mathbf{R}}

\DeclareMathOperator{\btau}{\boldsymbol{\tau}}
\DeclareMathOperator{\bzero}{\boldsymbol{0}}
\DeclareMathOperator{\bx}{\mathbf{x}}
\DeclareMathOperator{\bh}{\mathbf{h}}
\newcommand{\blind}{0}

\usepackage{hyperref}

\graphicspath{{./Figures/}}

\newcommand*\samethanks[1][\value{footnote}]{\footnotemark[#1]}

\begin{document}

\def\spacingset#1{\renewcommand{\baselinestretch}%
{#1}\small\normalsize} \spacingset{1}

%%%%%%%%%%%%%%%%%%%%%%%%%%%%%%%%%%%%%%%%%%%%%%%%%%%%%%%%%%%%%%%%%%%%%%%%%%%%%% 

\if0\blind  
{ 
  \title{\bf Super-Consistent Estimation of Points of Impact in Nonparametric Regression with Functional Predictors}
  \author{
  Dominik Po{\ss}\thanks{Institute of Finance and Statistics, University of Bonn, Bonn, Germany},
  %%%%%%%
  Dominik Liebl\thanks{Institute of Finance and Statistics and Hausdorff Center for Mathematics, University of Bonn, Bonn, Germany},
  %%%%%%%
  Alois Kneip\samethanks,
  %%%%%%%
  Hedwig Eisenbarth\thanks{School of Psychology, Victoria University of Wellington, Wellington, New Zealand},
  %%%%%%%
  Tor D.~Wager\thanks{Presidential Cluster in Neuroscience and Department of Psychological and Brain Sciences, Dartmouth College, Hanover, New Hampshire, USA}, \and
  %%%%%%%
  and Lisa Feldman Barrett\thanks{Department of Psychology, Northeastern University and Department of Psychiatry, Massachusetts General Hospital/Harvard Medical School, Boston, Massachusetts, USA; and Athinoula A.~Martinos Center for Biomedical Imaging, Massachusetts General Hospital, Charlestown, Massachusetts, USA}
  }
  \date{}
  \maketitle

} \fi

\if1\blind
{
  \bigskip
  \bigskip
  \bigskip
  \begin{center}
    {\LARGE\bf Super-Consistent Estimation of Points of Impact in Nonparametric Regression with Functional Predictors} 
\end{center}
  \medskip
} \fi

\bigskip

\begin{abstract}
Predicting scalar outcomes using functional predictors is a classic problem in functional data analysis. In many applications, however, only specific locations or time-points of the functional predictors have an impact on the outcome. Such ``points of impact'' are typically unknown and have to be estimated in addition to estimating the usual model components.  We show that our points of impact estimator enjoys a super-consistent convergence rate and does not require knowledge or pre-estimates of the unknown model components. This remarkable result facilitates the subsequent estimation of the remaining model components as shown in the theoretical part, where we consider the case of nonparametric models and the practically relevant case of generalized linear models. The finite sample properties of our estimators are assessed by means of a simulation study. Our methodology is motivated by data from a psychological experiment in which the participants were asked to continuously rate their emotional state while watching an affective video eliciting a varying intensity of emotional reactions.
\end{abstract}


\vfill

\newpage

\spacingset{1.5}

%%%%%%%%%%%%%%%%%%%%%%%%%%%%%%%%%%%%%%%%%%%%
\section{Introduction}\label{sec:intro}
%%%%%%%%%%%%%%%%%%%%%%%%%%%%%%%%%%%%%%%%%%%%
Identifying important time points in time continuous trajectories is a difficult but highly relevant problem. For instance, current psychological research on emotional experiences often includes time continuous stimuli such as videos to induce emotional states, say $X(t)\in\mathbb{R}$, with $t\in[a,b]$, where $a$ denotes the start of the video and $b$ the end (see Figure \ref{fig:appl_1}). The evaluation of such stimuli is based on asking participants whether the video made them feel negative, say $Y=0$, or positive, say $Y=1$. In this paper we consider a novel dataset where participants were asked to continuously report their emotional states while watching an affective documentary video on the persecution of African albinos. After watching the video, the participants were asked to rate their final overall feeling. Psychologists are interested in understanding how such concluding overall ratings relate to the fluctuations of the emotional states while watching the video, as this has implications for the way emotional states are assessed in research using such material. Due to a lack of appropriate statistical methods, existing approaches use heuristics such as the ``peak-and-end rule'' in order to link the overall ratings with the continuous emotional stimuli (see Section \ref{sec:RDA}).  Such heuristic approaches, however, can produce results that do not accurately capture the summary rating and can be easily over-interpreted, as there is no unbiased formal inference about which time points contribute to the summary rating.  By contrast, our new methodology allows us to identify the crucial affective video scenes -- the basic prerequisite to understanding the emergence of emotional states in this kind of experiment.

The identification of ``influential'' stimuli occurring in a video corresponds to identify corresponding time points $\tau\in(a,b)$. We aim to estimate such time points within the nonparametric model
%%%%%%%%%
\begin{align}\label{eq:IntoM1}
Y = g\big(X(\tau_1),\dots,X(\tau_S)\big) + \varepsilon,
\end{align}
%%%%%%%%%
where $\tau_1,\dots,\tau_S \in (a,b)$ and their number $S \in \mathbb{N}$ are unknown and need to be estimated.  The values $\tau_1,\dots,\tau_S$ are called \emph{points of impact} and provide specific locations at which the functional predictor $X\in L^2([a,b])$ influences the scalar outcome $Y$. In our real data application in Section \ref{sec:RDA}, $Y$ is a binary variable and the functional predictor $X$ is evaluated at two estimated points of impact $\hat\tau_1$ and $\hat\tau_2$; see Figure \ref{sec:RDA}.
%%%%%%%%%%%%%%%%%%%%%%%
\begin{figure}[!htb]
\center
\includegraphics[width=.9\textwidth]{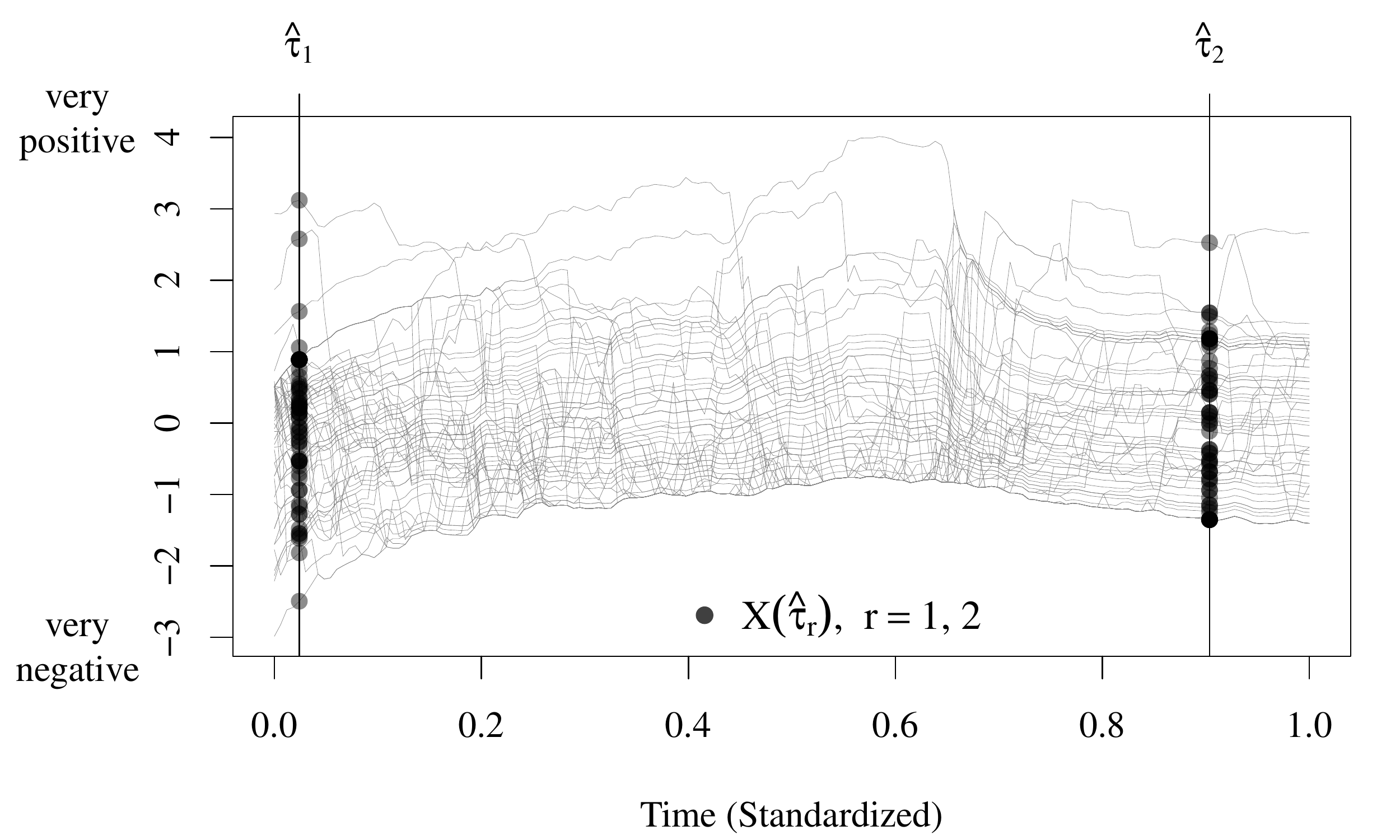}
\caption[]{Continuously self-reported emotion trajectories $\{X(t):0\leq t\leq 1\}$ of $n=65$ participants with two estimated points of impact $\hat\tau_1$ and $\hat\tau_2$; see application in Section \ref{sec:RDA}.}
\label{fig:appl_1}
\end{figure}
%%%%%%%%%%%%%%%%%%%%%%% 

Our method builds upon the work of \cite{KPS2015}, however, we consider the much more challenging case of estimating points of impact within a fully nonparametric function $g$.  A remarkable feature of our method is that identification and estimation of the points of impact $\tau_1,\dots,\tau_S$ neither require knowledge about the nonparametric function $g$ nor an estimate of $g$.  The estimation of the points of impact $\tau_1,\dots,\tau_S$ is thus robust to model misspecifications and free of additional contaminating estimation errors.  This result goes far beyond the special case of a functional linear model as considered by \cite{KPS2015}.

%Our main theoretical result (Theorem \ref{poicon}) on the super-consistency of our points of impact estimation procedure neither requires knowledge about the nonparametric function $g$ nor an estimate of it.  
%The estimation of the points of impact is thus robust to model mis-specifications and free of additional contaminating estimation errors. 
%  A remarkable feature of our method is that identification and estimation of the points of impact $\tau_1,\dots,\tau_S$ neither require knowledge about the possibly nonparametric function $g$ nor an estimate of $g$.  The estimation of the points of impact is thus robust to model mis-specifications and free of additional contaminating estimation errors. 

To the best of our knowledge, the problem of estimating $\tau_1,\dots,\tau_S$ in \eqref{eq:IntoM1} is, so far, only considered by \cite{FHV2010}, who propose to estimate $g$ nonparametrically for any combination of point of impact candidates $t^\ast_1,\dots,t^\ast_{S^\ast}\in\{t_j: t_j=a+j(b-a)/p\,\text{ with }\,j=1,\dots,p\}$ and to select the best model using cross-validation. This brute force method, however, becomes problematic in practice for $S\geq 2$ and large $p$. Furthermore, the nonparametric estimation of $g$ implies that the points of impact $\tau_1,\dots,\tau_S$ can be estimated at most with the non-parametric rate $n^{-2/(4+S)}$, where $n$ denotes the sample size. Here the speed of convergence decreases dramatically for dimensions $S\geq 2$. By contrast, we can estimate the points of impact with a super-consistent convergence rate, that is, faster than the parametric rate $n^{-1/2}$, and our estimation algorithm is applicable in practice for any fixed $S$ and large $p\gg n$.

The super-consistency result for our points of impact estimators is very beneficial for subsequent estimation problems and allows us to estimate the unknown function $g$ as if the points of impact were known. We demonstrate this for a nonparametric model $g$ as well as for the practically relevant case of generalized linear models with linear predictors, that is, $g(X(\tau_1),\dots,X(\tau_S))=g(\alpha+\sum_{r=1}^S\beta_rX(\tau_r))$ with assumed known parametric link function $g$.

% Literature
So far, the purely nonparametric framework is only considered by \cite{FHV2010}. The case of a known $g$ and linear predictor function $\alpha+\sum_{r=1}^S\beta_rX(\tau_r)$ had already been considered by previous studies; however, none of these studies provides a super-consistent estimation of points of impact independent of the model $g$. The term ``impact point'' was coined by \cite{LMcK2009} and \cite{LMcKB2010}.  \cite{LMcK2009} consider a logistic regression framework and \cite{LMcKB2010} consider a linear regression framework.  A point of impact model, where $S=1$ is assumed known, has also been studied in survival analysis for the Cox regression model \citep{Z2012}.  \cite{KPS2015} allow for an unknown number $S\geq 0$ of points of impact augmenting the functional linear regression model.  \cite{LRR2019} propose an improved estimation algorithm for the latter work. \cite{GP2014} consider a linear regression framework with multiple points of impact postulating the existence of some consistent estimation procedure.  \cite{BBC2017} consider a linear regression framework and propose a reproducing kernel Hilbert space approach.
% Prediction:
Selecting sparse features from functional data $X$ is also useful for clustering. For instance, \cite{FV2017} propose a method for sparse clustering of functional data. In a slightly different context, \cite{PJF2016} focus on selecting predictive subdomains of the functional data.
% Lindqist JASA papers
Related to this paper is also the work of \cite{L2012} and \cite{SL2014}. \cite{L2012} extends structural equation models to the functional data analysis setting and uses his methodology to select significantly impacting time-intervals in  functional magnetic resonance imaging (fMRI) data. \cite{SL2014} propose a mixed effects model which facilitates selecting significant impact regions in fMRI data by controlling for systematic measurement errors.
% FDA-General:
% Readers with a general interest in functional data analysis are referred to the textbooks of \cite{RS05_book}, \cite{FV06_book}, \cite{HK12_book}, \cite{HR15_book}, and \cite{KR17_book}.

% Paper-Outline
The rest of this work is structured as follows. Section \ref{sec:POI} considers the estimation of the points of impact $\tau_r$ and their number $S$ independent of the model $g$. Subsequent estimation of the function $g$ is discussed in Section \ref{sec:Inference}. The simulation study and the real data application are in Sections \ref{sec:SIM} and \ref{sec:RDA}. All proofs and additional simulation results can be found in the appendixes of the supplementary paper supporting this article \citep{Suppl2019}.  The \textsf{R}-package \texttt{fdapoi} and the \textsf{R}-scripts for reproducing our main empirical results are also provided as part of the online supplementary material \citep{Suppl2019}.

%%%%%%%%%%%%%%%%%%%%%%%%%%%%%%%%%%%%%%%%%%%%%%%%%%%%%%%
\section{Estimating points of impact}\label{sec:POI}
%%%%%%%%%%%%%%%%%%%%%%%%%%%%%%%%%%%%%%%%%%%%%%%%%%%%%%%
In the following we present our theoretical framework (Section \ref{ssec:theoreticalframe}), the estimation algorithm (Section \ref{SubSec:Estim}) and our asymptotic results (Section \ref{SubSec:Asymp}). The section concludes with a discussion of possibilities to generalize our theoretical results (Section \ref{ssec:general}).  

%%%%%%%%%%%%%%%%%%%%%%%%%%%%%%%%%%%%%%%%%%%%%%%%%%%%%%%%%%%%%%%
\subsection{Theoretical framework}\label{ssec:theoreticalframe}
%%%%%%%%%%%%%%%%%%%%%%%%%%%%%%%%%%%%%%%%%%%%%%%%%%%%%%%%%%%%%%%
In this section we present our theoretical framework for estimating the points of impact $\tau_1,\dots,\tau_S$ without knowing or (pre-)estimating the possibly nonparametric model function $g$. The identification of points of impact constitutes a particular variable selection problem. Since we consider the case where the functional predictor is observed over a densely discretized grid, one might be tempted to apply multivariate variable selection methods like Lasso or related procedures. Note, however, that the high correlation of the predictor at neighboring discretization points violates the basic requirements of these multivariate variable selection procedures. 

Suppose we are given an i.i.d.~sample of data $(X_i, Y_i)$, $i=1,\dots, n$, where $X_i=\{X_i(t),t\in [a,b]\}$ is a stochastic process with $\mathbb{E}(\int_a^b X_i(t)^2 \, dt)<\infty$, $[a,b]$ is a compact subset of $\mathbb{R}$ and $Y_i$ is a real valued random variable. It is assumed that the relationship between $Y_i$ and the functional predictor $X_i$ can be modeled as
%%%%%%%%%
\begin{align}\label{eq:whnp}
Y_i = g\big(X_i(\tau_1),\dots,X_i(\tau_S)\big) + \varepsilon_i,
\end{align}
%%%%%%%%%
where $\varepsilon_i$ denotes the statistical error term with $\mathbb{E}(\varepsilon_i|X_i(t)) = 0$ for all $t\in [a,b]$. The number $0\leq S<\infty$ and the points of impact $\tau_1, \dots,\tau_S$ are unknown and have to be estimated from the data -- without knowing the true model function $g$. The points of impact $\tau_1, \dots,\tau_S$ indicate the locations at which the functional values $X_i(\tau_1), \dots, X_i(\tau_S)$ have a specific influence on $Y_i$. Without loss of generality, we consider centered random functions $X_i$ with $\E(X_i(t))=0$ for all $t\in[a,b]$.

Surprisingly, the unknown function $g$ has to fulfill only some very mild regularity conditions and does not have to be estimated in order to estimate the points of impact $\tau_1,\dots,\tau_S$ (see Theorem \ref{thm:NP1}). Estimating points of impact, however, necessarily depends on the structure of $X_i$. Motivated by our application we consider stochastic processes with rough sample paths such as (fractional) Brownian motion, Ornstein-Uhlenbeck processes, Poisson processes, etc. These processes also have important applications in fields such as finance, chemometrics, econometrics, and the analysis of gene expression data \citep{LR1991,LWCM2007,DS2006,RHN2013}. Common to these processes are covariance functions $\sigma(s,t)=\E(X_i(s)X_i(t))$ which are two times continuously differentiable for all points $s\neq t$, but not two times differentiable at the diagonal $s=t$. The following assumption on the covariance function of $X_i$ describes a very large class of such stochastic processes and allows us to derive precise theoretical results:
%%%%%%%%%%%%%%%%%%%%%%%%%%%%%%%%%%
\begin{assumption}\label{assum1}
For some open subset $\Omega\subset\mathbb{R}^3$ with $[a,b]^2\times [0,b-a]\subset \Omega$,
there exists a twice continuously differentiable function $\omega:\Omega \rightarrow
\mathbb{R}$ as well as some $0<\kappa<2$ such that for all $s,t\in [a,b]$
\begin{align}
\sigma(s,t)=\omega(s,t,|s-t|^\kappa).\label{deromega}
\end{align}
Moreover, $0<\inf_{t\in [a,b]}c(t)$, where $c(t):=-\frac{\partial}{\partial z}\omega(t,t,z)|_{z=0}$.
\end{assumption}
%%%%%%%%%%%%%%%%%%%%%%%%%%%%%%%%%%
The parameter $\kappa$ quantifies the degree of smoothness of the covariance function $\sigma$ at the diagonal. While a twice continuously differentiable covariance function will satisfy \eqref{deromega} with $\kappa=2$, small values $0<\kappa<2$ indicate a process with non-smooth sample paths.

Assumption \ref{assum1} covers several important classes of stochastic processes. Recall, for instance, that the class of self-similar (not necessarily centered) processes $X_i=\{X_i(t):t\geq 0\}$ has the property that $X_i(c_1t)=c_1^HX_i(t)$ for any constant $c_1>0$ and some exponent $H>0$. It is then well known that the covariance function of any such process $X_i$ with stationary increments and $0<\mathbb{E}(X_i(1)^2)<\infty$ satisfies
%%%%%%%%%%%%%%%%%%%%
\begin{equation*}
\sigma(s,t)= \omega(s,t,|s-t|^{2H})=(s^{2H}+t^{2H}-|s-t|^{2H})\,c_2
\end{equation*}
%%%%%%%%%%%%%%%%%%%
for some constant $c_2>0$; see Theorem 1.2 in \cite{EM2000}. If $0<H<1$ such a process respects Assumption \ref{assum1} with $\kappa =2H$ and $c(t) = c_2$. A prominent example of a self-similar process is the fractional Brownian motion.

Another class of processes is given when $X_i=\{X_i(t):t\geq 0\}$ is a second order process with stationary and independent increments. In this case it is easy to show that
$\sigma(s,t) =\allowbreak \omega(s,t,|s-t|)=\allowbreak (s+t-|s-t|)\, c_3$
for some constant $c_3>0$. The Assumption \ref{assum1} then holds with $\kappa=1$ and $c(t)=c_3$.
%%%%%%%%%%%%%
% Proof: theorem 4 hier: http://www.math.uah.edu/stat/processes/Increments.html
% Siehe Lit-Ordner: Processes with Stationary, Independent Increments.pdf
%%%%%%%%%%%%%
The latter conditions on $X_i$ are, for instance, satisfied by second order L\'{e}vy processes which include important processes such as Poisson processes, compound Poisson processes, or jump-diffusion processes.

A third important class of processes satisfying Assumption \ref{assum1} are those with a Mat\'{e}rn covariance function. For this class of processes the covariance function depends only on the distance between $s$ and $t$ through
%%%%%%%%%
$$\sigma(s,t) = \omega_\nu(s,t,|s-t|) = \frac{\pi \phi}{2^{\nu-1}\Gamma(\nu+1/2)\alpha^{2\nu}}(\alpha|s-t|)^\nu K_\nu\big(\alpha|s-t|\big),$$
%%%%%%%%%
where $K_\nu$ is the modified Bessel function of the second kind, and $\rho$, $\nu$ and $\alpha$ are non-negative parameters of the covariance. It is known that this covariance function is $2m$ times differentiable if and only if $\nu > m$  \citep[cf.][Ch.~2.7, p.~32]{S1999_book}. Assumption \ref{assum1} is then satisfied for $\nu <1$. For the special case where $\nu = 0.5$ one may derive the long term covariance function of an Ornstein-Uhlenbeck process which is given as $\sigma(s,t) =\allowbreak\omega(s,t,|s-t|)= \allowbreak  0.5\,\exp(- \theta |s-t|)\sigma_{OU}^2/\theta$, for some parameter $\theta>0$ and $\sigma_{OU}>0$. Assumption \ref{assum1} is then covered with $\kappa=1$ and $c(t) = 0.5\,\sigma_{OU}^2$.

The remarkable result that identification and estimation of the points of impact $\tau_1,\dots,\allowbreak\tau_S$ requires neither knowledge about the possibly nonparametric function $g$ nor an estimate of $g$ is based on the following theorem.
%%%%%%%%%%%%%%%%%%%%%%%%%%%%%%%%%%
\begin{theorem}\label{thm:NP1}
Let $X_i$ be a Gaussian process. For any function $g(x_1,\dots,x_S)$ such that for all $r=1,\dots,S$ the partial derivative $\partial g(x_1,\dots,x_S)/\partial x_r$ is continuous almost everywhere and $0<|\E(\frac{\partial}{\partial x_r} g(X_i(\tau_1),\dots, X_i(\tau_S)))|<\infty$, we define $\vartheta_r=\E\big(\frac{\partial}{\partial x_r} g(X_i(\tau_1),\dots, X_i(\tau_S))\big)$. Then the equation
$\E\big(X_i(s)Y_i\big)= \sum_{r=1}^S \vartheta_r \sigma(s,\tau_r)$ holds for all $s\in[a,b]$.
%%%%%%%%%%%%%%%%%%%%%%%%%%%%%%%%%%

% Let $X_i$ be a Gaussian process. For any function $g(x_1,\dots,x_s)$ such that for all $r=1,\dots, S$, $\partial g /\partial x_r$ exists almost everywhere such that
%$0<|\vartheta_r|<\infty$, where $\vartheta_r=\E(\frac{\partial}{\partial x_r} g(X_i(\tau_1),\dots, X_i(\tau_S)))$ then

% Let $X_i$ be a Gaussian process. For any function $g(x_1,\dots,x_s)$ such that for all $r=1,\dots, S$, $\partial g /\partial x_r$ exists almost everywhere with $0<|\E(\frac{\partial}{\partial x_r} g(X_i(\tau_1),\dots, X_i(\tau_S)))|<\infty$ let $\vartheta_r=\E(\frac{\partial}{\partial x_r} g(X_i(\tau_1),\dots, X_i(\tau_S)))$. Then
% \begin{equation*}
% \E(X_i(s)Y_i)= \sum_{r=1}^S \vartheta_r \sigma(s,\tau_r)\quad\text{for all}\quad s \in [a,b].
% \end{equation*}
\end{theorem}
%%%%%%%%%%%%%%%%%%%%%%%%%%%%%%%%%%

% Theorem \ref{thm:NP1} simplifies our proofs, but the estimation of the points of impact may also be possible under more general situations; see also remark (iv) on Theorem \ref{poicon}.

Theorem \ref{thm:NP1} allows to decompose the cross-covariance $\E(X_i(s)Y_i)$ into coefficients $\vartheta_r$, which depend on the unknown function $g$, and the covariance function $\sigma$, which only depends on $X_i$. Our estimation strategy for the points of impact $\tau_r$ works for unknown $\vartheta_r$ with $0<|\vartheta_r|<\infty$. The latter imposes only mild regularity assumptions on $g$ and is fulfilled, for instance, by any nonparametric single-index model, $g(X_i(\tau_1),\dots, X_i(\tau_S))\equiv g(\eta_i)$ with $\eta_i=\alpha +\sum_{r=1}^S\beta_rX_i(\tau_r)$, where $0<|\E(g'(\eta_i))|<\infty$. Of course, the class of possible functions $g$ defined by Theorem \ref{thm:NP1} also contains much more complex cases than single-index models.

The intention of our estimator for the points of impact $\tau_r$ is to exploit the covariance structure of  processes described by Assumption \ref{assum1}. Covariance functions $\sigma(s,t)$ satisfying Assumption \ref{assum1} are obviously not two times differentiable at the diagonal $s=t$, but are two times differentiable for $s\neq t$.  Using Theorem \ref{thm:NP1} in conjunction with Assumption \ref{assum1} allows us to uniquely identify the locations of the points of impact from the cross-covariance $\E(X_i(s)Y_i)$.  Let us make this more precise by defining
%%%%%%%%%%%%%%%%%%%
\begin{equation*}
%f_{XY}(s):=\E(X_i(s)Y_i) = c_0\E(X_i(s)\eta_i)=c_0\sum_{r=1}^S \beta_r \sigma(s, \tau_r).
f_{XY}(s):=\E(X_i(s)Y_i) = \sum_{r=1}^S \vartheta_r \sigma(s, \tau_r)\quad\text{for}\quad s\in[a,b].
\end{equation*}
%%%%%%%%%%%%%%%%%%%
Since $\sigma(s,t)$ is not two times differentiable at $s=t$, the cross-covariance $f_{XY}(s)$ will not be two times differentiable at $s=\tau_r$, for all $r=1,\dots, S$, resulting in kink-like features at $\tau_r$ as depicted in the upper plot of Figure \ref{fig:cor}.
%%%%%%%%%%%%%%%%%%%%%%% 
\begin{figure}[!thb]
\centering
\includegraphics[width=0.9\textwidth]{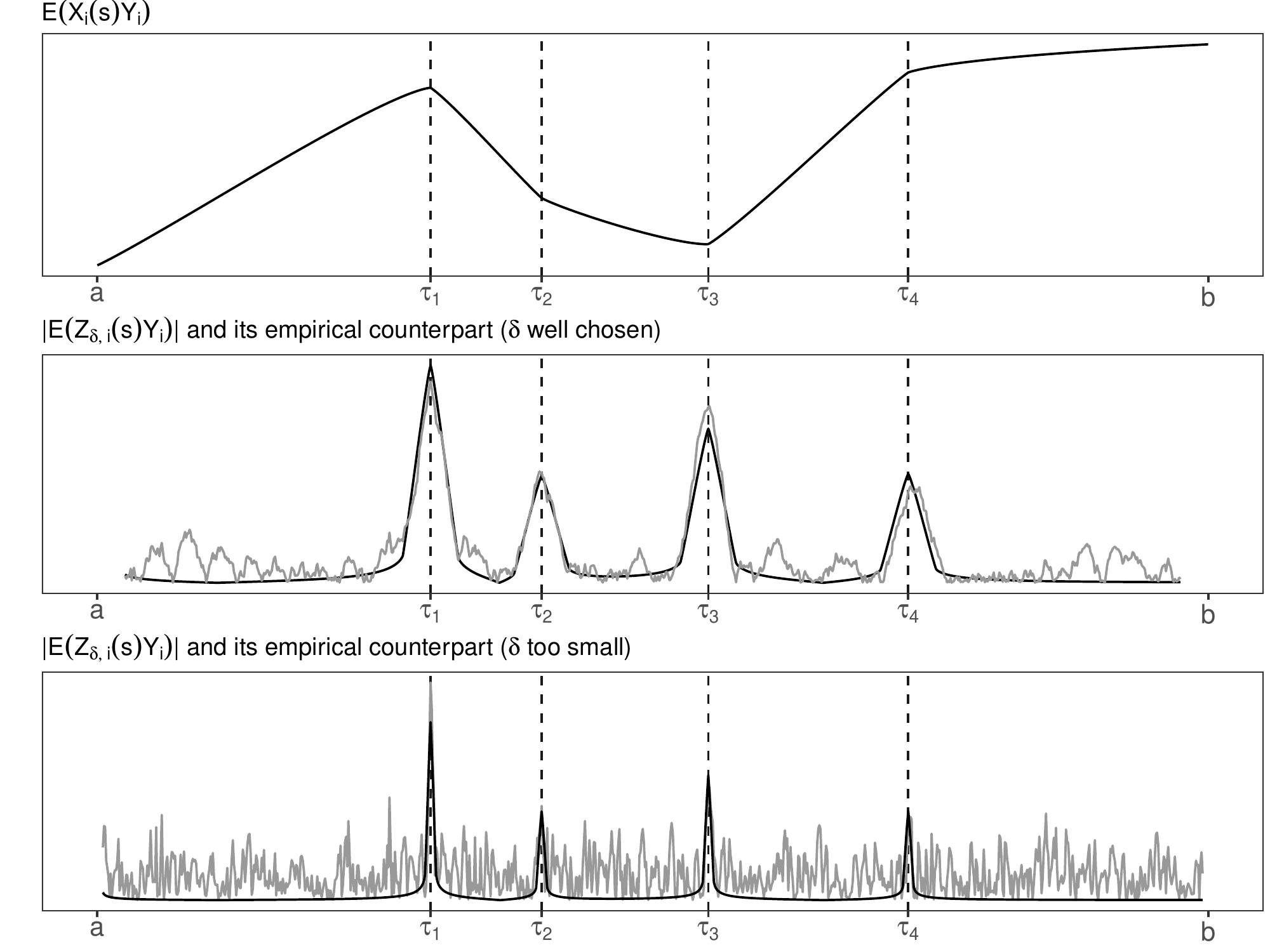} 
\caption[Illustrating estimation of points of impact]{The upper panel shows $\E(X_i(s)Y_i)$ as a function of $s$ with four kink-like features at the points of impact (dashed vertical lines). The two lower panels show $|\E(Z_{\delta,i}(s)Y_i)|$, $a+\delta\leq s\leq b-\delta$, (black solid line) and its empirical counterparts (gray solid line) for two different values of $\delta>0$.}
\label{fig:cor}
\end{figure}
%%%%%%%%%%%%%%%%%%%%%%

A natural strategy for estimating $\tau_r$ is to detect these kinks by considering the following modified central difference approximation of the second derivative of $f$ at a point $s\in [a-\delta, b-\delta]$ for some $\delta>0$:
%%%%%%%%%%%%%%%%%%%%%%%%%%%%%%%%%%
\begin{equation}\label{eq:2nd_diff}
f_{XY}(s)-\frac{1}{2}(f_{XY}(s+\delta)+ f_{XY}(s-\delta)).% \approx  -\frac{1}{2}\delta^2 f_{XY}''(s).
\end{equation}
%%%%%%%%%%%%%%%%%%%%%%%%%%%%%%%%%%
By defining the auxiliary process
%%%%%%%%%%%%%%%%%%
\begin{equation*}
Z_{\delta,i}(s):= X_i(s) - \frac{1}{2}(X_i(s-\delta) + X_i(s+\delta)) \quad\text{for}\quad s \in [a+\delta, b-\delta],
\end{equation*}
%%%%%%%%%%%%%%%%%%
we have the following equivalent moment expression for \eqref{eq:2nd_diff}:
%%%%%%%%%%%%%%%%%%
\begin{equation}\label{eq:va2nd}
\mathbb{E}(Z_{\delta,i}(s)Y_i)=f_{XY}(s) - \frac{1}{2}(f_{XY}(s+\delta)+ f_{XY}(s-\delta)).
\end{equation}
%%%%%%%%%%%%%%%%%%

At $s=\tau_r$, expression \eqref{eq:va2nd} will decline more slowly to zero as $\delta\to 0$ than for $s\neq \tau_r$, $r=1,\dots,S$.  For suitable values of $\delta$, the points of impact $\tau_r$ can then be estimated using the local extrema of the empirical counterpart of $|\mathbb{E}(Z_{\delta,i}(s)Y_i)|$ (see middle panel of Figure \ref{fig:cor}). \\
More precisely, Theorem \ref{thm:NP1} together with Proposition \ref{thmcharX} and Lemma \ref{lem3} in Appendix \ref{app:poiEST} of the supplementary paper \cite{Suppl2019} imply that as $\delta\to 0$
%%%%%%%%%%%%%%%%%%%%%%%%
\begin{equation*}
\begin{array}{ll}
\mathbb{E}(Z_{\delta,i}(s)Y_i) = \vartheta_r c(\tau_r)\delta^\kappa + o(\delta^\kappa) &\text{if } s\in\{\tau_1,\dots,\tau_S\}\\
% since two-continuous derivatives at off-diagonal points
\mathbb{E}(Z_{\delta,i}(s)Y_i) = O(\delta^2) &\text{if } s\notin\{\tau_1,\dots,\tau_S\},
\end{array}
\end{equation*}
where $0<\kappa<2$ and $c(\cdot)>0$ are as defined in Assumption \ref{assum1}.
%%%%%%%%%%%%%%%%%

Of course, $\mathbb{E}(Z_{\delta,i}(s)Y_i)$ is not known and we have to rely on $n^{-1}\sum_{i=1}^n Z_{\delta,i}(s)Y_i$ as its estimate. Under our setting we will have that the variance $\V(Z_{\delta,i}(s)Y_i) = O(\delta^\kappa)$ which implies
%%%%%%%%
$$\frac{1}{n}\sum_{i=1}^n Z_{\delta,i}(s)Y_i - \mathbb{E}(Z_{\delta,i}(s)Y_i) = O_P\left(\sqrt{\delta^\kappa/n}\right).$$
%%%%%%%%
Consequently, the identification of points of impact requires a sensible choice of $\delta$. For too small $\delta$-values (e.g., $\delta^\kappa\sim n^{-1}$) the estimation noise will overlay the signal; this situation is depicted in the bottom plot of Figure \ref{fig:cor}. For too large $\delta$-values, however, it will not be possible to distinguish between neighboring points of impact.

\smallskip

\noindent\textbf{Remark:} \textit{Even if the covariance function $\sigma(s,t)$ does not satisfy Assumption \ref{assum1}, the points of impact $\tau_r$ may still be estimated using the local extrema of $\E(Z_{\delta,i}(s)Y_i)$. Suppose, for instance, there exists a $m\geq 2$ times differentiable function $\widetilde{\sigma}: \mathbb{R} \to \mathbb{R}$ such that $\sigma(s,t) = \widetilde{\sigma}(|s-t|)$, where $\widetilde{\sigma}(|s-t|)$ decays fast enough, as $|s-t|$ increases, such that $X_i(s)$ is essentially uncorrelated with $X_i(\tau_r)$ for $|\tau_r-s|\gg 0$. If $|\widetilde{\sigma}''(0)| > |\widetilde{\sigma}''(|s-t|)|$, for $s\neq t$, and $\min_{r \neq k}|\tau_r - \tau_k|$ is large enough, then all points of impact might be identified as local extrema of $\E(Z_{\delta,i}(s)Y_i)$.}

%%%%%%%%%%%%%%%%%%%%%%%%%%%%%%%%%%%%%%%%%%%%%%%%%%%%%%%%%
\subsection{Estimation algorithm}\label{SubSec:Estim}
%%%%%%%%%%%%%%%%%%%%%%%%%%%%%%%%%%%%%%%%%%%%%%%%%%%%%%%%%
In the following we consider the case where each $X_i$ has been observed over $p$ equidistant points $t_j=a+(j-1)(b-a)/(p-1)$, $j=1, \dots, p$, where $p$ may be much larger than $n$. Estimators for the points of impact $\tau_r$ are determined by sufficiently large local maxima of $\left|n^{-1}\sum_{i=1}^n Z_{\delta,i}(t_j)Y_i\right|$. %This strategy is similar to \cite{KPS2015}; however, in contrast to \cite{KPS2015}, we avoid a direct computation of $Z_{\delta,i}(t_j)$ for every $t_j$ and propose the following computationally more efficient estimation procedure:\\

\begin{algo}{\bf(Estimating points of impact)}\label{Algo:1}\
%%%%%%%%%%%%%%%%%%%%%%%%%%%%%%%%%%%%%%%
%%%%%%%%%%%%%%%%%%%%%%%%%%%%%%%%%%%%%%%
\spacingset{1.1}
\begin{enumerate}
\item[{\bf 1.}] {\bf Calculate}
$\widehat{f}_{XY}(t_j) := \frac{1}{n}\sum_{i=1}^n X_i(t_j)Y_i, \quad\text{for each}\quad j=1,\dots, p$
%%%%%%%%%%%%%%%
\item[{\bf 2.}] {\bf Choose} $\delta>0$ such that there exists some $k_\delta\in\mathbb{N}$ with $1\leq k_\delta<(p-1)/2$
and $\delta=k_\delta(b-a)/(p-1)$.
%%%%%%%%%%%%%%
\item[{\bf 3.}] {\bf Calculate}
$\widehat{f}_{ZY}(t_j):=\widehat{f}_{XY}(t_j)-\frac{1}{2}(\widehat{f}_{XY}(t_j-\delta) + \widehat{f}_{XY}(t_j+\delta))$, for all $j\in {\cal J}_{\delta}$, where ${\cal J}_{\delta}:=\{k_\delta+1,\ldots,p-k_\delta\}$
%%%%%%%%%%%%%%
\item[{\bf 4.}] {\bf Repeat:}
\begin{itemize}
\item[]{\bf Initiate} the repetition by setting $l=1$.
\item[]{\bf Estimate} the $l$th point of impact candidate as $\widehat{\tau}_l=\underset{t_j: \,j\in {\cal J}_{\delta}}{\arg\max} |\widehat{f}_{ZY}(t_j)|.$
%%%%%%%%%
\item[]{\bf Update} ${\cal J}_{\delta}$ by eliminating all points in ${\cal J}_{\delta}$ in an interval of size $\sqrt{\delta}$ around $\widehat{\tau}_l$.\\
\phantom{{\bf Update}} Set $l=l+1$.
\item[]{\bf End} repetition if ${\cal J}_{\delta} = \emptyset$.
\end{itemize}
%%%%%%%%%%%%%%
\end{enumerate}
\spacingset{1.5}
%%%%%%%%%%%%%%%%%%%%%%%%%%%%%%%%%%%%%%%%%%
The procedure will result in estimates $\widehat{\tau}_1, \widehat{\tau}_2, \dots, \widehat{\tau}_{M_\delta}$, where $M_\delta<\infty$ denotes the maximum number of possible repetitions. The estimator of $S$ is
\begin{equation*}
\widehat S =\min\left\{l\in\mathbb{N}_0: \left|\frac{\frac{1}{n}\sum_{i=1}^n Z_{\delta,i}(\widehat{\tau}_{l+1})Y_i}{(\frac{1}{n}\sum_{i=1}^n Z_{\delta,i}(\widehat{\tau}_{l+1})^2)^{1/2}}\right|<\lambda\right\}\quad
\text{for some threshold }\;\lambda>0.
\end{equation*}
%%%%%%%%%%%%%%%%%
\end{algo}
%%%%%%%%%%%%%%%%%%%%%%%%%%%%%%%%%%%%%%%%%%

\smallskip

\noindent An asymptotically valid choice of the threshold $\lambda$ is presented in Theorem \ref{poicon} and a practical implementation of $\lambda$ is discussed below of Theorem \ref{poicon}. 

\smallskip

\noindent\textbf{Remark:} \textit{
This estimation algorithm is made for the case of densely observed functional data.  In practice this means functional data that are sampled at a high frequency such as in our real data example (Section \ref{sec:RDA}).  Unfortunately, we do not see a simple way to  generalize  our method to the case of irregularly or sparsely sampled functional data. Such a generalization would require a very different approach based on nonparametric smoothing procedures.}

% We note that the above estimation procedure is similar to, but deviates from the estimation procedure in \cite{KPS2015}.
% In contrast to \cite{KPS2015}, we avoid calculating
% = \frac{1}{n}\sum_{i=1}^n Z_{\delta,i}(t_j)Y_i
% the variation of the second finite difference directly to $\frac{1}{n}\sum_{i=1}^n X_i(t_j)Y_i$ instead of first determining $Z_{\delta,i}(t_j)$ and then calculating $\frac{1}{n}\sum_{i=1}^n Z
% _{\delta,i}(t_j)Y_i$ as it was original proposed in \cite{KPS2015}, may drastically decrease the computational costs of the estimation procedure for larger sample sizes.

%%%%%%%%%%%%%%%%%%%%%%%%%%%%%%%%%%
\subsection{Asymptotic results}\label{SubSec:Asymp}
%%%%%%%%%%%%%%%%%%%%%%%%%%%%%%%%%%
In this section, we consider asymptotics as $n\rightarrow \infty$ with $p\equiv p_n \geq L n^{1/\kappa}$ for some constant $(b-a)/2<L<\infty$. We introduce the following assumption:
%%%%%%%%%%%%%%%%%%%%%%%%%%%%%%%%%%%%%%%%%%%%%%
\begin{assumption} \label{assum2}\
\begin{itemize}
\item[a)] $X_1,\dots,X_n$ are i.i.d.~random functions distributed according to $X$. The process $X$ is
Gaussian with covariance function $\sigma(s,t)$.
%%%
\item[b)] There exists a $0<\sigma_{|y|}<\infty$ such that for each $m=1, 2, \dots$ we have \\
$\E(|Y_i|^{2m}) \leq 2^{m-1} m! \sigma_{|y|}^{2m}$.
\end{itemize}
\end{assumption}
%%%%%%%%%%%%%%%%%%%%%%%%%%%%%%%%%%%%%%%%%%%%%% 
The moment condition in b) is obviously fulfilled for bounded $Y_i$. For instance, in the functional logistic regression we have that $\E(|Y_i|^m) \leq 1$ for all $m =1,2,\dots$. Condition b) also holds for any centered sub-Gaussian $Y_i$, where a centering of $Y_i$ can always be achieved by substituting $g(X_i(\tau_1),\dots,X_i(\tau_S)) + \E(g(X_i(\tau_1),\dots,X_i(\tau_S)))$ for $g(X_i(\tau_1),\dots,X_i(\tau_S))$ in Model \eqref{eq:whnp}. If $X_i$ satisfies condition a), then condition b) in particular holds if the errors $\varepsilon_i$ are sub-Gaussian and $g$ is differentiable with bounded partial derivatives.
%%%%%%%%%%%%%%%%%%%%%%%%%%%%%

The following result shows consistency of our estimators for the points of impact $\widehat{\tau}_r$ and the estimator $\widehat{S}$:
%%%%%%%%%%%%%%%%%%%%%%%%%%%%%%%%%%%%%%%%%%%%%%
\begin{theorem} \label{poicon}
Under Assumptions \ref{assum1}, \ref{assum2}, and the assumptions of Theorem~\ref{thm:NP1}, let $\delta\equiv\delta_n\rightarrow 0$ as $n\rightarrow\infty$ such that $n\delta^\kappa/|\log \delta|\rightarrow \infty$ and  $\delta^\kappa/n^{-\kappa+1}\rightarrow 0$. We then obtain that
%%%%%%%%%%%%%%%%%
\begin{equation}\label{thm3eq1}
\max_{\phantom{\widehat{S}}r=1,\ldots,\widehat{S}}\min_{\phantom{\widehat{S}}s=1,\ldots,S\phantom{\widehat{S}}}|\widehat{\tau}_r-\tau_s|\ =\ O_P( n^{-1/\kappa}).
\end{equation}
%%%%%%%%%%%%%%%%%
Moreover, there exists a constant $0<D<\infty$ such that when  Algorithm \ref{Algo:1} is applied with threshold
$$\lambda\equiv \lambda_n=A\sqrt{ \frac{\sigma_{|y|}^2}{n} \log \left(\frac{b-a}{\delta}\right)}, \quad A>D, \text{\quad and \quad$\delta^2=O(n^{-1})$, \quad then}
$$
%%%%%%%%%%%%%%%%%%%%%%%%%%%%%%%%
\begin{equation} \label{thm3eq4}
P(\widehat{S}=S)\ \rightarrow\ 1 \quad \text{as}\quad  n\rightarrow \infty.
\end{equation}
%%%%%%%%%%%%%%%%%%%%%%%%%%%%%%%
\end{theorem}

Note that the rates of convergence in \eqref{thm3eq1} are super-consistent, since $0<\kappa<2$. For instance, for Ornstein-Uhlenbeck processs or Brownian motions we have $\kappa=1$, such that $\max_{r=1,\ldots,\widehat{S}}\min_{s=1,\ldots,S}|\widehat{\tau}_r-\tau_s|\ =\ O_P( n^{-1})$.  

In principle, arbitrarily fast rates of convergence can be achieved for $\kappa$-values close to zero, because small $\kappa$-values correspond to rough processes, $X_i$.  Roughness means that the process has strong local variations also within small intervals $[\tau_r-\epsilon,\tau_r+\epsilon]$, $\epsilon>0$, which facilitates differentiating a point of impact $\tau_r$, $r=1\dots,S$, from the neighboring points $t\in[\tau_r-\epsilon,\tau_r+\epsilon]$. By contrast, for smooth processes (large $\kappa$-values) all values of $X_i(t)$ with $t\in[\tau_r-\epsilon,\tau_r+\epsilon]$ will be almost identical which makes it hard to identify the correct point of impact $\tau_r$.

A practical and asymptotically valid threshold specification which performed well in our simulation studies is given by $\lambda = A ((\E(Y_i^4))^{1/2} \log\big((b-a)/\delta\big)/n)^{1/2}$, where $\E(Y_i^4)$ is estimated by $\widehat{\E}(Y_i^4) = n^{-1}\sum_{i=1}^n Y_i^4$ and $A=\sqrt{2 \sqrt{3}}$. This value is motivated by an argument using the central limit theorem in the derivations of the threshold for Theorem~\ref{poicon}. See the remark after the proof of Lemma~\ref{lem0} in Appendix~\ref{app:poiEST} for additional information.

The super-consistency result in Theorem \ref{poicon} is very general and does not require knowledge of $g$ or a pre-estimate of $g$; only a set of mild and verifiably assumptions on $g$ is postulated. Therefore, we expect that the theorem will be found useful for deriving inferential results for a broad variety of different models $g$. In the following we demonstrate the usefulness of Theorem \ref{poicon} for deriving inferential results for nonparametric models and parametric generalized linear models. Note that the related Corollary 1 in \cite{FHV2010} requires the simultaneous estimation of the nonparametric model function $g$ and the points of impact. This approach results in substantially slower nonparametric convergence rates and limits the applicability of their result considerably.

% The super-consistency result for our points of impact estimators is very beneficial in possible subsequent estimation problems. It allows, for instance, to estimate the unknown function $g$ as if the points of impact were known. We demonstrate this for a purely nonparametric framework as well as for the practically relevant case of generalized linear models with linear predictors, that is, $g(X(\tau_1),\dots,X(\tau_S))=g(\alpha+\sum_{r=1}^S\beta_rX(\tau_r))$ with known parametric link function $g$.

%%%%%%%%%%%%%%%%%%%%%%%%%%%%%%%%%%%%%%%%%%%%%%%%%%%%%%%%
\subsection{Generalizations}\label{ssec:general}
%%%%%%%%%%%%%%%%%%%%%%%%%%%%%%%%%%%%%%%%%%%%%%%%%%%%%%%%
The above theoretical assumptions provide a tractable setup that will be used also in the remaining parts of the paper.  In this subsection, however, we show that the Gaussian assumption of Theorem \ref{thm:NP1} and Theorem \ref{poicon} can be relaxed and that our approach for identifying and estimating the points of impact may also work for a large class of non-Gaussian processes (Section \ref{ssec:General_1}).  Moreover, we outline how our estimation procedure can be adapted to a more general version of the covariance Assumption \ref{assum1} (Section \ref{ssec:General_2}).

%%%%%%%%%%%%%%%%%%%%%%%%%%%%%%%%%%%%%%%%%%%
\subsubsection{Non-Gaussian processes}\label{ssec:General_1}
%%%%%%%%%%%%%%%%%%%%%%%%%%%%%%%%%%%%%%%%%%%
To generalize Theorem \ref{thm:NP1} one can build upon the framework of elliptical processes which includes the case of non-Gaussian, heavy-tailed distributions.  That is, one can consider processes $X_i$ that depend on some latent random variable $V_i$ such that the conditional distribution of $X_i$ given $V_i=v$ is Gaussian.  However, the (unconditional) distribution of $X_i$ then additionally depends on the distribution of $V_i$ and may be far from Gaussian.

Our conditions A and B in Appendix \ref{app:poiID-2} define a general framework for such non-Gaussian processes $X_i$ and Proposition \ref{idgeneral} in Appendix \ref{app:poiID-2} generalizes Theorem \ref{thm:NP1} for this general framework.  Here in this subsection, however, we focus on the arguably most important special case of our general framework -- namely, the case of elliptically distributed processes.  Elliptical distributions include the special case of a Gaussian distribution as considered in Theorem \ref{thm:NP1}, but also many important non-Gaussian distributions such as the t-distribution, the Laplace distribution, and the logistic distribution \cite[see, for instance,][]{BBT2014}.

Let $X_i$ be a (centered) elliptical process, that is, let $X_i(t)\overset{d}{=}V_iX_i^*(t)$, $t\in[a,b]$, where $V_i>0$ is a strictly positive real-valued random variable, $X_i^*$ is a zero mean Gaussian process with covariance function $\sigma^*(s,t)$, and where $V_i$ and $X_i^*$ are independent of each other.  Moreover, let the error term $\varepsilon_i$ in \eqref{eq:whnp} be independent of $V_i$ and $X_i$ and let $\V(V_i)<\infty$.  Then the elliptically distributed random function $X_i$ fulfills our conditions A and B in Appendix \ref{app:poiID-2} and it follows by Proposition \ref{idgeneral} in Appendix \ref{app:poiID-2} that 
\begin{equation*}
\E\big(X_{i}(s)Y_i\big)=\sum_{r=1}^S \sigma^*(s,\tau_r) \E\left( V_i^2  \vartheta_r(V_i)\right)=
\sum_{r=1}^S \sigma(s,\tau_r) \frac{\E\left(V_i^2\vartheta_r(V_i)\right)}{\V(V_i)},
\end{equation*}
where $\vartheta_r(V_i)=
\E\big(\frac{\partial}{\partial x_r}g(X_i(\tau_1),\dots,X_i(\tau_S))\big| V_i\big)$ and $\sigma(s,t)=\V(V_i)\sigma^*(s,t)$ is the covariance function of the elliptically distributed process $X_i$.  As in the case of Theorem \ref{thm:NP1}, the above result allows us to decompose the cross-covariance $\E(X_i(s)Y_i)$ into a scaling coefficient $\E\left(V_i^2\vartheta_r(V_i)\right)/\V(V_i)$ which depends on the unknown function $g$ (via $\vartheta_r$) and the covariance function $\sigma(s,\tau_r)$ which only depends on $X_i$.  This result holds for elliptically distributed $X_i$ and requires only mild regularity assumptions on $g$ which are essentially equivalent to those imposed by Theorem \ref{thm:NP1}; see conditions A and B in Appendix \ref{app:poiID-2}.

As in the preceding section, the identification of the points of impact relies only on the structural covariance Assumption \ref{assum1} which holds for rough -- Gaussian or non-Gaussian -- processes $X_i$.  Since $\sigma(s,t)=\V(V_i)\sigma^*(s,t)$, the requirements of Assumption \ref{assum1} may directly be applied to the covariance function $\sigma^*(s,t)$ of the Gaussian process component $X_i^*$ of the elliptical process $X_i$. If $\sigma^*$ satisfies Assumption \ref{assum1} for some $\omega^*:\Omega\to\mathbb{R}$, 
%, then $\omega(t,s;v)=v^2\omega^*(s,t,|s-t|^\kappa)$ 
then Proposition \ref{idgeneral} in Appendix \ref{app:poiID-2} leads to
%%%%%%%%%%%%%%%
\begin{equation*}
\begin{array}{ll}
\mathbb{E}(Z_{\delta,i}(s)Y_i) =  C(\tau_r)\delta^\kappa + o(\delta^\kappa)&\text{if } s\in\{\tau_1,\dots,\tau_r\}\\
% since two-continuous derivatives at off-diagonal points
\mathbb{E}(Z_{\delta,i}(s)Y_i)   = O(\delta^2) &\text{if } s\notin\{\tau_1,\dots,\tau_r\}
\end{array}
\end{equation*}
as $\delta\to 0$ 
with 
$C(\tau_r)=c^*(\tau_r)
\E\left( V_i^2  \vartheta_r(V_i)\right)$, where $c^*(\tau_r)= -\frac{\partial}{\partial z}\omega^*(\tau_r,\tau_r,z)|_{z=0}$, $r=1,\dots,S$.

Theorem \ref{poicon} can also be generalized to the case that $X_i$ is elliptically distributed.  
%Let $X_i\overset{d}{=}V_{i}X_i^\star$, where $V_i$ is a strictly positive, real valued random variable, and where $X_i^\star$ is a zero mean Gaussian process independent of $V_i$.  
Note that then $Z_{\delta,i}(s)Y_i\overset{d}{=}Z^\star_{\delta,i}(s)Y^\star_i$, where $Z^\star_{\delta,i}(s)=X^\star_i(s) - \frac{1}{2}(X^\star_i(s-\delta), + X^\star_i(s+\delta))$, for $s \in [a+\delta, b-\delta]$, and $Y_i^\star=V_iY_i$. 
Therefore, estimating points of impact from data $(X_i,Y_i)$ is equivalent to estimating points of impact from data $(X_i^\star,Y_i^\star)$. 
Thus, Theorem \ref{poicon} remains valid if all conditions on $X_i$ and $Y_i$ in Theorem \ref{poicon} now apply to $X_i^\star$ and $Y_i^\star$.

Our more general framework of conditions A and B in Appendix \ref{app:poiID-2} includes even more complex cases than the above discussed elliptical processes.  For instance, one may consider processes $X_i\overset{d}{=}V_{i1}(t)X_i^*(t)+V_{i2}(t)$, where $(V_{i1},V_{i2})$ is jointly independent of $X_i^*$ and where $V_{i1}$ and $V_{i2}$ are almost surely twice continuously differentiable functions on $[a,b]$ (see Appendix \ref{app:poiID-2} in the supplementary paper \cite{Suppl2019} for more details).

%%%%%%%%%%%%%%%%%%%%%%%%%%%%%%%%%%%%%%%%%%%%%%%%%%%%%%%%%%%%%%%%%%%%%%%%%%%%%%%%
\subsubsection{Generalizing covariance Assumption \ref{assum1}}\label{ssec:General_2} 
%%%%%%%%%%%%%%%%%%%%%%%%%%%%%%%%%%%%%%%%%%%%%%%%%%%%%%%%%%%%%%%%%%%%%%%%%%%%%%%%
Assumption~\ref{assum1} holds for non-smooth/rough processes $X_i$ with covariance function $\sigma(s,t)=\omega(s,t,|s-t|^\kappa)$, where the requirement $0<\kappa<2$ excludes all smooth, twice continuously differentiable processes, $X_i$, with $\kappa\geq 2$.  

However, the degree of roughness of the processes, $X_i$, is actually not a necessary requirement for identifying and estimating points of impact.  The crucial property is that the covariance function $\sigma(s,t)$ of $X_i$ is less smooth at the diagonal than for $|t-s|>0$.  For instance, let $\sigma(s,t)$ be $d=4$ times continuously differentiable at all off-diagonal points, $s\neq t$, but \emph{not} $d=4$ times differentiable at the diagonal points, $s=t$.  This scenario corresponds to a generalization of Assumption \ref{assum1} with $0<\kappa<d=4$ which now excludes only all four times continuously differentiable processes, $X_i$, with $\kappa\geq d=4$.  In this case, one may look at the modified $4$th central difference approximation of the $4$th derivative of $\mathbb{E}(X_i(s)Y_i)$ and replace $Z_{\delta,i}(s)$ by
%%%%%%%%%%%%%%%%%%%%%%%%%%%% 
$$\tilde{Z}^{(4)}_{\delta,i}(s):=X_i(s)-\frac{2}{3}\big(X_i(s-\delta) + X_i\big(s+\delta)\big) + \frac{1}{6}\big(X_i(s-2\delta)+X_i(s+2\delta)\big).$$
% or generally
% $$\tilde{Z}^{(d)}_{\delta,i}(s):=\sum_{\ell=0}^{d}(-1)^\ell\left({d}\choose{\ell}\right)$$
%%%%%%%%%%%%%%%%%%%%%%%%%%%%
Theoretical results may then be derived under a generalized version of Assumption \ref{assum1} demanding that there exists a $d=4$ times differentiable function $\omega$ such that \eqref{deromega} holds for any $\kappa<d=4$.  

Equivalent generalizations can, for instance, be made for any $d\in\{2,4,6,8,\dots\}$, which would involve then a modified $d$th order central difference processes $\tilde{Z}^{(d)}_{\delta,i}(s)$.  This way, Assumption \ref{assum1} can be generalized to the requirement $0<\kappa<d$ which also then includes smooth processes, $X_i$.  Deriving the estimation theory under this setup would then lead to even more accurate points of impact estimators with an even faster super-consistent convergence rate.  However, taking higher order differences in practice usually involves numerical instabilities.

%%%%%%%%%%%%%%%%%%%%%%%%%%%%%%%%%%
\section[]{Subsequent estimation of $g$}\label{sec:Inference}
%%%%%%%%%%%%%%%%%%%%%%%%%%%%%%%%%%
Given estimates of the points of impact $\tau_1,\dots,\tau_S$ and their number $S$, one is typically interested in the subsequent estimation and inference regarding the remaining model components. The following section considers the case of a nonparametric model $g$.  Section \ref{sec:PES} considers the case of a generalized linear model, which is of particular practical relevance.

In the following we assume the existence of some consistent estimation procedure for the points of impact satisfying $\max_{r=1,\ldots,\widehat{S}}|\widehat{\tau}_r-\tau_r|\ =\ O_P( n^{-1/\kappa})$ and $P(\widehat{S}=S)\to 1$, where we use matched labels in the sense that $\tau_r = \arg \min_{s=1,\dots, S}|\widehat{\tau}_r-\tau_s|$. These requirements are fulfilled by our estimation procedure described in Section \ref{SubSec:Estim}, but may also be fulfilled for alternative procedures.

%%%%%%%%%%%%%%%%%%%%%%%%%%%%%%%%%%
\subsection{Nonparametric estimation}\label{sec:NP}
%%%%%%%%%%%%%%%%%%%%%%%%%%%%%%%%%%
Estimating the nonparametric function $g$ in \eqref{eq:whnp} is a non-standard estimation problem, since the unknown points of impact $\tau_r$ of the predictor variables $X_i(\tau_r)$ must be replaced by their estimates $\widehat\tau_r$. That is, for given estimates $\widehat{\tau}_1,\dots,\widehat{\tau}_S$ we may estimate the unknown regression function $g$ by the following Nadaraya-Watson type estimator
%%%%%%%%%%%%%%%
\begin{align}\label{eq:nadaraya}
\widehat{g}_{\widehat{\tau}}(x_1,\dots,x_S) = \frac{\sum_{i=1}^n K\Big(\frac{X_i(\widehat{\tau}_1)-x_1}{h_1},\dots, \frac{X_i(\widehat{\tau}_S)-x_S}{h_S}\Big) Y_i}{\sum_{i=1}^n K\Big(\frac{X_i(\widehat{\tau}_1)-x_1}{h_1},\dots, \frac{X_i(\widehat{\tau}_S)-x_S}{h_S}\Big)},
\end{align}
%%%%%%%%%%%%%
where $K$ denotes a standard nonnegative symmetric bounded second-order kernel function with $\smallint K(u)du=1$, and where $h_1,\dots,h_S$ denote the bandwidth parameters.

For the following result we make use of our super-consistency result in Theorem \ref{poicon}. Note, however, that the rates of consistency for the point of impact estimators $\widehat{\tau}_r$ of Theorem \ref{poicon} cannot be used directly to quantity the errors $|X_i(\widehat{\tau}_r)-X_i(\tau_r)|$, $r=1,\dots,S$, since under Assumption \ref{assum1} we cannot make use of Taylor-expansions of $X_i$. Therefore, the following result is non-standard because of the additional error component $\widehat{g}_{\widehat{\tau}}(x_1,\dots,x_S)-\widehat{g}_{\tau}(x_1,\dots,x_S)=O_p\big(\sum_{r=1}^S1/(n^{\min\{1,1/\kappa\}}(h_1\cdots h_S) h_r^2)\big)$ contained in \eqref{eq:NWkernelREG}, where $\widehat{g}_\tau$ is defined as in \eqref{eq:nadaraya}, but using the true predictor variables $X_i(\tau_1),\dots,X_i(\tau_S)$.
%%%%%%%%%%%%%%%%%%%%%%%%%%%%%%%%

% Er hat jedoch völlig vergessen, dass bei BEKANNTEM \tau auch noch die üblichen Zusatzannahmen an den Kern und die Funktion g gelten müssen, damit die behaupteten Bias und Varianzabschätzungen richtig sind (Second order kernel \int K dx =1, g twice continously differentiable). Die im Theorem genannten Annahmen reichen daher nicht aus. Das ist ziemlich ärgerlich, denn selbst einem oberflächlicher Referee fällt vielleicht auf, dass bei dem Theorem etwas nicht stimmen kann (bounded derivative für g reicht bekanntermaßen nicht aus). Das könnte dann auch das Vertrauen in den schwierigen Teil der Theorie erschüttern.

\begin{theorem}\label{thm:lcr}
Let $\widehat{S}=S$, $\max_{r=1,\dots, S}|\widehat{\tau}_r-\tau_r| = O_p(n^{-1/\kappa})$, and let Assumptions \ref{assum1} and \ref{assum2} and the assumptions of Theorem \ref{poicon} hold.  
%  let $X_i$ be Gaussian process satisfying Assumption \ref{assum1} and let Assumption \ref{assum2} hold.
% \st{If the Kernel $K: \mathbb{R}^S \to \mathbb{R}$ is bounded and twice continuously differentiable with bounded derivatives and if the regression function $g$ is bounded with a bounded derivative, we then have} \\
Moreover, let the kernel function $K: \mathbb{R}^S \to \mathbb{R}$ be a second-order kernel (i.e., a density function that is symmetric around zero) with continuous second-order partial derivatives and let the regression function $g$ have continuous second-order partial derivatives. We then have for any points $x_1,\dots,x_S$ in the interior of the support of $X$ that
\begin{align}\label{eq:NWkernelREG}
\begin{split}
&\widehat{g}_{\widehat{\tau}}(x_1,\dots,x_S) - g(x_1,\dots,x_S)=\\
&\quad O_p\left(\sum_{r=1}^S h_r^2 + (nh_1\cdots h_S)^{-1/2}+\sum_{r=1}^S\frac{1}{n^{\min\{1,1/\kappa\}} (h_1\cdots h_S) h_r^2}\right)
\end{split}
\end{align}
for $n\to\infty$, and $h_1,\dots,h_S\to 0$ with $n^{\min\{1,1/\kappa\}}(h_1\cdots h_S)h_r^2 \to \infty$, for each $r=1,\dots, S$.
\end{theorem}
%%%%%%%%%%%%%%%%%%%%%%%%%%%%%%%

If each bandwidth has the same order of magnitude and $0<\kappa\leq 1$, the well-known optimal bandwidth choice $h_r\sim n^{-1/(S+4)}$, $r=1,\dots,S$, can be used to simplify Theorem \ref{thm:lcr} as following.
%%%%%%%%%%%%%%%%%%%%%%%%%%%%%%%%%%%
\begin{corollary}\label{cor:lcr}
Under the assumptions of Theorem \ref{thm:lcr}, let $0<\kappa\leq 1$ and $h_r\sim n^{-1/(S+4)}$ for all $r=1,\dots,S$. Then
$$\widehat{g}_{\widehat{\tau}}(x_1,\dots,x_S)-g(x_1,\dots,x_S)=  O_p\big(n^{-2/(S+4)}\big).$$
\end{corollary}
%%%%%%%%%%%%%%%%%%%%%%%%%%%%%%%%%%%
\noindent That is, under the conditions of Corollary \ref{cor:lcr}, we have the same optimal rates of convergence as in the case where the points of impact were known.

%%%%%%%%%%%%%%%%%%%%%%%%%%%%%%%%%%
\subsection{Parametric estimation}\label{sec:PES}
%%%%%%%%%%%%%%%%%%%%%%%%%%%%%%%%%%
In this section it is assumed that the relationship between $Y_i$ and the functional predictor $X_i$ can be modeled using the framework of generalized linear models with known parametric function $g$,
%%%%%%%%%%%%%%%%%%%%%%%%%%%
\begin{align}\label{eq:wh}
Y_i = g\Big(\alpha + \sum_{r=1}^S \beta_r X_i(\tau_r)\Big) + \varepsilon_i,
\end{align}
%%%%%%%%%%%%%%%%%%%%%%%%%%%
in which the i.i.d.~error term $\varepsilon_i$ respects $\mathbb{E}(\varepsilon_i|X_i(t)) = 0$ for all $t\in [a,b]$ and where $\V(\varepsilon_i|X_i(t),t\in [a,b]) = \sigma^2(g(\eta_i))<\infty$ with strictly positive variance function $\sigma^2(\cdot)$ defined over the range of $g$. For simplicity the function $g$ is assumed to be a known, strictly monotone and smooth function with bounded first and second order derivatives and hence invertible \citep[see, for instance,][for similar assumptions]{MS2005}. The constant $\alpha$ allows us to consider centered random functions $X_i$ with $\E(X_i(t))=0$ for all $t\in[a,b]$. Note that we do not assume that the conditional distribution of $Y_i$ belongs to the exponential family of distributions. Denoting the linear predictor
%%%%%%%%%%%%%%%%%%%%%%%%%%%%%
\begin{align}\label{eq:lp}
\eta_i = \alpha + \sum_{r=1}^S \beta_r X_i(\tau_r)
\end{align}
%%%%%%%%%%%%%%%%%%%%%%%%%%%%%
allows us to write $\E(Y_i|X_i) = g(\eta_i)$ as well as $\V(Y_i|X_i)=\sigma^2(g(\eta_i))<\infty$. Hence, this setup of model \eqref{eq:wh} belongs to the broad class of quasi-likelihood models which can be seen as a generalization of a generalized linear model framework \citep[cf.][Ch.~9]{MCN1989}.

Identifiability of the model parameters in \eqref{eq:wh} is not obvious due to the functional predictor $X_i(\cdot)$, which, in principle, allows for infinitely many alternative model candidates. The following Theorem \ref{thmident} shows that any possible kind of model-misspecification in $\alpha$, $\beta_r$, $\tau_r$, $r=1,\dots,S$, or $S$, will lead to a different model in the mean squared error sense implying the identifiability of model \eqref{eq:wh}.

% The following theorem shows that all model-parameters can be uniquely identified under the assumptions of model \eqref{eq:wh}. Identifiability of the model parameters in \eqref{eq:wh} is not obvious due to the involved functional predictor $X_i$ and the fact that we do not assume identifiability by imposing a classical full-rank assumption. %For instance, we do not impose the assumption that the mean of the gram-matrix of $(1, X_i(\tau_1), \dots, X_i(\tau_S))$

% Identifiability of the model parameters in \eqref{eq:wh} is not obvious as we do not assume that $\E(\bX_i(\btau) \bX_i(\btau)^T)$, $\bX_i(\btau) = (1, X_i(\tau_1), \dots, X_i(\tau_S))^T$, is invertible.
%
%%%%%%%%%%%%%%%%%%%%%%%%%%%%%%%%%%%%%%%
\begin{theorem}\label{thmident}
Let $g(\cdot)$ be invertible and assume that $X_i$ satisfies Assumptions~\ref{assum1} and \ref{assum2}. Then for all $S^*\geq S$, all $\alpha^*, \beta_1^*,\ldots,\beta_{S^*}^*\in\mathbb{R}$, and all $\tau_1,\dots,\tau_{S^*}\in(a,b)$ with $\tau_{k}\notin\{\tau_1,\dots,\tau_S\}$, $k=S+1,\dots,S^*$, we obtain
%%%%%%%%%%%%%%%%%%
\begin{align}
\mathbb{E}\left(\bigg( g(\alpha + \sum_{r=1}^{S}\beta_rX_i(\tau_r) ) - g(\alpha^* + \sum_{r=1}^{S^*}\beta_r^*X_i(\tau_r)) \bigg)^2\right)>0,
\label{Eq:eqident}
\end{align}
%%%%%%%%%%%%%%%%%%
whenever $|\alpha-\alpha^*|>0$,
or $\sup_{r=1,\ldots,S}|\beta_r-\beta^*_r|>0$,
or $\sup_{r=S+1,\ldots,S^*}|\beta^*_r|>0$.
\end{theorem}
%%%%%%%%%%%%%%%%%%%%%%%%%%%%%%%%%%%%%%%

Note that the proof of Theorem \ref{thmident} does only require the existence of second moments and, therefore, may be generalized also to the case of non-Gaussian processes $X_i$.

%For the following, we assume the existence of appropriate estimators $\widehat{\tau}_r$ of $\tau_r$, $r=1,\dots,S$ and $\widehat{S}$ of $S$ such as, for instance, our estimators described in Section \ref{sec:POI}. Indeed, under the assumptions of Theorem \ref{poicon}, the (super-)consistency results \eqref{thm3eq1} and \eqref{thm3eq4} for our estimators also apply to the case of \eqref{eq:wh}, which is only a special case of our nonparametric model \eqref{eq:whnp}.

Estimation of $\bbeta_0 = (\alpha,\beta_1,\dots, \beta_S)^T$ is performed by quasi-maximum likelihood.  Define $\bX_i(\widehat{\btau}) = (1, X_i(\widehat{\tau}_1), \dots, X_i(\widehat{\tau}_S))^T$ and denote the $j$th, $1\leq j\leq S+1$, element of the latter vector as $\widehat{X}_{ij}$.
For $\bbeta \in \mathbf{R}^{S+1}$ let
$\widehat{\eta}_i(\bbeta) = \bX_i(\widehat{\btau})^T\bbeta$,
$\widehat{\bmu}_n(\bbeta) = (g(\widehat{\eta}_1(\bbeta)),\dots, g(\widehat{\eta}_n(\bbeta)))^T$,
$\widehat{\bD}_n(\bbeta)$ be the $n\times (S+1)$ matrix with entries
$g'(\widehat{\eta}_i(\bbeta))\widehat{X}_{ij}$, and let
$\widehat{\bV}_n(\bbeta)$ be a $n\times n$ diagonal matrix with elements
$\sigma^2(g(\widehat{\eta}_i(\bbeta)))$.
Furthermore, denote the corresponding objects evaluated at the true points of impact $\tau_r$ by
$\bX_i(\btau)$, $X_{ij}$,  $\eta_i(\bbeta)$, $\bmu_n(\bbeta)$, $\bD_n(\bbeta)$, and $\bV_n(\bbeta)$; this notational convention applies also to the below defined objects.

Our estimator $\widehat{\bbeta}$ for $\bbeta_0=(\alpha, \beta_1,\dots,\beta_S)^T$ is defined as the solution of the $S+1$ score equations $\widehat{\bU}_n(\widehat{\bbeta})=0$, where
%%%%%%%%%%%%%
\begin{align}
\widehat{\bU}_n(\bbeta) = \widehat{\bD}_n(\bbeta)^T\widehat{\bV}_n(\bbeta)^{-1}(\bY_n-\widehat{\bmu}_n(\bbeta)). \label{eq:QML}
\end{align}
%%%%%%%%%%%%%
Note that these are non-classic score equations evaluated at the estimates $\widehat{\tau}_r$ instead of $\tau_r$.

In the following, it will be convenient to define
%%%%%%%%%%%%%
\begin{align*}
\bF_n(\bbeta) = \bD_n(\bbeta)^T\bV_n(\bbeta)^{-1}\bD_n(\bbeta)
\quad\text{and}\quad
\widehat{\bF}_n(\bbeta) = \widehat{\bD}_n(\bbeta)^T\widehat{\bV}_n(\bbeta)^{-1}\widehat{\bD}_n(\bbeta).
\end{align*}
%%%%%%%%%%%%%
By definition it holds that $\E(n^{-1}\bF_n(\bbeta))=[\E(g'(\eta_i(\bbeta))^2/\sigma^2(g(\eta_i(\bbeta)))\, X_{ik}X_{il})]_{k,l}$ with $k,l=1,\dots, S+1$. Let $\eta(\bbeta)$ and $X_j$ be generic copies of $\eta_i(\bbeta)$ and of the $j$th component of $\bX_i(\btau)$, respectively. This allows us to write $\E(n^{-1}\bF_n(\bbeta))=\E(\bF(\bbeta))$ with $\E(\bF(\bbeta))=[\E(g'(\eta(\bbeta))^2/\sigma^2(g(\eta(\bbeta)))\, X_{k}X_{l})]_{k,l}$, where we point out that $\E(\bF(\bbeta))$ is for all $\bbeta \in \mathbf{R}^{S+1}$ a symmetric and strictly positive definite matrix with inverse $\E(\bF(\bbeta))^{-1}$. Indeed, suppose $\E(\bF(\bbeta))$ were not strictly positive definite, we would then derive the contradiction $\E((\sum_{j=1}^{S+1} a_j X_j g'(\eta(\bbeta))/\sigma(g(\eta(\bbeta))))^2)=0$ for nonzero constants $a_1,\dots, a_{S+1}$. A similar argument can be used to show that $\E(\widehat{\bF}(\bbeta))$ is strictly positive definite, where $\E(\widehat{\bF}(\bbeta))=[\E(g'(\widehat{\eta}(\bbeta))^2/\sigma^2(g(\widehat{\eta}(\bbeta)))\,\widehat{X}_{k}\widehat{X}_{l})]_{k,l}$.

%In the rest of this section we assume $X_i(\cdot)$ to be i.i.d.~Gaussian distributed with covariance $\sigma(s,t)$ satisfying Assumption~\ref{assum1}. 

The following additional set of assumptions are used to derive more precise theoretical statements:
%%%%%%%%%%%%%%%%%%%%%%%%%%%%%%%%%%%%%%%%%%%%%%
\begin{assumption} \label{assum3}\
\begin{itemize}
\item[a)] There exists a constant $0<M_{\varepsilon}<\infty$, such that $\E(\varepsilon_i^p|X_i(t))\leq M_{\varepsilon}$, for all $t\in[a,b]$ and for some even $p$ with $p \geq \max\{2/\kappa+\epsilon, 4\}$ and some $\epsilon>0$.
\item[b)] The function $g$ is monotone, invertible with two bounded derivatives $|g'(\cdot)|\leq c_g$, $|g''(\cdot)|\leq c_g$, for some constant $0\leq c_g<\infty$.
\item[c)] $h(\cdot):=g'(\cdot)/\sigma^2(g(\cdot))$ is a bounded function with two bounded derivatives.
%this implies that h' is continuous, a condition which is actually needed.
\end{itemize}
\end{assumption}
%%%%%%%%%%%%%%%%%%%%%%%%%%%%%%%%%%%%%%%%%%%%%%
Condition a) states that some higher moments of $\varepsilon_i$ exist. While the condition on $p\geq 4$ and $p$ being even simplifies the proofs, the condition $p> 2/\kappa$ is a more crucial one and is used in the proof of Proposition~\ref{lem:score1} in the supplementary Appendix~\ref{app:PARAest}. Conditions a) to c) hold, for example, in the important case of a functional logistic regression with points of impact, where $g$ is the standard logistic function. Condition c) is satisfied, for instance, in the special case of generalized linear models with natural link functions. For the latter case, we have $\sigma^2(g(x)) = g'(x)$ such that $h(x)=1$.

%%%%%%%%%%%%%%%%%%%%%%%%%%%%%%%%%%%%
\begin{theorem}\label{th:paraestML1}
Let $\widehat{S}=S$, $\max_{r=1,\dots, S}|\widehat{\tau}_r-\tau_r|= O_P(n^{-1/\kappa})$ and let $X_i$ be a Gaussian process satisfying Assumption~\ref{assum1}. Under Assumption \ref{assum3} we then obtain
\begin{align}
\sqrt{n}(\widehat{\bbeta}-\bbeta_0) \stackrel{d}{\to} N\big(\bzero, (\mathbb{E}(\bF(\bbeta_0)))^{-1}\big).\label{eq:thMLE2a}
\end{align}
\end{theorem}
%%%%%%%%%%%%%%%%%%%%%%%%%%%%%%%%%%%%
%$\E(\bF(\bbeta))=(\E(g'^2(\eta)/\sigma^2(g(\eta))\, X_{k}X_{l})))_{k=1,\dots, S+1, l=1,\dots, S+1}$, where $\eta$ and $X_j$ are a generic copies of $\eta_i(\bbeta)$ and the $j$th component of $\bX_i$ respectively.
%%%%%%%%%%%%%%%%%%%%%%%%%%%%%%%%%%%%%
That is, our estimator based on $\widehat{\tau}_r$ enjoys the same asymptotic efficiency properties as if the true points of impact $\tau_r$ were known.  In fact, it achieves the same asymptotic efficiency properties as under classic multivariate setups \citep[cf.][]{MC1983}. In practice one might replace $\mathbb{E}(\bF(\bbeta_0))$ with its consistent estimator $n^{-1}\widehat{\bF}_n(\widehat{\bbeta})$ in order to derive approximate results. This is a direct consequence of Equations \eqref{eq:thQML3mean1} and \eqref{eq:exist0} in the supplementary Appendix~\ref{app:PARAest}.

%%%%%%%%%%%%%%%%%%%%%%%%%%%%%%%%%%%%%%%%%
\subsubsection{Parametric estimation: Practical implementation}\label{sec:PIM}
%%%%%%%%%%%%%%%%%%%%%%%%%%%%%%%%%%%%%%%%%
An implementation of our parametric estimation procedure comprises, first, the estimation of the points of impact $\tau_r$ and, second, the estimation of the parameters $\alpha$ and $\beta_r$. Estimating the points of impact $\tau_r$ relies on the choice of $\delta$ and a choice of the threshold parameter $\lambda$ (see Section \ref{SubSec:Estim}). Asymptotic specifications are given in Theorem~\ref{poicon}; however, these determine the tuning parameters $\delta$ and $\lambda$ only up to constants and are generally of a limited use in practice. In the following we propose an alternative fully data-driven model selection approach.

For a given $\delta$, our estimation procedure leads to a set of potential point of impact candidates $\{\widehat{\tau}_1,\widehat{\tau}_2,\dots,\widehat{\tau}_{M_\delta}\}$ (see Section \ref{SubSec:Estim}). Selecting final point of impact estimates from this set of candidates corresponds to a classic variable selection problem. In the case where the distribution of $Y_i|X_i$ belongs to the exponential family (as in the logistic regression) one may perform a best subset selection optimizing a standard model selection criterion such as the Bayesian Information Criterion (BIC),
%%%%%%%%%%%%%%%%%%%%%%%%%%%%%%%%
\begin{align}\label{eq:BIC}
\operatorname{BIC}_{\mathcal{X}}(\delta)= -2 \log \mathcal{L}_{\mathcal{X}} + K_{\mathcal{X}}\log{(n)}.
\end{align}
%%%%%%%%%%%%%%%%%%%%%%%%%%%%%%%%
Here, $\log\mathcal{L}_{\mathcal{X}}$ is the log-likelihood of the model with intercept and predictor variables $\mathcal{X}\subseteq \{X_i(\widehat{\tau}_1),X_i(\widehat{\tau}_2),\dots,X_i(\widehat{\tau}_{M_\delta})\}$, where $K_{\mathcal{X}}=1+|\mathcal{X}|$ denotes the number of predictors. Minimizing $\operatorname{BIC}_{\mathcal{X}}(\delta)$ over $0<\delta<(b-a)/2$ leads to the final model choice.

In the more general case of quasi-likelihood models \citep[cf.][Ch.~9]{MCN1989} where only the first two moments $\E(Y_i|X_i) = g(\eta_i)$ and $\V(Y_i|X_i)=\sigma^2(g(\eta_i))$ are known, one may replace the deviance $-2 \log \mathcal{L}_{\mathcal{X}}$ by the expression for the quasi-deviance
%%%
$-2Q_{\mathcal{X}} = -2\sum_{i=1}^n\int_{y_i}^{g(\widehat{\eta}_{\mathcal{X},i})}(y_i-t)/(\sigma^2(t))\,dt$,
%%%
where $\widehat{\eta}_{\mathcal{X},i}$ is the linear predictor with intercept and predictor variables $\mathcal{X}$.

In order to calculate $\operatorname{BIC}_{\mathcal{X}}(\delta)$, we need the estimates $\widehat{\bbeta}$ solving the estimation equations $\widehat{\bU}_n(\widehat{\bbeta})=0$. In practice these equations are solved iteratively, for instance, by the usual Newton-Raphson method with Fisher-type scoring. That is, for an arbitrary initial value $\widehat{\bbeta}_0$ sufficiently close to $\widehat{\bbeta}$ one generates a sequence of estimates $\widehat{\bbeta}_m$, with $m=1,2,\dots$,
%%%%%%%%%%%%%%%%%%%%%%%%%%%%%%
\begin{align}
\widehat{\bbeta}_m = \widehat{\bbeta}_{m-1} + \big(\widehat{\bF}_n(\widehat{\bbeta}_{m-1})\big)^{-1}
\widehat{\bU}_n(\widehat{\bbeta}_{m-1}).\label{eq:FS}
\end{align}
%%%%%%%%%%%%%%%%%%%%%%%%%%%%%%
Iteration is executed until convergence and the final step of the procedure yields the estimate $\widehat{\bbeta}$. Here, $\widehat{\bF}_n(\bbeta)$ and $\widehat{\bU}_n(\bbeta)$ replace $\bF_n(\bbeta)$ and $\bU_n(\bbeta)$ in the usual Fisher scoring algorithm, since the unknown $\tau_r$, $1\leq r\leq S$, are replaced by their estimates $\widehat{\tau}_r$. The latter is justified asymptotically by our results in Corollary \ref{cor:uniform} and Proposition \ref{th:paraestQML} in Appendix~\ref{app:PARAest}.
%%%%%%%%

%%%%%%%%%%%%%%%%%%%%%%%%%%%%%%%%%%%%%
\section{Simulation}\label{sec:SIM} 
%%%%%%%%%%%%%%%%%%%%%%%%%%%%%%%%%%%%%
We investigate the finite sample performance of our estimators using Monte Carlo simulations. After simulating a trajectory $X_i$ over $p$ equidistant grid points $t_j$, $j=1,\dots,p$, on $[a,b]=[0,1]$, linear predictors of the form $\eta_i = \alpha + \sum_{r=1}^S \beta_r X_i(\tau_r)$ are constructed for some predetermined model parameters $\alpha$, $\beta_r$, $\tau_r$, and $S$, where a point of impact is implemented as the smallest observed grid point $t_j$ closest to $\tau_r$. The response $Y_i$ is derived as a realization of a Bernoulli random variable with success probability $g(\eta_i)=\exp(\eta_i)/(1+\exp(\eta_i))$, resulting in a logistic regression framework with points of impact.
%%%%%%%%%%%%
The simulation study is implemented in \textsf{R} \citep{Rcite}, where we use the \textsf{R}-package \texttt{glmulti} \citep{GLMULTI} in order to implement the fully data-driven BIC-based best subset selection estimation procedure described in Section~\ref{sec:PIM}. The threshold estimator from Section~\ref{SubSec:Estim} requires appropriate choices of $\delta = \delta_n$ and $\lambda = \lambda_n$. Theorem~\ref{poicon} suggests that a suitable choice of $\delta$ is given by $\delta= c_{\delta}\,n^{-1/2}$ for some constant $c_{\delta}>0$. Our simulation results are based on $c_{\delta} = 1.5$; similar qualitative results were derived for a broader range of values $c_{\delta}$. For the threshold $\lambda$ we use $\lambda=A((\widehat{\E}(Y^4))^{1/2}\log((b-a)/\delta)/n)^{1/2}$, where $A=\sqrt{2\sqrt{3}}$ and $\widehat{\E}(Y^4) = n^{-1}\sum_{i=1}^n Y_i^4$, as motivated below of Theorem~\ref{poicon}.

In what follows, we denote the BIC-based selection (see Section \ref{sec:PIM}) of points of impact by POI and the threshold-based selection (Algorithm \ref{Algo:1}) by TRH.
%%%
%%where  $\widehat{E}(Y^4) = n^{-1}\sum_{i=1}^n Y_i^4$ is the estimator of $\E(Y^4)$.
%%%%%%%%%%%%%%
Estimated points of impact candidates are related to the true impact points by the following matching rule: In a first step the interval $[a,b]$ is partitioned into $S$ subintervals of the form $I_j=[m_{j-1},m_{j})$, where $m_0 = a$, $m_S=b$ and $m_j = (\tau_{j}+\tau_{j+1})/2$ for $0<j<S$. The candidate $\widehat{\tau}_l$ in interval $I_j$ with the closest distance to $\tau_j$ is then taken as the estimate of $\tau_j$. %No impact point estimate in an interval results in an unmatched $\tau_j$ and a missing value when calculating statistics for the estimator.

The simulation results for our parametric estimation procedure (Section \ref{sec:PES}) are based on $1000$ Monte Carlo iterations for each constellation of $n\in\{100,200,500,1000,3000\}$ and $p\in\{100,500,1000\}$. The results for our nonparametric estimation procedure (Section \ref{sec:NP}), are based on the same general setup, but consider the reduced set of sample sizes $n\in\{100,200,500\}$.  Estimation errors for the parametric estimation procedure are illustrated by boxplots with error bars representing the $10\%$ and $90\%$ quantiles. The estimation errors for the nonparametric estimation procedure are quantified by the Mean Average Squared Error, $\operatorname{MASE}=1000^{-1}\sum_{r=1}^{1000}n^{-1}\sum_{i=1}^n\big(g(\eta_i)-\widehat{g}^r_{\widehat{\tau}}(X^r_{i}(\widehat{\tau}^r_{1}),\dots,X^r_{i}(\widehat{\tau}^r_{\hat{S}}))\big)^2$, where the superscript $r$ denotes the $r$th simulation run.
%%%%%%%%%

Five data generating processes (DGP) are considered (see Table \ref{tab:simulatione_1}) using the following three processes $\{X_i(t):0\leq t\leq 1\}$ covering a broad range of situations:
%%%%%%%%%%%%%%%%%%%%
\begin{description}
\item[OUP] {\sc Ornstein-Uhlenbeck Process}. A Gaussian process with covariance function $\sigma(s,t) = \sigma^2_u/(2\theta) (\exp(-\theta|s-t|) - \exp(-\theta(s+t)))$. We choose $\theta=5$ and $\sigma^2_u=3.5$.
%%%
\item[GCM] {\sc Gaussian Covariance Model}. A Gaussian process with covariance function $\sigma(s,t) = \sigma(|s-t|) = \exp(-(|s-t|/d)^2)$. We choose $d=1/10$.
%%%
\item[EBM] {\sc Exponential Brownian Motion}. A non Gaussian process with covariance function $\sigma(s,t) = \exp((s+t+ |s-t|)/2)-1$. It is defined by $X_i(t) = \exp(B_i(t))$, where $B_i(t)$ is a Brownian motion.
\end{description}

%%%%%%%%%%%%%%%%%%%%%%%%%%%%%%%%%%%%%%
%%%%%%%%%%%%%%%%%%%%%%
\spacingset{1}
\begin{table}[!htb]
\addtolength{\tabcolsep}{-3.pt}
\centering
\caption[]{Data generating processes considered in the simulations\\}\label{tab:simulatione_1}
\begin{adjustbox}{max width=0.9\textwidth}
\begin{tabular}{l*{11}{c}}
\midrule
\multicolumn{2}{c}{\textbf{Model}} & \multicolumn{5}{c}{\textbf{Points of impact}} & \multicolumn{5}{c}{\textbf{Parameters}} \\
\cmidrule(lr){1-2}    \cmidrule(lr){3-7 } \cmidrule(lr){8-12}
\textbf{Label} & \textbf{Process} & $\boldsymbol{S}$ & $\boldsymbol{\tau_1}$ &$\boldsymbol{\tau_2}$ & $\boldsymbol{\tau_3}$  & $\boldsymbol{\tau_4}$ & $\boldsymbol{\alpha}$  & $\boldsymbol{\beta_1}$ & $\boldsymbol{\beta_2}$ & $\boldsymbol{\beta_3}$  & $\boldsymbol{\beta_4}$\\
\midrule
DGP 1 & OUP & $1^\ast$                  & 1/2 &     &     &     & 1 &  4 &   &    &   \\
DGP 2 & OUP & $2\phantom{^\ast}$        & 1/3 & 2/3 &     &     & 1 & -6 & 5 &    &   \\
DGP 3 & OUP & $4\phantom{^\ast}$        & 1/6 & 2/6 & 4/6 & 5/6 & 1 & -6 & 6 & -5 & 5 \\
DGP 4 & GCM & $2\phantom{^\ast}$        & 1/3 & 2/3 &     &     & 1 & -6 & 5 &    &   \\
DGP 5 & EBM & $2\phantom{^\ast}$        & 1/3 & 2/3 &     &     & 1 & -6 & 5 &    &   \\
\bottomrule
\multicolumn{12}{c}{$^\ast$Note: $S=1$ is assumed known (only for DGP 1).}
\end{tabular}
\end{adjustbox}
\end{table}
\spacingset{1.5}
%%%%%%%%%%%%%%%%%%%%%%%%%%%%%%%%%%%%%%

DGP 1-3 are increasingly complex, but satisfy our theoretical assumptions. The general setups of DGP 4 and DGP 5 are equivalent to DGP 2, but the processes $X_i$ (GCM and EBM) violate our theoretical assumptions. The covariance function in DGP 4 is infinitely many times differentiable, even at the diagonal where $s=t$, contradicting Assumption \ref{assum1}, but fitting the remark underneath this Assumption. The process in DGP 4 contradicts the Gaussian Assumption \ref{assum2}.

\subsection{Evaluation of the parametric estimation procedure}

%%%%%%%%%%%
%% DGP 1 %% 
%%%%%%%%%%%
%%%%%%%%%%%%%%%%%%%%
%%%%%%%%%%%%%%%%%%%%
\begin{figure}[!htb]
\centering 
{\sc DGP 1: Estimation errors for $n\in\{100,200,500,1000,3000\}$}
\par\medskip
\includegraphics[width=0.9\textwidth]{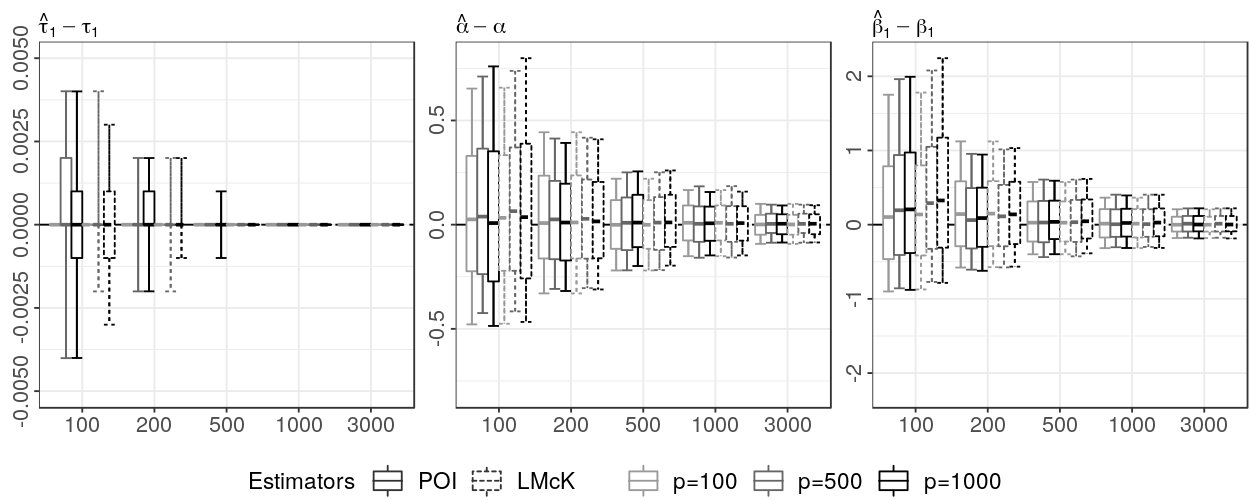}
\caption[Estimation errors for DGP 1]{Comparison of the estimation errors from using our BIC-based method POI (solid lines) and the method of \cite{LMcK2009} (dashed lines). }
\label{fig:sim_mod1}
\end{figure}
%%%%%%%%%%%%%%%%%%%%
%%%%%%%%%%%%%%%%%%%%
DGP 1 allows us to compare our data-driven BIC-based estimation procedure from Section~\ref{sec:PIM} (denoted as POI) with the estimation procedure of \cite{LMcK2009} (denoted as LMcK). \cite{LMcK2009} consider situations where $S=1$ is known and propose estimating the unknown parameters $\alpha, \beta_1$ and $\tau_1$ by simultaneously maximizing the likelihood over $\alpha$, $\beta_1$ and the grid points $t_j$. Our estimation procedure does not require knowledge about $S$, but profits from a situation where $S=1$ is known. Therefore, for reasons of comparability, we restrict the BIC-based model selection process to allow only for models containing one point of impact candidate. The simulation results are depicted in Figure \ref{fig:sim_mod1} and are virtually identical for both methods and show a satisfying behavior of the estimates. It should be noted, however, that our estimator is computationally advantageous as it greatly thins out the number of possible point of impact candidates by allowing only the local maxima of $|n^{-1}\sum_{i=1}^n Z_{\delta,i}(s)Y_i|$ as possible point of impact candidates. Our threshold-based estimation procedure leads to similar qualitative results. We omit these results, however, in order to allow for a clear display in Figure \ref{fig:sim_mod1}. The performance of our threshold-based procedure is reported in detail for the remaining simulation studies (DGP 2-5).

%%%%%%%%%%%
%% DGP 2 %%
%%%%%%%%%%%
%%%%%%%%%%%%%%%%%%%%
%%%%%%%%%%%%%%%%%%%%
\begin{figure}[!htb]
\centering
{\sc DGP 2: Estimation errors for $n\in\{100,200,500,1000,3000\}$}
\par\medskip
\includegraphics[width=1\textwidth]{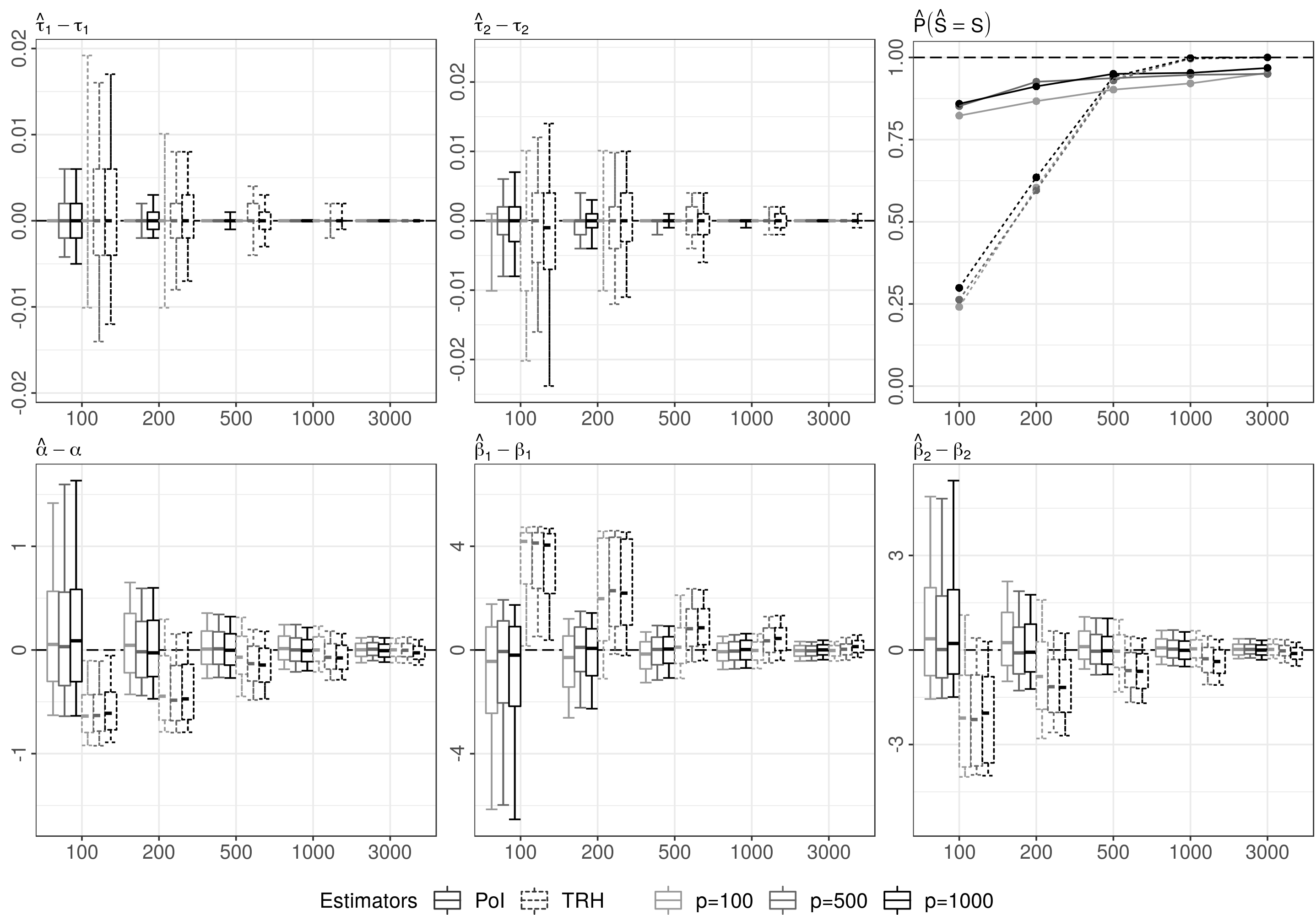}
\caption[Estimation errors for simulation model 2 (2 impact points)]{Comparison of the estimation errors from using our BIC-based method POI (solid lines) and our threshold-based method TRH (dashed lines). }
\label{fig:sim_mod2}
\end{figure}
%%%%%%%%%%%%%%%%%%%%
DGP 2 is more complex than DGP 1 since $S=2$ and considered unknown. Figure~\ref{fig:sim_mod2} compares the estimation errors from using our BIC-based POI estimator with those from our threshold-based estimator (denoted as TRH). For smaller sample sizes $n$, the POI estimator seems to be preferable to the TRH estimator. Although, estimating the locations of the points of impact $\tau_1$ and $\tau_2$ is quite accurate for both procedures, the number $S$ is estimated correctly more often  using the POI estimator (see upper right panel). The more precise estimation of $S$ when using the POI estimator results in essentially unbiased estimates of the parameters $\alpha$, $\beta_1$, and $\beta_2$. By contrast, the less precise estimation of $S$ using the TRH estimator leads to clearly visible omitted variable biases in the estimates of the parameters $\alpha$, $\beta_1$, and $\beta_2$. As the sample size increases, however, the accuracy of estimating $\widehat{S}$ improves for the TRH estimator such that both estimators show eventually a similar performance.

%%%%%%%%%%%
%% DGP 3 %%
%%%%%%%%%%%
DGP 3 with $S=4$ unknown points of impact comprises an even more complex situation than DGP 2. For reasons of space, Figure \ref{fig:sim_mod3} is referred to Appendix \ref{app:SIM_MORE}. It shows that the qualitative results from DGP 2 still hold. For large $n$, the POI and TRH estimators both lead to accurate estimates of the model parameters for all choices of $p$. As already observed in DGP 2, however, the TRH estimator leads to imprecise estimates of $S$ for small $n$, which results in omitted variables biases in the estimates of the parameters $\alpha$, $\beta_1$, $\beta_2$, $\beta_3$, and $\beta_4$. Because of the increased complexity of DGP 3, these biases are even more pronounced than in DGP 2. The reason for this is partly due to the construction of the TRH estimator, where we set the value of $\delta$ to $\delta=c_{\delta}n^{-1/2}$ with $c_{\delta}=1.5$. Asymptotically, the choice of $c_{\delta}$ has a negligible effect, but may be inappropriate for small $n$, since the estimation procedure eliminates all points within a $\sqrt{\delta}$-neighborhood around a chosen candidate $\widehat{\tau}_r$ (see Section \ref{SubSec:Estim}). For DGP 3, the choice of $c_{\delta}=1.5$ results in a too large $\sqrt{\delta}$-neighborhood, such that the estimation procedure also eliminates true point of impact locations for small $n$.
%%%%
By contrast, the POI estimator is able to avoid such adverse eliminations as the
BIC criterion is also minimized over $\delta$.

%%%%%%%%%%%
%% DGP 4 %%
%%%%%%%%%%%
%%%%%%%%%%%%%%%%%%%%
%%%%%%%%%%%%%%%%%%%%
\begin{figure}[!htb]
\centering
{\sc DGP 4: Estimation errors for $n\in\{100,200,500,1000,3000\}$}
\par\medskip
\includegraphics[width=1\textwidth]{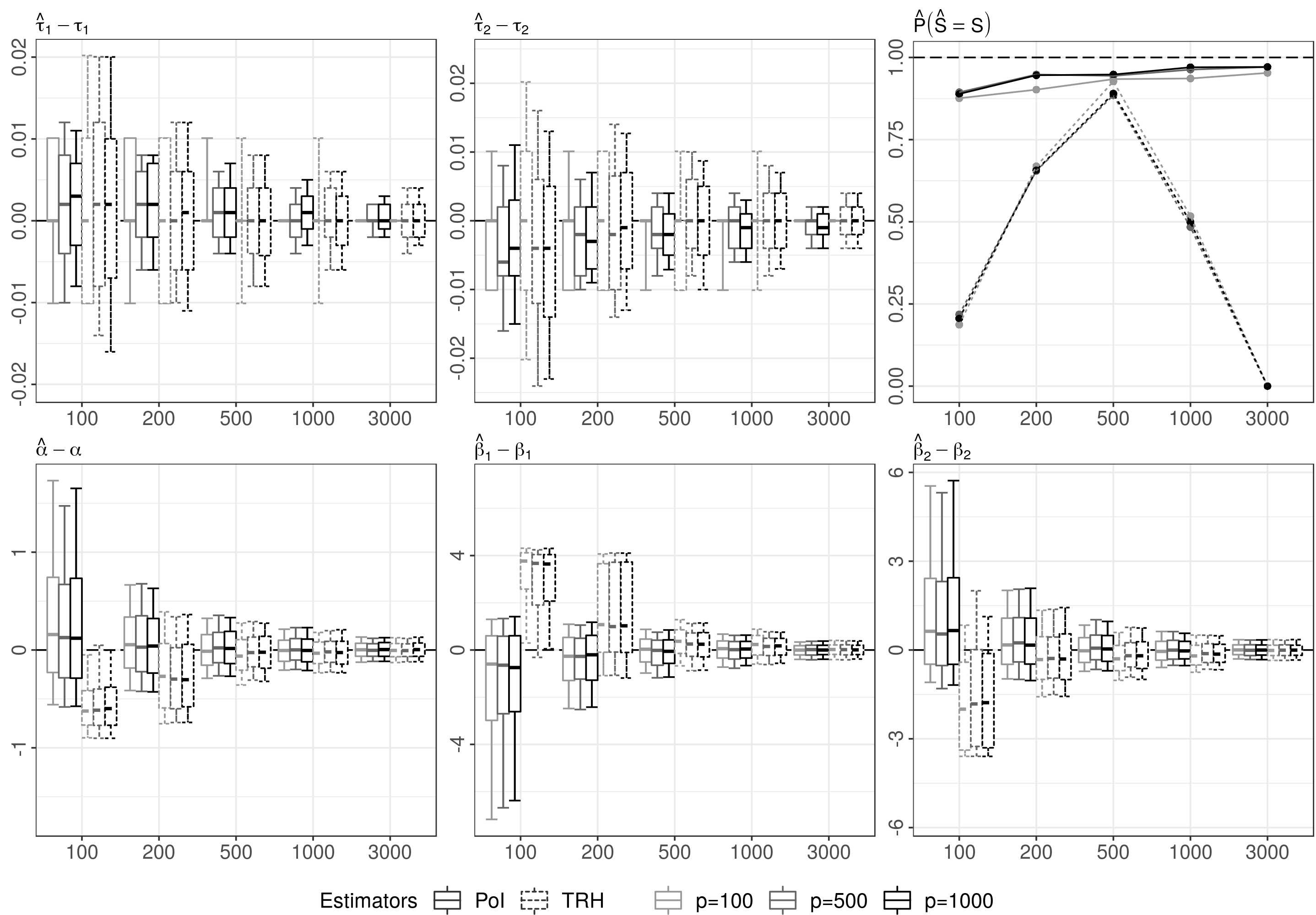}
\caption[]{Comparison of the estimation errors from using our BIC-based method POI (solid lines) and our threshold-based method TRH (dashed lines).}
\label{fig:sim_mod4}
\end{figure}
%%%%%%%%%%%%%%%%%%%%
%%%%%%%%%%%%%%%%%%%%
DGP 4 takes up the general setup of DGP 2, but the functional data $X_i$ are simulated using a Gaussian covariance model (GCM) which is characterized by an infinitely many times differentiable covariance function. This setting contradicts our basic Assumption~\ref{assum1}, but fits our remark at the end of Section \ref{ssec:theoreticalframe}.  From Figure~\ref{fig:sim_mod4} it can be concluded that even under the failure of Assumption~\ref{assum1}, both estimation procedures are capable of consistently estimating the points of impact and the model parameters. The TRH estimator, however, fails to estimate the number of points of impact $S$ even for large $n$, since the $\lambda$-threshold is tailored for situations under Assumption~\ref{assum1}. Here the TRH estimator is able to estimate the true points of impact, but additionally selects more and more redundant point of impact candidates as $n$ becomes large. That is, the TRH estimator becomes more a screening than a selection procedure which can be problematic in practice.
%%%%
By contrast, the POI estimator is able to avoid such redundant selections of point of impact candidates, as the BIC criterion only selects points of impact candidates if they result in a sufficiently large improvement of the model fit.

%%%%%%%%%%%
%% DGP 5 %%
%%%%%%%%%%%
DGP 5 also takes up the setup of DGP 2; however, the process $X_i$ is simulated as an exponential Brownian Motion (EBM) violating Assumption \ref{assum2}, but still satisfying Assumption~\ref{assum1}. Here we set the asymptotically negligible tuning parameter $c_\delta$ of the TRH estimator equal to 3.
%% See (also in our folder `Online_LookUps`):
%https://math.stackexchange.com/questions/2547076/covariance-function-of-brownian-motion-process-yt-x2t
The evolution of the estimation errors can be seen in Figure~\ref{fig:sim_mod5} in Appendix \ref{app:SIM_MORE}. The results are comparable with our previous simulations in DGP 2 and DGP 3, indicating that the estimation procedure is robust to at least some violations of Assumption~\ref{assum2}.

% \noindent\textbf{Resume:} Both of our point of impact estimation procedures POI and TRH work well (and fast) in our simulations. The effect of increasing $p$ is generally negligible for all considered sample sizes $n$. Estimates of $\tau_r$ are very accurate, especially if we keep in mind that the distance between two successive grid points is given by approximately $0.01(b-a)$, $0.002(b-a)$ and $0.001(b-a)$ for our choices of $p$. This high accuracy facilitates their use in subsequent parametric as well as nonparametric estimations. In small samples and for violations of the model assumptions, there is a clear advantage when using the POI-estimator.

%%%%%%%%%%%%%%%%%%%%%%%%%%
\subsection{Evaluation of the nonparametric estimation procedure}
%%%%%%%%%%%%%%%%%%%%%%%%%%
Table \ref{tab:NPsim} contains the simulation results for our nonparametric estimation procedure described in Section \ref{sec:NP}. We focus on the more challenging data generating processes, DGP2-5, with at least two points of impact and compare our nonparametric method with the Most-Predictive Design Points (MPDP) method of \cite{FHV2010}. To the best of our knowledge, the MPDP method is the only comparable method in the literature.  We tried hard to carry out the full simulation study for the MPDP method; however, \cite{FHV2010} use a brute force minimization approach based on cross-validation considering $2^p$ grid point combinations, which makes their method computationally extremely expensive.\footnote{Due to the high computational costs, the simulation study in \cite{FHV2010} is based on only 50 Monte Carlo replications. In a readme-file, provided at Frederic Ferraty's homepage, the authors report that one run with a dataset of $n=149$ curves and $p=700$ grid points lasts about 30 minutes.} For the MPDP method, we, therefore, had to limit the number of Monte Carlo replications to 500, the number of grid points to $p\in\{100,500\}$ and the sample sizes to $n\in\{100,200\}$.

\spacingset{1}
\begin{table}[!tb]
\centering
\caption[]{Mean Average Squared Errors (MASE) for the nonparametric estimator $\widehat{g}_{\widehat{\tau}}$}\label{tab:NPsim}
%\smallskip
\vspace{2ex}
\begin{minipage}[t]{.45\textwidth}
\centering
\begin{tabular}{ccccc}
\toprule
 & & & \multicolumn{2}{c}{MASE}\\
\cline{4-5}\\[-2ex]
DGP & $p$ & $n$ & TRH  & MPDP  \\
\midrule
2 &  100 & 100 & 0.098 & 0.100 \\
2 &  100 & 200 & 0.061 & 0.089 \\
2 &  100 & 500 & 0.017 &       \\
2 &  500 & 100 & 0.097 & 0.098 \\
2 &  500 & 200 & 0.064 & 0.092 \\
2 &  500 & 500 & 0.023 &       \\
2 & 1000 & 100 & 0.094 &       \\
2 & 1000 & 200 & 0.060 &       \\
2 & 1000 & 500 & 0.022 &       \\
\midrule
3 &  100 & 100 & 0.155 & 0.175 \\
3 &  100 & 200 & 0.105 & 0.156 \\
3 &  100 & 500 & 0.058 &       \\
3 &  500 & 100 & 0.150 & 0.173 \\
3 &  500 & 200 & 0.102 & 0.161 \\
3 &  500 & 500 & 0.060 &       \\
3 & 1000 & 100 & 0.149 &       \\
3 & 1000 & 200 & 0.100 &       \\
3 & 1000 & 500 & 0.059 &       \\
\bottomrule
\end{tabular}
\end{minipage}%
\begin{minipage}[t]{.45\textwidth}
  \centering
\begin{tabular}{ccccc}
  \toprule
  & & & \multicolumn{2}{c}{MASE}\\
  \cline{4-5}\\[-2ex]
DGP & $p$ & $n$ & TRH  & MPDP  \\
\midrule
4 &  100 & 100 & 0.089 & 0.093 \\
4 &  100 & 200 & 0.044 & 0.087 \\
4 &  100 & 500 & 0.011 &       \\
4 &  500 & 100 & 0.085 & 0.096 \\
4 &  500 & 200 & 0.045 & 0.087 \\
4 &  500 & 500 & 0.010 &       \\
4 & 1000 & 100 & 0.086 &       \\
4 & 1000 & 200 & 0.045 &       \\
4 & 1000 & 500 & 0.010 &       \\
\midrule
5 &  100 & 100 & 0.096 & 0.180 \\
5 &  100 & 200 & 0.089 & 0.177 \\
5 &  100 & 500 & 0.069 &       \\
5 &  500 & 100 & 0.094 & 0.179 \\
5 &  500 & 200 & 0.090 & 0.176 \\
5 &  500 & 500 & 0.069 &       \\
5 & 1000 & 100 & 0.092 &       \\
5 & 1000 & 200 & 0.091 &       \\
5 & 1000 & 500 & 0.066 &       \\
\bottomrule
\end{tabular}
\end{minipage}\\[2ex]%
%$^\ast$ Based only on 500 Monte Carlo replications.
\end{table}
\spacingset{1.5}

The results in Table \ref{tab:NPsim} show that the MASE decreases with increasing sample size $n$ and that the effect of different numbers of grid points $p$ is essentially negligible for both methods. The differences in the simulation results for the different data generating processes are generally equivalent to those discussed for the parametric estimation procedure. DGP 3 with its four points of impact is the most challenging case and, therefore, produces the largest estimation errors. The MPDP method of \cite{FHV2010} has throughout larger estimation errors than our nonparametric estimation results based on the TRH estimator (Algorithm \ref{Algo:1}). The larger estimation errors in $\widehat{g}$ of the MPDP method can be explained by its larger estimation errors when estimating the points of impact $\tau_1,\dots,\tau_S$ (see Table \ref{tab:NPsim2}). In fact, our super-consistent points of impact estimator has substantially smaller estimation errors (factor $1/10$ to $1/100$) than the MPDP method.

\spacingset{1}
\begin{table}[!htb]
\centering
\caption[]{Average Mean Squared Errors (AvgMSE) for $\widehat{\tau}_1,\dots,\widehat{\tau}_S$}\label{tab:NPsim2}
%\smallskip
\vspace{2ex}
\begin{minipage}[t]{.45\textwidth}
\centering
\begin{tabular}{ccccc}
\toprule
 & & & \multicolumn{2}{c}{AvgMSE}\\
\cline{4-5}\\[-2ex]
DGP & $p$ & $n$ & TRH   & MPDP\\
\midrule
  2 & 100  & 100 & 0.0002 & 0.0063 \\
  2 & 100  & 200 & 0.0001 & 0.0023 \\
  2 & 100  & 500 & 0.0000 &  \\
  2 & 500  & 100 & 0.0002 & 0.0084 \\
  2 & 500  & 200 & 0.0001 & 0.0013 \\
  2 & 500  & 500 & 0.0000 &  \\
  2 & 1000 & 100 & 0.0002 &  \\
  2 & 1000 & 200 & 0.0001 &  \\
  2 & 1000 & 500 & 0.0000 &  \\
  \midrule
  3 & 100  & 100 & 0.0002 & 0.0186 \\
  3 & 100  & 200 & 0.0001 & 0.0036 \\
  3 & 100  & 500 & 0.0000 &  \\
  3 & 500  & 100 & 0.0002 & 0.0218 \\
  3 & 500  & 200 & 0.0001 & 0.0035 \\
  3 & 500  & 500 & 0.0000 &  \\
  3 & 1000 & 100 & 0.0002 &  \\
  3 & 1000 & 200 & 0.0001 &  \\
  3 & 1000 & 500 & 0.0000 &  \\
\bottomrule
\end{tabular}
\hspace*{.25ex} Legend: $\operatorname{AvgMSE}=S^{-1}\sum_{l=1}^S\operatorname{MSE}(\widehat{\tau}_l)$
\end{minipage}%
\begin{minipage}[t]{.45\textwidth}
  \centering
\begin{tabular}{ccccc}
  \toprule
  & & & \multicolumn{2}{c}{AvgMSE}\\
  \cline{4-5}\\[-2ex]
DGP & $p$ & $n$ & TRH   & MPDP\\
\midrule
4 & 100  & 100 & 0.0003 & 0.0072 \\
4 & 100  & 200 & 0.0001 & 0.0029 \\
4 & 100  & 500 & 0.0001 &  \\
4 & 500  & 100 & 0.0002 & 0.0062 \\
4 & 500  & 200 & 0.0001 & 0.0023 \\
4 & 500  & 500 & 0.0000 &  \\
4 & 1000 & 100 & 0.0002 &  \\
4 & 1000 & 200 & 0.0001 &  \\
4 & 1000 & 500 & 0.0000 &  \\
\midrule
5 & 100  & 100 & 0.0004 & 0.0111 \\
5 & 100  & 200 & 0.0006 & 0.0025 \\
5 & 100  & 500 & 0.0002 &  \\
5 & 500  & 100 & 0.0004 & 0.0097 \\
5 & 500  & 200 & 0.0006 & 0.0009 \\
5 & 500  & 500 & 0.0001 &  \\
5 & 1000 & 100 & 0.0004 &  \\
5 & 1000 & 200 & 0.0007 &  \\
5 & 1000 & 500 & 0.0002 &  \\
\bottomrule
\end{tabular}
\end{minipage}\\[2ex]%
\end{table}
\spacingset{1.5}

%%%%%%%%%%%%%%%%%%%%%%%%%%%%%%%%%%%%%%%%%%%%%%%%%%%%%%%%%%%%%%%%%%%%%%
\section{Points of impact in continuous emotional stimuli}\label{sec:RDA}
%%%%%%%%%%%%%%%%%%%%%%%%%%%%%%%%%%%%%%%%%%%%%%%%%%%%%%%%%%%%%%%%%%%%%%
Current psychological research on emotional experiences increasingly includes continuous emotional stimuli such as videos to induce emotional states as an attempt to increase ecological validity \citep{TFH2009}. Asking participants to evaluate those stimuli is most often done using an overall rating such as ``How positive or negative did this video make you feel?''. Such global overall ratings are guided by the participant's affective experiences while watching the video \citep{S1999,MLC2005} which makes it crucial to identify the relevant parts of the stimulus impacting the overall rating in order to understand the emergence of emotional states and to make use of specific ``impacting'' parts of the stimuli.

Due to a lack of appropriate statistical methods, existing approaches use heuristics such as the ``peak-and-end rule'' in order to link the overall ratings with the continuous emotional stimuli. The peak-and-end rule states that people's evaluations can be well predicted using just two characteristics: the moment of emotional peak intensity and the ending of the emotional stimuli \citep{F2000}. Such a heuristic approach, however, is only of limited practical use. The peak intensity moment and the ending are not necessarily good predictors. Furthermore, the peak intensity moment can vary strongly across participants, which prevents linking the overall rating to specific moments in the continuous emotional stimuli that are of a \textit{common} relevance. 

% Data-Content
Our case study comprises data from $n=65$ participants, who were asked to continuously report their emotional state (from very negative to very positive) while watching a documentary video (112 sec.) on the persecution of African albinos.  
%The video does not contain emotionally arousing visual material, but the spoken words contain some emotionally arousing descriptions. 
A version of the video can be found online at YouTube.\footnote{Link to the video: \url{https://youtu.be/9F6UpuJIFaY}. The video clip used in the experiment corresponds approximately to the first 115 sec.~of the video at YouTube.} The first six data points ($<1$ sec.) are removed as they contain some obviously erratic components. Figure \ref{fig:appl_1} shows the standardized emotion trajectories $X_i(t_j)$, where $t_j$ are equidistant grid points within the unit-interval $0=t_1<\dots<t_p=1$ with $p=167$. After watching the video, the participants were asked to rate their final overall feeling. This overall rating was coded as a binary variable $Y_i\in\{0,1\}$, where $Y_i=0$ denotes ``I feel negative'' ($48\%$ of the participants) and $Y_i=1$ denotes ``I do not feel negative'' ($52\%$ of the participants).
%% Data collection:
The data were collected in May 2013. Participants were recruited through Amazon Mechanical Turk (\url{www.mturk.com}) and received 1USD for completing the ratings via the online survey platform SoSci Survey (\url{www.soscisurvey.de}). The study was approved by the local institutional review board (IRB, University of Colorado Boulder). The documentary video is taken from the Interdisciplinary Affective Science Laboratory Movie Set (Feldman Barrett, L., unpublished).

%%%%%%%%%%%%%%%%%%%%%%%%%%%%%%%%%%%%%%
%%%%%%%%%%%%%%%%%%%%%%%%%%%%%%%%%%%%%% 
\spacingset{1}
\begin{table}[!htb]
\addtolength{\tabcolsep}{-3.pt}
\centering
\caption[]{Estimation results.}\label{tab:appl_1} 
\smallskip
\begin{adjustbox}{max width=0.9\textwidth}
\begin{tabular}{l|cc|cc|cc|}
\toprule
\multicolumn{1}{l}{} & \multicolumn{2}{c}{\textbf{POI}}  &   \multicolumn{2}{c}{\textbf{PER-1}} & \multicolumn{2}{c}{\textbf{PER-2}}\\
                            \cmidrule(lr){2-3} \cmidrule(lr){4-5} \cmidrule(lr){6-7}
\multicolumn{1}{l}{\textbf{Regressor}} & \textbf{Coeff.} & \multicolumn{1}{c}{\textbf{(S.E.)}} &\textbf{Coeff.} & \multicolumn{1}{c}{\textbf{(S.E.)}}  & \textbf{Coeff.} & \multicolumn{1}{c}{\textbf{(S.E.)}}  \\
\midrule
$X(\widehat{\tau}_1)$       & $\phantom{-}1.16^{***}$   &$(0.41)$ &          &            &           &  \\
$X(\widehat{\tau}_2)$& $\phantom{-}0.71^{**\phantom{*}}$&$(0.32)$ &         &            &           &   \\
$X(p^{\operatorname{abs}})$ &                 &          & $\phantom{-}0.41$ & $(0.36)$  &           &   \\
$X(p^{\operatorname{pos}})$ &                 &          &          &            & $\phantom{-}0.46$  &$(0.29)$\\
$X(p^{\operatorname{neg}})$ &                 &          &          &            &  $\phantom{-}0.54$ &$(0.43)$\\
$X(1)$                      &                &          &  $\phantom{-}0.20$ & $(0.26)$  & $\phantom{-}0.04$   & $(0.28)$\\
Constant                    &  $-0.29\phantom{^{***}}$        & $(0.29)$& $-0.98$  & $(0.71) $ & $-0.29$   & $(0.73)$\\
\midrule
\multicolumn{1}{l}{Log Likelihood}    & \multicolumn{2}{c}{$-36.03$}           & \multicolumn{2}{c}{$-43.48$}           & \multicolumn{2}{c}{$-41.58$} \\
\multicolumn{1}{l}{Akaike Inf.~Crit.} & \multicolumn{2}{c}{$\phantom{-}78.07$} & \multicolumn{2}{c}{\phantom{-4}$92.96$} & \multicolumn{2}{c}{\phantom{-4}$91.15$} \\
\multicolumn{1}{l}{McFadden Pseudo-R$^2$}& \multicolumn{2}{c}{$\phantom{-7}0.19$}& \multicolumn{2}{c}{\phantom{-42}$0.03$}& \multicolumn{2}{c}{\phantom{-41}$0.07$} \\
\multicolumn{1}{l}{Somers' $D_{xy}$} & \multicolumn{2}{c}{$\phantom{-7}0.53$} & \multicolumn{2}{c}{\phantom{-42}$0.20$} & \multicolumn{2}{c}{\phantom{-41}$0.34$} \\
\bottomrule
\multicolumn{7}{c}{Note: $^{*}\text{pvalue}<0.1$; $^{**}\text{pvalue}<0.05$; $^{***}\text{pvalue}<0.01$} \\
\end{tabular}
\end{adjustbox}
\end{table}
\spacingset{1.5}
%%%%%%%%%%%%%%%%%%%%%%%%%%%%%%%%%%%%%%
%%%%%%%%%%%%%%%%%%%%%%%%%%%%%%%%%%%%%%

To analyze the data we use our parametric estimation procedure (Section \ref{sec:PES}) using a logit link function $g$ and the BIC-based selection of points of impact (Section \ref{sec:PIM}).  We compare our estimation procedure with the performance of the following two logit regression models based on peak-and-end rule (PER) predictor variables:
\begin{description}
\item[PER-1] Logit regression with peak intensity predictor $X_i(p^{\operatorname{abs}}_i)$ and the end-feeling predictor $X_i(1)$, where $p^{\operatorname{abs}}_i=\arg\max_t(|X_i(t)|)$
\item[PER-2] Logit regression with peak intensity predictors  $X_i(p^{\operatorname{pos}}_i)$ and $X_i(p^{\operatorname{neg}}_i)$ and end-feeling predictor $X_i(1)$, where
$p^{\operatorname{pos}}_i=\arg\max_t(X_i(t))$ and $p^{\operatorname{neg}}_i=\arg\min_t(X_i(t))$
\end{description}

Table \ref{tab:appl_1} shows the estimated coefficients, standard errors, as well as summary statistics for each of the three models, where our estimation procedure is denoted by POI. In comparison to our POI estimator, both benchmark models (PER-1 and PER-2) have significantly lower model fits (McFadden Pseudo R$^2$) and significantly lower predictive abilities (Somers' $D_{xy}$), where $D_{xy}=0$ means that a model is making random predictions and $D_{xy}=1$ means that a model discriminates perfectly.
%%%%%%%%%%%%%%%%%%%%%%%%%%%%%%%%%%%%%%%%% 
%%%%%%%%%%%%%%%%%%%%%%%%%%%%%%%%%%%%%%%%% 
\begin{figure}[!ht]
\centering
\includegraphics[width=.9\textwidth]{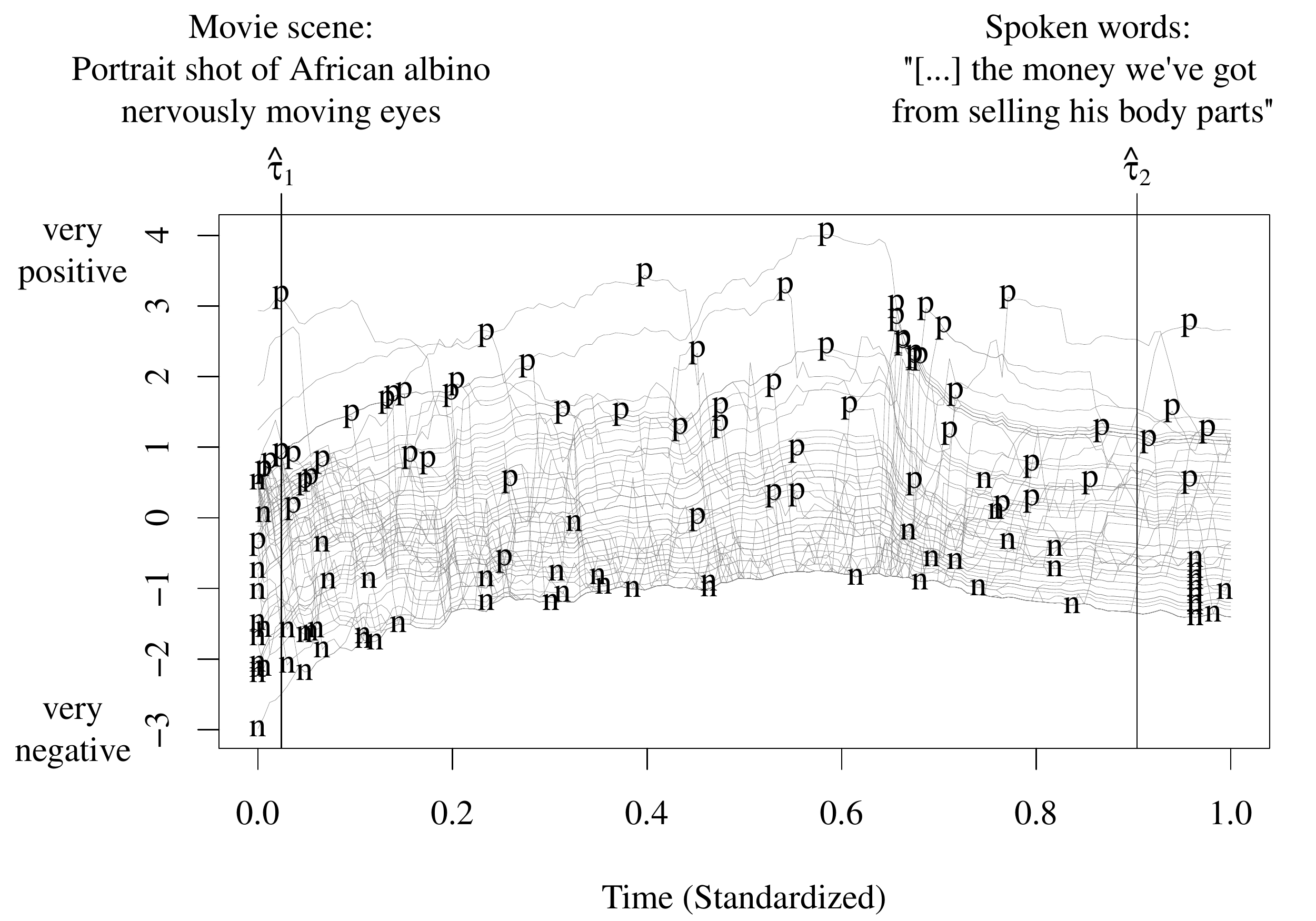}
\caption[]{Visualization of the impact points $\widehat{\tau}_1$ and $\widehat{\tau}_2$ (vertical lines).  The positive and the negative peak intensity predictors, $X_i(p_i^{\operatorname{pos}})$ and $X_i(p_i^{\operatorname{neg}})$, are marked by ``p'' and ``n''.}
\label{fig:appl_2}
\end{figure}
%%%%%%%%%%%%%%%%%%%%%%%%%%%%%%%%%%%%%%%%
%%%%%%%%%%%%%%%%%%%%%%%%%%%%%%%%%%%%%%%%

Figure \ref{fig:appl_2} shows the positive (p) and negative (n) peak intensity predictors, $X_i(p_i^{\operatorname{pos}})$ and $X_i(p_i^{\operatorname{neg}})$, for all participants; the absolute intensity predictors, $X_i(p_i^{\operatorname{abs}})$, form a subset of these. The peak intensity predictors are distributed across the total domain and, therefore, do not allow linking the overall ratings $Y_i$ to specific common time points $t\in[0,1]$ in the continuous emotional stimuli. By contrast, the estimated points of impact $\widehat{\tau}_1$ and $\widehat{\tau}_2$ allow for such a link and point to two emotionally arousing movie scenes: 
\begin{description}
\item[$\widehat{\tau}_1$:] Portrait shot of the traumatized African albino protagonist nervously moving eyes.
\item[$\widehat{\tau}_2$:] Spoken words: ``[\dots]the money we've got from selling his body parts.''  
\end{description}

\bigskip

\noindent\textbf{Supplementary materials.} The online supplementary materials \citep{Suppl2019} include the supplementary paper containing additional simulation results and the proofs of our theoretical results, the \textsf{R}-package \texttt{fdapoi} and \textsf{R}-scripts for reproducing our main empirical results.

\bigskip

\noindent\textbf{Acknowledgments.} The online rating tool for the data collection was kindly provided by Dominik Leiner (SoSci Survey, Germany).  Many thanks go to the Editor, the Associate Editor and two anonymous referees whose constructive comments helped us to improve our manuscript and motivated Section \ref{ssec:general}.

\bigskip

\noindent\textbf{Fundings.} Data collection was funded by the National Institutes of Health Director's Pioneer Award (DP1OD003312) to Lisa Feldman Barrett, and the National Institute on Drug Abuse grant (R01DA035484) to Tor D.~Wager. The development of the Interdisciplinary Affective Science Laboratory (IASLab) Movie Set was supported by a grant from the U.S.~Army Research Institute for the Behavioral and Social Sciences (W5J9CQ-11-C-0046) to Lisa Feldman Barrett and Tor D.~Wager. The views, opinions, and/or findings contained in this paper are those of the authors and shall not be construed as an official U.S.~Department of the Army position, policy, or decision, unless so designated by other documents.

\spacingset{1}
%%%%%%%%%%%%%%%%%%%%%%%%%%%%%
% References
%%%%%%%%%%%%%%%%%%%%%%%%%%%%%
\bibliographystyle{Chicago}
\bibliography{bibfile}

\begin{thebibliography}{}

\bibitem[\protect\citeauthoryear{Aneiros and Vieu}{Aneiros and
  Vieu}{2014}]{GP2014}
Aneiros, G. and P.~Vieu (2014).
\newblock Variable selection in infinite-dimensional problems.
\newblock {\em Statistics \& Probability Letters\/}~{\em 94}, 12--20.

\bibitem[\protect\citeauthoryear{Berrendero, Bueno-Larraz, and
  Cuevas}{Berrendero et~al.}{2019}]{BBC2017}
Berrendero, J.~R., B.~Bueno-Larraz, and A.~Cuevas (2019).
\newblock An {RKHS} model for variable selection in functional linear
  regression.
\newblock {\em Journal of Multivariate Analysis\/}~{\em 170}, 25--45.

\bibitem[\protect\citeauthoryear{Boente, Barrera, and Tyler}{Boente
  et~al.}{2014}]{BBT2014}
Boente, G., M.~S. Barrera, and D.~E. Tyler (2014).
\newblock A characterization of elliptical distributions and some optimality
  properties of principal components for functional data.
\newblock {\em Journal of Multivariate Analysis\/}~{\em 131}, 254--264.

\bibitem[\protect\citeauthoryear{Calcagno}{Calcagno}{2013}]{GLMULTI}
Calcagno, V. (2013).
\newblock {\em glmulti: Model selection and multimodel inference made easy}.
\newblock R package version 1.0.7.

\bibitem[\protect\citeauthoryear{Dagsvik and Str{\o}m}{Dagsvik and
  Str{\o}m}{2006}]{DS2006}
Dagsvik, J.~K. and S.~Str{\o}m (2006).
\newblock Sectoral labour supply, choice restrictions and functional form.
\newblock {\em Journal of Applied Econometrics\/}~{\em 21\/}(6), 803--826.

\bibitem[\protect\citeauthoryear{Embrechts and Maejima}{Embrechts and
  Maejima}{2000}]{EM2000}
Embrechts, P. and M.~Maejima (2000).
\newblock An introduction to the theory of self-similar stochastic processes.
\newblock {\em International Journal of Modern Physics B\/}~{\em 14\/}(12n13),
  1399--1420.

\bibitem[\protect\citeauthoryear{Ferraty, Hall, and Vieu}{Ferraty
  et~al.}{2010}]{FHV2010}
Ferraty, F., P.~Hall, and P.~Vieu (2010).
\newblock Most-predictive design points for functional data predictors.
\newblock {\em Biometrika\/}~{\em 97\/}(4), 807--824.

\bibitem[\protect\citeauthoryear{Floriello and Vitelli}{Floriello and
  Vitelli}{2017}]{FV2017}
Floriello, D. and V.~Vitelli (2017).
\newblock Sparse clustering of functional data.
\newblock {\em Journal of Multivariate Analysis\/}~{\em 154}, 1--18.

\bibitem[\protect\citeauthoryear{Fredrickson}{Fredrickson}{2000}]{F2000}
Fredrickson, B.~L. (2000).
\newblock Extracting meaning from past affective experiences: The importance of
  peaks, ends, and specific emotions.
\newblock {\em Cognition \& Emotion\/}~{\em 14\/}(4), 577--606.

\bibitem[\protect\citeauthoryear{Kneip, Po{\ss}, and Sarda}{Kneip
  et~al.}{2016}]{KPS2015}
Kneip, A., D.~Po{\ss}, and P.~Sarda (2016).
\newblock Functional linear regression with points of impact.
\newblock {\em The Annals of Statistics\/}~{\em 44\/}(1), 1--30.

\bibitem[\protect\citeauthoryear{Lee and Ready}{Lee and Ready}{1991}]{LR1991}
Lee, C. and M.~J. Ready (1991).
\newblock Inferring trade direction from intraday data.
\newblock {\em The Journal of Finance\/}~{\em 46\/}(2), 733--746.

\bibitem[\protect\citeauthoryear{Levina, Wagaman, Callender, Mandair, and
  Morris}{Levina et~al.}{2007}]{LWCM2007}
Levina, E., A.~Wagaman, A.~Callender, G.~Mandair, and M.~Morris (2007).
\newblock Estimating the number of pure chemical components in a mixture by
  maximum likelihood.
\newblock {\em Journal of Chemometrics\/}~{\em 21\/}(1-2), 24--34.

\bibitem[\protect\citeauthoryear{Liebl, Rameseder, and Rust}{Liebl
  et~al.}{2020}]{LRR2019}
Liebl, D., S.~Rameseder, and C.~Rust (2020).
\newblock Improving estimation in functional linear regression with points of
  impact: Insights into {Google AdWords}.
\newblock {\em Working Paper arXiv:1709.02166\/}.

\bibitem[\protect\citeauthoryear{Lindquist}{Lindquist}{2012}]{L2012}
Lindquist, M.~A. (2012).
\newblock Functional causal mediation analysis with an application to brain
  connectivity.
\newblock {\em Journal of the American Statistical Association\/}~{\em
  107\/}(500), 1297--1309.

\bibitem[\protect\citeauthoryear{Lindquist and McKeague}{Lindquist and
  McKeague}{2009}]{LMcK2009}
Lindquist, M.~A. and I.~W. McKeague (2009).
\newblock Logistic regression with brownian-like predictors.
\newblock {\em Journal of the American Statistical Association\/}~{\em
  104\/}(488), 1575--1585.

\bibitem[\protect\citeauthoryear{Mauss, Levenson, McCarter, Wilhelm, and
  Gross}{Mauss et~al.}{2005}]{MLC2005}
Mauss, I.~B., R.~W. Levenson, L.~McCarter, F.~H. Wilhelm, and J.~J. Gross
  (2005).
\newblock The tie that binds? {C}oherence among emotion experience, behavior,
  and physiology.
\newblock {\em Emotion\/}~{\em 5\/}(2), 175.

\bibitem[\protect\citeauthoryear{McCullagh}{McCullagh}{1983}]{MC1983}
McCullagh, P. (1983).
\newblock Quasi-likelihood functions.
\newblock {\em The Annals of Statistics\/}~{\em 11\/}(1), 59--67.

\bibitem[\protect\citeauthoryear{McCullagh and Nelder}{McCullagh and
  Nelder}{1989}]{MCN1989}
McCullagh, P. and J.~Nelder (1989).
\newblock {\em Generalized Linear Models\/} (2 ed.).
\newblock Monographs on Statistics \& Applied Probability (37). Chapman and
  Hall/CRC.

\bibitem[\protect\citeauthoryear{McKeague and Sen}{McKeague and
  Sen}{2010}]{LMcKB2010}
McKeague, I.~W. and B.~Sen (2010).
\newblock Fractals with point impact in functional linear regression.
\newblock {\em The Annals of Statistics\/}~{\em 38\/}(4), 2559--2586.

\bibitem[\protect\citeauthoryear{M\"{u}ller and Stadtm\"{u}ller}{M\"{u}ller and
  Stadtm\"{u}ller}{2005}]{MS2005}
M\"{u}ller, H.-G. and U.~Stadtm\"{u}ller (2005).
\newblock Generalized functional linear models.
\newblock {\em The Annals of Statistics\/}~{\em 33\/}(2), 774--805.

\bibitem[\protect\citeauthoryear{Park, Aston, and Ferraty}{Park
  et~al.}{2016}]{PJF2016}
Park, A.~Y., J.~A. Aston, and F.~Ferraty (2016).
\newblock Stable and predictive functional domain selection with application to
  brain images.
\newblock {\em Working Paper arXiv:1606.02186\/}.

\bibitem[\protect\citeauthoryear{Po{\ss}, Liebl, Kneip, Eisenbarth, Wager, and
  Feldman~Barrett}{Po{\ss} et~al.}{2020}]{Suppl2019}
Po{\ss}, D., D.~Liebl, A.~Kneip, H.~Eisenbarth, T.~D. Wager, and
  L.~Feldman~Barrett (2020).
\newblock Supplement to ``{Super}-consistent estimation of points of impact in
  nonparametric regression with functional predictors''.

\bibitem[\protect\citeauthoryear{{R Core Team}}{{R Core Team}}{2020}]{Rcite}
{R Core Team} (2020).
\newblock {\em R: A language and environment for statistical computing}.
\newblock Vienna, Austria: R Foundation for Statistical Computing.

\bibitem[\protect\citeauthoryear{Rohlfs, Harrigan, and Nielsen}{Rohlfs
  et~al.}{2013}]{RHN2013}
Rohlfs, R.~V., P.~Harrigan, and R.~Nielsen (2013).
\newblock Modeling gene expression evolution with an extended
  {Ornstein}-{Uhlenbeck} process accounting for within-species variation.
\newblock {\em Molecular Biology and Evolution\/}~{\em 31\/}(1), 201--211.

\bibitem[\protect\citeauthoryear{Schubert}{Schubert}{1999}]{S1999}
Schubert, E. (1999).
\newblock Measuring emotion continuously: Validity and reliability of the
  two-dimensional emotion-space.
\newblock {\em Australian Journal of Psychology\/}~{\em 51\/}(3), 154--165.

\bibitem[\protect\citeauthoryear{Sobel and Lindquist}{Sobel and
  Lindquist}{2014}]{SL2014}
Sobel, M.~E. and M.~A. Lindquist (2014).
\newblock Causal inference for fmri time series data with systematic errors of
  measurement in a balanced on/off study of social evaluative threat.
\newblock {\em Journal of the American Statistical Association\/}~{\em
  109\/}(507), 967--976.

\bibitem[\protect\citeauthoryear{Stein}{Stein}{1999}]{S1999_book}
Stein, M. (1999).
\newblock {\em Interpolation of Spatial Data: Some Theory for Kriging}.
\newblock Springer Series in Statistics. Springer.

\bibitem[\protect\citeauthoryear{Trautmann, Fehr, and Herrmann}{Trautmann
  et~al.}{2009}]{TFH2009}
Trautmann, S.~A., T.~Fehr, and M.~Herrmann (2009).
\newblock Emotions in motion: Dynamic compared to static facial expressions of
  disgust and happiness reveal more widespread emotion-specific activations.
\newblock {\em Brain Research\/}~{\em 1284}, 100--115.

\bibitem[\protect\citeauthoryear{Zhang}{Zhang}{2012}]{Z2012}
Zhang, Y. (2012).
\newblock {\em Sparse selection in Cox models with functional predictors}.
\newblock Ph.\ D. thesis.

\end{thebibliography}


\begin{thebibliography}{}

\bibitem[\protect\citeauthoryear{Brillinger}{Brillinger}{2012a}]{B2012a}
Brillinger, D.~R. (2012a).
\newblock A generalized linear model with ``gaussian'' regressor variables.
\newblock In {\em Selected Works of David Brillinger}, pp.\  589--606.
  Springer.

\bibitem[\protect\citeauthoryear{Brillinger}{Brillinger}{2012b}]{B2012b}
Brillinger, D.~R. (2012b).
\newblock The identification of a particular nonlinear time series system.
\newblock In {\em Selected Works of David Brillinger}, pp.\  607--613.
  Springer.

\bibitem[\protect\citeauthoryear{Fahrmeir and Kaufmann}{Fahrmeir and
  Kaufmann}{1985}]{FK1985}
Fahrmeir, L. and H.~Kaufmann (1985).
\newblock Consistency and asymptotic normality of the maximum likelihood
  estimator in generalized linear models.
\newblock {\em The Annals of Statistics\/}~{\em 13\/}(1), 342--368.

\bibitem[\protect\citeauthoryear{Johnstone and Lu}{Johnstone and
  Lu}{2009}]{JL2009}
Johnstone, I.~M. and A.~Y. Lu (2009).
\newblock On consistency and sparsity for principal components analysis in high
  dimensions.
\newblock {\em Journal of the American Statistical Association\/}~{\em
  104\/}(486), 682--693.

\bibitem[\protect\citeauthoryear{Kneip, Po{\ss}, and Sarda}{Kneip
  et~al.}{2016}]{KPS_S_2015}
Kneip, A., D.~Po{\ss}, and P.~Sarda (2016).
\newblock Supplement to ``{F}unctional linear regression with points of
  impact''.
\newblock {\em The Annals of Statistics\/}.

\bibitem[\protect\citeauthoryear{Li and Racine}{Li and Racine}{2006}]{LR06}
Li, Q. and J.~Racine (2006).
\newblock {\em Nonparametric Econometrics: Theory and Practice}.
\newblock Princeton University Press.

\bibitem[\protect\citeauthoryear{Liu}{Liu}{1994}]{L1994}
Liu, J.~S. (1994).
\newblock Siegel's formula via stein's identities.
\newblock {\em Statistics \& Probability Letters\/}~{\em 21\/}(3), 247--251.

\bibitem[\protect\citeauthoryear{Stein}{Stein}{1981}]{S1981}
Stein, C.~M. (1981).
\newblock Estimation of the mean of a multivariate normal distribution.
\newblock {\em The Annals of Statistics\/}~{\em 9\/}(6), 1135--1151.

\bibitem[\protect\citeauthoryear{van~de Geer and Lederer}{van~de Geer and
  Lederer}{2013}]{vandeGeer2013}
van~de Geer, S. and J.~Lederer (2013).
\newblock The bernstein-orlicz norm and deviation inequalities.
\newblock {\em Probability Theory and Related Fields\/}~{\em 157\/}(1),
  225--250.

\bibitem[\protect\citeauthoryear{van~der Vaart and Wellner}{van~der Vaart and
  Wellner}{1996}]{vanderVaart1996}
van~der Vaart, A.~W. and J.~A. Wellner (1996).
\newblock {\em Weak Convergence and Empirical Processes: With Applications to
  Statistics}.
\newblock Springer.

\end{thebibliography}
%%%%%%%%%%%%%%%%%%%%%%%%%%%%%

%%%%%%%%%%%%%%%%%%%%%%%%%%%%%%%%%%%%%%%%
% Appendix-Supplement Paper
%%%%%%%%%%%%%%%%%%%%%%%%%%%%%%%%%%%%%%%%
\newpage

\setcounter{page}{1}
\thispagestyle{plain}
\pagenumbering{Roman}

\appendix

\vspace*{.5cm}
\spacingset{1}
%%%%%%%%%%%%%%%%%%%%%%%%%%%%%%%%%%%%%%%%%%

\if0\blind
{ \begin{center}
  {\LARGE \bf Supplement to\\[1ex]
  ``Super-Consistent Estimation of Points of\\[.5ex]
  Impact in Nonparametric Regression with\\[1ex]
  Functional Predictors''}

  \bigskip

  %%%%%%%%%%%%%%
  Dominik Po{\ss}{}\\Institute of Finance and Statistics, University of Bonn, Bonn, Germany\\[2ex]
  %%%%%%%
  Dominik Liebl\\Institute of Finance and Statistics and Hausdorff Center for Mathematics, University of Bonn, Bonn, Germany\\[2ex]
  %%%%%%%
  Alois Kneip\\Institute of Finance and Statistics and Hausdorff Center for Mathematics, University of Bonn, Bonn, Germany\\[2ex]
  %%%%%%%
  Hedwig Eisenbarth\\School of Psychology, Victoria University of Wellington, Wellington, New Zealand\\[2ex]
  %%%%%%%
  Tor D.~Wager\\Presidential Cluster in Neuroscience and Department of Psychological and Brain Sciences, Dartmouth College, Hanover, New Hampshire, USA\\[1ex]and\\[1ex]
  %%%%%%
  Lisa Feldman Barrett\\
  Department of Psychology, Northeastern University and Department of Psychiatry, Massachusetts General Hospital/Harvard Medical School, Boston, Massachusetts, USA; and Athinoula A.~Martinos Center for Biomedical Imaging, Massachusetts General Hospital, Charlestown, Massachusetts, USA
  %%%%%%%%%%%%%%
  \end{center}
} \fi

\if1\blind
{
  \bigskip
  \bigskip
  \bigskip
  \begin{center}
    {\LARGE \bf Supplement to\\[1ex] ``Points of Impact in Generalized Linear\\[1ex]Models with Functional Predictors''}
\end{center}
  \medskip
} \fi

\bigskip

%%%%%%%%%%%%%%%%%%%%%%%%%%%%%%%%%%%%%%%%%%%%%%%%%%%%%%%%%%%%%%%%%%
\section{Additional simulation results}\label{app:SIM_MORE}
%%%%%%%%%%%%%%%%%%%%%%%%%%%%%%%%%%%%%%%%%%%%%%%%%%%%%%%%%%%%%%%%%%
This appendix contains the additional simulation results discussed in Section \ref{sec:SIM} of the main paper. Figure~\ref{fig:sim_mod3} depicts the results for DGP 3 and Figure~\ref{fig:sim_mod5} illustrates the results for DGP 5.

%%%%%%%%%%%
%% DGP 3 %%
%%%%%%%%%%%
\begin{sidewaysfigure}  
%\begin{figure}[!htb]
\centering
{\sc DGP 3: Estimation errors for $n\in\{100,200,500,1000,3000\}$}
\par\medskip
\includegraphics[width=0.8\textwidth]{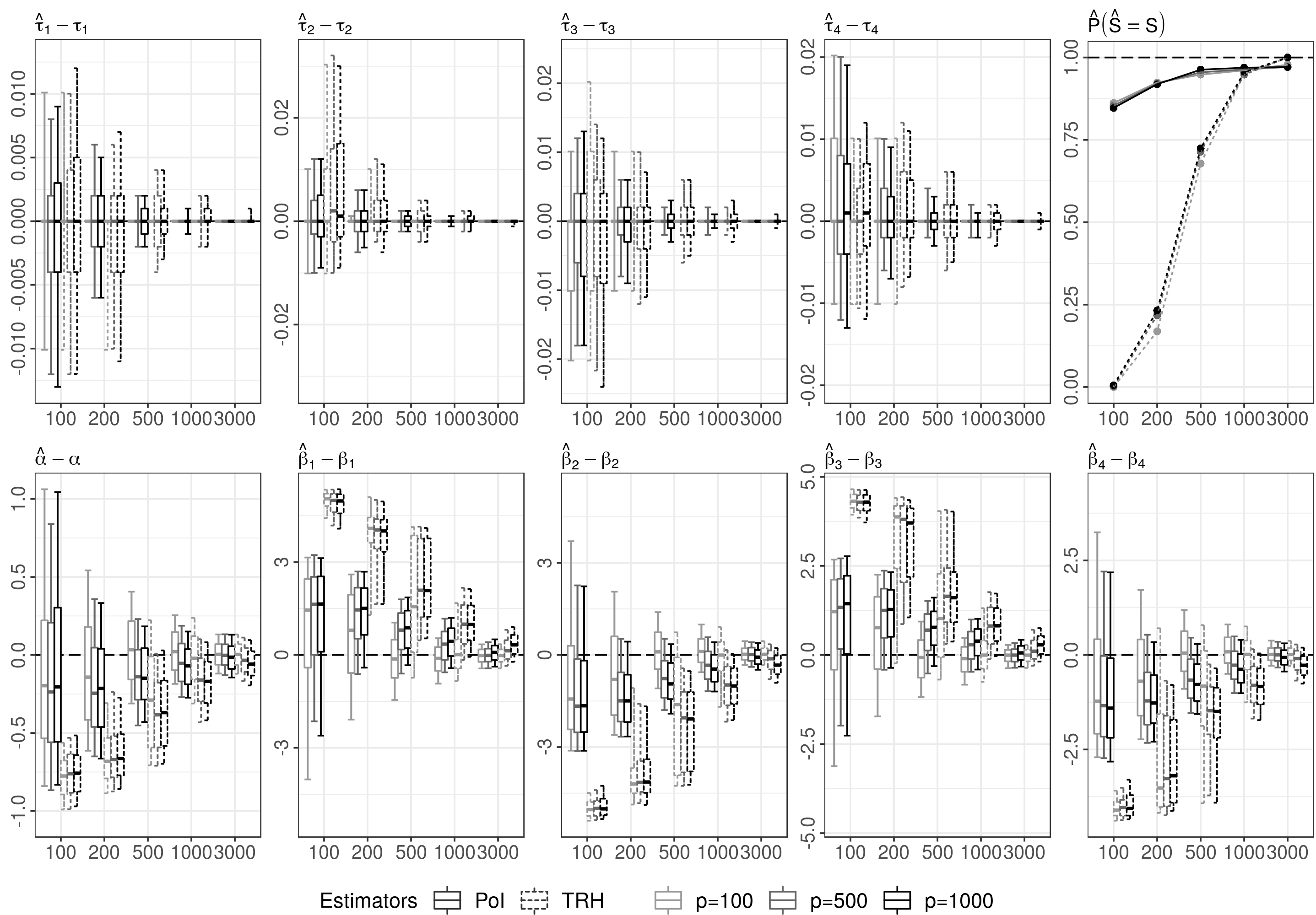}
\caption[Estimation errors for DGP 3 (4 impact points, BIC vs TRH)]{Comparison of the estimation errors from using our BIC-based method POI (solid lines) and our threshold-based method TRH (dashed lines). }
\label{fig:sim_mod3}
%\end{figure}
\end{sidewaysfigure}

%%%%%%%%%%%
%% DGP 5 %%
%%%%%%%%%%%
\begin{sidewaysfigure}
%\begin{figure}[!htb]
\centering
{\sc DGP 5: Estimation errors for $n\in\{100,200,500,1000,3000\}$}
\par\medskip
\includegraphics[width=0.8\textwidth]{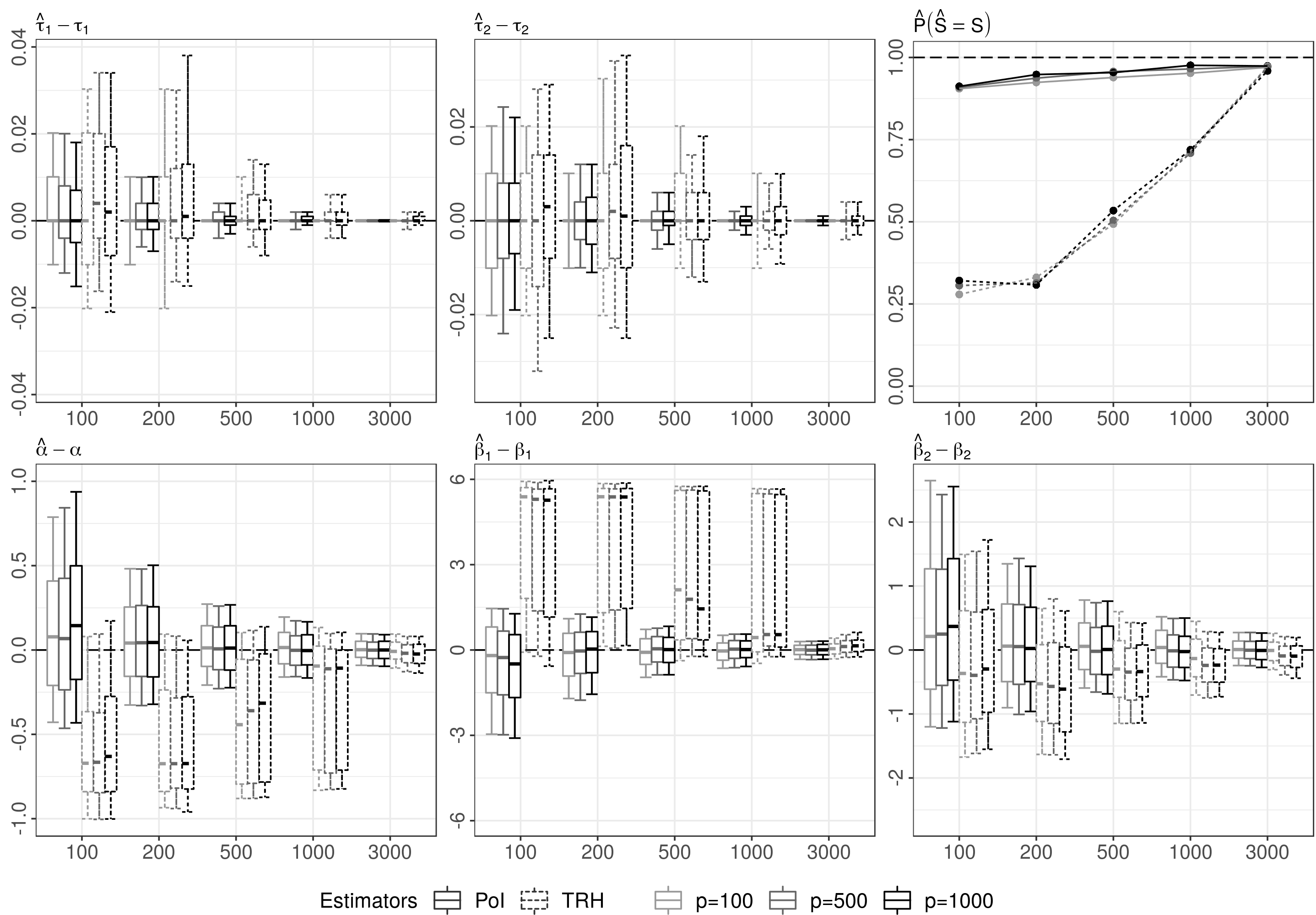}
\caption[Estimation errors for DGP 5 (2 impact points, BIC vs TRH, EBM)]{Comparison of the estimation errors from using our BIC-based method POI (solid lines) and our threshold-based method TRH (dashed lines). }
\label{fig:sim_mod5}
%\end{figure}
\end{sidewaysfigure}

\newpage

%%%%%%%%%%%%%%%%%%%%%%%%%%%%%%%%%%%%%%%%%%%%%%%%%%%%%%%%%%%%%%%%%%%%%%
\section{Identifying points of impact}\label{app:poiID}
%%%%%%%%%%%%%%%%%%%%%%%%%%%%%%%%%%%%%%%%%%%%%%%%%%%%%%%%%%%%%%%%%%%%%%

%When assuming a Gaussian process, Theorem \ref{thm:NP1} rests upon a modification of the famous Stein's lemma:

\subsection{Proof of Theorem \ref{thm:NP1}}\label{app:poiID-1}
%% %%%%%%%%%%%%%%%%%%%%%%%%%%%%%%%%%%%%%%%%%%%%%%%%%%%%%%%%
%% Proof of ``Stein's Lemma''
%% To do: Der Beweis liest sich noch nicht flüssig. Füllwörter nutzen...
%% %%%%%%%%%%%%%%%%%%%%%%%%%%%%%%%%%%%%%%%%%%%%%%%%%%%%%%%%
\begin{proof}[{\bf Proof of Theorem \ref{thm:NP1}}]
Since $X_i$ is a Gaussian process satisfying $\E(X_i(s))=0$ for all $s \in [a,b]$, $\bX_{s,i}=(X_i(s),X_i(\tau_1),\dots, X_i(\tau_S))^T$ follows for all $s\in[a,b]$ a multivariate normal distribution with mean zero and covariance matrix $\bSigma_s=\E(\bX_{s,i}\bX_{s,i}^T)$.

\noindent Let $\mathbf{Z}_{s,i}$ be an $(S+1)$-dimensional standard normal distributed random vector, then $\bX_{s,i}=\bSigma_s^{1/2}\mathbf{Z}_{s,i}$. Furthermore, define $\tilde{g}(\bX_{s,i})=g(X_i(\tau_1),\dots, X_i(\tau_S))$ for all $s\in[a,b]$. It then follows from the proof of Lemma 1 in \citeappendix{L1994} that
%\noindent Let $\mathbf{Z}_{s,i}$ be an $(S+1)$-dimensional standard normal distributed random vector, then $\bX_{s,i}=\bSigma_s^{1/2}\mathbf{Z}_{s,i}$. Define $\tilde{g}(\bX_{s,i})=g(X_i(\tau_1),\dots, X_i(\tau_S))$ for all $s\in[a,b]$. It then follows from the proof of Lemma 1 in \citeappendix{L1994} that
\begin{align*}
\E &\big(\bX_{s,i}g(X_i(\tau_1),\dots, X_i(\tau_S))\big)=
\E\big(\bX_{s,i}\tilde{g}(\bX_{s,i})\big)=\\
&\bSigma_s\E\Big(\big(0,\frac{\partial}{\partial x_1}g(X_i(\tau_1),\dots,X_i(\tau_S)),\dots,\frac{\partial}{\partial x_S}g(X_i(\tau_1),\dots,X_i(\tau_S))\big)\Big)^T.
\end{align*}
Specifically,
\begin{align*}
&\E\big(X_{i}(s)Y_i\big)=
 \E\big(X_{i}(s)g(X_i(\tau_1),\dots, X_i(\tau_S))\big)=\sum_{r=1}^S\sigma(s,\tau_r) \E\big(\frac{\partial}{\partial x_r}g(X_i(\tau_1),\dots,X_i(\tau_S)\big).
\end{align*}
Setting $\vartheta_r  = \E(\frac{\partial}{\partial x_r}g(X_i(\tau_1),\dots,X_i(\tau_S))$ completes the proof.
\end{proof}

%%%%%%%%%%%%%%%%%%%%%%%%%%%%%%%%%%%%%%%%%%%%%%%%%%%%%%%%%%%%%%%%%%%%%%%%%%%%
\subsection{Identification for non-Gaussian processes}\label{app:poiID-2}
%%%%%%%%%%%%%%%%%%%%%%%%%%%%%%%%%%%%%%%%%%%%%%%%%%%%%%%%%%%%%%%%%%%%%%%%%%%%

So far, our strategy for identifying and estimating points of impact rests upon assuming a non-smooth Gaussian process. Theorem \ref{thm:NP1} together with Assumption \ref{assum1} then implies that $\E\big(X_{i}(s)Y_i\big)$ is not twice differentiable at the points of impact $\tau_r$, $r=1,\dots,S$.  For the auxiliary process $Z_{\delta,i}$ defined  in Section \ref{sec:POI} we then obtain peaks of the function $\E\big(Z_{\delta,i}(s)Y_i\big)$ for $s\in\{\tau_1,\dots,\tau_S\}$.

In the following we will show that the Gaussian assumption can be relaxed. We will provide more general assumptions under which the same identification strategy can be pursued.
% We want to note that subsequent arguments only refer to this point.
% Deriving an asymptotic theory for the resulting estimators will require additional assumptions when leaving the Gaussian setup.
The generalization is based on the idea that realizations of  $X_i$ may depend on some latent random variable $V_i$ such that the conditional distributions of $X_i$ given $V_i=v$ are Gaussian. The (unconditional) distribution of $X_i$ then additionally depends on the distribution of $V_i$ and may be far from Gaussian.

A simple example are {\bf elliptical} processes: For instance, for some strictly positive real-valued random variable $V_i>0$ with $E(V_i^2)<\infty$ it holds
\begin{equation}
X_i(t)=V_i X_i^*(t),  \quad V_i \text{  independent of } X_i^*,
\label{poiID 1}
\end{equation}
where $X_i^*(t)$ is a zero mean Gaussian process with covariance function $\sigma^*(s,t)$. In this case the conditional distribution of $X_i$ given $V_i=v$ is Gaussian with mean zero and covariance function
$v^2\sigma^*(s,t)$.

A more general framework is given by the following condition:

\begin{itemize}
\item[A)] $X_i$ is a zero mean stochastic process on $[a,b]$. %with covariance function $\sigma(s,t)$.
Realizations of the process $X_i$ depend on the realizations of a latent random variable $V_i$ defined on a metric space ${\cal V}$. The joint distribution of $(X_i,V_i)$ is such that for each $v\in {\cal V}$ the conditional distribution of $X_i$ given $V_i=v$ is {\bf Gaussian} with conditional mean function
$$\mu(s;v)=E(X_i(s)|V_i=v), \quad s\in [a,b]$$
and continuous conditional covariance function
$$\sigma(s,t;v):=\E\left((X_i(s)-\mu(s;v))(X_i(t)-\mu(t;v))\big| V_i=v\right)<\infty,\quad (s,t)\in [a,b]^2$$
Moreover, the error term $\varepsilon_i$ in \eqref{eq:whnp} is independent of $V_i$ and $X_i$, and for all $r=1,\dots,S$
\begin{equation*}
\vartheta_r(v)  =
\E\big(\frac{\partial}{\partial x_r}g(X_i(\tau_1),\dots,X_i(\tau_S))\big|\ V_i=v\big)<\infty
%\label{poiID 3}
\end{equation*}
is a measurable function of $v\in {\cal V}$. In the following we will additionally assume that the joint distribution of $(X_i,V_i)$ is such that
all conditional and unconditional expectations used in subsequent arguments exist. This will go without saying.
\end{itemize}
An additional condition then ensures identifiability of points of impact:
\begin{itemize}
\item[B)] $M(s):=\E\left[\mu(s;V_i)\E\big( g(X_i(\tau_1),\dots,X_i(\tau_r))\big| V_i\big)\right]$ is a twice continuously differentiable function of $s\in[a,b]$. Furthermore,
the conditional covariance functions satisfy Assumption \ref{assum1}. With $\Omega:= [a,b]^2\times [0,b-a]$, there exists a function $\omega:\Omega \times {\cal V} \rightarrow
\mathbb{R}$ such that
\begin{align}
\sigma(s,t;v)=\omega(s,t,|s-t|^\kappa;v),\label{poiID 3}
\end{align}
where $\omega(\cdot;v)$ is measurable in $v$, and
\begin{align}
W_r(s,t,z):=\E\left(\omega(s,t,z;V_i)\vartheta_r(V_i)\right), \quad r=1,\dots,S,\label{poiID 4}
\end{align}
are twice continuously differentiable functions of $s,t,z\in [a,b]^2\times [0,b-a]$. Moreover, $0\neq C(\tau_r):=-\frac{\partial}{\partial z}W_r(\tau_r,\tau_r,z)|_{z=0}$ for all $r=1,\dots,S$.
\end{itemize}

\bigskip

\begin{proposition}\label{idgeneral}

\begin{itemize}
\item[i)] Under Condition A) we obtain
\begin{align}
\E\big(X_{i}(s)Y_i\big)=\sum_{r=1}^S \E\left( \sigma(s,\tau_r;V_i) \vartheta_r(V_i)\right) +
\E\left[\mu(s;V_i)\E\big( g(X_i(\tau_1),\dots,X_i(\tau_r))\big| V_i\big)\right]
\label{poiID 5}
\end{align}
\item[ii)] Under Conditions A) and B) we have $\E\big(X_{i}(s)Y_i\big)=\sum_{r=1}^S W_r(s,\tau_r,|s-\tau_r|^\kappa)+M(s)$, and
\begin{equation}
\begin{array}{ll}
\mathbb{E}(Z_{\delta,i}(s)Y_i) =  C(\tau_r)\delta^\kappa + o(\delta^\kappa)&\text{if } s=\tau_r, \text{ for some }  r\in\{1,\dots,S\}\\
 \mathbb{E}(Z_{\delta,i}(s)Y_i)   = o(\delta^2) &\text{if } s\notin\{\tau_1,\dots,\tau_r\}
\end{array}
\label{poiID 6}
\end{equation}
as $\delta\rightarrow 0$.
\end{itemize}

\end{proposition}

\bigskip

The elliptical process introduced in \eqref{poiID 1} provides an example. In this case we have
$\sigma(s,t;v)=v^2 \sigma^*(s,t)$ as well as $\mu(s;v)=0$. If all relevant moments exist, then \eqref{poiID 5} simplifies to
$$\E\big(X_{i}(s)Y_i\big)=\sum_{r=1}^S \sigma^*(s,\tau_r) \E\left( V_i^2  \vartheta_r(V_i)\right)=
\sum_{r=1}^S \sigma(s,\tau_r) \frac{\E\left( V_i^2  \vartheta_r(V_i)\right)}{\V(V_i)},$$
where $\sigma(s,t)=\V(V_i)\sigma^*(s,t)$ is the covariance function of $X_i$. Additionally, if the Gaussian process $X_i^*$ satisfies Assumption \ref{assum1} for some $\omega^*:\Omega  \rightarrow
\mathbb{R}$, then $\omega(t,s;v)=v^2 \omega^*(s,t,|s-t|^\kappa)$, and \eqref{poiID 4} leads to
$$W_r(s,t,z)= \omega^*(s,t,z)\E\left( V_i^2  \vartheta_r(V_i)\right)$$
which is twice continuously differentiable in $(s,t,z)$. Result \eqref{poiID 6} then holds with
$C(\tau_r)=c^*(\tau_r)
\E\left( V_i^2  \vartheta_r(V_i)\right)$, where $c^*(\tau_r)= -\frac{\partial}{\partial z}\omega^*(\tau_r,\tau_r,z)|_{z=0}$, $r=1,\dots,S$.

\bigskip

A more complex example is given by the following situation: There exist smooth, non-Gaussian stochastic processes $V_{i1}(t),V_{i2}(t)$, $t\in [a,b]$, as well as a zero mean Gaussian process $X_i^*(t)$ such that
\begin{equation*}
X_i(t)=V_{i1}(t) X_i^*(t)+V_{i2}(t), \quad  t\in [a,b],
%\label{poiID 7}
\end{equation*}
where $V_i=(V_{i1},V_{i2})$ are independent of $X_i^*$. All relevant moments exist,
 and with
probability 1 any realization of $V_{i1}$ as well as any realization of $V_{i2}$  are  twice continuously differentiable functions on $[a,b]$. If $\sigma^*$ denotes the covariance function of $X_i^*$, then given $(V_{i1},V_{i2})=(v_1,v_2)$ the conditional distribution of $X_i$ is Gaussian, and
$$\mu(t,v)=v_2(t), \quad \sigma(t,s;v)=v_1(s)v_1(t)\sigma^*(s,t).$$
Consequently, \eqref{poiID 5}  becomes
$$\E\big(X_{i}(s)Y_i\big)=\sum_{r=1}^S \sigma^*(s,\tau_r) \E\left( V_{i1}(s)V_{i1}(\tau_r)  \vartheta_r(V_i)\right)+
\E\left( V_{i2}(s) \E\big( g(X_i(\tau_1),\dots,X_i(\tau_r))\big| V_i\big)\right)$$
Smoothness of  $v_2$ implies smoothness of $v_2(s) \E\big( g(X_i(\tau_1),\dots,X_i(\tau_r))\big| V_i=v_2\big)$ in $s\in[a,b]$, and hence
$M(s)=\E\left( V_{i2}(s) \E\big( g(X_i(\tau_1),\dots,X_i(\tau_r))\big| V_i\big)\right)$ is twice continuously differentiable for
$s\in [a,b]$. Furthermore, if $X_i^*$ satisfies Assumption \ref{assum1} for some $\omega^*:\Omega  \rightarrow
\mathbb{R}$, then for almost every realization $v=(v_1,v_2)$ of $V_i$ and any $r=1,\dots,S$
$$\omega(s,t,z;v)\vartheta_r(v)=\omega^*(s,t,z;v)v_1(s)v_1(t)\vartheta_r(v)$$
is a.s. a twice continuously differentiable function of $(s,t,z)$. This in turn implies that
$$W_r(s,t,z)= \omega^*(s,t,z)\E\left(  V_{i1}(s)V_{i1}(\tau_r)  \vartheta_r(V_i)\right),\quad r=1,\dots,S,$$
are twice continuously differentiable functions of $s,t,z\in [a,b]^2\times [0,b-a]$.

\begin{proof}[{\bf Proof of Proposition \ref{idgeneral}}]

For $v\in\Omega$ the conditional distribution of $X_i-\mu(\cdot;v)$ given $V_i=v$ corresponds to a zero mean Gaussian process, and hence the arguments in the
proof of Theorem \ref{thm:NP1} imply that
\begin{align*}
\E\big((X_{i}(s)-\mu(s;v))Y_i\big |V_i=v \big)&=
\sum_{r=1}^S\sigma(s,\tau_r;v) \E\big(\frac{\partial}{\partial x_r}g(X_i(\tau_1),\dots,X_i(\tau_S))\big|V_i=v  \big)\\
&=\sum_{r=1}^S\sigma(s,\tau_r;v)\vartheta_r(v).
\end{align*}
Therefore
\begin{align*}
\E\big(X_{i}(s)Y_i\big |V_i=v \big)=\sum_{r=1}^S\sigma(s,\tau_r;v)\vartheta_r(v)+\mu(s;v)\E\big( g(X_i(\tau_1),\dots,X_i(\tau_r))\big| V_i\big).
\end{align*}
 Assertion i) thus follows from
$$\E\big(X_{i}(s)Y_i\big)=\E\left(\E\big(X_{i}(s)Y_i\big |V_i \big)\right)$$
Now note that under Condition B) our definition of $Z_{\delta,i}$ leads to
\begin{align*}
 \mathbb{E}(Z_{\delta,i}(s)Y_i) &=\sum_{r=1}^S\left( W_r(s,\tau_r,|s-\tau_r|^\kappa)-\frac{1}{2}  W_r(s+\delta,\tau_r,|s+\delta-\tau_r|^\kappa)-\frac{1}{2} W_r(s-\delta,\tau_r,|s-\delta-\tau_r|^\kappa)\right) \\
 &+ M(s)-\frac{1}{2}  M(s+\delta)-\frac{1}{2} M(s-\delta)
\end{align*}
Since $M(s)$ and $W(t,s,z)$ are twice continuously differentiable, Taylor expansions then immediately lead to $ \mathbb{E}(Z_{\delta,i}(s)Y_i) =O(\delta^2)$ for all $s\notin\{\tau_1,\dots,\tau_r\}$. Furthermore,  for $s\in \{\tau_1,\dots,\tau_r\}$ Taylor expansions imply that
\begin{align*}
 \mathbb{E}(Z_{\delta,i}(\tau_r)Y_i) = W_r(\tau_r,\tau_r,0)-  W_r(\tau_r,\tau_r,\delta^\kappa)+o(\delta^\kappa)=
 -\frac{\partial}{\partial z}W_r(\tau_r,\tau_r,z)|_{z=0}\delta^\kappa+o(\delta^\kappa)
\end{align*}
as $\delta\rightarrow 0$,  $r=1,\dots,S$. 

\end{proof}
\bigskip

\section{Estimating points of impact}\label{app:poiEST}
%%%%%%%%%%%%%%%%%%%%%%%%%%%%%%%%%%%%%%%%%%%%%%%%%%%%%%%%%%%%%%%%%%%%%%
Parts of the proofs of Proposition \ref{thmcharX}, Lemmas \ref{lem1}-\ref{lem3}, Theorem \ref{poicon} and Lemma~\ref{lem:1new} follow similar arguments as in \cite{KPS2015}. However, we emphasize that we consider a fundamentally different, more challenging nonparametric statistical model, which requires additional new theoretical arguments. Furthermore, we correct some minor mistakes occurring in \cite{KPS2015}; especially the arguments around Equation \eqref{poicon-eq11} of Theorem \ref{poicon} correct the arguments around Equation (C.36) in Appendix C of the supplementary paper \citeappendix{KPS_S_2015}.
\bigskip

%% %%%%%%%%%%%%%%%%%%%%%%%%%%%%%%%%%%%%%%%%%%%%%%%%%%%%%%%
%% New contribution:  proof Taylorexpansions of $\E(Z_{\delta}(t)X(t))$
%% This part has been improved at Manuscript_03.tex (Beweis auch abgeändert)
%% - reihenfolge von $Z_{\delta,i}(t)X(s)$ in sämtlichen aussagen verändert
%% - $\E(Z_\delta(t+ur)X(t)) statt \E(Z_\delta(t)X(t+ur))$
%% %%%%%%%%%%%%%%%%%%%%%%%%%%%%%%%%%%%%%%%%%%%%%%%%%%%%%%%
\begin{proposition}\label{thmcharX}
Under Assumption~\ref{assum1} we have for all $t \in (a,b)$, and any sufficiently small $\delta>0$ for some constants $M_1<\infty$ and $M_2<\infty$:
\begin{align}
\E(Z_{\delta,i}(t)&X_i(t)) = \delta^\kappa c(t) +R_{1;\delta,t},
\quad \text{with } \sup_{t\in [a+\delta,b-\delta]} |R_{1;\delta,t}|\leq M_1\delta^{\min\{2\kappa,2\}}, \label{Zdelta1} \\
\E(Z_{\delta,i}(t)^2) &=\delta^\kappa\left(2c(t)-\frac{2^\kappa}{2}c(t)\right) +R_{2;\delta,t},
\quad \text{with } \sup_{t\in [a+\delta,b-\delta]} |R_{2;\delta,t}|<M_2 \delta^{\min\{2\kappa,2\}}.
 \label{Zdelta2}
\end{align}
\smallskip

%%OLDMoreover, for any $0<c<\infty$  we have for any sufficiently small $\delta>0$ and all $u\in [-c,c]$
%%OLD{\small \begin{align}
%%OLD\mathbb{E} (X_i(t+u\delta)Z_{\delta,i}(t)) %%OLD%=\sigma(t+u\delta,t)-\frac{1}{2}\sigma(t+u\delta,t-\delta)
																						%-\frac{1}{2}\sigma(t+u\delta,t+\delta)\nonumber \\
%%OLD&= \omega(t,t,0)
%%OLD-\frac{1}{2}\omega(t,t-\delta,\delta^\kappa)
%%OLD-\frac{1}{2}\omega(t,t+\delta,\delta^\kappa)\nonumber\\
%%OLD&\phantom{=} \ -c(t)\delta^\kappa\left(|u|^\kappa
%%OLD-\frac{1}{2}(|u+1|^\kappa-1)-\frac{1}{2}(|u-1|^\kappa-1)\right)
%%OLD+ R_{3;c,u,\delta,t}
 %%OLD\label{Zdelta3}\\
%%OLD=& -c(t)\delta^\kappa\left(|u|^\kappa
%%OLD-\frac{1}{2}|u+1|^\kappa-\frac{1}{2}|u-1|^\kappa\right)
%%OLD+ R_{4;c,u,\delta,t},
 %%OLD\label{Zdelta4}
%%OLD\end{align}}
Moreover, for any $0<c<\infty$  we have for any sufficiently small $\delta>0$ and all $u\in [-c,c]$:
\begin{align}
\E(Z_{\delta,i}(t+u\delta)X_i(t))
%=
%=\sigma(t+u\delta,t)-\frac{1}{2}\sigma(t+(u-1)\delta,t)
																						%-\frac{1}{2}\sigma(t+(u+1)\delta,t)\nonumber \\
&= \E(Z_{\delta,i}(t)X_i(t)) \nonumber\\
&\qquad -c(t)\delta^\kappa\left(|u|^\kappa
-\frac{1}{2}(|u+1|^\kappa-1)-\frac{1}{2}(|u-1|^\kappa-1)\right)
+ R_{3;c,u,\delta,t}
 \label{Zdelta3}\\
=&-c(t)\delta^\kappa\left(|u|^\kappa
-\frac{1}{2}|u+1|^\kappa-\frac{1}{2}|u-1|^\kappa\right)
+ R_{4;c,u,\delta,t},
 \label{Zdelta4}
\end{align}
where for some constants $M_{3,c}<\infty$ and $M_{4,c}<\infty$
{\small
$$\sup_{t\in [a+\delta,b-\delta]}
R_{3;c,u,\delta,t}\leq M_{3,c}(|u|^{1/2}\delta)^{\min\{2\kappa,2\}}, \quad
 \sup_{t\in [a+\delta,b-\delta]} R_{4;c,u,\delta,t}\leq M_{4,c}\delta^{\min\{2\kappa,2\}}
$$}
hold for all $u\in [-c,c]$.
\smallskip

Finally, for all $s\in[a,b]$ with $|t-s|\geq \delta$ we have for some constant $M_5<\infty$
\begin{align}
|\mathbb{E} (Z_{\delta,i}(t)X_i(s))|= \begin{cases}
 M_5 \frac{\delta^2}{|t-s|^{2-\kappa}} & \text{if } \kappa\neq 1\\
  M_5 \delta^2 & \text{if } \kappa= 1.
\end{cases}
 \label{Zdelta5}
\end{align}
\end{proposition}

\begin{proof}[{\bf Proof of Proposition~\ref{thmcharX}}]
Assumption \ref{assum1} implies that the absolute values of all
first and second order partial derivatives of $\omega(t,s,z)$ are uniformly bounded
by some constant $M<\infty$ for all $(t,s,z)$ in the compact subset $[a,b]^2\times [0,b-a]$
of $\Omega$.

\noindent By definition of $Z_{\delta,i}$  it thus follows from a Taylor expansion
 of $\omega$ that for $t\in(a,b)$, any sufficiently small $\delta>0$ and some constant $M_1<\infty$
\begin{align}
\mathbb {E}(Z_{\delta,i}(t)&X_i(t))=\sigma(t,t)-\frac{1}{2}\sigma(t-\delta,t)
-\frac{1}{2}\sigma(t+\delta,t)\nonumber \\
=&\ \omega(t,t,0)-\frac{1}{2}\omega(t-\delta,t,\delta^\kappa)
-\frac{1}{2}\omega(t+\delta,t,\delta^\kappa)\nonumber\\
=&\ \delta^\kappa c(t) +R_{1;\delta,t},
\quad \text{with } \sup_{t\in [a+\delta,b-\delta]} |R_{1;\delta,t}|\leq M_1\delta^{\min\{2\kappa,2\}},
 \label{Zdelta1.pf}
\end{align}
which proofs assertion \eqref{Zdelta1}.
For $\E(Z_{\delta,i}(t)^2)$, i.e.~the variance of $Z_{\delta,i}(t)$,  we obtain by similar arguments
{\small\begin{align}
&var(Z_{\delta,i}(t))=
2\omega(t,t,0)-\omega(t,t-\delta,\delta^\kappa)-\omega(t,t+\delta,\delta^\kappa)
- \frac{1}{2}\left(\omega(t,t,0)-\omega(t+\delta,t-\delta,(2\delta)^\kappa)\right) \nonumber\\
 &\quad - \frac{1}{4}\left(2\omega(t,t,0)-\omega(t-\delta,t-\delta,0)
-\omega(t+\delta,t+\delta,0)\right)
\nonumber\\
&=\delta^\kappa\left(2c(t)-\frac{2^\kappa}{2}c(t)\right) +R_{2;\delta,t},
\quad \text{with } \sup_{t\in [a+\delta,b-\delta]} |R_{2;\delta,t}|<M_2\delta^{\min\{2\kappa,2\}}
 \label{Zdelta2.pf}
\end{align}}
for some constant $M_2<\infty$, i.e.~assertion \eqref{Zdelta2}.

Assertion \eqref{Zdelta3} and \eqref{Zdelta4} follow again from Taylor expansions of $\omega$, i.e.~we have for any $0<c<\infty$  and any sufficiently small $\delta>0$ and all $u\in [-c,c]$
\begin{align}
\E(Z(t+u\delta)X(t)) &=
%\E(X(t+u\delta)X(t)) - 0.5 \E(X(t+(u-1)\delta)X(t)) - 0.5 \E(X(t+(u+1)\delta)X(t))\\
 \sigma(t+u\delta,t) - \frac{1}{2} \sigma(t+(u-1)\delta,t) - \frac{1}{2} \sigma(t+(u+1)\delta,t)\nonumber \\
%&= \omega(t+u\delta, t , |u\delta|^\kappa)
%		-0.5 \omega(t+(u-1)\delta, t , |(u-1)\delta|^\kappa)
%		-0.5 \omega(t+(u+1)\delta, t , |(u+1)\delta|^\kappa)\\
%&=\omega(t,t,0) + u\delta\omega_1(t,t,0) +  |u\delta|^\kappa \omega_3(t,t,0) + R_1\\
%&\quad - 0.5 (\omega(t-\delta, t, \delta^\kappa) + u\delta \omega_1(t-\delta, t,0) + (|u-1|^\kappa-1)|\delta^\kappa \omega_{3}(t-\delta, t, 0)  +R_2) \\
%&\quad 	-0.5 (\omega(t+\delta, t, \delta^\kappa) + u\delta\omega_1(t+\delta, t, 0) + (|u+1|^\kappa-1) \delta^\kappa \omega_{3}(t+\delta, t, 0) + R_3) \\
%&= \omega(t,t,0) - 0.5 \omega(t-\delta, t, \delta^\kappa)  - 0.5\omega(t+\delta, t, \delta^\kappa)\\
%&\quad   + \omega_{3}(t, t, 0)\delta^\kappa( |u|^\kappa - 0.5(|u-1|^\kappa-1) - 0.5(|u+1|^\kappa-1)) \\
%&\quad   + R_1 + R_2 + R_3 \\
%&\quad   + u\delta\omega_1(t,t,0) -0.5(u\delta (\omega_1(t, t,0) + R_4)) - 0.5(u\delta \omega_1(t,t,0) + R_5)\\
%&\quad -0.5(|u-1|^\kappa-1)\delta^\kappa R_6 -0.5(|u+1|^\kappa-1)\delta^\kappa R_7 \\
\begin{split}\label{zdelta.aux1}
			&=\omega(t,t,0) - \frac{1}{2} \omega(t-\delta, t, \delta^\kappa)  - \frac{1}{2}\omega(t+\delta, t, \delta^\kappa) \\
&\quad - c(t)\delta^\kappa\left( |u|^\kappa - \frac{1}{2}(|u-1|^\kappa-1) - \frac{1}{2}(|u+1|^\kappa-1)\right) + R_{3;c,u;\delta,t}
\end{split} \\
&= -c(t) \delta^\kappa \left(|u|^\kappa - \frac{1}{2}(|u-1|^\kappa) - \frac{1}{2}(|u+1|^\kappa) \right) + R_{4;c,u;\delta,t} \label{zdelta.aux2}
\end{align}
%Where:
%\begin{itemize}
%\item $|R_1|\leq M((u\delta)^2 + |u\delta|^{2\kappa})$
%\item $|R_2|\leq M((u\delta)^2 + (|(|u-1|^\kappa-1)|\delta^\kappa)^2)$
%\item $|R_3|\leq M((u\delta)^2 + (|(|u+1|^\kappa-1)|\delta^\kappa)^2)$
%\item $|R_4|\leq M(\delta + \delta^\kappa)$
%\item $|R_5|\leq M(\delta + \delta^\kappa)$
%\item $|R_6|\leq M(\delta)$
%\item $|R_7|\leq M(\delta)$
%\end{itemize}
%
%Note:
%\begin{itemize}
%\item $|R_1|\leq M_c (|u|^{1/2}\delta)^{\min\{2,2\kappa\}}$
%\item $|R_2|\leq M_c (|u|^{1/2}\delta)^{\min\{2,2\kappa\}}$
%\item $|R_3|\leq M_c (|u|^{1/2}\delta)^{\min\{2,2\kappa\}}$
%\item $|u\delta R_4|\leq M_c (|u|^{1/2}\delta)^{\min\{2,2\kappa\}}$
%\item $|u\delta R_5|\leq M_c (|u|^{1/2}\delta)^{\min\{2,2\kappa\}}$
%\item $|(|u-1|^\kappa-1)\delta^\kappa R_6| \leq M_c (|u|^{1/2}\delta)^{\min\{2,2\kappa\}}$
%\item $|(|u+1|^\kappa-1)\delta^\kappa R_7| \leq M_c (|u|^{1/2}\delta)^{\min\{2,2\kappa\}}$
%\end{itemize}
%
%Hence $$\sup_{t\in [a+\delta, b-\delta]}|R_{3;c,u;\delta,t}| \leq M_{3,c} (|u|^{1/2}\delta)^{\min\{2,2\kappa\}}$$
%and
%$$\sup_{t\in [a+\delta, b-\delta]}|R_{4;c,u;\delta,t}| \leq M_{4,c} \delta^{\min\{2,2\kappa\}}.$$
%
where for some constants $M_{3,c}<\infty$ and $M_{4,c}<\infty$
$$\sup_{t\in [a+\delta,b-\delta]}
R_{3;c,u,\delta,t}\leq M_{3,c}(|u|^{1/2}\delta)^{\min\{2\kappa,2\}}, \quad
 \sup_{t\in [a+\delta,b-\delta]} R_{4;c,u,\delta,t}\leq M_{4,c}\delta^{\min\{2\kappa,2\}}
$$
hold for all $u\in [-c,c]$. Assertion \eqref{Zdelta3} now follows immediately from \eqref{zdelta.aux1} together with the expression of $\E(Z_{\delta,i)}(t)X_i(t))$ in terms of $\omega$, while assertion \eqref{Zdelta4} corresponds to \eqref{zdelta.aux2}.

In order to proof Assertion \eqref{Zdelta5}, let $s\in[a,b]$ with $|t-s|\geq \delta$. Another Taylor expansion yields:
\small\begin{align}
\mathbb{E} (Z_{\delta,i}(t)X_i(s))= & \omega(t,s,|t-s|^\kappa)
-\frac{1}{2}\omega(t-\delta,s,|t-s-\delta|^\kappa)
-\frac{1}{2}\omega(t+\delta,s,|t-s+\delta|^\kappa) \nonumber \\
&= -\frac{1}{2} (\omega_3(t,s,|t-s|^\kappa)(|t-s+\delta|^\kappa + |t-s-\delta|^\kappa - 2|t-s|^\kappa)) + R_{5;\delta,t}\label{eq:Zdelta5.a}
\end{align}
where for some constant $M_{5,1}<\infty$ we have $$|R_{5;\delta,t}| \leq M_{5,1}((\delta^2 + (|t-s+\delta|^\kappa - |t-s|^\kappa)^2) + (\delta^2 + (|t-s-\delta|^\kappa - |t-s|^\kappa)^2)),$$
and $\omega_3(x,y,z)$ denotes the partial derivative of $\omega(x,y,z)$ with respect to its third argument.

For $\kappa =1$ we then have for $|t-s|\geq \delta$,
$||t-s+\delta| - |t-s|| = \delta$,
$||t-s-\delta| - |t-s|| = \delta$
as well as
$|t-s+\delta| + |t-s-\delta| -2|t-s| =0$,
which leads together with \eqref{eq:Zdelta5.a} to
\begin{align}
|\mathbb{E}(Z_{\delta,i}(t)X_i(s))|\leq M_{5,2}\delta^2 \label{eq:Zdelta5.b}
\end{align}
for some constant $M_{5,2}<\infty$.

For $\kappa \neq 1$,  define $v := t-s$ to shorten the notation and suppose $v \geq \delta>0 $ (i.e. suppose $t>s$, the case for which $s>t$ follows from similar arguments as below).
\begin{align*}
|t-s+\delta|^\kappa + |t-s-\delta|^\kappa - 2|t-s|^\kappa &= |v + \delta|^\kappa + |v-\delta|^\kappa - 2|v|^\kappa =\underbrace{(v + \delta)^\kappa + (v-\delta)^\kappa - 2v^\kappa}_{=: f(\delta)}.
\end{align*}
Suppose $v \geq 2\delta$. It follows from a Taylor expansion of $f(\delta)$ around $0$, that there exists a $\xi$ between $0$ and $\delta$ such that
%f(\delta) = f(0) + \delta f'(0) + \delta^2/2 f''(\xi) = 0 + 0 + \delta^2 \frac{1}{2}\kappa (\kappa-1)((v+\xi)^{\kappa-2}+(v-\xi)^{\kappa-2}
$$|f(\delta)| \leq \frac{1}{2}\kappa|(\kappa-1)| \delta^2 (|(v+\xi)^{\kappa-2}| + |(v-\xi)^{\kappa-2}|).$$
Since $\kappa<2$ we have
\begin{align}
(v+\xi)^{\kappa-2}\leq v^{\kappa-2} = \frac{1}{|t-s|^{2-\kappa}}.\label{eq:Zdelta5.c}
\end{align}
 On the other hand, since $v \geq 2\delta$, we have for all $0\leq \xi \leq \delta$
\begin{align}
|(v-\xi)^{\kappa-2}| \leq |(v-\frac{v}{2})^{\kappa-2}| %= \frac{|v|^{\kappa-2}}{2^{\kappa-2}}
= \frac{2^{2-\kappa}}{|t-s|^{2-\kappa}}.\label{eq:Zdelta5.d}
\end{align}
Now consider the case $\delta \leq v \leq 2\delta$. If $\kappa >1$, then $f(\delta)\geq 0$ and
\begin{align}
|f(\delta)| &= (v + \delta)^\kappa + (v-\delta)^\kappa - 2v^\kappa  \leq (v + \delta)^\kappa + (v-\delta)^\kappa 	\nonumber \\
&\leq 2 (u+\delta)^\kappa
\leq 2 (2\delta)^\kappa = 2 (2 \delta)^2 (2\delta)^{\kappa-2} \nonumber \\
	&\leq % 8 \delta^2 v^{\kappa-2} = 8 \frac{\delta^2}{v^{2-\kappa}} =
		 8 \frac{\delta^2}{|t-s|^{2-\kappa}}. \label{eq:Zdelta5.e}
\end{align}
For $\kappa<1$ we have $f(\delta) \leq 0$. With $\delta \leq v \leq 2\delta$, we then have
\begin{align}
|f(\delta)| = 2v^\kappa - (v+\delta)^\kappa - (v-\delta)^\kappa \leq 2v^\kappa
\leq 2 \cdot (2 \delta)^\kappa = 8 \delta^2 (2\delta)^{\kappa-2}
\leq 8 \frac{\delta^2}{v^{2-\kappa}}  = 8 \frac{\delta^2}{|t-s|^{2-\kappa}}.\label{eq:Zdelta5.f}
\end{align}
Finally, similar arguments may be used to show that additionally we have for some constant $M_{5,3}<\infty$
\begin{align}
|R_{5;\delta,t}| \leq M_{5,3}\left( \frac{\delta^2}{|t-s|^{2-\kappa}}\right).\label{eq:Zdelta5.g}
\end{align}
Assertion~\eqref{Zdelta5} then follows from \eqref{eq:Zdelta5.a} - \eqref{eq:Zdelta5.g}.
\end{proof}
%% %%%%%%%%%%%%%%%%%%%%%%%%%%%%%%%%%%%%%%%%%%%%%%%%%%%%%%%
%% End proof Taylorexpansions of $\E(Z_{\delta}(t)X(t))$
%% %%%%%%%%%%%%%%%%%%%%%%%%%%%%%%%%%%%%%%%%%%%%%%%%%%%%%%%

%%%%%%%%%%%%%%%%%%%%%%%%%%%%%%%%%%%%%%%%%%%%%%%%%%%%%%%%%%%%%%%%%%%%%%
\subsection*{Proofs for estimation of impact points}
%%%%%%%%%%%%%%%%%%%%%%%%%%%%%%%%%%%%%%%%%%%%%%%%%%%%%%%%%%%%%%%%%%%%%%

%% %%%%%%%%%%%%%%%%%%%%%%%%%%%%%%%%%%%%%%%%%%%%%%%%%%%%%%%
%% NEW Contribution (Instead of Proposition 1?!)
%% %%%%%%%%%%%%%%%%%%%%%%%%%%%%%%%%%%%%%%%%%%%%%%%%%%%%%%%
We begin by stating a deviation bound for the central $\chi^2$ distribution.
\begin{lemma}\label{lem:chi2dev}
Let $W \sim \chi^2_{n}$ then for all $0\leq x < 1/2$ we have
\begin{align}
P( |n^{-1}W - 1|>x) \leq 2 \exp(-3nx^2/16).\label{eq:chidevi}
\end{align}
\end{lemma}
\begin{proof}[{\bf Proof of Lemma \ref{lem:chi2dev}}]
Equations (A.2) and (A.3) in \citeappendix{JL2009} imply that for $0\leq x < 1/2$ we have
\begin{align*}
P(|n^{-1}W - 1|>x) & = P(W \geq n(1+x)) + P(W \leq n(1-x)) \\
												 &\leq \exp(-3nx^2/16) + \exp(-nx^2/4) \\
												 &\leq 2\exp(-3nx^2/16).
\end{align*}
\end{proof}

\bigskip

%% %%%%%%%%%%%%%%%%%%%%%%%%%%%%%%%%%%%%%%%%%%%%%%%
%% Lemma 1 Replacement
%% %%%%%%%%%%%%%%%%%%%%%%%%%%%%%%%%%%%%%%%%%%%%%%%
\begin{lemma}\label{lem1}
Under Assumption \ref{assum1} there exist constants $0<D_{1}<\infty$ and
$0<D_{2}<\infty$,
such that for all
$n$, all
$0<\delta<(b-a)/2$, all
$t\in [a+\delta,b-\delta]$, all $0<s\leq  1/2$ with $\delta^\kappa s^\kappa\geq s\delta^2$,
and every $0<z\leq  \sqrt{n}$
we obtain
{\small
\begin{align} \label{lem1eq1}
P\biggl(\sup_{t-s\delta\leq u\leq t+s\delta}  &|\frac{1}{n}\sum_{i=1}^n[(Z_{\delta,i}(t) -Z_{\delta,i}(u))Y_i-\mathbb{E}((Z_{\delta,i}(t) -Z_{\delta,i}(u))Y_i)]|
\leq z D_1 \sqrt{\frac{\delta^k s^\kappa}{n}}\biggr) \nonumber \\
&\geq 1-2 \exp(-z^2)
\end{align}}
and
{\small
\begin{align} \label{lem1eq2}
P\biggl(\sup_{t-s\delta\leq u\leq t+s\delta}  &|\frac{1}{n}\sum_{i=1}^n[(Z_{\delta,i}(t)^2 -Z_{\delta,i}(u)^2)-\mathbb{E}(Z_{\delta,i}(t)^2 -Z_{\delta,i}(u)^2)]|
\leq z D_{2} \delta^{\kappa}\sqrt{\frac{s^\kappa}{n}}\biggr) \nonumber \\
&\geq 1 - 2 \exp(-z^2).
\end{align}}
\end{lemma}

\bigskip

\begin{proof}[{\bf Proof of Lemma~\ref{lem1}.}]
Choose some arbitrary $0<\delta<(b-a)/2$, $t\in [a+\delta,b-\delta]$, as well as $0< s\leq 1/2$. For $q_1,q_2\in [-1,1]$, Taylor expansions then yield
\begin{align}
	\mathbb{E}& ((Z_{\delta,i}(t+q_1s\delta)-Z_{\delta,i}(t+q_2s\delta))^2)= \mathbb{E}((Z_{\delta,i}(t+q_1s\delta)-Z_{\delta,i}(t+q_1s\delta+(q_2-q_1)s\delta))^2)\nonumber\\
					=& c(t+q_1s\delta)s^\kappa\delta^\kappa 2\left(3/2|q_2-q_1|^\kappa -(|q_2-q_1+\frac{1}{s}|^\kappa +|q_2-q_1-\frac{1}{s}|^\kappa-2\frac{1}{s^\kappa})\right)\nonumber\\
						&+\frac{1}{2}c(t+q_1s\delta)2^\kappa s^\kappa\delta^\kappa
							\left(|\frac{q_2-q_1}{2}+\frac{1}{s}|^\kappa+ |\frac{q_2-q_1}{2}-\frac{1}{s}|^\kappa-2\frac{1}{s^\kappa}\right) + R_{6;\delta,t}\nonumber\\
						& \quad \text{with } |R_{6;\delta,t}|\leq L_{2,1} ||q_2-q_1|^{1/2}s^{1/2}\delta|^{\min\{2\kappa,2\}} \label{lem1pr01}
\end{align}
for some constant $L_{2,1}<\infty$.\\

Note that there exists a constant $L_{2,2}<\infty$ such that for all $0<s\leq 0.5$ we have
$||q_2-q_1+\frac{1}{s}|^\kappa +|q_2-q_1-\frac{1}{s}|^\kappa-2\frac{1}{s^\kappa}|\leq L_{2,2} |q_2-q_1|^2$
as well as
$||\frac{q_2-q_1}{2}+\frac{1}{s}|^\kappa +|\frac{q_2-q_1}{2}-\frac{1}{s}|^\kappa-2\frac{1}{s^\kappa}|\leq L_{2,2} |q_2-q_1|^2$.
%\textcolor{red}{Die beiden Aussagen sind schon relativ heftig zu beweisen: eventuell details für eine der beiden?!)}

Together with \eqref{lem1pr01} this implies that there exists a constant $L_{2,3}<\infty$, which can be chosen independent of $s$ and $\delta$, such that for all $q_1,q_2\in [-1,1]$
 \begin{align}
\mathbb{E} ((Z_{\delta,i}(t+q_1s\delta)-Z_{\delta,i}(t+q_2s\delta))^2)\leq L_{2,3} s^\kappa\delta^\kappa |q_1-q_2|^{\min\{1,\kappa\}}.
\label{lem1pr02}
\end{align}
Define $Z_{\delta,i}^*(q) := \frac{1}{\sqrt{s^\kappa \delta^\kappa}} (Z_{\delta,i}(t+qs\delta)Y_i - \mathbb{E}(Z_{\delta,i}(t+qs\delta)Y_i))$ and
$Z_{\delta}^*(q) := \frac{1}{\sqrt{n}} \sum_{i=1}^n Z_{\delta,i}^*(q)$.
By bounding the absolute moments of $E(|Y_i|^{2m})$ according to Assumption~\ref{assum2}, one can easily verify that for $K= 4\sqrt{L_{2,3} |q_2-q_1|^{\min\{1,\kappa\}}}\sigma_{|y|}$,
the Bernstein condition
$$E(|Z_{\delta,i}^*(q_1) - Z_{\delta,i}^*(q_2)|^m) \leq \frac{m!}{2} K^{m-2} K^2$$
holds for all $0<s\leq 0.5$, all integers $m \geq 2$, all $q_1,q_2 \in [-1,1]$ and all $0<\delta<(b-a)/2$.\\

An application of Corollary 1 in \citeappendix{vandeGeer2013} then guarantees that the Orlicz norm of $Z^*(q_1) - Z^*(q_2)$ is bounded, i.e., one has for all $q_1, q_2 \in [-1,1]$
\begin{align}
|| Z_{\delta}^*(q_1) - Z_{\delta}^*(q_2)||_{\Psi}\leq L_{2,4} |q_1-q_2|^{\min\{\frac{1}{2},\frac{1}{2}\kappa\}} \label{eq:orliczbd}
\end{align}
for some constant $0<L_{2,4}<\infty$.  %, where $\Psi(x) = \exp(\frac{n}{6}(\sqrt{1+\frac{2\sqrt{6}x}{\sqrt{n}}}-1)^2)-1.$

The proof then follows from well known maximal inequalities of empirical process theory. In particular, by \eqref{eq:orliczbd} one may apply theorem 2.2.4 of \citeappendix{vanderVaart1996}. It is immediately seen that the covering integral appearing in this theorem is finite, and we can thus infer that there exists a constant $0<D_{1,1}<\infty$ such that
\begin{align*}
\mathbb{E}\left(\exp{\left( \sup\limits_{q_1,q_2 \in \left[-1,1\right]} n/6  \left(\sqrt{1+2\sqrt{\frac{6}{nD_{1,1}^2}}|Z_{\delta}^*(q_1)-Z_{\delta}^*(q_2)|}-1\right)^2\right)}\right)
\leq 2.
\end{align*}
For every $z>0$, the Markov inequality then yields
{\small \begin{align*}
&P\bigg( \sup\limits_{q_1,q_2 \in [-1,1]} |Z_{\delta}^*(q_1)-Z_{\delta}^*(q_2)|
\geq z \frac{D_{1,1}}{2\sqrt{6}}  \bigg)\\
& =  P\Bigg(\exp{\bigg( \sup\limits_{q_1,q_2 \in \left[-1,1\right]} n/6
\bigg(\sqrt{1+2\sqrt{\frac{6}{n D_{1,1}^2}}|Z_{\delta}^*(q_1)-Z_{\delta}^*(q_2)|}-1\bigg)^2\bigg)}
 \geq \exp\left(n/6 \bigg(\sqrt{1+z/\sqrt{n}}-1\bigg)^2\right)\Bigg)\\
&\leq 2 \exp\bigg(-n/6 \bigg(\sqrt{1+ z/\sqrt{n}}-1\bigg)^2\bigg).
\end{align*}
}
At the same time it follows from a Taylor expansion that for any $0<z  \leq  \sqrt{n}$  there exists a constant $0<D_{1,2}<\infty$ such that
\begin{equation}\label{lem1:tayl}
\frac{n}{6}\left(\sqrt{1+z/\sqrt{n}}-1\right)^2 \geq  D_{1,2} z^2.
\end{equation}
 Assertion \eqref{lem1eq1} is an immediate consequence.

In order to prove (\ref{lem1eq2}) first note that
$Z_{\delta,i}(t_1)^2 -Z_{\delta,i}(t_2)^2=
(Z_{\delta,i}(t_1) -Z_{\delta,i}(t_2))(Z_{\delta,i}(t_1) +Z_{\delta,i}(t_2))$.
Equation \eqref{Zdelta5} %% in Kneip et al (2013) %(\ref{Zdelta5})
 implies the existence of a constant
$0<L_{2,5}<\infty$ such that
$\mathbb{E}((Z_{\delta,i}(t+q_1s\delta) +Z_{\delta,i}(t+q_2s\delta))^2)\leq L_{2,5} \delta^\kappa$
for all $q_1,q_2\in [-1,1]$, and all $n$, $t$, $s$ and $\delta$.
With $Z_{\delta}^{**}(q)=
\frac{1}{\sqrt{\delta^{2\kappa}s^\kappa}}\frac{1}{\sqrt{n}}\sum_{i=1}^n
(Z_{\delta,i}(t+qs\delta)^2 - \mathbb{E}(Z_{\delta,i}(t+qs\delta)^2))$,
 similar steps as above now imply the existence of a constant $0<L_{2,6}<\infty$ such that
{\small
\begin{align*}
\|Z_{\delta}^{**}(q_1)-Z_{\delta}^{**}(q_2)\|_{\Psi} \leq L_{2,6} |q_1-q_2|^{\min\{\frac{1}{2},\frac{1}{2}\kappa\}}.
\end{align*}}
Using again  maximal inequalities of empirical process theory and \eqref{lem1:tayl},
Assertion \eqref{lem1eq2} now follows from arguments similar to those used to prove
\eqref{lem1eq1}.
%% %%%%%%%%%%%%%%%%%%%%%%%%%%%%%%%%%%%%%%%%%%%%%%%
%% Beweisidee von den Bernstein Equations/Aussage
%% Geht analog wie oben:
%% Z_{\delta,i}(t_1)^2 -Z_{\delta,i}(t_2)^2
%% ist das produkt von zwei normalverteilten und zentrierten ZV
%% Man Schätzt wieder den $\E$-Wert ab, wendet CS an
%% nutzt die normalität und schaut sich wie man die Varianzen abschätzen kann
%% einerseits mit s^\kappa \delta^\kappa , andererseits mit \delta^\kappa.
%% deshalb ist hier auch die kalierung mit $1/(\sqrt{s^kappa}\delta^\kappa)$
%% anstatt mit $1/\sqrt{s^\kappa \delta^\kappa}$ sinnvoll.
%% %%%%%%%%%%%%%%%%%%%%%%%%%%%%%%%%%%%%%%%%%%%%%%%
\end{proof}
%% %%%%%%%%%%%%%%%%%%%%%%%%%%%%%%%%%%%%%%%%%%%%%%%
%% End Lemma 1 Replacement
%% %%%%%%%%%%%%%%%%%%%%%%%%%%%%%%%%%%%%%%%%%%%%%%%

\bigskip

%% %%%%%%%%%%%%%%%%%%%%%%%%%%%%%%%%%%%%%%%%%%%%%%%
%% Lemma 2 Replacement
%% %%%%%%%%%%%%%%%%%%%%%%%%%%%%%%%%%%%%%%%%%%%%%%%
\begin{lemma}\label{lem0}
Under the assumptions of Theorem \ref{poicon} there exist constants $0<D_3<D_4<\infty$ and
$0<D_5<\infty$ such that
\begin{align}
&0<D_3\delta^\kappa\leq \inf_{t\in [a+\delta,b-\delta]}\mathbb{E}(Z_{\delta,i}(t)^2) \leq
\sigma_{z,sup}^2:=  \sup_{t\in [a+\delta,b-\delta]}\mathbb{E}(Z_{\delta,i}(t)^2)
\leq D_4\delta^\kappa \label{lem0eq0} \\
&\lim_{n\rightarrow\infty} P\left(\sup_{t\in [a+\delta,b-\delta]}
|\frac{1}{n}\sum_{i=1}^n[Z_{\delta,i}(t)^2 -\mathbb{E}(Z_{\delta,i}(t)^2)]| \leq D_5
\delta^\kappa\sqrt{\frac{1}{n}\log \Big(\frac{b-a}{\delta}\Big)}\right)=1.
\label{lem0eq1}
\end{align}

Moreover, there exist a constant $0<D<\infty$ such that for any $A^*$ with $D<A^*\leq A$ we obtain as $n\rightarrow\infty$:
\begin{align}\label{lem0eq2}
\begin{split}
P\Bigg(\sup_{t\in [a+\delta,b-\delta]} &(\frac{1}{n}\sum_{i=1}^n Z_{\delta,i}(t)^2)^{-\frac{1}{2}} |\frac{1}{n}\sum_{i=1}^n
( Z_{\delta,i}(t)Y_i-\mathbb{E}(Z_{\delta,i}(t)Y_i))| \\
&\qquad \qquad \qquad \qquad \leq A^*\sqrt{\frac{\sigma_{|y|}^2}{n}\log \Big(\frac{b-a}{\delta}\Big)}
\Bigg) \rightarrow 1,
\end{split}\\
\begin{split}\label{lem0eq3}
P\Bigg(\sup_{t\in [a+\delta,b-\delta]} &|\frac{1}{n}\sum_{i=1}^n
( Z_{\delta,i}(t)Y_i-\mathbb{E}(Z_{\delta,i}(t)Y_i))|\\
&\qquad \qquad \qquad \qquad  \leq A^*\sqrt{\frac{\sigma_{|y|}^2 D_{4,\epsilon} \delta^\kappa}{n}\log \Big(\frac{b-a}{\delta}\Big)}\Bigg)
\rightarrow 1,
\end{split}
\end{align}
for any constant $D_{4,\epsilon}<\infty$ with $D_4<D_{4,\epsilon}$.
\end{lemma}

\bigskip

\begin{proof}[{\bf Proof of Lemma \ref{lem0}.}]
Assertion~\ref{lem0eq0} follows directly from Assumption~\ref{assum1} and equation \eqref{Zdelta2}.

\noindent Let ${\cal J}_\delta:=\{j| \ t_j\in[a+\delta,b-\delta], j\in\{1,\dots,p\}\}$.
%Assumption \ref{assum2} implies that all joint distributions of
%vectors with elements  $\{Z_{\delta,i}(t_j)\}_{i=1,\dots,n,j\in {\cal J}_\delta}$ or $\{Y_i\}_{i=1,\dots,n}$ are multivariate normal.
Choose some constants $w_1,w_2$ with
 $1<w_1<w_2$ and determine an equidistant grid
 $s_1=a+\delta<s_2<\dots<s_{N_{w1}}=b-\delta$ of $N_{w_1}=[(\frac{b-a}{\delta})^{w_1}]$ points in $[a+\delta,b-\delta]$.
 Obviously, $\ell_{w_1}:=|s_{j}-s_{j-1}|=O(\delta^{w_1})$, $j=2,\dots,N_{w_1}$, as $\delta\rightarrow 0$.
 Then
{\small \begin{align*}
\sup_{t\in [a+\delta,b-\delta]}& |\frac{1}{n}\sum_{i=1}^n
 Z_{\delta,i}(t)^2-\mathbb{E}(Z_{\delta,i}(t)^2)|\leq
\sup_{j\in \{2,\dots,N_{w_1}\}} |\frac{1}{n}\sum_{i=1}^n
 Z_{\delta,i}(s_j)^2-\mathbb{E}(Z_{\delta,i}(s_j)^2)|\\
&+ \sup_{j\in \{2,\dots,N_{w_1}\}}\sup_{t\in [s_{j-1},s_{j}]}
|\frac{1}{n}\sum_{i=1}^n
 Z_{\delta,i}(t)^2-\mathbb{E}(Z_{\delta,i}(t)^2)-(Z_{\delta,i}(s_j)^2-\mathbb{E}(Z_{\delta,i}(s_j)^2))|.
\end{align*}}

By Assumption \ref{assum2}, using the deviation bound \eqref{eq:chidevi} as well as $\sup_{t\in [a+\delta,b-\delta]}\mathbb{E}(Z_{\delta,i}(t)^2) \leq D_4\delta^\kappa$, it follows from
 the Bonferroni-inequality
that as $n\rightarrow \infty$
\begin{align*}
P\Bigg(&\sup_{j\in \{2,\dots,N_{w_1}\}} |\frac{1}{n}\sum_{i=1}^n
 Z_{\delta,i}(s_j)^2-\mathbb{E}(Z_{\delta,i}(s_j)^2)|
\leq \sqrt{\frac{16}{3}w_2}D_4\delta^\kappa
\sqrt{\frac{1}{n} \log \Big(\frac{b-a}{\delta}\Big)}
\Bigg)\\ \nonumber
&\geq   1-
2 N_{w_1} \cdot \exp\left(-w_2\log \Big(\frac{b-a}{\delta}\Big)\right)
\geq 1- 2 \Big(\frac{b-a}{\delta}\Big)^{w_1-w_2} \rightarrow 1,
\end{align*}
%% %%%%%%%%%%%%%%%%%%%%%%%%%%%%%%%%%%%%%%%%%%%%%%%%%%%%%%%%%%%%%%%%%%%%%%%%%%%%
%% Anmerkung:
%% Natürlich ist jedes einzelne der
%% $\sum_{i=1}^n(Z_{\delta,i}(s_j)/\E(Z_{\delta,i}(s_j)^2))^2$
%% wegen unserer Gaussanahme in Assumption~\ref{assum2}
%% Chi-Quadrat verteilt - Die Deviation Bound lässt sich also anwenden
%% Die Bedingung an die Deviation bound ist erfüllt, da
%% $1/n\log((b-a)/\delta) \to 0$ da $n\delta^\kappa/|\log(\delta)| \to \infty$
%% %%%%%%%%%%%%%%%%%%%%%%%%%%%%%%%%%%%%%%%%%%%%%%%%%%%%%%%%%%%%%%%%%%%%%%%%%%%%
while Lemma \ref{lem1} implies that as $n\rightarrow\infty$
\begin{align*}
P\Bigg(&\sup_{j\in \{2,\dots,N_{w_1}\}}\sup_{t\in [s_{j-1},s_{j}]} |\frac{1}{n}\sum_{i=1}^n
 Z_{\delta,i}(t)^2-\mathbb{E}(Z_{\delta,i}(t)^2)-(Z_{\delta,i}(s_j)^2-\mathbb{E}(Z_{\delta,i}(s_j)^2))|  \\
& \qquad \quad
\leq D_2 \sqrt{w_2} \delta^\kappa
\sqrt{\frac{\ell_{w_1}^\kappa}{\delta^\kappa n} \log \Big(\frac{b-a}{\delta}\Big)}
\Bigg) \rightarrow 1.
\end{align*}
Recall that $\frac{\ell_{w_1}^\kappa}{\delta^\kappa }=O(\delta^{\kappa(w_1-1)})$ and
hence $\sqrt{\frac{\ell_{w_1}^\kappa}{\delta^\kappa n} \log (\frac{b-a}{\delta})}=o(\sqrt{\frac{1}{n} \log (\frac{b-a}{\delta})})$. When combining the above arguments we thus obtain (\ref{lem0eq1}).\\

Before considering \eqref{lem0eq2}  note that it follows from \eqref{lem0eq0} and \eqref{lem0eq1} that there exists a constant $0<L_{3,1}<\infty$
such that as $n\rightarrow\infty$.
\begin{align}\label{eq:lem0inf}
P(\inf_{u\in [a+\delta,b-\delta]}\frac{1}{n}\sum_{i=1}^n Z_{\delta,i}(u)^2\geq L_{3,1}\delta^\kappa)
\rightarrow 1
\end{align}
Now consider \eqref{lem0eq2} and keep in mind that
$$(\frac{1}{n}\sum_{i=1}^n Z_{\delta,i}(t)^2)^{\frac{1}{2}}-(\frac{1}{n}\sum_{i=1}^n Z_{\delta,i}(s)^2)^{\frac{1}{2}}
=\frac{\frac{1}{n}\sum_{i=1}^n Z_{\delta,i}(t)^2-\frac{1}{n}\sum_{i=1}^n Z_{\delta,i}(s)^2}{(\frac{1}{n}\sum_{i=1}^n Z_{\delta,i}(t)^2)^{\frac{1}{2}}+(\frac{1}{n}\sum_{i=1}^n Z_{\delta,i}(s)^2)^{\frac{1}{2}}}.$$
Some straightforward computations lead to
\begin{align}\begin{split}\label{lem0pr01}
&\sup_{t\in [a+\delta,b-\delta]} (\frac{1}{n}\sum_{i=1}^n Z_{\delta,i}(t)^2)^{-\frac{1}{2}}|\frac{1}{n}\sum_{i=1}^n
 Z_{\delta,i}(t)Y_i-\mathbb{E}(Z_{\delta,i}(t)Y_i)|\\
 &\leq \sup_{j\in \{2,\dots,N_{w_1}\}}(\frac{1}{n}\sum_{i=1}^n Z_{\delta,i}(s_j)^2)^{-\frac{1}{2}} |\frac{1}{n}\sum_{i=1}^n
 Z_{\delta,i}(s_j)Y_i-\mathbb{E}(Z_{\delta,i}(s_j)Y_i)|\\
&+  \sup_{j\in \{2,\dots,N_{w_1}\}}\left(\sup_{t\in [s_{j-1},s_{j}]}
\frac{|\frac{1}{n}\sum_{i=1}^n Z_{\delta,i}(t)^2-\frac{1}{n}\sum_{i=1}^n Z_{\delta,i}(s_j)^2|}{2\inf_{u\in [s_{j-1},s_{j}]}\frac{1}{n}\sum_{i=1}^n Z_{\delta,i}(u)^2}
 \frac{|\frac{1}{n}\sum_{i=1}^n
 Z_{\delta,i}(s_j)Y_i-\mathbb{E}(Z_{\delta,i}(s_j)Y_i)|}{(\frac{1}{n}\sum_{i=1}^n Z_{\delta,i}(s_j)^2)^{\frac{1}{2}}}\right)\\
 &+
 \sup_{j\in \{2,\dots,N_{w_1}\}} \sup_{t\in [s_{j-1},s_{j}]}
\frac{|\frac{1}{n}\sum_{i=1}^n
 Z_{\delta,i}(t)Y_i-\mathbb{E}(Z_{\delta,i}(t)Y_i)-(Z_{\delta,i}(s_j)Y_i-\mathbb{E}(Z_{\delta,i}(s_j)Y_i))|}{\inf_{u\in [s_{j-1},s_{j}]}(\frac{1}{n}\sum_{i=1}^n Z_{\delta,i}(u)^2)^{\frac{1}{2}}}.
\end{split}
\end{align}
Furthermore,
{\small \begin{align*}
|\frac{1}{n}\sum_{i=1}^n Z_{\delta,i}(t)^2&-\frac{1}{n}\sum_{i=1}^n Z_{\delta,i}(s_j)^2|\leq |\mathbb{E}(Z_{\delta,i}(t)^2)-\mathbb{E}(Z_{\delta,i}(s_j)^2)|\\
&+
|\frac{1}{n}\sum_{i=1}^n Z_{\delta,i}(t)^2-\frac{1}{n}\sum_{i=1}^n Z_{\delta,i}(s_j)^2-
(\mathbb{E}(Z_{\delta,i}(t)^2)-\mathbb{E}(Z_{\delta,i}(s_j)^2))|,
\end{align*}}
\noindent and it follows from \eqref{lem1pr02} and \eqref{lem0eq0} that there is a constant $0<L_{3,2}<\infty$ such that
 for every $j\in \{2,\dots,N_{w_1}\}$
{\small $$|\mathbb{E}(Z_{\delta,i}(t)^2)-\mathbb{E}(Z_{\delta,i}(s_j)^2)|
\leq \sqrt{\mathbb{E}((Z_{\delta,i}(t)-Z_{\delta,i}(s_j))^2)\mathbb{E}((Z_{\delta,i}(t)+Z_{\delta,i}(s_j))^2)}
\leq L_{3,2} \delta^\kappa \sqrt{\frac{\ell_{w_1}^\kappa}{\delta^\kappa}}$$}
\noindent holds for all sufficiently
large $n$. Together with \eqref{eq:lem0inf} We can therefore infer from Lemma \ref{lem1} that for some constants $0<L_{3,3}<\infty$,
$0<L_{3,4}<\infty$
\begin{align} \label{lem0pr02}
P\Bigg(\sup_{j\in \{2,\dots,N_{w_1}\}}\sup_{t\in [s_{j-1},s_{j}]} \frac{|\frac{1}{n}\sum_{i=1}^n Z_{\delta,i}(t)^2-\frac{1}{n}\sum_{i=1}^n Z_{\delta,i}(s_j)^2|}{2\inf_{u\in [s_{j-1},s_{j}]}\frac{1}{n}\sum_{i=1}^n Z_{\delta,i}(u)^2}
\leq L_{3,3} \sqrt{\frac{\ell_{w_1}^\kappa}{\delta^\kappa}}
\Bigg) \rightarrow 1,
\end{align}
\begin{align}
\begin{split}\label{lem0pr03}
P\Bigg(&\sup_{j\in \{2,\dots,N_{w_1}\}} \sup_{t\in [s_{j-1},s_{j}]}
\frac{|\frac{1}{n}\sum_{i=1}^n
 Z_{\delta,i}(t)Y_i-\mathbb{E}(Z_{\delta,i}(t)Y_i)-(Z_{\delta,i}(s_j)Y_i-\mathbb{E}(Z_{\delta,i}(s_j)Y_i))|}{\inf_{u\in [s_{j-1},s_{j}]}(\frac{1}{n}\sum_{i=1}^n Z_{\delta,i}(u)^2)^{\frac{1}{2}}} \\
& \qquad \quad
\leq L_{3,4} \sqrt{\frac{\ell_{w_1}^\kappa}{\delta^\kappa n} \log \Big(\frac{b-a}{\delta}\Big)}
\Bigg) \rightarrow 1,
\end{split}
\end{align}
as $n\rightarrow\infty$.
Moreover, by \eqref{eq:lem0inf}, we have with probability $1$ as $n \to \infty$,
\begin{align}\label{eq:lem0inf1}
\sup_{j\in\{2,3,\dots, N_{\omega_1}\}} \frac{|\frac{1}{n}\sum_{i=1}^n Z_{\delta,i}(s_j)Y_i - \E(Z_{\delta,i}(s_j)Y_i)|}{(\frac{1}{n}\sum_{i=1}^n Z_{\delta,i}(s_j)^2 )^{\frac{1}{2}}} \leq \sup_{j\in\{2,3,\dots, N_{\omega_1}\}} \frac{|\frac{1}{n}\sum_{i=1}^n Z_{\delta,i}(s_j)Y_i - \E(Z_{\delta,i}(s_j)Y_i)|}{(L_{3,1}\delta^\kappa )^{\frac{1}{2}}}.
\end{align}
Chose an arbitrary point $s_j$, and define $W_i(s_j):= \frac{1}{\sqrt{D_4 \delta^\kappa\sigma_{|y|}^2}} (Z_{\delta,i}(s_j)Y_i  - \E(Z_{\delta,i}(s_j)Y_i))$. Then $E(W_i(s_j))=0$ and it is easy to show that under Assumption \ref{assum2} with %$K=2^{\frac{4}{2}}a_1^{\frac{1}{2}}$,
$K=4$, %new 18.01.2018 to account for \mathcal{A}=2
a constant which is independent of $s_j$, %, where $a_1:= \max\{1,\mathcal{A}\}$,
$W_i(s_j)$ satisfies the Bernstein condition in Corollary 1 of \citeappendix{vandeGeer2013}, i.e., we have
$$
\E(|W_i(s_j)|^m) \leq \frac{m!}{2} K^{m-2} K^2\quad\text{for all}\quad m=2,3,\dots
$$
It immediately follows from an application of Corollary 1 in \citeappendix{vandeGeer2013} that there exists a constant $0<L_{3,5}<\infty$ %in fact: 0<"`c"'\leq L_3 = D_{1,1} \leq \sqrt{6}\sigma = \sqrt{6}K
such that % for $\Psi(x) = \exp(\frac{n}{6}(\sqrt{1+2\sqrt{\frac{6}{n}}x}-1)^2)-1$ \textcolor{red}{HATTEN WIR IN LEMMA 1 SCHON DEFINIERT!!! KANN MAN KUERZEN?!}
the Orlicz-Norm $||\frac{1}{\sqrt{n}}\sum_{i=1}^n W_i(s_j)||_{\Psi}$ can be bounded by  $L_{3,5}<\infty$. And hence we can infer that
$$\mathbb{E}\Bigg(\exp(\frac{n}{6}(\sqrt{1+2\sqrt{\frac{6}{L_{2,5}^2 n}}|\frac{1}{\sqrt{n}}\sum_{i=1}^n W_i(s_j)|}-1)^2)\Bigg)\leq 2.$$
It then follows from similar steps as in the proof of Lemma~\ref{lem1} that there exists a constant $0<L_{3,6}<\infty$ such that for all $0<z\leq \sqrt{n}$ we obtain for all $s_j$
\begin{align*}
P\Bigg(&|\frac{1}{\sqrt{n}}\sum_{i=1}^n W_i(s_j)| > z L_{2,6}\Bigg)\\
& = P\Bigg(\frac{|\frac{1}{n}\sum_{i=1}^n Z_{\delta,i}(s_j)Y_i  - E(Z_{\delta,i}(s_j)Y_i)|}{\sqrt{L_{3,1} \delta^\kappa}} > z L_{3,6} \sqrt{\frac{D_4 \delta^\kappa\sigma_{|y|}^2}{n L_{3,1} \delta^\kappa}}\Bigg) \leq 2\exp(-z^2).
\end{align*}
We can thus conclude that there exists a constant $0<L_{3,7}<\infty$ such that
\begin{align*}
%P(&\frac{ |\frac{1}{n}\sum_{i=1}^n Z_{\delta,i}(s_j)Y_i  - E(Z_{\delta,i}(s_j)Y_i)| }{\sqrt{L_{2,1} \delta^\kappa}}  > z L_{4} \sqrt{\frac{D_4 \delta^\kappa\sigma_{|y|}^2}{n L_{2,1} \delta^\kappa} } ) \\
%&=
P\Bigg(\frac{|\frac{1}{n}\sum_{i=1}^n Z_{\delta,i}(s_j)Y_i  - E(Z_{\delta,i}(s_j)Y_i)|}{(L_{3,1} \delta^\kappa)^\frac{1}{2}} > z L_{3,7} \sqrt{\frac{\sigma_{|y|}^2}{n}}\Bigg) \leq 2\exp(-z^2).
\end{align*}
%%%%%%%%%%%%%%%%%
%\textcolor{red}{HIER KANN MAN NOCH ABKUERZEN: SOFORT $L_5$ ZWEI ZEILEN DRUEBER EINFUEHREN UND $L_4$ DAFUER AUFGEBEN!!}
%%%%%%%%%%%%%%%%%
Choose some  $\omega_3$ with $\omega_3>\omega_1>1$ and note that our conditions on $\delta$ imply that $z =  \sqrt{\omega_{3} \log{(\frac{b-a}{\delta}})} \leq \sqrt{n}$ for all sufficiently large $n$.
Using the union bound this leads to
\begin{align}
P\Bigg(\sup_{j\in\{2,3,\dots,  N_{\omega_1}\}}& \frac{|\frac{1}{n}\sum_{i=1}^n Z_{\delta,i}(s_j)Y_i - E(Z_{\delta,i}(s_j)Y_i)|}{(L_{3,1}\delta^\kappa )^{\frac{1}{2}}} \leq  \sqrt{\omega_{3} \log{\Big(\frac{b-a}{\delta}}\Big)} L_{3,7} \sqrt{\frac{\sigma_{|y|}^2}{n}}\Bigg) \nonumber  \\
%& =1 - P(\sup_{j\in\{2,3,\dots, N_{\omega_1} \}} \frac{|\frac{1}{n}\sum_{i=1}^n Z_{\delta,i}(s_j)Y_i - E(Z_{\delta,i}(s_j)Y_i)|}{(L_1\delta^\kappa )^{\frac{1}{2}}} >  z L_{5} \sqrt{\frac{\sigma_{|y|}^2}{n}}) \\
& \geq 1 - \sum_{j=1}^{N_{\omega_1}} P\Bigg(\frac{|\frac{1}{n}\sum_{i=1}^n Z_{\delta,i}(s_j)Y_i - E(Z_{\delta,i}(s_j)Y_i)|}{(L_{3,1}\delta^\kappa )^{\frac{1}{2}}} >  \sqrt{\omega_{3} \log{\Big(\frac{b-a}{\delta}}\Big)} L_{3,7} \sqrt{\frac{\sigma_{|y|}^2}{n}}\Bigg) \nonumber \\
%&\geq 1- N_{\omega_1} 2 \exp(-\sqrt{\omega_{3} \log{(\frac{b-a}{\delta}})}^2)\\
&\geq 1- N_{\omega_1} 2 \exp\left(-\omega_{3} \log{\Big(\frac{b-a}{\delta}\Big)}\right) \geq 1 - 2\Big(\frac{b-a}{\delta}\Big)^{\omega_{1}-\omega_3} \to 1. \label{eq:lem0new}
\end{align}
With $D = \sqrt{\omega_3}L_{3,7}<\infty$, assertion~\eqref{lem0eq2} holds for all $A^*>D$ by \eqref{lem0pr01}-\eqref{eq:lem0new}.

Finally, \eqref{lem0eq3} follows from \eqref{lem0eq0}, \eqref{lem0eq1} and \eqref{lem0eq2} by noting that the assertions in particular imply that with probability converging to $1$ (as $n\to \infty$) we have $\sup_{t\in(a,b)} \frac{1}{n}\sum_{i=1}^n Z_{\delta,i}(t)^2 \leq D_{4,\epsilon} \delta^\kappa$ for any constant $D_{4,\epsilon}>D_4$.
\end{proof}
%% %%%%%%%%%%%%%%%%%%%%%%%%%%%%%%%%%%%%%%%%%%%%%%%
%% END Lemma 2 Replacement
%% %%%%%%%%%%%%%%%%%%%%%%%%%%%%%%%%%%%%%%%%%%%%%%%

\bigskip

%% %%%%%%%%%%%%%%%%%%%%%%%%%%%%%%%%%%%%%%%%%%%%%%%
%% Remarks on the choice of  $\lambda$
%% %%%%%%%%%%%%%%%%%%%%%%%%%%%%%%%%%%%%%%%%%%%%%%%
Contrary to Lemma 2 in the supplementary paper to Kneip et al.~(2016), we don't have the simple expression $D=\sqrt{2}$; this is the price to pay for not assuming $Y_i$ to be Gaussian.
%%%%%%%%%%%%%%%%%%%%%%

\bigskip

\noindent \textbf{Remarks to Lemma~\ref{lem0} concerning the threshold $\lambda$:}
\begin{enumerate}
\item Using a slight abuse of notation, first note that there is a close connection between  $\lambda = A\sqrt{\sigma_{|y|}^2\log (\frac{b-a}{\delta})/n}$ for some $A>D$ given in Theorem~\ref{poicon} and $\widetilde\lambda := A\sqrt{\sqrt{\E(Y^4)}\log (\frac{b-a}{\delta})/n}$ for $A=\sqrt{2\sqrt{3}}$ as used in our simulations. Indeed, set $\sigma_{|y|}^2 = \E(Y^2)$. Jensen's inequality implies that there exists a constant $0<\widetilde{D}\leq 1$ such that $\E(Y^2) \widetilde{D} = \sqrt{\E(Y^4)}$.
We can therefore rewrite the expression for $\widetilde\lambda$ in the form of $\lambda$ presented in Theorem~\ref{poicon} as $A\sqrt{\sigma_{|y|}^2\log (\frac{b-a}{\delta})/n}$ with $A=\sqrt{2\sqrt{3}\widetilde{D}}$.
%%%
%, where now the constant $D$ appearing in this theorem can be extracted as $\sqrt{2\sqrt{3}D_1}$.
\end{enumerate}
We proceed to give more details about the motivation for the threshold used in the simulation:

\begin{enumerate}[resume]
\item Arguments for the applicability of the threshold $\lambda$ in the proof of Theorem~\ref{poicon} follow from Lemma~\ref{lem0}. The crucial step for determining an operable threshold $\lambda$ is to derive useful bounds on
 $$\sup_{j\in\{2,3,\dots, N_{\omega_1}\}} \frac{|\frac{1}{n}\sum_{i=1}^n Z_{\delta,i}(s_j)Y_i - E(Z_{\delta,i}(s_j)Y_i)|}{(\frac{1}{n}\sum_{i=1}^n Z_{\delta,i}(s_j)^2 )^{\frac{1}{2}}}.$$
Define $V_{\delta}(t) :=(1/n\sum_{i=1}^n Z_{\delta,i}(t)Y_i - \E(Z_{\delta,i}(t)Y_i))/(1/n \sum_{i=1}^n Z_{\delta,i}(t)^2)^{1/2}$. It is then easy to see that under our assumptions
$\sqrt{n}(1/n\sum_{i=1}^n Z_{\delta,i}(t)Y_i - \E(Z_{\delta,i}(t)Y_i))$ satisfies the Lyapunov conditions. We hence can conclude that $\sqrt{n}V_{\delta}(t)$ converges for all $t$ in distribution to $N(0, \V(Z_{\delta,i}(t)Y_i)/\E(Z_{\delta,i}(t)^2))$, while at the same time the Cauchy-Schwarz inequality implies $\V(Z_{\delta,i}(t)Y_i)/ \E(Z_{\delta,i}(t)^2) \leq \sqrt{3 \E(Y_i^4)}$.\\ %hier wurde auch noch normalität der Z genutzt!
If the convergence to the normal distribution is sufficiently fast, the union bound in the proof of Lemma~\ref{lem0} together with an elementary bound on the tails of the normal distribution leads to
$$P\left(\sup_{j\in\{2,3,\dots, N_{\omega_1}\}} \frac{|\frac{1}{n}\sum_{i=1}^n Z_{\delta,i}(s_j)Y_i - E(Z_{\delta,i}(s_j)Y_i)|}{(\frac{1}{n}\sum_{i=1}^n Z_{\delta,i}(s_j)^2 )^{\frac{1}{2}}} \leq A^*\sqrt{ \frac{\sqrt{\E(Y_i^4)}}{n} \log (\frac{b-a}{\delta})} \right) \to 1,$$
for some $A^*\geq \sqrt{2\sqrt{3}}$. The threshold $A\sqrt{\sqrt{\E(Y_i^4)}\log (\frac{b-a}{\delta})/n}$ for some $A\geq \sqrt{2\sqrt{3}}$ is then an immediate consequence.
\end{enumerate}

\bigskip

%% %%%%%%%%%%%%%%%%%%%%%%%%%%%%%%%%%%%%%%%%%%%%%%%
%% Lemma 4 Replacement
%% %%%%%%%%%%%%%%%%%%%%%%%%%%%%%%%%%%%%%%%%%%%%%%%
\begin{lemma}\label{lem3}
Under the assumptions of Theorem \ref{poicon} let $I_r:=\{t\in[a,b]|\ |t-\tau_r|\leq
\min_{s\neq r} |t-\tau_s|\}$, $r=1,\dots,S$. If $S>0$, there then exist a constant $0<Q_1<\infty$ and
for each $0<\kappa<2$ %%NEW NEW NEW NEW!!!!
a constant $0<Q_2<\infty$ such that for all sufficiently small $\delta>0$ and all $r=1,\dots,S$ we have
\begin{align}
|\mathbb{E}(Z_{\delta,i}(t)Y_i)|\leq  Q_1\frac{\delta^2}{\max\{\delta,|t-\tau_r|\}^{2-\kappa}} %+ M_{sup}^* \delta^{\min\{2,\kappa+1\}}
\quad \text{ for every } t\in I_r,
\label{lem3eq1}
\end{align}
%%%%%%%%%%%%%%
as well as
%%%%%%%%%%%%%%
\begin{align}
\sup_{t\in I_r,\ |t-\tau_r|\geq \frac{\delta}{2}}
|\mathbb{E}(Z_{\delta,i}(t)Y_i)|\leq  (1-Q_2) |\vartheta_r| c(\tau_r)\delta^\kappa,
\label{lem3eq2}
\end{align}
and for any $u\in [-0.5,0.5]$
\begin{align} \label{lem3eq3}
&\mathbb{E}(Z_{\delta,i}(\tau_r+u\delta)Y_i) - \mathbb{E}(Z_{\delta,i}(\tau_r)Y_i)\nonumber \\
& \qquad =
- c \vartheta_r c(\tau_r)\delta^\kappa\left(|u|^\kappa
-\frac{1}{2}(|u+1|^\kappa-1)-\frac{1}{2}(|u-1|^\kappa-1)\right)
+ R_{7;r}(u),
\end{align}
where $|R_{7;r}(u)|\leq \widetilde{M}_{r}||u|^{1/2}\delta|^{\min\{2\kappa,2\}}$ for some constants
$\widetilde{M}_r<\infty$, $r=1,\dots,S$ and $\vartheta_r = \E(\frac{\partial}{\partial x_r}g(X_i(\tau_1),\dots,X_i(\tau_S))$.
\end{lemma}

\bigskip

\begin{proof}[{\bf Proof of Lemma~\ref{lem3}.}]
Theorem~\ref{thm:NP1} guarantees us the existence of constants $\vartheta_1,\dots, \vartheta_S$ with $\vartheta_r = \E(\frac{\partial}{\partial x_r}g(X_i(\tau_1),\dots,X_i(\tau_S)))$, $r=1,\dots, S$,  such that that for all $t\in [a+\delta,b-\delta]$ we have
\begin{align}
\mathbb{E}(Z_{\delta,i}(t)Y_i) =\sum_{r=1}^S \vartheta_r \mathbb{E}(Z_{\delta,i}(t)X_i(\tau_r)).
\label{lem3eqp1}
\end{align}
 Since $\tau_1,\dots,\tau_S\in(a,b)$ are fixed, we have
$\tau_r\in [a+\delta,b-\delta]$, $r=1,\dots,S$, as well as
  $\delta\ll \frac{1}{2}\min_{r\neq s} |\tau_r-\tau_s|$ for all sufficiently small $\delta>0$.
 Using \eqref{lem3eqp1}, assertions \eqref{lem3eq1} and \eqref{lem3eq2} are thus immediate consequences of
\eqref{Zdelta4} and \eqref{Zdelta5}.

\noindent In order to prove \eqref{lem3eq3}) note that similar to \eqref{Zdelta3}
straightforward Taylor expansions can be used to show that
\begin{align}
\E(&Z_{\delta,i}(\tau_r + u\delta)X(t)  - Z_{\delta,i}(\tau_r)X(t)) =
\omega(\tau_r + u\delta,t,|\tau_r - t + u\delta|^\kappa) - \omega(\tau_r,t,|\tau_r - t|^\kappa) \nonumber\\
&\quad -0.5(\omega(\tau_r + (u+1)\delta,t,|\tau_r - t + (u+1)\delta|^\kappa) -\omega(\tau_r + \delta,t,|\tau_r - t +\delta|^\kappa))\nonumber\\
&\quad -0.5(\omega(\tau_r + (u-1)\delta,t,|\tau_r - t + (u-1)\delta|^\kappa) -\omega(\tau_r - \delta,t,|\tau_r - t -\delta|^\kappa))\nonumber\\
\begin{split}\label{eq:lem4.aux}
& =
 \omega_3(\tau_r ,t,|\tau_r - t|^\kappa) \Big((|\tau_r - t + u\delta|^\kappa-|\tau_r - t|^\kappa)\\
&\quad -0.5(|\tau_r - t + (u+1)\delta|^\kappa- |\tau_r - t + \delta|^\kappa)
-0.5(|\tau_r - t + (u-1)\delta|^\kappa- |\tau_r - t - \delta|^\kappa)\Big)\\
&\quad + R_{8;u,t,\delta}
\end{split}
\end{align}
%\begin{itemize}
%\item $|R_1|\leq M((u\delta)^2 + (u\delta)^{2\kappa}) \leq M_c ((|u|^{1/2}\delta)^{\min\{2,2\kappa}) $
%\item $|R_2|\leq M((u\delta)^2 + |u\delta|^{2\kappa}) \leq M_c ((|u|^{1/2}\delta)^{\min\{2,2\kappa}) $
%\item $|R_3|\leq M((u\delta)^2 + |u\delta|^{2\kappa}) \leq M_c ((|u|^{1/2}\delta)^{\min\{2,2\kappa}) $
%\item $|R_4|\leq M(\delta + ||\tau_r - t + \delta|^\kappa - |\tau_r - t|^\kappa | ) \leq M(\delta + \delta^\kappa)$
%\item $|R_5|\leq M(\delta + ||\tau_r - t - \delta|^\kappa - |\tau_r - t|^\kappa | ) \leq M(\delta + \delta^\kappa)$
%\item $|R_6|\leq M(\delta + \delta^\kappa)$
%\item $|R_7|\leq M(\delta + \delta^\kappa)$
%\item $|R_8| \leq M(|u|^{0.5}\delta)^{\min\{2,2\kappa\}}$
%\end{itemize}
%
%Note:
%\begin{itemize}
%\item $|u\delta R_4|\leq M_c (|u|^{0.5}\delta)^{\min\{2,2\kappa\}}$
%\item $|u\delta R_5|\leq M_c (|u|^{0.5}\delta)^{\min\{2,2\kappa\}}$
%\item $||\tau_r - t + (u+1)\delta|^\kappa- |\tau_r - t + \delta|^\kappa| |R_6|\leq M |u\delta|^\kappa |R_6| \leq M_c(|u|^{0.5}\delta)^{\min\{2,2\kappa\}}$
%\item $||\tau_r - t + (u-1)\delta|^\kappa- |\tau_r - t - \delta|^\kappa| |R_7|\leq M_c(|u|^{0.5}\delta)^{\min\{2,2\kappa\}}$
%\end{itemize}
Where for $|u|\leq 0.5$ we have $|R_{8;u,t,\delta}| \leq L_{4,1}(|u|^{0.5}\delta)^{\min\{2,2\kappa\}}$ for some constant $L_{4,1}$ with $|L_{4,1}|<\infty$.
Since $\delta \ll \frac{1}{2}\min_{r\neq s}|\tau_r-\tau_s|$ for all sufficiently small $\delta$. We can conclude from \eqref{eq:lem4.aux} that $$|\E(Z_{\delta,i}(\tau_r + u\delta)X(\tau_s)  - Z_{\delta,i}(\tau_r)X(\tau_s))| \leq L_{4,2} ||u|^{1/2}\delta|^{\min\{2,2\kappa\}}$$
for some constant $L_{4,2}$ with  $|L_{4,2}|<\infty$ and all $r,s \in \{1,\dots, S\}$, $r\neq s$.
Assertion \eqref{lem3eq3} is then an immediate consequence of \eqref{lem3eqp1} and \eqref{Zdelta3}.
\end{proof}
%% %%%%%%%%%%%%%%%%%%%%%%%%%%%%%%%%%%%%%%%%%%%%%%%
%% END Lemma 4 Replacement
%% %%%%%%%%%%%%%%%%%%%%%%%%%%%%%%%%%%%%%%%%%%%%%%%

\bigskip

%% %%%%%%%%%%%%%%%%%%%%%%%%%%%%%%%%%%%%%%%%%%%%%%%
%% Proof of Theorem 4 - Replacement
%% %%%%%%%%%%%%%%%%%%%%%%%%%%%%%%%%%%%%%%%%%%%%%%%
\begin{proof}[{\bf Proof of Theorem \ref{poicon}.}]
Let $\lambda_n = A\sqrt{\frac{\sigma_{|y|}^2}{n}\log\big(\frac{b-a}{\delta} \big)}$  and let
 $I_{\delta}:=\{t_j|\ t_j\in[a+\delta,b-\delta], j\in\{1,\dots,p\}\}$.
 For any $t\in I_{\delta}$ we have
 \begin{align}
\frac{1}{n}\sum_{i=1}^n
 Z_{\delta,i}(t)Y_i=
\mathbb{E}(Z_{\delta,i}(t)Y_i) + \frac{1}{n}\sum_{i=1}^n
( Z_{\delta,i}(t)Y_i-\mathbb{E}(Z_{\delta,i}(t)Y_i)).
\label{poicon-eq1}
\end{align}

\noindent In the case in which there are no points of impact, i.e. $S=0$,
we then have $\E(Z_{\delta,i}(t)Y_i) = 0$ for all $t \in [a-\delta, b-\delta]$ and all $\delta>0$. Lemma \ref{lem0} implies that
$$P\left(\sup_{t\in [a+\delta,b-\delta]} \frac{|\frac{1}{n}\sum_{i=1}^n
 Z_{\delta,i}(t)Y_i|}{(\frac{1}{n}\sum_{i=1}^n Z_{\delta,i}(t)^2)^{1/2}}\leq \lambda_n\right)
\rightarrow 1,$$
 and hence $P(\widehat S=S)\rightarrow 1$, as $n\rightarrow\infty$.

%%%%%%
%\textcolor{red}{Anmerkung: hier wurde ursprünglich bereits die Annahme $\delta^2 = O(1/n)$ im Beweis genutzt, damit der im Ursprungspaper vorhandene funktionale Part schnell genug verschwindet ($\delta^{\min\{2,\kappa+1\}}=o(\sqrt{\delta^\kappa/n}$) folgt m.E. in unserem Setting NUR aus der Bedingung..) - beachte, dass man im Ursprungspaper an dieser Stelle auch lediglich die zweite Aussage zeigt, aber nicht, dass der Algorithmus selbst nicht auch schon ohne die zusätzliche Bedingung an $\delta$ auch früh genug abbricht (also tatsächlich der Algorithmus selbst auch $\widehat{S}=0$ Kandidaten findet..)}
%%%%%%
% - Im Fall ohne funktionalen Part ist die Sache klarer: Mit WS gegen 1 wird nicht nur $P(\widehat{S}=S}=1$ gelten, sondern auch
%- hierzu müsste für die allgemeineren Bedingungen an $n$ und $\delta$ bereits $\delta^{2, \kappa+1}/\delta^{\kappa/2} ) = o(\lambda_n)$ gelten...
%NUR mit $\delta^2 = O(n^{-1})$ gilt jedoch $\delta^{\min\{2,\kappa+1\}} = o (\sqrt{\delta^\kappa/n})$, somit auch $\delta^{2, \kappa+1}/\delta^{\kappa/2} = o(\sqrt{1/n}) = o(\lambda_n)$...}

\noindent Now consider the case $S\ge 1$.
Select some arbitrary $\alpha>2$. As $n\rightarrow\infty$ we have $\delta\equiv\delta_n
\rightarrow 0$. Therefore, $\tau_r\in [a+\delta,b-\delta]$, $r=1,\dots,S$, as well as
  $\sqrt{\delta}/\alpha<\frac{1}{2}\min_{r\neq s} |\tau_r-\tau_s|$, provided that $n$ is sufficiently large.
	
Let $I_{r,\delta,\alpha}:=\{t\in I_{\delta}\ | \ |t-\tau_r|\leq \sqrt{\delta}/\alpha\}$, $r=1,\dots,S$, as well as $I_{\delta,\alpha}=\bigcup_{r=1}^S I_{r,\delta,\alpha}$ and
$I_{\delta,\alpha}^C:=I_{\delta}\backslash I_{\delta,\alpha}$.

\noindent By our assumptions on the sequence $\delta\equiv\delta_n$ we can infer from \eqref{poicon-eq1}, \eqref{lem3eq1}, and \eqref{lem0eq3} that there exist constants $0<C_1<\infty$
and $0<C_2<\infty$  such
that the event
\begin{align}
\sup_{t\in I_{\delta,\alpha}^C} |\frac{1}{n}\sum_{i=1}^n
Z_{\delta,i}(t)Y_i|\leq C_1\sqrt{\frac{\delta^\kappa}{n}|\log \delta|}+
C_2\alpha^{2-\kappa} \delta^{1+\kappa/2}
\label{poicon-eq2}
\end{align}
holds with probability tending to 1 as $n\rightarrow\infty$.
%% %%%%%%%%%%%%%%%%%%%%%%%%%%%%%%
%% Proof \eqref{poicon-eq2}
%% %%%%%%%%%%%%%%%%%%%%%%%%%%%%%%
% Indeed:
% We have $t\in I_{\delta,\alpha}^C$ and hence for all $r=1,2,\dots, S$: $|t-\tau_r| > \sqrt{\delta}/\alpha > \delta$ as $n \to \infty$. By \eqref{lem3eq1} we then have for all $t\in I_{\delta,\alpha}^C$
% $$ |\E(Z_{\delta,i}(t)Y_i)| \leq Q_1 \frac{\delta^2}{\Big(\sqrt{\delta}/\alpha \Big)^{2-\kappa}} = Q_1 \delta^{1+0.5\kappa}\alpha^{2-\kappa}.$$
%
% At the same time note that as $n\to \infty$ there exists a constant $M>1$ such that $\log((b-a)/\delta) \leq M|\log(\delta)|$.
% Hence,
% \begin{align*}
% A^* \sqrt{\sigma_{|y|}^2D_{4,\epsilon} \frac{\delta^\kappa}{n}\log((b-a)/\delta)}
% & \leq A^* \sqrt{\sigma_{|y|}^2D_{4,\epsilon} M \frac{\delta^\kappa}{n}\|log(\delta)|}\\
% & = C_1 \sqrt{\delta^\kappa/n |\log \delta|},
% \end{align*}
% where $C_1 := A^* \sqrt{\sigma_{|y|}^2D_{4,\epsilon}M}$.
% Since
% $$\sup_{t\in I_{\delta,\alpha}^C} |f(t)| \leq \sup_{t\in \left[ a+\delta,b-\delta \right]}|f(t)|,$$
% we can hence conclude from \eqref{lem0eq3} that with probability converging to 1 we additionally have
% $$\sup_{t\in I_{\delta,\alpha}^C} |\frac{1}{n}\sum_{i=1}^n Z_{\delta,i}(t)Y_i - \E(Z_{\delta,i}(t)Y_i)| \leq C_1 \sqrt{\frac{\delta^\kappa}{n}|\log \delta|}.$$
% The assertion then follows from \eqref{poicon-eq1}
%% %%%%%%%%%%%%%%%%%%%%%%%%%%%%%%
%% End Proof \eqref{poicon-eq2}
%% %%%%%%%%%%%%%%%%%%%%%%%%%%%%%%
 Since by assumption
$\frac{|\log \delta|}{n\delta^\kappa}\rightarrow 0$ and hence
 $\sqrt{\frac{\delta^\kappa}{n}|\log \delta|}=o(\delta^\kappa)$,  the decomposition given in \eqref{poicon-eq1} together with \eqref{lem3eq2} and \eqref{lem0eq3}
imply the existence of a constant $0<C_3<Q_2$ such that
\begin{align}
\sup_{t\in I_{r,\delta,\alpha}, |t-\tau_r|\geq \delta/2} |\frac{1}{n}\sum_{i=1}^n
Z_{\delta,i}(t)Y_i|\leq (1-C_3) |\vartheta_r| c(\tau_r)\delta^\kappa, \quad r=1,\dots,S
\label{poicon-eq3}
\end{align}
hold with probability tending to 1 as $n\rightarrow\infty$.

\noindent For $r=1,\dots,S$ let $j(r)$ be an index satisfying $|\tau_r-t_{j(r)}|=\min_{j=1,\dots,p} |\tau_r-t_{j}|$.
Obviously $|\tau_r-t_{j(r)}|\leq \frac{b-a}{2(p-1)}$  and
 by  \eqref{Zdelta4} our conditions on $p\equiv p_n$ imply %that there exists a constant $0<C_4<\infty$ such that
%\begin{align*}
%|\mathbb{E}(Z_{\delta,i}(t_{j(r)})X_i(\tau_r))-c(\tau_r)\delta^\kappa|\leq C_4 n^{-1/\kappa}, \, r=1,\dots, S.
%\end{align*}
\begin{align*}
|\mathbb{E}(Z_{\delta,i}(t_{j(r)})X_i(\tau_r))-c(\tau_r)\delta^\kappa| = o(\delta^\kappa), \, r=1,\dots, S.
\end{align*}

Using again \eqref{lem0eq3} together with $\sqrt{\frac{\delta^\kappa}{n}|\log \delta|}=o(\delta^\kappa)$ and \eqref{Zdelta5}, we can thus conclude that there exists a sequence
$\{\epsilon_{n}\}$ of positive numbers with $\lim_{n\rightarrow\infty} \epsilon_{n}\rightarrow 0$ such that
\begin{align}
|\frac{1}{n}\sum_{i=1}^nZ_{\delta,i}(t_{j(r)})Y_i|\geq (1-\epsilon_n) |\vartheta_r| c(\tau_r)\delta^\kappa, \quad r=1,\dots, S
\label{poicon-eq4}
\end{align}
holds with probability tending to 1 as $n\rightarrow\infty$.

 Now define
 \begin{align}
\widetilde{\tau}_r:=arg \max_{t\in I_{\delta}:\ |t-\tau_r|\leq \delta/2} |\frac{1}{n}\sum_{i=1}^nZ_{\delta,i}(t)Y_i|.
\label{poicon-eq5}
\end{align}
 Since $\delta^{1+\kappa/2}=o(\delta^\kappa)$ and since $\alpha>2$, one can infer from \eqref{poicon-eq2} -
\eqref{poicon-eq4} that the following assertions hold
 with probability tending to 1
 as $n\rightarrow\infty$:
  \begin{align}
\widetilde{\tau}_r=arg \max_{t\in I_{r,\delta,\alpha}} |\frac{1}{n}\sum_{i=1}^nZ_{\delta,i}(t)Y_i|=
arg \max_{t\in I_{r,\delta,\alpha}\cup I_{\delta,\alpha}^C} |\frac{1}{n}\sum_{i=1}^nZ_{\delta,i}(t)Y_i|,
\quad r=1,\dots,S,
\label{poicon-eq6}
\end{align}
%% %%%%%%%%%%%%%%%%%%%%%%%%%%%%%%%%%%%%%%%%%%%%%%%
%% Proof of assertion \eqref{poicon-eq6}
%% %%%%%%%%%%%%%%%%%%%%%%%%%%%%%%%%%%%%%%%%%%%%%%%
% This is easy. We have
% \begin{enumerate}
% \item $|\tau_r - t_{j(r)}| \leq L n^{-1/\kappa} \leq \delta/2$, as $n\to \infty$
% \item $\sup_{t\in I_{\delta,\alpha}^C}|\frac{1}{n}\sum_{i=1}^nZ_{\delta,i}(t)Y_i| \leq C_1\sqrt{\delta^\kappa/n |\log(\delta)|} + C_2 \alpha^{2-\kappa}\delta^{1+\kappa/2}=o(\delta^\kappa)$
% \item $\sup_{t\in I_{r,\delta,\alpha}:|t-\tau_r|\geq \delta/2} |\frac{1}{n}\sum_{i=1}^nZ_{\delta,i}(t)Y_i| \leq (1-C_3)|\vartheta_r|c(\tau_r)\delta^\kappa$
% \item $|\frac{1}{n}\sum_{i=1}^n Z_{\delta,i}(t_{j(r)})Y_i| \geq (1-\epsilon_n)|\vartheta_r|c(\tau_r)\delta^\kappa, \, r=1,2,\dots, S$.
% \end{enumerate}
% Since $|\tau_r - t_{j(r)}| \leq \delta/2$ we hence know that
% $$\sup_{t\in I_{\delta}:|t-\tau_r| \leq \delta/2}|\frac{1}{n}\sum_{i=1}^n Z_{\delta,i}(t)Y_i| \geq (1-\epsilon_n)|\vartheta_r|c(\tau_r)\delta^\kappa$$
% which is larger then the supremum of $|\frac{1}{n}\sum_{i=1}^n Z_{\delta,i}(t)Y_i|$ over all other feasible points, since as $n\to \infty$ we have
% $$(1-\epsilon_n)|\vartheta_r| c(\tau_r)\delta^\kappa > \max\{(1-C_3)|\vartheta_r|c(\tau_r)\delta^\kappa, C_1\sqrt{\delta^\kappa/n |\log(\delta)|} + C_2 \alpha^{2-\kappa}\delta^{1+\kappa/2}\}.$$
%% %%%%%%%%%%%%%%%%%%%%%%%%%%%%%%%%%%%%%%%%%%%%%%%
%% End Proof of assertion \eqref{poicon-eq6}
%% %%%%%%%%%%%%%%%%%%%%%%%%%%%%%%%%%%%%%%%%%%%%%%%
as well as
\begin{align}
I_{r,\delta,\alpha}\subset [\widetilde{\tau}_r-\sqrt{\delta}/2,\widetilde{\tau}_r+\sqrt{\delta}/2]
\quad   \quad r=1,\dots,S.
\label{poicon-eq7}
\end{align}
%% %%%%%%%%%%%%%%%%%%%%%%%%%%%%%%%%%%%%%%%%%%%%%%%
%% Proof of assertion \eqref{poicon-eq7}
%% %%%%%%%%%%%%%%%%%%%%%%%%%%%%%%%%%%%%%%%%%%%%%%%
% Chose any $t\in I_{r,\delta,\alpha}$.
% If $\widetilde{\tau}_r=\tau_r$ then the assertion
% follows immediately from the fact that $\alpha>2$.
% If $\widetilde{\tau}_r\neq \tau_r$, then we know that
% |\widetilde{\tau}_r - \tau_r| \leq \frac{\delta}{2}.
% We then know that $|\widetilde{\tau}_r -\tau_r|\leq \delta/2$
% hence, since $\alpha>2$ and $\delta \to 0$:
% \begin{align*}
% |t-\widetilde{\tau}_r| &= |(t-\tau_r) + (\tau_r - \widetilde{\tau}_r)| \\
% &\leq |t-\tau_r| + |\widetilde{\tau}_r-\tau_r|\\
% &\leq \frac{\sqrt{\delta}}{\alpha} + \frac{\delta}{2}\\
% &\leq \frac{\sqrt{\delta}}{2}.
% \end{align*}
% This proofs the assertion.
%% %%%%%%%%%%%%%%%%%%%%%%%%%%%%%%%%%%%%%%%%%%%%%%%
%% End Proof of assertion \eqref{poicon-eq7}
%% %%%%%%%%%%%%%%%%%%%%%%%%%%%%%%%%%%%%%%%%%%%%%%%
But under \eqref{poicon-eq6} and \eqref{poicon-eq7} construction of the estimators $\widehat{\tau}_k$,
$k=1,\dots,S$, for
the first $S$ steps of our estimation procedure implies that $\{ \widehat{\tau}_1,\dots,\widehat{\tau}_S\} =\{\widetilde{\tau}_1,\dots,\widetilde{\tau}_S\}$. Therefore,
\begin{align}
&P\left(\{ \widehat{\tau}_1,\dots,\widehat{\tau}_S\} =\{\widetilde{\tau}_1,\dots,\widetilde{\tau}_S\}\right) \rightarrow 1
\label{poicon-eq8}\\
& P\left(I_{\delta}\backslash \bigcup_{r=1}^S [\widehat{\tau}_r-\sqrt{\delta}/2,\widehat{\tau}_r+\sqrt{\delta}/2]\subset
 I_{\delta,\alpha}^C\right) \rightarrow 1
 \label{poicon-eq9}
\end{align}
 as $n\rightarrow\infty $.
%% %%%%%%%%%%%%%%%%%%%%%%%%%%%%%%%%%%%%%%%%%%%%%%%
%% Comment on \eqref{poicon-eq8} and \eqref{poicon-eq9}
%% %%%%%%%%%%%%%%%%%%%%%%%%%%%%%%%%%%%%%%%%%%%%%%%
% The assertion indeed follow immediately from the proclaimed equations
% and the definition of our estimator. (Note that $\tau_s \notin [\widehat{\tau}_r - \sqrt{\delta}/2, \widehat{\tau}_r + \sqrt{\delta}/2 ]$)
%% %%%%%%%%%%%%%%%%%%%%%%%%%%%%%%%%%%%%%%%%%%%%%%%
%% End comment on \eqref{poicon-eq8} and \eqref{poicon-eq9}
%% %%%%%%%%%%%%%%%%%%%%%%%%%%%%%%%%%%%%%%%%%%%%%%%

\noindent By definition of $\widetilde{\tau}_r$, $r=1,\dots,S$, in \eqref{poicon-eq5} it already follows from
 \eqref{poicon-eq8} that $\widehat{\tau}_1,\dots,\widehat{\tau}_S$ provide consistent estimators of the true points of impact. Some more precise approximations are, however, required to show Assertion \eqref{thm3eq1}.
%% %%%%%%%%%%%%%%%%%%%%%%%%%%%%%%%%%%%%%%%%%%%%%%%
%% Bis hierhin scheint alles okay zu sein.
%% %%%%%%%%%%%%%%%%%%%%%%%%%%%%%%%%%%%%%%%%%%%%%%%

\noindent Note that for all $r=1,\dots,S$ and all $t\in (a,b)$
\begin{align}
\begin{split}\label{poicon-eq10}
\frac{1}{n}\sum_{i=1}^nZ_{\delta,i}(t)Y_i=&\mathbb{E}(Z_{\delta,i}(t)Y_i)
- (\E(Z_{\delta,i}(\tau_r)Y_i)-c(\tau_r)\vartheta_r \delta^\kappa)\\
&+\frac{1}{n}\sum_{i=1}^n[(Z_{\delta,i}(t) -Z_{\delta,i}(\tau_r))Y_i-\mathbb{E}((Z_{\delta,i}(t) -Z_{\delta,i}(\tau_r))Y_i)]\\
&+\frac{1}{n}\sum_{i=1}^nZ_{\delta,i}(\tau_r)Y_i
-\mathbb{E}(Z_{\delta,i}(\tau_r)Y_i)
+ (\E(Z_{\delta,i}(\tau_r)Y_i)-c(\tau_r)\vartheta_r \delta^\kappa)
\end{split}
\end{align}

\noindent Note that for all sufficiently large $n$, our assumptions on $p_n$ guarantee that $|\tau_r-t_{j(r)}|\leq \frac{b-a}{2(p-1)} \leq M_L n^{-1/\kappa}$  for some constant $M_L<1$. Let $M_p:=arg \max \{m\in\mathbb{N}| \ \frac{\delta}{2^m}\geq 2n^{-1/\kappa}\}$.
\noindent Our assumptions on the sequence $\delta\equiv\delta_n$ yield
$\sup_{m=1,\dots,M_p} \frac{|2^{-m/2}\delta|^{\min\{2\kappa,2\}}}{2^{-\kappa m}\delta^\kappa}\rightarrow 0$.
%% %%%%%%%%%%%%%%%%%%%%%%%%%%%%%%%%%%%%%%%%%%%%%%%
%% Proof of cvg to 0:
%% %%%%%%%%%%%%%%%%%%%%%%%%%%%%%%%%%%%%%%%%%%%%%%%
% \nointend Assume $\kappa \geq 1$. By the definition of $M_p\equiv M_{p,n}$ we have for all $m=1,\dots, M_p$:
% $$ 2^m \leq \frac{1}{2} \delta n ^{1/\kappa} \qquad \Rightarrow \qquad 2^m \leq \frac{1}{2} \delta n ^{1/\kappa}
% \qquad \Rightarrow \qquad 2^{m(\kappa-1)} \leq \delta^{\kappa-1} n^{\frac{\kappa-1}{\kappa}}.$$
% Hence for any $m = 1, \dots, M_p$:
% \begin{align*}
% \frac{|2^{-m/2} \delta|^{\min\{2,2\kappa\}}}{2^{-\kappa m}\delta^\kappa} & \stackrel{\kappa \geq 1}{=} 2^{m(\kappa-1)} \delta^{2-\kappa} \\
% &\leq \delta^{\kappa-1} n^{\frac{\kappa-1}{\kappa}} \delta^{2-\kappa}\\
% &= \frac{\delta}{n^{\frac{-\kappa+1}{\kappa}}} = \Big(\frac{\delta^\kappa}{n^{-\kappa+1}}\Big)^{1/\kappa} \rightarrow 0
% \end{align*}
% Since we have required $\frac{\delta^\kappa}{n^{-\kappa+1}}\to 0$. (Note: used here for the first time - not needed to get consistency)\\
% \nointend Now assume  $\kappa<1$
% we then have
% $$\frac{ |2^{-m/2} \delta|^{\min\{2,2\kappa\}} }{ 2^{-\kappa m}\delta^\kappa } \stackrel{\kappa<1}{=}
% \frac{2^{-\kappa m }\delta^{2\kappa}}{2^{-\kappa m }\delta^\kappa} = \delta^\kappa \to 0.$$\\
%% %%%%%%%%%%%%%%%%%%%%%%%%%%%%%%%%%%%%%%%%%%%%%%%
%% End Proof of cvg to 0:
%% %%%%%%%%%%%%%%%%%%%%%%%%%%%%%%%%%%%%%%%%%%%%%%%

\noindent
We can therefore infer from \eqref{lem3eq3} that there are
constants $0<C_{5,m}<C_{6,m}<\infty$ with $\inf(C_{6,m}-C_{5,m}) \geq M_C$ for some constant $M_C>0$ such that for all $m=1,2,\dots,M_p$ and all sufficiently large $n$
\begin{align}
\begin{split}\label{poicon-eq11}
\sup_{t\in I_{r,\delta,\alpha},\ |t-\tau_r|\geq \frac{\delta}{2^{m+1}}}
|\E(Z_{\delta,i}(t)Y_i) - (\E(Z_{\delta,i}(\tau_r)Y_i)-c(\tau_r)\vartheta_r \delta^\kappa)|&\leq  |\vartheta_r| c(\tau_r) \left(\delta^\kappa-
C_{6,m}  \frac{\delta^\kappa}{2^{\kappa m}}\right),\\
 |\mathbb{E}(Z_{\delta,i}(t_{j(r)})Y_i) - (\E(Z_{\delta,i}(\tau_r)Y_i)-c(\tau_r)\vartheta_r \delta^\kappa)|&> |\vartheta_r| c(\tau_r) \left(\delta^\kappa-
C_{5,m}  \frac{\delta^\kappa}{2^{\kappa m}}\right)
\end{split}
\end{align}
hold for every $r=1,\dots,S$.
On the other hand, the exponential inequality \eqref{lem1eq1} obviously implies the
 existence of a constant $0<C_7<\infty$ such that for any $0<q\leq \sqrt{n}$
\begin{align}
P\left(\sup_{|t-\tau_r|\leq \delta/2^{m}} |\frac{1}{n}\sum_{i=1}^n[(Z_{\delta,i}(t) -Z_{\delta,i}(\tau_r))Y_i-\mathbb{E}((Z_{\delta,i}(t) -Z_{\delta,i}(\tau_r))Y_i)]|
\leq C_7 q \sqrt{\frac{\delta^\kappa}{2^{\kappa m} n}}\right) \geq  1- \frac{1}{q},
\label{poicon-eq12}
\end{align}
holds for all $m=1,2,\dots$ and each $r=1,\dots,S$.

For all $m=1,2,\dots$ and $r=1,\dots,S$ let ${\cal A}(n,m,r)$ denote the event that
$$\sup_{|t-\tau_r|\leq \delta/2^{m}}|\frac{1}{n}\sum_{i=1}^n[(Z_{\delta,i}(\tau_r) -Z_{\delta,i}(t))Y_i-\mathbb{E}((Z_{\delta,i}(\tau_r) -Z_{\delta,i}(t))Y_i)]|< (C_{6,m}-C_{5,m})|\vartheta_r| c(\tau_r)\frac{\delta^\kappa}{2^{\kappa m}}. $$
Inequality \eqref{poicon-eq12} implies that with $C_8:=\frac{C_7}{M_C (\min_{r=1,\dots,S}|\vartheta_r|c(\tau_r))}$
 the complementary events
${\cal A}(n,m,r)^C$ can be bounded by
\begin{align}
P\left({\cal A}(n,m,r)^C\right)
\leq C_8
\sqrt{\frac{2^{\kappa m}}{\delta^\kappa n}}
\label{poicon-eq13}
\end{align}
 for all
$m=1,2,\dots$ and $r=1,\dots,S$.
%% %%%%%%%%%%%%%%%%%%%%%%%%%%%%%%%%%%%%%%%%%%%%%%%
%% Proof of \eqref{poicon-eq13}
%% %%%%%%%%%%%%%%%%%%%%%%%%%%%%%%%%%%%%%%%%%%%%%%%
%  The assertion follows indeed directly from \eqref{poicon-eq12} with $q = 1/C_8 \sqrt{\frac{\delta^\kappa}{2^{\kappa m}}}$.
%% %%%%%%%%%%%%%%%%%%%%%%%%%%%%%%%%%%%%%%%%%%%%%%%
%% End Proof of \eqref{poicon-eq13}
%% %%%%%%%%%%%%%%%%%%%%%%%%%%%%%%%%%%%%%%%%%%%%%%%
If $m\leq M_p$, then
 \eqref{poicon-eq10} and  \eqref{poicon-eq11}   imply that
 under  ${\cal A}(n,m,r)$ we have
\begin{align}
\widetilde{\tau}_{r,m}:=arg \sup_{t\in I_{r,\delta,\alpha},|t-\tau_r|\leq \delta/2^{m}} |\frac{1}{n}\sum_{i=1}^nZ_{\delta,i}(t)Y_i| \ \in \left[\tau_r-\frac{\delta}{2^{m+1}},\tau_r
+\frac{\delta}{2^{m+1}}\right]
\label{poicon-eq14}
\end{align}
for each $r=1,\dots,S$ and all sufficiently large $n$.
%% %%%%%%%%%%%%%%%%%%%%%%%%%%%%%%%%%%%%%%%%%%%%%%%
%% Proof of \eqref{poicon-eq14}
%% %%%%%%%%%%%%%%%%%%%%%%%%%%%%%%%%%%%%%%%%%%%%%%%
% Can't see it immediately....
%% %%%%%%%%%%%%%%%%%%%%%%%%%%%%%%%%%%%%%%%%%%%%%%%
%% End Proof of \eqref{poicon-eq14}
%% %%%%%%%%%%%%%%%%%%%%%%%%%%%%%%%%%%%%%%%%%%%%%%%

\noindent Choose an arbitrary $\epsilon>0$ and set
\begin{equation*}
m^*(\epsilon):= \max \left\{m=1,2,\dots | \ \epsilon\geq C_8
\sqrt{\frac{2^{\kappa m}}{\delta^\kappa n}}\right\}
\end{equation*}
whenever there exists an integer $m>0$ such that $\epsilon\geq C_8
\sqrt{\frac{2^{\kappa m}}{\delta^\kappa n}}$ and set $m^*(\epsilon) := 1$ otherwise. Furthermore define
\begin{equation*}
m(\epsilon):=\min\{m^*(\epsilon), M_p\}.
\end{equation*}
%% %%%%%%%%%%%%%%%%%%%%%%%%%%%%%%%%%%%%%%%%%%%%%%%
%% A note on m(\epsilon)
%% %%%%%%%%%%%%%%%%%%%%%%%%%%%%%%%%%%%%%%%%%%%%%%%
% Note: it follows from the definitions of $m^*(\epsilon)$ and $M_p$
% that $$m(\epsilon) = m^*(\epsilon)  \qquad \Leftrightarrow \qquad \frac{\epsilon^2}{C_8^2} \leq (\frac{1}{2})^\kappa $$
%% %%%%%%%%%%%%%%%%%%%%%%%%%%%%%%%%%%%%%%%%%%%%%%%
%% End note on m(\epsilon)
%% %%%%%%%%%%%%%%%%%%%%%%%%%%%%%%%%%%%%%%%%%%%%%%%

By our assumptions on $\delta\equiv\delta_n$
there then obviously exists a constant $A(\epsilon)<\infty$ such that for all sufficiently large $n$,
\begin{align}
\frac{\delta}{2^{m(\epsilon)}}\leq A(\epsilon) n^{-1/\kappa}.
\label{poicon-eq15}
\end{align}
\noindent Now consider the event ${\cal A}(n,\epsilon):=\bigcap_{r=1}^{S}\bigcap_{m=1}^{m(\epsilon)}{\cal A}(n,m,r)$. By \eqref{poicon-eq13} the Bonferroni inequality  implies that
\begin{align}
P\left({\cal A}(n,\epsilon)\right)
\geq 1- 2^{\kappa/2}S\sum_{m=1}^{m(\epsilon)}  C_8
\sqrt{\frac{2^{\kappa m}}{\delta^\kappa n}}
\geq 1- 2^{\kappa/2}S\sum_{m=0}^{m(\epsilon)-1}
(\frac{1}{2^{\kappa/2}})^{m(\epsilon)-m} \epsilon \geq 1- \frac{2^{\kappa/2}S\epsilon}{1-(\frac{1}{2})^{\kappa/2}}.
\label{poicon-eq16}
\end{align}

But under event ${\cal A}(n,\epsilon)$ we can infer from \eqref{poicon-eq14} that
\begin{align}
\widetilde{\tau}_{r,1}=\widetilde{\tau}_{r,2}=\dots = \widetilde{\tau}_{r,m(\epsilon)}.
\label{poicon-eq17}
\end{align}
%% %%%%%%%%%%%%%%%%%%%%%%%%%%%%%%%%%%%%%%%%%%%%%%%
%% Proof of \eqref{poicon-eq17}
%% %%%%%%%%%%%%%%%%%%%%%%%%%%%%%%%%%%%%%%%%%%%%%%%
% This is indeed trivial, since we neither change $\delta$ nor the processes and $\widetilde{\tau}_{r,m}$ has to be
% in all intervalls, given by \eqref{poicon-eq14}.
%% %%%%%%%%%%%%%%%%%%%%%%%%%%%%%%%%%%%%%%%%%%%%%%%
%% End proof of \eqref{poicon-eq17}
%% %%%%%%%%%%%%%%%%%%%%%%%%%%%%%%%%%%%%%%%%%%%%%%%
Additionally let ${\cal A}^*(n)$ denote the event that $\{ \widehat{\tau}_1,\dots,\widehat{\tau}_S\} =\{\widetilde{\tau}_1,\dots,\widetilde{\tau}_S\}$. The definitions in (\ref{poicon-eq5}) and (\ref{poicon-eq14}) yield
$\widetilde{\tau}_{r,1}=\widetilde{\tau}_r$, $r=1,\dots,S$, and
we can thus conclude from
\eqref{poicon-eq15} and \eqref{poicon-eq17} that under events ${\cal A}^*(n)$ and
${\cal A}(n,\epsilon)$ we have
\begin{equation}
\max_{r=1,\ldots,S}\min_{s=1,\ldots,S}|\widehat{\tau}_r-\tau_s|
=\max_{r=1,\ldots,S}\min_{s=1,\ldots,S}|\widetilde{\tau}_{r,m(\epsilon)}-\tau_s|
\leq \frac{\delta}{2^{m(\epsilon)+1}}\leq
\frac{A(\epsilon)}{2}  n^{-1/\kappa}
\label{poicon-eq18}
\end{equation}
for all $n$ sufficiently large.

Recall that by \eqref{poicon-eq8} we have $P({\cal A}^*(n))\rightarrow 1$ as
$n\rightarrow\infty$. Since $\epsilon$ is arbitrary,  \eqref{thm3eq1} thus follows from \eqref{poicon-eq16} and \eqref{poicon-eq18}.
%% %%%%%%%%%%%%%%%%%%%%%%%%%%%%%%%%%%%%%%%%%%%%%%%
%% Proof of \eqref{poicon-eq18} and the conclusion
%% %%%%%%%%%%%%%%%%%%%%%%%%%%%%%%%%%%%%%%%%%%%%%%%
% This is trivial.
%% %%%%%%%%%%%%%%%%%%%%%%%%%%%%%%%%%%%%%%%%%%%%%%%
%% End proof of \eqref{poicon-eq18} and the conclusion
%% %%%%%%%%%%%%%%%%%%%%%%%%%%%%%%%%%%%%%%%%%%%%%%%

It remains to prove Assertion \eqref{thm3eq4}.
%need $D<A^*<A$
For some $A^*<A$ with $D\leq A^*$ define $\lambda_n^*<\lambda_n$ by $\lambda_n^*:=A^*\sqrt{\frac{\sigma_{|y|}^2}{n}\log (\frac{b-a}{\delta})}$. By \eqref{lem0eq2}
it is immediately seen that in addition to \eqref{poicon-eq2} also
\begin{align*}
\sup_{t\in I_{\delta,\alpha}^C} \frac{|\frac{1}{n}\sum_{i=1}^n
Z_{\delta,i}(t)Y_i|}{\sqrt{\frac{1}{n}\sum_{i=1}^n
Z_{\delta,i}(t)^2}}\leq \lambda_n^*
+C_2 \alpha^{2-\kappa} \frac{\delta^{1+\kappa/2}}{\inf_{t\in I_{\delta,\alpha}^C} \sqrt{\frac{1}{n}\sum_{i=1}^n
Z_{\delta,i}(t)^2}}
\end{align*}
holds with probability tending to 1 as $n\rightarrow\infty$.

But \eqref{lem0eq1} and our assumptions on the sequence $\delta\equiv \delta_n$ lead to
$\frac{\delta^{1+\kappa/2}}{\inf_{t\in I_{\delta,\alpha}^C} \sqrt{\frac{1}{n}\sum_{i=1}^n
Z_{\delta,i}(t)^2}}=o_P(\lambda_n^*)$.
%% %%%%%%%%%%%%%%%%%%%%%%%%%%%%%%%%%%%%%%%%%%%%%%%
%% Proof of this unlabeled assertion
%% %%%%%%%%%%%%%%%%%%%%%%%%%%%%%%%%%%%%%%%%%%%%%%%
% with probability 1 we have by \eqref{lem0eq0} and \eqref{lem0eq1} that for some constant $L_{2,1}>0$ we have
% $$\inf_{t\in I_{\delta,\alpha}^C} \sqrt{\frac{1}{n}\sum_{i=1}^nZ_{\delta,i}(t)^2}
% \geq \inf_{t \in [a+\delta, b-\delta]} \sqrt{\frac{1}{n}\sum_{i=1}^nZ_{\delta,i}(t)^2}
% \geq   \sqrt{L_{2,1} \delta^\kappa}.$$
% (Compare this with \eqref{eq:lem0inf} in Lemma \ref{lem0}.)
% Hence, with probability converging to 1 we have for some constant $L<\infty$:
% \begin{align*}
% \frac{\delta^{1+\kappa/2}}{\inf_{t\in I_{\delta,\alpha}^C} \sqrt{\frac{1}{n}\sum_{i=1}^n
% Z_{\delta,i}(t)^2}} \leq L \delta
% \end{align*}
% Now, chose some arbitrary $\epsilon>0$. Since $\delta = c \frac{1}{\sqrt{n}})$ we then have, as $n \to \infty$, for some constant $K<\infty$
% \begin{align*}
% \frac{L \delta}{A^* \sqrt{\frac{\sigma_{|y|}^2}{n}\log\left(\frac{b-a}{\delta}\right)}}
% = K \frac{\frac{1}{\sqrt{n}}}{\sqrt{\frac{1}{n}\log\left(\frac{b-a}{\delta}\right)}} = K \frac{1}{\sqrt{\log\left(\frac{b-a}{\delta}\right)}} \leq \epsilon,
% \end{align*}
% since $\log\left(\frac{b-a}{\delta}\right) \to \infty$. We thus may conclude that we indeed have
% $\frac{\delta^{1+\kappa/2}}{\inf_{t\in I_{\delta,\alpha}^C} \sqrt{\frac{1}{n}\sum_{i=1}^n
% Z_{\delta,i}(t)^2}}=o_P(\lambda_n^*)$.
%% %%%%%%%%%%%%%%%%%%%%%%%%%%%%%%%%%%%%%%%%%%%%%%%
%% End proof of this unlabeled assertion
%% %%%%%%%%%%%%%%%%%%%%%%%%%%%%%%%%%%%%%%%%%%%%%%%

 Using \eqref{poicon-eq9}, the construction of the estimator $\widehat{\tau}_{S+1}$ therefore implies that as $n\rightarrow\infty$,
\begin{align*}
P\left(\frac{|\frac{1}{n}\sum_{i=1}^n
Z_{\delta,i}(\widehat{\tau}_{S+1})Y_i|}{\sqrt{\frac{1}{n}\sum_{i=1}^n
Z_{\delta,i}(\widehat{\tau}_{S+1})^2}}
<\lambda_n\right)=P\left(
\sup_{t\in I_{\delta,\alpha}^C} \frac{|\frac{1}{n}\sum_{i=1}^n
Z_{\delta,i}(t)Y_i|}{\sqrt{\frac{1}{n}\sum_{i=1}^n
Z_{\delta,i}(t)^2}}\leq \lambda_n \right)\rightarrow 1,
\end{align*}
%% %%%%%%%%%%%%%%%%%%%%%%%%%%%%%%%%%%%%%%%%%%%%%%%
%% Proof of this unlabeled assertion
%% %%%%%%%%%%%%%%%%%%%%%%%%%%%%%%%%%%%%%%%%%%%%%%%
% This is indeed true. By \eqref{poicon-eq9}, the estimator of \widehat{\tau}_{S+1} must be in $I_{\delta,\alpha}^C$ and by definition its supremum.
%% %%%%%%%%%%%%%%%%%%%%%%%%%%%%%%%%%%%%%%%%%%%%%%%
%% End proof of this unlabeled assertion
%% %%%%%%%%%%%%%%%%%%%%%%%%%%%%%%%%%%%%%%%%%%%%%%%
while \eqref{poicon-eq4} together with \eqref{poicon-eq8}, \eqref{lem0eq0} and \eqref{lem0eq1}  yield
\begin{align*}
P\left(\min_{r=1,\dots,S} \frac{|\frac{1}{n}\sum_{i=1}^n
Z_{\delta,i}(\widehat{\tau}_{r})Y_i|}{\sqrt{\frac{1}{n}\sum_{i=1}^n
Z_{\delta,i}(\widehat{\tau}_{r})^2}}
>\lambda_n\right)\rightarrow 1.
\end{align*}
By definition of our estimator $\widehat{S}$,  \eqref{thm3eq4} is an immediate consequence.
\end{proof}
%% %%%%%%%%%%%%%%%%%%%%%%%%%%%%%%%%%%%%%%%%%%%%%%%
%% End Proof of Theorem 4 Replacement
%% %%%%%%%%%%%%%%%%%%%%%%%%%%%%%%%%%%%%%%%%%%%%%%%
%\textcolor{red}{ACHTUNG: ORLICZ NORM IN PSI ANSTELLE VON PHI UMBENENNEN!!!!!}

%%%%%%%%%%%%%%%%%%%%%%%%%%%%%%%%%%%%%%%
\section[]{Subsequent estimation of $g$}
%%%%%%%%%%%%%%%%%%%%%%%%%%%%%%%%%%%%%%%

In this appendix, $||X||_{\Psi} = \inf\{C>0|\E(\Psi(|X|/C))\leq 1\}$ refers to the Orlicz norm of a random variable $X$ with respect to $\Psi(x) = \exp(n/6(\sqrt{1+2\sqrt{6}x/\sqrt{n}}-1)^2)-1$. Similar we use for $p\geq 1$ the Orlicz norm $||X||_{p} = \{ \inf C>0: (\E(|X|^p))^{1/p}<C \}$ which corresponds to the usual $L_p$-norm.

\bigskip

For the proofs of this section we need the following Lemma:

%% %%%%%%%%%%%%%%%%%%%%%%%%%%%%%%%%%%%%%%%%%%%%%%%%%%%%%%%%%%%%%%%%%%%%%%%%
%% New: Lemma \ref{lem:1new}
%% %%%%%%%%%%%%%%%%%%%%%%%%%%%%%%%%%%%%%%%%%%%%%%%%%%%%%%%%%%%%%%%%%%%%%%%%
%% Parts of this Lemma corresponds (C.44) in \cite{KPS2015},
%% but this assertion is corrected for a mistake.
%% %%%%%%%%%%%%%%%%%%%%%%%%%%%%%%%%%%%%%%%%%%%%%%%%%%%%%%%%%%%%%%%%%%%%%%%%
\begin{lemma}\label{lem:1new}
Under Assumption \ref{assum1},  there exist constants
$M<\infty$, $M_{*}<\infty$ and $M_{**}<\infty$ such that for all sufficiently small $s>0$, all $q\in[-1,1]$, all $t \in (a,b)$ and all $t^*\in[a,b]$ we have
\begin{align}\label{eq:diffexpect}
|\E((X_i(t+qs)-X_i(t))X_i(t^*))| & \leq M |qs|^{\min\{1,\kappa\}}.
%Taylor erster ordnung:
%$\omega(t+qs,t^*, |t+qs-t^*|^\kappa)- \omega(t,t^*,|t-t^*|^\kappa) = O(qs + (|t+qs-t^*|^\kappa - |t-t^*|^\kappa))$
%aber  $||t+qs-t^*|^\kappa - |t-t^*|^\kappa|\leq qs^{\min\{1,\kappa\}}$ und $|qs|\leq qs^{\min\{1,\kappa\}$
\intertext{and for all $q_1, q_2 \in [-1,1]$}
\begin{split}\label{eq:vari}
%\sigma_{(X_i(\tau_r+sq_1)-X_i(\tau_r+sq_2))}^2 & =
\E\Big((X_i(\tau_r+sq_1)-X_i(\tau_r+sq_2))^2\Big) & \leq M_{*}|q_1-q_2|^\kappa s^\kappa \leq M_{**} s^\kappa.
\end{split}
\end{align}
\end{lemma}
%% %%%%%%%%%%%%%%%%%%%%%%%%%%%%%%%%%%%%%%%%%%%%%%%%%%%%%%%%%%%%%%%%%%%%%%%%
%% End: Lemma \ref{lem:1new}
%% %%%%%%%%%%%%%%%%%%%%%%%%%%%%%%%%%%%%%%%%%%%%%%%%%%%%%%%%%%%%%%%%%%%%%%%%

\bigskip

%% %%%%%%%%%%%%%%%%%%%%%%%%%%%%%%%%%%%%%%%%%%%%%%%%%%%%%%%%%%%%%%%%%%%%%%%%
%% Proof of Lemma \ref{lem:1new}
%% %%%%%%%%%%%%%%%%%%%%%%%%%%%%%%%%%%%%%%%%%%%%%%%%%%%%%%%%%%%%%%%%%%%%%%%%
\begin{proof}[Proof of Lemma \ref{lem:1new}]
Let $q \in [-1,1]$ and choose an arbitrary $t^* \in [a,b]$ as well as $t\in (a,b)$ for all sufficiently small $s$ it then follows from Taylor expansions that

\begin{align*}
\E((X_i(t+qs)-X_i(t))X_i(t^*)) & =  \omega(t+qs,t^*, |t+qs-t^*|^\kappa)- \omega(t,t^*,|t-t^*|^\kappa)\\
& = O(qs + (|t+qs-t^*|^\kappa - |t-t^*|^\kappa)).
\end{align*}
Assertion \eqref{eq:diffexpect} then follows immediately from the fact that $||t+qs-t^*|^\kappa - |t-t^*|^\kappa| = O(|qs|^{\min\{1,\kappa\}})$ as well as $|qs|=O(|qs|^{\min\{1,\kappa\}})$.

Moreover, further Taylor expansions can be used to show that for all $q_1,q_2\in[-1,1]$, all $t\in(a,b)$ and all sufficiently small $s>0$ we have %for all t in ???????
%Taylor um (t_1,t_1,0)
%it follows from a Taylor expansion around $(t+q_1s,t+q_1s,0)$, that
{\small
\begin{align*}
\E\Big((X_i(t+q_1s) - X_i(t+q_2s))^2\Big) &= \omega(t+q_1s,t+q_1s,0) +\omega(t+q_2s,t+q_2s,0) \\
     &\quad  - \omega(t+q_1s,t+q_2s,|(q_1-q_2)s|^\kappa)  - \omega(t+q_2s,t+q_1s,|(q_1-q_2)s|^\kappa)\\
  & = \omega(t+q_1s,t+q_1s,0) + \omega(t+q_2s,t+q_2s,0) \\
    &\quad - \omega(t+q_1s,t+q_2s,0)  - \omega(t+q_2s,t+q_1s,0) + O(|(q_1-q_2)s|^\kappa)\\
  & = O\Big((q_1-q_2)^2s^2\Big) + O\Big(|q_1-q_2|^\kappa s^\kappa\Big)  = O\Big(|q_1-q_2|^\kappa s^\kappa\Big),
\end{align*}
}
%%%%%%%%extended proof: write t_1 = t+q_1s and t_2 = t+q_2s, then
%\begin{align*}
%\E\Big((X_i(t+q_1s) &- X_i(t+q_2s))^2\Big)
     %= \omega(t_1,t_1,0) - \omega(t_1,t_2,|t_1-t_2|^\kappa)  - \omega(t_2,t_1,|t_1-t_2|^\kappa) + \omega(t_2,t_2,0)\\
    %& = \omega(t_1,t_1,0)
        %- (\omega(t_1,t_1,0) + (t_2-t_1)\omega_2(t_1,t_1,0) + O(|(q_2-q_1)s|^\kappa)) \\
				%& - (\omega(t_1,t_1,0) + (t_2-t_1)\omega_1(t_1,t_1,0) + O(|(q_2-q_1)s|^\kappa)) \\
				%& + (\omega(t_1,t_1,0) + (t_2-t_1)\omega_1(t_1,t_1,0) + (t_2-t_1)\omega_2(t_1,t_1,0)  +  O((t_2-t_1)^2))\\
    %& = O(|(q_2-q_1)s|^\kappa) + O(|(q_2-q_1)s|^2)
%\end{align*}
proving Assertion \eqref{eq:vari}.
\end{proof}
%% %%%%%%%%%%%%%%%%%%%%%%%%%%%%%%%%%%%%%%%%%%%%%%%%%%%%%%%%%%%%%%%%%%%%%%%%
%% End proof of Lemma~\ref{lem:1new}
%% %%%%%%%%%%%%%%%%%%%%%%%%%%%%%%%%%%%%%%%%%%%%%%%%%%%%%%%%%%%%%%%%%%%%%%%%

\bigskip

%%%%%%%%%%%%%%%%%%%%%%%%%%%%%%%%%%%%%%%
\subsection{The nonparametric case}
%%%%%%%%%%%%%%%%%%%%%%%%%%%%%%%%%%%%%%%
In order to proof Theorem \ref{thm:lcr} we need auxiliary results.

\bigskip

\begin{prop}\label{prop:1new}
Let $X_i=(X_i(t):t\in[a,b])$, i=1,\dots, n, be i.i.d. Gaussian processess with covariance functions $\sigma(s,t)$ satisfying Assumption~\ref{assum1}. For any differentiable and bounded function $F(x_1,\dots, x_S)$ with
$|\E\Big(\frac{\partial}{\partial x_r} F(X_i(\tau_1),\dots, X_i(\tau_S))\Big)|<\infty$ for $r=1,\dots, S$, we then have
{\small
\begin{align}
\frac{1}{n}\sum_{i=1}^n (X_i(\widehat{\tau}_r)- X_i(\tau_r))F(X_i(\tau_1),\dots, X_i(\tau_S))& = O_p\Bigg(n^{-\min\{1,\frac{1}{\kappa}\}}\sum_{r=1}^S|\E\Big(\frac{\partial}{\partial x_r} F(X_i(\tau_1),\dots, X_i(\tau_S))\Big)|\Bigg) \label{eq:prop1new1}\\ %benögitgt F gebounded und diffbarkeitsannahme
\frac{1}{n}\sum_{i=1}^n (X_i(\widehat{\tau}_r)-X_i(\tau_r))^2 &= O_p(n^{-1})\label{eq:lcdev1}\\
\frac{1}{n}\sum_{i=1}^n (X_i(\widehat{\tau}_r)-X_i(\tau_r))\varepsilon_i F(X_i(\tau_1),\dots, X_i(\tau_S))&= O_p(n^{-1}),\label{eq:lcdev2}\\ %only boundedness of F is needed
\frac{1}{n}\sum_{i=1}^n (X_i(\widehat{\tau}_r)-X_i(\tau_r))^4 &= O_p(n^{-2})\label{eq:lcdev3}
\end{align}
}
\end{prop}

\bigskip

%% %%%%%%%%%%%%%%%%%%%%%%%%%%%%%%%%%%%%%%%%%%%%%%%%%%%%%%%%%%%%%%%%%
%% Proof of Abschätzungen für Kernelestimates
%% Achtung: Hier sind Textstücke aus einen der Lemmas zuvor reinkopiert,
%% da im wesentlichen dieselben Schritte angewandt werden
%% %%%%%%%%%%%%%%%%%%%%%%%%%%%%%%%%%%%%%%%%%%%%%%%%%%%%%%%%%%%%%%%%%

\bigskip

\begin{proof}[Proof of Proposition \ref{prop:1new}]
To ease notation, let $F(\bX_i) = F(X_i(\tau_1),\dots, X_i(\tau_S))$. For an arbitrary choice of $r \in \{1,\dots, S\}$ choose any $0<s$ sufficiently small and define for $q_1,q_2 \in [-1,1]$
{\small \begin{align*}
\chi_i(q_1,q_2) & := (X_i(\tau_r + sq_1)-X_i(\tau_r))F(\bX_i) - (X_i(\tau_r + sq_2)-X_i(\tau_r))F(\bX_i)\\
                & \qquad - \E\big((X_i(\tau_r + sq_1)-X_i(\tau_r))F(\bX_i) - (X_i(\tau_r + sq_2)-X_i(\tau_r))F(\bX_i) \big)\\
                & = (X_i(\tau_r + sq_1)-X_i(\tau_r + sq_2))F(\bX_i) - \E\big((X_i(\tau_r + sq_1)-X_i(\tau_r + sq_2))F(\bX_i) \big)
\end{align*}
}
We then have $\E(\chi_i(q_1,q_2)) = 0$ and it follows from some straightforward calculations, since $|F(\bX_i)|\leq M_F<\infty$, that there exists a constant $0<M_{1}<\infty$ such that for $m=2,3,\dots$ we have
\begin{align}\label{eq:BERNIE}
\E(|\frac{1}{s^{\kappa/2}}\chi_i(q_1,q_2)|^m) \leq \frac{m!}{2}(M_{1}|q_1-q_2|^{\kappa/2})^m. %eq:BERNIE
\end{align}
%%%%%%%%%%%%%%%%%%%%%%%%%%%%%%%%%%%%%%%%%%%%
%proof of eq:BERNIE
%%We use:
%%the fact that |F(\bX_i)|\leq M_F < \infty
%%moments of centered normal distributed
%%2^{m/2} \frac{\Gamma(\frac{m+1}{2}}{\sqrt{\pi}} \leq 2^m  m!/2
%%\sigma_{(X_i(\tau_r+sq_1)-X_i(\tau_r+sq_2))}^2 \leq M_3 s^\kappa
%\begin{align*}
%\E(|\frac{1}{s^{\kappa/2}}\chi_i(q_1,q_2)|^m)
%&= s^{-\frac{m\kappa}{2}}\E(|\chi_i(q_1,q_2)|^m)\\
%&= s^{-\frac{m\kappa}{2}}\E(|(X_i(\tau_r + sq_1)-X_i(\tau_r + sq_2))F(\bX_i) - \E\big((X_i(\tau_r + sq_1)-X_i(\tau_r + sq_2))F(\bX_i)\big)|^m\\
%&\leq s^{-\frac{m\kappa}{2}} 2^{m-1} (\E(|(X_i(\tau_r + sq_1)-X_i(\tau_r + sq_2))F(\bX_i)|^m) + |\E\big((X_i(\tau_r + sq_1)-X_i(\tau_r + sq_2))F(\bX_i)\big)|^m)\\
%&\leq s^{-\frac{m\kappa}{2}} 2^{m}\E(|(X_i(\tau_r + sq_1)-X_i(\tau_r + sq_2))F(\bX_i)|^m) \\
%&\leq M_F^m s^{-\frac{m\kappa}{2}} 2^{m}\E(|(X_i(\tau_r + sq_1)-X_i(\tau_r + sq_2))|^m)\\
%&= M_F^m s^{-\frac{m\kappa}{2}} 2^{m} \sigma_{(X_i(\tau_r + sq_1)-X_i(\tau_r + sq_2))}^m 2^{m/2} \frac{\Gamma(\frac{m+1}{2})}{\sqrt{\pi}} \\
%&\leq  M_F^m s^{-\frac{m\kappa}{2}} 2^m \sigma_{(X_i(\tau_r + sq_1)-X_i(\tau_r + sq_2))}^m 2^m \frac{m!}{2}\\
%&\leq  M_F^m 4^m s^{-\frac{m\kappa}{2}} M_3^{m/2} s^{\frac{m \kappa}{2}}|q_1-q_2|^\frac{m\kappa}{2} \\
%&\leq  M_4^m 4^m |q_1-q_2|^\frac{m\kappa}{2}\\
%&\leq  \frac{m!}{2}(M_4 |q_1-q_2|^\frac{\kappa}{2})^m\\
%\end{align*}
%%%%%%%%%%%%%%%%%%%%%%%%%%%%%%%%%%%%%%%%%%%%
Corollary 1 in \citeappendix{vandeGeer2013} now guarantees that there exists a constant $0<M_{2}<\infty$ such that the Orlicz norm of $\frac{1}{\sqrt{ns^\kappa}}\sum_{i=1}^n \chi_i(q_1,q_2)$ can be bounded, i.e., we have for some $0<M_{2}<\infty$:
\begin{align}\label{eq:BERNIE1}
||\frac{1}{\sqrt{ns^\kappa}}\sum_{i=1}^n \chi_i(q_1,q_2)^m||_{\Psi}\leq M_{2} |q_1-q_2|^{\kappa/2}.
\end{align}
The proof then follows from maximum inequalities of empirical processes. By \eqref{eq:BERNIE1} one can apply Theorem 2.2.4 of \citeappendix{vanderVaart1996}. The covering integral in this theorem can easily be seen to be finite and one can thus infer that there exists a constant $0<M_{3}<\infty$ such that
%%%%%%%%%%%%%%%%%%%
$$\mathbb{E}\Bigg( \exp\Big(\sup_{q_1,q_2\in [-1,1]} n/6 \big(\sqrt{1+2\sqrt{\frac{6}{n M_{3}^2}}|\frac{1}{\sqrt{n s^\kappa}}\sum_{i=1}^n\chi_i(q_1,q_2)|}-1 \big)^2 \Big) \Bigg)\leq 2.$$
%%%%%%%%%%%%%%%%%%%
For every $x>0$, the Markov inequality then yields
\begin{align*}
P&\Bigg(\sup_{q_1,q_2\in [-1,1]}| \frac{1}{\sqrt{n s^\kappa}}\sum_{i=1}^n\chi_i(q_1,q_2)| \geq x\frac{M_{3}}{2\sqrt{6}} \Bigg)\leq 2\exp\Big(-\frac{n}{6}(\sqrt{1+x/\sqrt{n}}-1 )^2\Big)\nonumber.
\end{align*}
For improving the readability, it then follows from a Taylor expansion of $\frac{n}{6}(\sqrt{1+x/\sqrt{n}}-1 )^2$ that there exists a constant $0<M_{4}<\infty$ such that for all $0<x\leq \sqrt{n}$ we have
\begin{align}\label{eq:supaux1}
P&\Bigg(\sup_{q_1,q_2\in [-1,1]}| \frac{1}{\sqrt{n s^\kappa}}\sum_{i=1}^n\chi_i(q_1,q_2)| < M_{4}x\Bigg) \geq 1-2 \exp(-x^2).
\end{align}
%Indeed:
%\begin{align*}
%P&\Bigg(\sup_{q_1,q_2\in [-1,1]}| \frac{1}{\sqrt{n s^\kappa}}\sum_{i=1}^n\chi_i(q_1,q_2)| \geq x\frac{L_{1,7}}{2\sqrt{6}} \Bigg)\\
%& = P\Bigg(\exp\Big(\sup_{q_1,q_2} \frac{n}{6}(\sqrt{1+ 2\sqrt{\frac{6}{L_{1,7}^2}n}| \frac{1}{\sqrt{ns^\kappa}}\sum_{i=1}^n\chi_i(q_1,q_2)|}-1)^2\Big) \geq \exp\Big(\frac{n}{6}(\sqrt{1+x/\sqrt{n}}-1 )^2\Big)\Bigg)\nonumber\\
%&\leq 2\exp\Big(-\frac{n}{6}(\sqrt{1+x/\sqrt{n}}-1 )^2\Big)\nonumber\\
%&\leq 2\exp\Big(-D_7x^2\Big) \quad \text{(for all $0<x\leq \sqrt{n}$)}\nonumber
%\end{align*}
%
%Now observe that
%\begin{align*}
%\sup_{q_1,q_2\in [-1,1]} &| \frac{1}{\sqrt{n s^\kappa}}\sum_{i=1}^n\chi_i(q_1,q_2)|\\
%& \geq \sup_{q_1\in [-1,1], q_2=0}| \frac{1}{\sqrt{n s^\kappa}}\sum_{i=1}^n\chi_i(q_1,0)|\\
%& \geq \sup_{q_1\in [-1,1]} |\frac{1}{\sqrt{n s^\kappa}}\sum_{i=1}^n (X_i(\tau_r+sq_1)-X_i(\tau_r))f(\xi_i(u))| \\
%&\qquad -  \sup_{q_1}|\frac{\sqrt{n}}{\sqrt{s^\kappa}} \mathbb{E}((X_i(\tau_r+sq_1)-X_i(\tau_r))f(\xi_i(u)))|
%\end{align*}
Now, note that it follows from Theorem~\ref{thm:NP1} that
$$\E(X(t^*)F(\bX_i)) = \sum_{r=1}^S \E(X_i(t^*)X_i(\tau_r)) \E(\frac{\partial}{\partial x_r}F(\bX_i)).$$
Together with \eqref{eq:diffexpect} we we then have for all $q_1 \in [-1,1]$ for some constant $M_5<\infty$:
\begin{align}
|\E((X_i(\tau_r + sq_1)-X_i(\tau_r))F(\bX_i))|& = |\sum_{l=1}^S\E(\frac{\partial}{\partial x_r}F(\bX_i))\E((X_i(\tau_r + sq_1)-X_i(\tau_r))X_i(\tau_l))| \nonumber \\ \
%&\leq \sum_{l=1}^S|\E(\frac{\partial}{\partial x_r}F(\bX_i))| |\E((X_i(\tau_r + sq_1)-X_i(\tau_r))X_i(\tau_l))| \nonumber \\
%&\leq \sum_{l=1}^S|\E(\frac{\partial}{\partial x_r}F(\bX_i))| M_1 s^{\min\{1,\kappa\}} \nonumber\\
&\leq M_{5} s^{\min\{1,\kappa\}}(\sum_{l=1}^S|\E(\frac{\partial}{\partial x_l}F(\bX_i))|). \label{eq:scpfaux1}
\end{align}
Using \eqref{eq:scpfaux1} together with \eqref{eq:supaux1} we then can conclude that for all $0<x\leq \sqrt{n}$:
{\small \begin{align}
\label{eq:tatatat1}
\begin{split}
P\Bigg(\sup_{\tau_r-s\leq  u_r \leq \tau_r +s} |\frac{1}{n} \sum_{i=1}^n (X_i(u_r) - X_i(\tau_r))F(\bX_i)| < M_{5} s^{\min\{1,\kappa\}} (\sum_{l=1}^S|\E(\frac{\partial}{\partial x_l}F(\bX_i))|) & + M_{4}\frac{s^\frac{\kappa}{2}}{\sqrt{n}}x\Bigg)\\
& \geq 1-2\exp(-x^2).
\end{split}
\end{align}
}
Since $|\widehat{\tau}_r-\tau_r| = O_p(n^{-\frac{1}{\kappa}})$, Assertion \eqref{eq:prop1new1} then follows immediately.\\ %\eqref{thm3eq1}.\\

The proof of Assertion~\eqref{eq:lcdev1} follows from similar steps. Define
\begin{align*}
\chi_i^{(1)}(q_1,q_2) & := (X_i(\tau_r + sq_1)-X_i(\tau_r))^2 - (X_i(\tau_r + sq_2)-X_i(\tau_r))^2\\
                & - \E\big((X_i(\tau_r + sq_1)-X_i(\tau_r))^2 - (X_i(\tau_r + sq_2)-X_i(\tau_r))^2 \big).\\
\end{align*}
We then have $\E\Big(\chi_i^{(1)}(q_1,q_2)\Big) = 0$ and with
\begin{align*}
(X_i(\tau_r)&-X_i(\tau+sq_1))^2-(X_i(\tau_r)-X_i(\tau+sq_2))^2 \\
&= (X_i(\tau+sq_2)-X_i(\tau+sq_1))(2X_i(\tau_r)-X_i(\tau_r+sq_1)-X_i(\tau+sq_2)),
\end{align*}
it follows from some straightforward calculations, that there exists a constant $0<M_{6}<\infty$ such that for $m=2,3,\dots$ we have
\begin{align}\label{eq:BERNIEAUX}
\E(|\frac{1}{s^{\kappa}}\chi_i^{(1)}(q_1,q_2)|^m) \leq \frac{m!}{2}(M_{6}|q_1-q_2|^{\kappa/2})^m %eq:BERNIE
\end{align}
Once more, Corollary 1 in \citeappendix{vandeGeer2013} now guarantees that there exists a constant $M_{7}<\infty$ such that
\begin{align}\label{eq:BERNIEAUX1}
||\frac{1}{\sqrt{n}s^\kappa}\sum_{i=1}^n \chi_i(q_1,q_2)^m||_{\Psi}\leq M_{7} |q_1-q_2|^{\kappa/2}.
\end{align}
By \eqref{eq:BERNIEAUX1} another application of the maximal inequalities of empirical processes, using similar steps as in the proof of Assertion~\eqref{eq:prop1new1}, and since by \eqref{eq:vari} we have $\E( (X_i(\tau_r+sq_1)-X_i(\tau_r))^2) \leq M_{8} s^\kappa$ for some constant $M_8<\infty$, we can conclude that for all $0<x\leq \sqrt{n}$ we have for some constant $0\leq M_{9}<\infty$:
\begin{align}
\begin{split}\label{eq:tatatat2}
P\Bigg(\sup_{\tau_r-s\leq  u_r \leq \tau_r +s} |\frac{1}{n} \sum_{i=1}^n (X_i(u_r) - X_i(\tau_r))^2| < M_{8} s^{\kappa} + M_{9}\frac{s^\kappa}{\sqrt{n}}x\Bigg) \geq 1-2\exp(-x^2).
\end{split}
\end{align}
Since $|\widehat{\tau}_r-\tau_r| = O_p(n^{-\frac{1}{\kappa}})$, Assertion~\eqref{eq:lcdev2} then follows immediately.\\
%%%%%

In order to show assertion $\eqref{eq:lcdev2}$ we make use the Orlicz-norm $||X||_p$.\\
%https://math.stackexchange.com/questions/2480236/convergence-of-higher-absolute-moments-given-convergence-in-distribution
Choose some  $p>\frac{2}{\kappa}=p_{\kappa}$, and let $p$ be even. Note that $\E((X_i(\tau_r+sq_1)-X_i(\tau_r+sq_2))\varepsilon_i F(\bX_i)) = 0$. For all sufficiently small $0<s$ and all $q_1, q_2 \in [-1,1]$ it is easy to show that there exists a constant $M_{10}<\infty$ such that
$$\E(|s^{-\kappa/2} \frac{1}{\sqrt{n}}\sum_{i=1}^{n}(X_i(\tau_r+sq_1)-X_i(\tau_r+sq_2))\varepsilon_i F(\bX_i)|^p) \leq M_{10}^p |q_1-q_2|^{\frac{p \kappa}{2}}.
$$
%Idee: man addiert viele unkorrelierte terme und wenig terme die etwas bedeuten. von den bedeutungsvollen stellt größte anzahl (insgesamt n(n-1)...(n-p/2)-1) möglichkeiten) das produkt der varianzen $E(y_i^2)$, von $p/2$ terme....
%Mathematical Statistics for Economics and Business, p.271
%https://math.stackexchange.com/questions/2480236/convergence-of-higher-absolute-moments-given-convergence-in-distribution/2480252?noredirect=1#comment5124333_2480252
%(proof habe ich hier vorgekaut bekommen...: der trick ist $$(X_i(\tau_r+sq_1)-X_i(\tau_r+sq_2) = \sigma_{q_1,q_2} (X_i(\tau_r+sq_1)-X_i(\tau_r+sq_2)/\sigma_{q_1,q_2}$$ zu schreiben, letzter ist $N(0,1)$ verteilt. damit kommt dann schonmal die standardabweichung raus. alle momente gebounded, daher gilt dann auch: (achtung, hier ist noch nicht so richtig standardisisiert...aber CS  https://math.stackexchange.com/questions/2480236/convergence-of-higher-absolute-moments-given-convergence-in-distribution)
%Sieh auch "`central limit theorem - convergence of moments"' (von mir) im Online_LookUps ordner
We may conclude that
\begin{align}
||s^{-\kappa/2} \frac{1}{\sqrt{n}}\sum_{i=1}^{n}(X_i(\tau_r+sq_1)-X_i(\tau_r+sq_2))\varepsilon_i F(\bX_i)||_p\leq M_{10}|q_1-q_2|^{\frac{\kappa}{2}}.\label{eq:on1aux}
\end{align}
By inequality \eqref{eq:on1aux} one may once again apply Theorem 2.2.4 in \citeappendix{vanderVaart1996}.
Our condition on $p$ ensures that the covering integral appearing in this theorem is finite.
%%%%%%%%%%%%%%%
%Indeed
%For $x\geq 0$, the inverse of $\Phi(x) = |x|^p=y$ is $\Phi^{-1}(x) =y^{\frac{1}{p}}$ and $[a,b]$ can be covered (w.r.t to our semimetric) with $O(\epsilon^{-\frac{1}{\alpha}})$ balls of radius $\frac{1}{2}\epsilon$\\
%A $\epsilon$-Ball around $t$ w.r.t to the semimetric is $[t-\epsilon^{1/\alpha}, t+\epsilon^{1/\alpha}]$...
%But:
%$$
%\Phi^{-1}(\epsilon^{-\frac{1}{\alpha}})= \epsilon^{- \frac{1}{\alpha p}} = \epsilon^{- \frac{2}{\kappa p}}
%$$
%and
%$$\int_{0}^{b-a} \Phi^{-1}(\epsilon^{-\frac{1}{\alpha}}) \, d\epsilon = \int_{0}^{b-a}  \epsilon^{- \frac{2}{\kappa p}} <\infty$$
%if and only if
%\begin{align*}
%\frac{2}{\kappa p} &<1 \\
%\Leftrightarrow  \quad \frac{2}{\kappa} &< p\\
%\end{align*}
%which holds by assumption.
%%%%%%%%%%%%%%%
The maximum inequalities of empirical processes then imply:
$$||\sup_{q_1,q_2\in [-1,1]}|s^{-\kappa/2} \frac{1}{\sqrt{n}}\sum_{i=1}^{n}(X_i(\tau_r+sq_1)-X_i(\tau_r+sq_2))\varepsilon_i F(\bX_i)|||_{\Psi_p} \leq M_{11}$$
for some constant $M_{11}<\infty$. At the same time, the Markov inequality implies
\begin{align*}
P\Bigg(&\sup_{\tau_r-s\leq u \leq \tau_r + s} |\frac{1}{n}\sum_{i=1}^n(X_i(u)- X_{i}(\tau_r))\varepsilon_i F(\bX_i)|>s^{\kappa/2}\frac{x}{\sqrt{n}}\Bigg) \\
%&= P(\sup_{\tau_r-s\leq u \leq \tau_r + s} |s^{-\frac{\kappa}{2}}\frac{1}{\sqrt{n}}\sum_{i=1}^n X_i(t^*) (X_i(u)- X_{i}(\tau_r))\varepsilon_i F(\bX_i(\btau))|> x)\\
%&\leq P(\sup_{q_1,q_2} |s^{-\frac{\kappa}{2}}\frac{1}{\sqrt{n}}\sum_{i=1}^n X_i(t^*) (X_i(\tau_r +sq_1)- X_{i}(\tau_r+sq_2))\varepsilon_iF(\bX_i(\btau))| >x)\\
&\leq P\Bigg( |\sup_{q_1,q_2\in [-1,1]} |s^{-\frac{\kappa}{2}}\frac{1}{\sqrt{n}}\sum_{i=1}^n (X_i(\tau_r +sq_1)- X_{i}(\tau_r+sq_2))\varepsilon_i F(\bX_i)||^p > x^p\Bigg)
\leq \frac{M_{11}^p}{x^p}.
\end{align*}
Assertion \eqref{eq:lcdev2} then follows from $|\widehat{\tau}_r-\tau_r| = O_p(n^{-\frac{1}{\kappa}})$.

It remains to proof \eqref{eq:lcdev3}. For real numbers $x$ and $y$ it obviously holds that $x^4-y^4=(x-y)(x+y)(x^2+y^2)$. With the help of this equation and inequality \eqref{eq:vari}
it is easy to see that there exists a constant $M_{12}<\infty$ such that for all $p\geq 1$, all sufficiently small $s$, all $q_1,q_2\in [-1,1]$ we now have
\begin{align}
\begin{split}\label{eq:welldoneaux}
\E\Bigg(|s^{-2\kappa}\frac{1}{\sqrt{n}}\sum_{i=1}^n & ((X_i(\tau_r+sq_1)-X_i(\tau_r))^4- (X_i(\tau_r+sq_2)-X_i(\tau_r))^4 \\
& - \E\left((X_i(\tau_r+sq_1)-X_i(\tau_r))^4- (X_i(\tau_r+sq_2)-X_i(\tau_r))^4\right))|^p\Bigg)\leq M_{12}^p |q_1-q_2|^{\frac{p\kappa}{2}}.
\end{split}
\end{align}
%%%%%%%%%%%%%%%%%%%%%%%%%%%%%%%%%%%%%%%%%%%%%%
%Proof:
%Der Ausdruck konvergiert wieder gegen das p-te moment einer "`normalverteilten zufallsvariable"' mit varianz var((X_i(\tau_r+sq_1)-X_i(\tau_r))^4- (X_i(\tau_r+sq_2)-X_i(\tau_r))^4))
%Die Varianz ist das produkt von normalverteilten zufallsvariablen, lässt sich also mit den einzelnen varianzen der produkte abschchätzen
%pi mal daumen: var((x-y)(x+y)(x^2+y^2)) \leq M E((x-y)^2(x^2+y^2)(x^4+y^4))
% - $E((x-y)^2) \leq |q_1-q_2|^\kappa s^\kappa$
% - $E(x^2) \leq M s^\kappa$ und $E(x^2) \leq M s^\kappa$
% - $E(x^4) \leq M s^{2\kappa}$ und $E(y^4) \leq M s^{2\kappa}$
% insgesamt ist also $var((x-y)(x+y)(x^2+y^2)) \leq |q_1-q_2|^\kappa s^{4\kappa}$
% und das p-te moment der normalverteilten Z.V.gegen den der Ausdruck konvergiert: $\leq M_1^p (s^{2p/kappa} |q_1-q_2|^{p \kappa/2})$
%%%%%%%%%%%%%%%%%%%%%%%%%%%%%%%%%%%%%%%%%%%%%%
At the same time inequality \eqref{eq:vari} implies that there exists a constant $0<M_{13}<\infty$ such that $|\E((X_i(\tau_r+sq_1)-X_i(\tau_r))^4)| \leq M_{13}s^{2\kappa}$. Choose some $p>\frac{2}{\kappa}$, by \eqref{eq:welldoneaux} and with the help of another application of the maximum inequalities for empirical processes we can then conclude that there exists a constant $M_{14}<\infty$ such that
\begin{align*}
P\left(\sup_{\tau_r-s\leq u \leq \tau_r + s} |\frac{1}{n}\sum_{i=1}^n(X_i(u)- X_{i}(\tau_r))^4|>M_{13}s^{2\kappa} + \frac{s^{2\kappa}}{\sqrt{n}}x\right) \leq \frac{M_{14}^p}{x^p},
\end{align*}
Assertion~\eqref{eq:lcdev3} then follows once more from $|\widehat{\tau}_r-\tau_r| = O_p(n^{-\frac{1}{\kappa}})$. %\eqref{thm3eq1}.
%$\square$\\
\end{proof}
%% %%%%%%%%%%%%%%%%%%%%%%%%%%%%%%%%%%%%%%%%%%%%%%%%%%%%%%%%%%%%%%%%%
%% End: Proof of Abschätzungen für Kernelestimates
%% %%%%%%%%%%%%%%%%%%%%%%%%%%%%%%%%%%%%%%%%%%%%%%%%%%%%%%%%%%%%%%%%%

\bigskip

Suppose $\widehat{S}=S$. The feasible kernel density estimator for the density $f_{\tau}(\bx)=f(x_1,\dots, x_S)$ of $(X_i(\tau_1),\dots, X_i(\tau_S))$ replaces the unknown $\tau_r$ with their estimates $\widehat{\tau}_r$:
\begin{align*}
\widehat{f}_{\widehat{\tau}}(\bx) = \widehat{f}_{\widehat{\tau}}(x_1,\dots, x_S) = \frac{1}{n h_1\cdots h_S}\sum_{i=1}^n K\Big(\frac{X_i(\widehat{\tau}_1)-x_1}{h_1},\dots, \frac{X_i(\widehat{\tau}_S)-x_S}{h_S}\Big).
\end{align*}

\bigskip

%% %%%%%%%%%%%%%%%%%%%%%%%%%%%%%%%%%%%%%%%%%%%%%%%%%%%%%%%%%%%%%%%%%
%% Theorem KernelDensityEstimates
%% %%%%%%%%%%%%%%%%%%%%%%%%%%%%%%%%%%%%%%%%%%%%%%%%%%%%%%%%%%%%%%%%%
\begin{theorem}\label{thm:KDE}
Let $\widehat{S}=S$, $\max_{r=1,\dots, S}|\widehat{\tau}_r-\tau_r| = O_p(n^{-1/\kappa})$ and let $X_i$ be Gaussian process satisfying Assumption~\ref{assum1}. If the Kernel $K: \mathbb{R}^S \to \mathbb{R}$ is bounded and twice continuously differentiable with bounded derivatives, %die bedingung continuously differentiable ist etwas hart: erste ableitung muss stetig sein, die zweite muss eigentlich nur existieren für den restterm..
we then have
\begin{align}\label{eq:thmkde1}
\widehat{f}_{\widehat{\tau}}(\bx)  = \widehat{f}_{\tau}(\bx) + O_p\Big(\sum_{r=1}^S\frac{1}{n^{\min\{1,1/\kappa\}}(h_1\cdots h_S) h_r^2}\Big),
\end{align}
where $\widehat{f}_\tau(\bx)$ denotes the kernel density estimator for the density of $(X_i(\tau_1),\dots,X_i(\tau_S))$. %, where all points of impacts are known.
\end{theorem}
%%%%%%%%%%%%%%%%%%%%%%%

\bigskip

\begin{proof}[Proof of Theorem \ref{thm:KDE}]
By the usual decomposition we have
\begin{align}
\label{eq:kdeq0}
\begin{split}
\widehat{f}_{\widehat{\tau}}(\bx) = \widehat{f}_{\tau}(\bx) + \frac{1}{n h_1\cdots h_S}\sum_{i=1}^n
\Bigg(K\Big(\frac{\bX_i(\widehat{\btau})-\bx}{\bh}\Big)- K\Big(\frac{\bX_i(\btau)-\bx}{\bh}\Big)\Bigg),
%\widehat{f}_{\widehat{\tau}}(\bx) &= \widehat{f}_{\tau}(\bx) \\
%&+ \frac{1}{n h_1\dots h_S}\sum_{i=1}^n
%\Bigg(K\Big(\frac{X_i(\widehat{\tau}_1)-x_1}{h_1},\dots, \frac{X_i(\widehat{\tau}_S)-x_S}{h_S}\Big)- K\Big(\frac{X_i(\tau_1)-x_1}{h_1},\dots, \frac{X_i(\tau_S)-x_S}{h_S}\Big)\Bigg).
\end{split}
\end{align}
with $K\Big(\frac{\bX_i(\widehat{\btau})-\bx}{\bh}\Big) = K\Big(\frac{X_i(\widehat{\tau_1})-x_1}{h_1},\dots,\frac{X_i(\widehat{\tau_S})-x_S}{h_S}\Big)$ and
$K\Big(\frac{\bX_i(\btau)-\bx}{\bh}\Big) = K\Big(\frac{X_i(\tau_1)-x_1}{h_1},\dots,\frac{X_i(\tau_S)-x_S}{h_S}\Big)$.
%On the other hand, a Taylor expansion yields
%\begin{align*}
%&K\Big(\frac{X_i(\widehat{\tau}_1)-x}{h_1},\dots, \frac{X_i(\widehat{\tau}_S)-x}{h_S}\Big)- K\Big(\frac{X_i(\tau_1)-x}{h_1},\dots, \frac{X_i(\tau_S)-x}{h_S}\Big)\\
%&= \sum_{r=1}^S \frac{X_i(\widehat{\tau}_r)-X_r(\tau_r)}{h_r}K_r'\Big(\frac{X_i(\tau_1)-x}{h_1},\dots, \frac{X_i(\tau_S)-x}{h_S}\Big)
%+ R_i(x),
%\end{align*}
A Taylor expansion then yields
{\small \begin{align}
\frac{1}{n h_1\dots h_S}&\sum_{i=1}^n
\Bigg(K\Big(\frac{\bX_i(\widehat{\btau})-\bx}{\bh}\Big) -  K\Big(\frac{\bX_i(\btau)-\bx}{\bh}\Big)\Bigg)\nonumber \\
%\begin{split}
&=\sum_{r=1}^S \frac{1}{h_r}\, \frac{1}{h_1\cdots h_S} \, \frac{1}{n}\sum_{i=1}^n \big(X_i(\widehat{\tau}_r)-X_r(\tau_r)\big)K_r'\Big(\frac{\bX_i(\btau)-\bx}{\bh}\Big)
+\frac{1}{n h_1 \cdots h_S} \sum_{i=1}^n R_i(\bx), \label{eq:kdeq1}
%&\frac{1}{n h_1\dots h_S}\sum_{i=1}^n
%\Bigg(K\Big(\frac{X_i(\widehat{\tau}_1)-x_1}{h_1},\dots, \frac{X_i(\widehat{\tau}_S)-x_S}{h_S}\Big)- K\Big(\frac{X_i(\tau_1)-x_1}{h_1},\dots, \frac{X_i(\tau_S)-x_S}{h_S}\Big)\Bigg)\\
%&=\sum_{r=1}^S \frac{1}{h_r}\, \frac{1}{h_1\dots h_S} \, \frac{1}{n}\sum_{i=1}^n \big(X_i(\widehat{\tau}_r)-X_r(\tau_r)\big)K_r'\Big(\frac{X_i(\tau_1)-x_1}{h_1},\dots, \frac{X_i(\tau_S)-x_S}{h_S}\Big)\\
%&+\frac{1}{n h_1 \dots h_S} \sum_{i=1}^n R_i(\bx),
%\end{split}
\end{align}}
where $|R_i(\bx)| \leq M \sum_{r=1}^S \left(\frac{X_i(\widehat{\tau}_r)-X_i(\tau_r)}{h_r}\right)^2$ for some constant $M<\infty$. Remember that under our assumptions we have by assertion \eqref{eq:lcdev2} $\frac{1}{n}\sum_{i=1}^n(X_i(\widehat{\tau}_r)-X_r(\tau_r))^2 = O_p(n^{-1})$, hence:
\begin{align}
\left|\frac{1}{n h_1 \cdots h_S} \sum_{i=1}^n R_i(\bx)\right| & \leq \frac{1}{h_1\cdots h_S} \frac{1}{n}\sum_{i=1}^n |R_i(\bx)| \nonumber\\
&\leq M \sum_{r=1}^S \frac{1}{(h_1\cdots h_S) h_r^2} \frac{1}{n}\sum_{i=1}^n (X_i(\widehat{\tau}_r)-X_r(\tau_r))^2 %\nonumber \\
 = O_p\left(\sum_{r=1}^S\frac{1}{(n h_1\cdots h_S) h_r^2}\right).\label{eq:kdeq2}
\end{align}
It then remains to bound $\frac{1}{n}\sum_{i=1}^n \big(X_i(\widehat{\tau}_r)-X_r(\tau_r)\big) K_r'\Big((\bX_i(\btau)-\bx)/\bh \Big)$.
%$$\frac{1}{n}\sum_{i=1}^n \big(X_i(\widehat{\tau}_r)-X_r(\tau_r)\big) K_r'\Big(\frac{X_i(\tau_1)-x_1}{h_1},\dots, \frac{X_i(\tau_S)-x_S}{h_S}\Big).$$
With $F(X_i(\tau_1),\dots, X_i(\tau_S))=K_r'\Big((X_i(\tau_1)-x_1)/h_1,\dots, (X_i(\tau_S)-x_S)/h_S \Big)$
%$$F(X_i(\tau_1),\dots, X_i(\tau_S))=K_r'\Big(\frac{X_i(\tau_1)-x_1}{h_1},\dots, \frac{X_i(\tau_S)-x_S}{h_S}\Big),$$
we have
\begin{align*}
\frac{\partial}{\partial u_l} F(X_i(\tau_1),\dots, X_i(\tau_S)) %&=
 %\frac{\partial}{\partial u_l}K_r'\Big(\frac{u_1-x_1}{h_1},\dots, \frac{u_l-x}{h_S}\Big)|_{u_r=X_i(\tau_r)} \\
&= \frac{1}{h_l}K_{rl}''\Big(\frac{X_i(\tau_1)-x_1}{h_1},\dots, \frac{X_i(\tau_S)-x_S}{h_S}\Big).
\end{align*}
Since $|K_{rl}''(\cdot)|<\infty$, it then follows from \eqref{eq:prop1new1} in Proposition \ref{prop:1new} that
{\small
\begin{align}\label{eq:kdeqaux}
\frac{1}{n}\sum_{i=1}^n \big(X_i(\widehat{\tau}_r)-X_r(\tau_r)\big) K_r'\Big(\frac{\bX_i(\btau)-\bx}{\bh}\Big) &= O_p\Big(n^{-\min\{1,1/\kappa\}}\sum_{l=1}^S{\frac{1}{h_l}}\Big).
%\frac{1}{n}\sum_{i=1}^n \big(X_i(\widehat{\tau}_r)-X_r(\tau_r)\big) K_r'\Big(\frac{X_i(\tau_1)-x_1}{h_1},\dots, \frac{X_i(\tau_S)-x_S}{h_S}\Big) &= O_p\Big(n^{-\min\{1,1/\kappa\}}\sum_{l=1}^S{\frac{1}{h_l}}\Big).
\end{align}}
 Note that for some constant $0<M_S<\infty$ we have $(\sum_{r=1}^S 1/h_r)^2 \leq M_S \sum_{r=1}^S 1/h_r^2$, we can thus conclude from \eqref{eq:kdeqaux}  that
{\small \begin{align}
\sum_{r=1}^S \frac{1}{h_r}\, \frac{1}{h_1\cdots h_S} \,  \frac{1}{n}\sum_{i=1}^n \big(X_i(\widehat{\tau}_r)-X_r(\tau_r)\big)K_r'\Big(\frac{\bX_i(\btau)-\bx}{\bh}\Big) %\nonumber \\
%&= \frac{1}{h_1\dots h_S}\, \sum_{r=1}^S \frac{1}{h_r}\, O_p\Big(n^{-\min\{1,1/\kappa\}}\sum_{l=1}^S{\frac{1}{h_l}}\Big) \nonumber \\
%& O_p(\frac{1}{n^{\min\{1,1/\kappa\}}(h_1 \dots h_S)} (\sum_{l=1}^S{\frac{1}{h_l}})^2)\\
= O_p\left(\sum_{r=1}^S \frac{1}{n^{\min\{1,1/\kappa\}}(h_1 \cdots h_S)h_r^2}\right).\label{eq:kdeq3}
%%%%%%
%\sum_{r=1}^S \frac{1}{h_r}\, \frac{1}{h_1\dots h_S} \, & \frac{1}{n}\sum_{i=1}^n \big(X_i(\widehat{\tau}_r)-X_r(\tau_r)\big)K_r'\Big(\frac{X_i(\tau_1)-x_1}{h_1},\dots, \frac{X_i(\tau_S)-x_S}{h_S}\Big) \nonumber \\
%&= \frac{1}{h_1\dots h_S}\, \sum_{r=1}^S \frac{1}{h_r}\, O_p\Big(n^{-\min\{1,1/\kappa\}}\sum_{l=1}^S{\frac{1}{h_l}}\Big) \nonumber \\
%%& O_p(\frac{1}{n^{\min\{1,1/\kappa\}}(h_1 \dots h_S)} (\sum_{l=1}^S{\frac{1}{h_l}})^2)\\
%&= O_p\Big(\sum_{r=1}^S \frac{1}{n^{\min\{1,1/\kappa\}}(h_1 \dots h_S)h_r^2}\Big)\label{eq:kdeq3}
\end{align}}
The assertion of the theorem now follows immediately from \eqref{eq:kdeq0}, \eqref{eq:kdeq1}, \eqref{eq:kdeq2} and \eqref{eq:kdeq3}.
\end{proof}
%% %%%%%%%%%%%%%%%%%%%%%%%%%%%%%%%%%%%%%%%%%%%%%%%%%%%%%%%%%%%%%%%%%
%% End: Theorem KernelDensityEstimates
%% %%%%%%%%%%%%%%%%%%%%%%%%%%%%%%%%%%%%%%%%%%%%%%%%%%%%%%%%%%%%%%%%%

\bigskip

%% %%%%%%%%%%%%%%%%%%%%%%%%%%%%%%%%%%%%%%%%%%%%%%%%%%%%%%%%%%%%%%%%%
%% Theorem Nonparametric-Regression
%% %%%%%%%%%%%%%%%%%%%%%%%%%%%%%%%%%%%%%%%%%%%%%%%%%%%%%%%%%%%%%%%%%
\begin{proof}[Proof of Theorem \ref{thm:lcr}]
The proof follows from generalizations of arguments used in proofs for nonparametric regression (see for example \citeappendix[][Ch. 2]{LR06}). We have
\begin{align}\label{eq:lceq1}
\widehat{g}_{\widehat{\tau}}(\bx) - g(\bx) = \frac{(\widehat{g}_{\widehat{\tau}}(\bx) - g(\bx))\widehat{f}_{\widehat{\tau}}(\bx)}{\widehat{f}_{\widehat{\tau}}(\bx)} =: \frac{\widehat{m}_{\widehat{\tau}}(\bx)}{\widehat{f}_{\widehat{\tau}}(\bx)},
\end{align}
where $\widehat{m}_{\widehat{\tau}}(\bx) = (\widehat{g}_{\widehat{\tau}}(\bx) - g(\bx))\widehat{f}_{\widehat{\tau}}(\bx)$.
With $Y_i=g(\bX_i(\btau))+\varepsilon_i$ we can then write
%$Y_i=g(X_i(\tau_1),\dots, X_i(\tau_S))+\varepsilon_i$
%$$\widehat{g}_{\widehat{\tau}}(x) = \frac{\sum_{i=1}^n Y_i K((X_i(\widehat{\tau_1})-x)/h_1,\dots,(X_i(\widehat{\tau_S})-x)/h_S)}{\sum_{i=1}^n K((X_i(\widehat{\tau_1})-x)/h_1,\dots,(X_i(\widehat{\tau_S})-x)/h_S)}$$
\begin{align}\label{eq:lceq2}
\widehat{m}_{\widehat{\tau}}(\bx) = \widehat{m}_{\widehat{\tau},1}(\bx) + \widehat{m}_{\widehat{\tau},2}(\bx),
\end{align}
where
\begin{align}
\widehat{m}_{\widehat{\tau},1}(\bx) & = \frac{1}{n (h_1\cdots h_S)} \sum_{i=1}^n (g(\bX_i(\btau))-g(\bx)) K\Big(\frac{\bX_i(\widehat{\btau})-\bx}{\bh}\Big)\label{eq:lceq3}\\
\intertext{and}
\widehat{m}_{\widehat{\tau},2}(\bx) &= \frac{1}{n (h_1\cdots h_S)} \sum_{i=1}^n \varepsilon_i K\Big(\frac{\bX_i(\widehat{\btau})-\bx}{\bh}\Big).\label{eq:lceq4}
%\widehat{m}_{\widehat{\tau},1}(\bx) & = \frac{1}{n (h_1\dots h_S)} \sum_{i=1}^n (g(\bX_i(\btau))-g(\bx)) K(\frac{X_i(\widehat{\tau_1})-x_1}{h_1},\dots, \frac{X_i(\widehat{\tau_S})-x_S}{h_S})\label{eq:lceq3}\\
%\intertext{and}
%\widehat{m}_{\widehat{\tau},2}(\bx) &= \frac{1}{n (h_1\dots h_S)} \sum_{i=1}^n \varepsilon_i K(\frac{X_i(\widehat{\tau_1})-x_1}{h_1},\dots, \frac{X_i(\widehat{\tau}_S)-x_S}{h_S}).\label{eq:lceq4}
\end{align}
With $\widehat{m}_{\tau,1}(\bx) = (nh_1\cdots h_S)^{-1}\sum_{i=1}^n g(\bX_i(\btau))-g(\bx))K((\bX_i(\btau)-\bx)/\bh)$ we have by the usual decomposition:
\begin{align*}
\widehat{m}_{\widehat{\tau},1}(\bx) & = \widehat{m}_{\tau,1}(\bx)
 + \frac{1}{n (h_1\cdots h_S)} \sum_{i=1}^n (g(\bX_i(\btau))-g(\bx)) \Bigg( K\Big(\frac{\bX_i(\widehat{\btau})-\bx}{\bh}\Big)-K\Big(\frac{\bX_i(\btau)-\bx}{\bh}\Big)\Bigg).
%\widehat{m}_{\widehat{\tau},1}(\bx) & = \widehat{m}_{\tau,1}(\bx)\\
%& + \frac{1}{n (h_1\dots h_S)} \sum_{i=1}^n (g(\bX_i(\btau))-g(\bx)) \Bigg( K\Big(\frac{X_i(\widehat{\tau_1})-x_1}{h_1},\dots, \frac{X_i(\widehat{\tau_S})-x_S}{h_S})-K(\frac{X_i(\tau_1)-x_1}{h_1},\dots, \frac{X_i(\tau_S)-x_S}{h_S}\Big)\Bigg)
\end{align*}
A Taylor expansion then yields
\begin{align}
\widehat{m}_{\widehat{\tau},1}(\bx)&  =\widehat{m}_{\tau,1}(\bx) \nonumber\\
& + \frac{1}{(h_1\cdots h_S)}\sum_{r=1}^S \frac{1}{h_r}  \frac{1}{n} \sum_{i=1}^n (g(\bX_i(\btau))-g(\bx))
    (X_i(\widehat{\tau}_r)-X_i(\tau_r))K_r'\Big(\frac{\bX_i(\btau)-\bx}{\bh}\Big)\label{eq:lceq5}\\
& + \frac{1}{n (h_1\cdots h_S)} \sum_{i=1}^n (g(\bX_i(\btau))-g(\bx))R_{i,1}(\bx), \label{eq:lceq6}
%\widehat{m}_{\widehat{\tau},1}(\bx)&  =\widehat{m}_{\tau,1}(\bx) \nonumber\\
%&  + \frac{1}{(h_1\dots h_S)}\sum_{r=1}^S \frac{1}{h_r}  \frac{1}{n} \sum_{i=1}^n (g(\bX_i)-g(\bx))
    %(X_i(\widehat{\tau_r})-X_i(\tau_r))K_r'\Big(\frac{X_i(\tau_1)-x_1}{h_1},\dots, \frac{X_i(\tau_S)-x_S}{h_S}\Big)\label{eq:lceq5}\\
%&    + \frac{1}{n (h_1,\dots h_S)} \sum_{i=1}^n (g(\bX_i(\btau))-g(\bx))R_{i,1}(\bx) \label{eq:lceq6}
\end{align}
where $|R_{i,1}(\bx)| \leq M \sum_{r=1}^S \left(\frac{X_i(\widehat{\tau}_r)-X_i(\tau_r)}{h_r}\right)^2$ for some constant $0<M<\infty$, since all partial second derivatives of $K$ are bounded.\\
Now, assuming that $g$ and its first derivative is bounded, we obtain by Proposition \ref{prop:1new}:
$$\frac{1}{n}\sum_{i=1}^n (g(\bX_i(\btau))-g(\bx))(X_i(\widehat{\tau}_r)-X_i(\tau_r)) K_r'\Big(\frac{\bX_i(\btau)-\bx}{\bh}\Big) = O_p\Big(\sum_{r=1}^S \frac{1}{n^{\min\{1,1/\kappa\}} h_r}\Big).$$
%$$\frac{1}{n}\sum_{i=1}^n (g(\bX_i(\btau))-g(\bx))(X_i(\widehat{\tau}_r)-X_i(\tau_r)) K_r'\Big(\frac{X_i(\tau_1)-x_1}{h_1},\dots, \frac{X_i(\tau_S)-x_S}{h_S}\Big) = O_p\Big(\sum_{r=1}^S \frac{1}{n^{\min\{1,1/\kappa\}} h_r}\Big).$$
On the other hand, since by \eqref{eq:lcdev1} we have $\frac{1}{n}\sum_{i=1}^n (X_i(\widehat{\tau}_r)- X_i(\tau_r))^2 = O_p(n^{-1})$ and since $g$ is bounded, we obtain:
$$|\frac{1}{n}\sum_{i=1}^n (g(\bX_i(\btau))-g(\bx))R_{i,1}(\bx)|=O_p\Bigg(\sum_{r=1}^S \frac{1}{n h_r^2}\Bigg).$$
Together with \ref{eq:lceq5} and \ref{eq:lceq6} we then obtain
\begin{align}
\widehat{m}_{\widehat{\tau},1}(\bx)&  =
\widehat{m}_{\tau,1}(\bx)
+ O_p\Bigg(\sum_{r=1}^S\frac{1}{n^{\min\{1,1/\kappa\}} (h_1\dots h_S) h_r^2}\Bigg)
+ O_p\Bigg(\sum_{r=1}\frac{1}{n(h_1\cdots h_S)h_r^2}\Bigg)\nonumber \\
& = \widehat{m}_{\tau,1}(\bx) + O_p\Bigg(\sum_{r=1}^S\frac{1}{n^{\min\{1,1/\kappa\}} (h_1\cdots h_S) h_r^2}\Bigg).\label{eq:m1}
\end{align}

\noindent At the same time, since $K$ is bounded, \eqref{eq:lcdev2} implies
$$\frac{1}{n}\sum_{i=1}^n (X_i(\widehat{\tau}_r) - X_i(\tau_r))\varepsilon_i K\Big(\frac{\bX_i(\btau)-\bx}{\bh}\Big) = O_p(n^{-1}).$$
%$$\frac{1}{n}\sum_{i=1}^n (X_i(\widehat{\tau}_r) - X_i(\tau_r))\varepsilon_i K(\frac{X_i(\tau_1)-x_1}{h_1},\dots, \frac{X_i(\tau_S)-x_S}{h_S}) = O_p(n^{-1}).$$
%%%%%%%%%%%%%%%%%%%%%%%%%%%%%%%%%%%
Another Taylor expansion then yields %together with the CS-inequality, since $\frac{1}{n}\sum_{i=1}^n (X_i(\widehat{\tau}_r)-X_i(\tau_r))^4 = O_p(n^{-2})$,
\begin{align*}
\widehat{m}_{\widehat{\tau},2}(\bx) & = \widehat{m}_{\tau,2}(\bx) \\
&+ \frac{1}{h_1\cdots h_S}\,\sum_{r=1}^S \frac{1}{h_r} \frac{1}{n} (X_i(\widehat{\tau}_r)-X_i(\tau_r))\varepsilon_iK_r'\Big(\frac{\bX_i(\btau)-\bx}{\bh}\Big)\\
&+ \frac{1}{h_1\cdots h_S}\, \frac{1}{n}\sum_{i=1}^n R_{i,2}(\bx)
%\widehat{m}_{\widehat{\tau},2}(\bx) & = \widehat{m}_{\tau,2}(\bx) \\
%&+ \frac{1}{h_1\dots h_S}\,\sum_{r=1}^S \frac{1}{h_r} \frac{1}{n} (X_i(\widehat{\tau}_r)-X_i(\tau_r))\varepsilon_iK_r'\Big(\frac{X_i(\tau_1)-x_1}{h_1},\dots, \frac{X_i(\tau_S)-x_S}{h_S}\Big)\\
%&+ \frac{1}{h_1\dots h_S}\, \frac{1}{n}\sum_{i=1}^n R_{i,2}(\bx)
\end{align*}
With $|R_{i,2}(\bx)| \leq M_K\sum_{r=1}^S  (X_i(\widehat{\tau}_r)-X_i(\tau_r))^2/h_r^2 \cdot |\varepsilon_i|$, for some constant $M_K<\infty$, it follows from \eqref{eq:lcdev2} and \eqref{eq:lcdev3} together with the Cauchy-Schwarz inequality, that
\begin{align}
\label{eq:m2}
\begin{split}
\widehat{m}_{\widehat{\tau},2}(\bx) & = \widehat{m}_{\tau,2}(\bx)  + O_p\Bigg(\sum_{r=1}^S\frac{1}{n (h_1\cdots h_S) h_r}\Bigg) + O_p\Bigg(\sum_{r=1}^S\frac{1}{n (h_1\cdots h_S) h_r^2}\Bigg)\\
&= \widehat{m}_{\tau,2}(\bx)  + O_p\Bigg(\sum_{r=1}^S\frac{1}{n (h_1\cdots h_S) h_r^2}\Bigg).
\end{split}
\end{align}
With \eqref{eq:m1} and \eqref{eq:m2} we can conclude from \eqref{eq:lceq2} that
$$\widehat{m}_{\widehat{\tau}}= \widehat{m}_{\tau} + O_p\Bigg(\sum_{r=1}^S \frac{1}{n^{\min\{1,1/\kappa\}}(h_1\cdots h_S)h_r^2}\Bigg).$$
Since $f_{\tau}(\bx)>0$ and $\widehat{m}_{\tau}= O_p\Big(\sum_{r=1}^S h_r^2 + (h_1\dots h_S)^{-1/2}\Big)$ (see for example \citeappendix[][Ch.2]{LR06}) we then arrive together with \eqref{eq:thmkde1} at
\begin{align}
\begin{split}
\widehat{g}_{\widehat{\tau}}(\bx)-g(\bx) &= O_p\Bigg(\frac{\widehat{m}_{\widehat{\tau}}}{\widehat{f}_{\widehat{\tau}}(\bx)} \Bigg) =O_p\Bigg(\frac{\widehat{m}_{\widehat{\tau}}}{f_{\tau}(\bx)+o_p(1)} \Bigg)  \\
& =O_p\Bigg(\sum_{r=1}^S h_r^2 + (h_1\cdots h_S)^{-1/2} + \sum_{r=1}^S \frac{1}{n^{\min\{1,1/\kappa\}}(h_1\cdots h_S)h_r^2}\Bigg),
\end{split}
\end{align}
provided $n^{\min\{1,1/\kappa\}}(h_1\cdots h_S)h_r^2 \to \infty$ for $r=1, \dots, S$
%\bigskip
%In total
%$$\widehat{m}_{\widehat{\tau}}(\bx) = O_p\left(\eta_2 + \eta_1^{1/2} + n^{-\min\{1,1/\kappa\}} (h_1\dots h_S)^{-1} \sum_{r=1}^S (\frac{1}{h_r})^2\right).$$
%Finally:
%\begin{align*}
%\widehat{g}_{\widehat{\tau}}(\bx) - g(\bx)& = O_p\left(\frac{\widehat{m}_{\widehat{\tau}}(\bx)}{\widehat{f}_{\widehat{\tau}}(\bx)} \right)\\
%& = O_p\left(\frac{\widehat{m}_{\widehat{\tau}}(\bx)}{f_{\tau}(\bx)+o_p(1)} \right)\\
%& = O_p\left(\sum_{r=1}^S h_r^2 + (nh_1\dots h_S)^{-1/2} + \sum_{r=1}^S \frac{1}{n^{\min\{1,1/\kappa\}} (h_1\dots h_S) h_r^2}\right)
%\end{align*}
\end{proof}
%% %%%%%%%%%%%%%%%%%%%%%%%%%%%%%%%%%%%%%%%%%%%%%%%%%%%%%%%%%%%%%%%%%
%% End: Theorem Nonparametric-Regression
%% %%%%%%%%%%%%%%%%%%%%%%%%%%%%%%%%%%%%%%%%%%%%%%%%%%%%%%%%%%%%%%%%%

\bigskip

%% %%%%%%%%%%%%%%%%%%%%%%%%%%%%%%%%%%%%%%%%%%%%%%%%%%%%%%%%%%%%%%%%%
%% Corollary: Nonparametric-Regression - optimal h if kappa<1
%% %%%%%%%%%%%%%%%%%%%%%%%%%%%%%%%%%%%%%%%%%%%%%%%%%%%%%%%%%%%%%%%%%
\begin{proof}[Proof of Corollary~\ref{cor:lcr}] For $\kappa \leq 1$ and $h_r\sim n^{-1/(S+4)}$, $r=1,\dots, S$, the assertion of Corollary~\ref{cor:lcr} is a direct consequence of \eqref{eq:NWkernelREG} in Theorem~\ref{thm:lcr}.
\end{proof}
%% %%%%%%%%%%%%%%%%%%%%%%%%%%%%%%%%%%%%%%%%%%%%%%%%%%%%%%%%%%%%%%%%%
%% End Corollary: Nonparametric-Regression - optimal h if kappa<1
%% %%%%%%%%%%%%%%%%%%%%%%%%%%%%%%%%%%%%%%%%%%%%%%%%%%%%%%%%%%%%%%%%%

\bigskip

%%%%%%%%%%%%%%%%%%%%%%%%%%%%%%%%%%%%%%%%%%%%%%%%%%%%%%
\subsection{The parametric case}\label{app:PARAest} %\label{app:parametric}
In proofs of this appendix, we will make use of the following Lemma which summarizes results from \citeappendix{B2012b} and \citeappendix{B2012a}:

\begin{lemma}\label{lem:1aux}
Let $X_1$ and $X_2$ be two random variables with $0<\V(X_j)<\infty$, $j=1,2$. For any real valued function $f:\mathbb{R}\to \mathbb{R}$ with $\E(|f(X_1)|)<\infty$ and $\E(|f(X_1)X_1|)<\infty$ we have
\begin{align}\label{eq:lem1aux}
cov(f(X_1),X_2)= \frac{cov(f(X_1),X_1)}{\V(X_1)}cov(X_1,X_2),
\end{align}
provided that $cov\left(X_2-\frac{cov(X_1,X_2)}{\V(X_1)}X_1, f(X_1)\right) = 0$.
\end{lemma}
%% %%%%%%%%%%%%%%%%%%%%%%%%%%%%%%%%%%%%%%%%%%%%%%%%%%%%%%%%
%% End Lemma: Steins Lemma for parametric case AUX-Edition
%% %%%%%%%%%%%%%%%%%%%%%%%%%%%%%%%%%%%%%%%%%%%%%%%%%%%%%%%%

\bigskip

%% %%%%%%%%%%%%%%%%%%%%%%%%%%%%%%%%%%%%%%%%%%%%%%%%%%%%%%%%
%% Proof of Lemma: Steins Lemma for parametric case AUX-Edition
%% %%%%%%%%%%%%%%%%%%%%%%%%%%%%%%%%%%%%%%%%%%%%%%%%%%%%%%%%
\begin{proof}[{\bf Proof of Lemma \ref{lem:1aux}.}]
With $X_2 = \frac{cov(X_1, X_2)}{\V(X_1)} X_1 +  \left(X_2 - \frac{cov(X_1, X_2)}{\V(X_1)} X_1  \right)$, the assertion of the lemma follows immediately, since
\begin{align*}
cov(f(X_1),X_2) & = \frac{cov(f(X_1),X_1)}{\V(X_1)} cov(X_1,X_2) + cov\left(X_2-\frac{cov(X_1,X_2)}{\V(X_1)}X_1, f(X_1)\right)\\
								& = \frac{cov(f(X_1),X_1)}{\V(X_1)} cov(X_1,X_2).
\end{align*}
\end{proof}
%%%%%%%%%%%%%%%
%% %%%%%%%%%%%%%%%%%%%%%%%%%%%%%%%%%%%%%%%%%%%%%%%%%%%%%%%%
%% End proof of Lemma: Steins Lemma for parametric case AUX-Edition
%% %%%%%%%%%%%%%%%%%%%%%%%%%%%%%%%%%%%%%%%%%%%%%%%%%%%%%%%%

%\section{Proofs of the Theoretical Results from Section \ref{sec:PES}}\label{app:PARAest}

%In this appendix the proofs leading to our theoretical results concerning the parameter estimates as discussed in Section~\ref{sec:PES} are given.

%%%%%%%%%%%%%%%%%%%%%%%%%%%%%%%%%%%%%%%%%%%%%%%%%%%%%%%%%%%%%%%%%%%%%%%%%%%%%%%%%%%%%%%
%%%%%%%%%%%%%%%%%%%%%%%%%%%%%%%%%%%%%%%%%%%%%%%%%%%%%%%%%%%%%%%%%%%%%%%%%%%%%%%%%%%%%%%
%%%%%%%%%%%%%%%%%%%%%%%%%%%%%%%%%%%%%%%%%%%%%%%%%%%%%%%%%%%%%%%%%%%%%%%%%%%%%%%%%%%%%%%
%Identifiability of the model

\bigskip

We begin with the proof of Theorem~\ref{thmident}.

%But instead of proofing Theorem~\ref{thmident} directly, we proof a more general statement for future references, allowing again for a functional part  $\int_{a}^b\beta(t)X_i(t)\, dt$ in the linear predictor $\eta_i$ to be present:

%\begin{theorem} \label{thmidentext}
%Under our setup assume that $X_i$ satisfies Assumption~\ref{assum1}. Then for all $S^*\geq S$, all $\alpha^*, \beta_1^*,\ldots,\beta_{S^*}^*\in\mathbb{R}$, and
%all $\tau_1,\dots,\tau_{S^*}\in(a,b)$ with $\tau_{k}\notin\{\tau_1,\dots,\tau_S\}$, $k=S+1,\dots,S^*$, we obtain
%{\small \begin{align}
%\mathbb{E}\left(\bigg(
%g(\alpha + \sum_{r=1}^{S}\beta_rX_i(\tau_r) )
%% + \int_a^b\beta(t)X_i(t)\,dt )
%- g(\alpha^* + \sum_{r=1}^{S^*}\beta_r^*X_i(\tau_r) )
%% + \int_a^b\beta^*(t)X_i(t)\,dt)
%\bigg)^2\right)>0,
%\label{eqident}
%\end{align}}
%whenever $|\alpha-\alpha^*|>0$,
%% or $\mathbb{E}\Big(\big(\int_a^b(\beta(t)-\beta^*(t))X(t)dt\big)^2\Big)>0$,
%or $\sup_{r=1,\ldots,S}|\beta_r-\beta^*_r|>0$,
%or $\sup_{r=S+1,\ldots,S^*}|\beta^*_r|>0$.
%\end{theorem}

%\bigskip

%\begin{proof}[{\bf Proof of Theorem \ref{thmidentext}.}]
\begin{proof}[{\bf Proof of Theorem \ref{thmident}.}]
Since $X_i$ satisfies Assumption~\ref{assum1}, Theorem 3 in \cite{KPS2015} implies that the assumptions of Theorem 1 in \cite{KPS2015} are met.
Since
{\small
\begin{align*}
\E&\Bigg(\Big(
(\alpha + \sum_{r=1}^{S}\beta_rX_i(\tau_r) )
% +\int_a^b\beta(t)X_i(t)\,dt)
- (\alpha^* + \sum_{r=1}^{S^*}\beta_r^*X_i(\tau_r) )
% + \int_a^b\beta^*(t)X_i(t)\,dt )
\Big)^2\Bigg)\\
& = (\alpha-\alpha^*)^2 +
\E\Bigg(\Big(
\sum_{r=1}^{S}\beta_rX_i(\tau_r)
% + \int_a^b\beta(t)X_i(t)\,dt
- \sum_{r=1}^{S^*}\beta_r^*X_i(\tau_r)
% - \int_a^b\beta^*(t)X_i(t)\,dt
\Big)^2\Bigg).
\end{align*}}
It follows from Theorem 1 in \cite{KPS2015} that
{\small
\begin{align}
\E\Bigg(\Big(
(\alpha + \sum_{r=1}^{S}\beta_rX_i(\tau_r) )
% +\int_a^b\beta(t)X_i(t)\,dt)
- (\alpha^* + \sum_{r=1}^{S^*}\beta_r^*X_i(\tau_r) )
% + \int_a^b\beta^*(t)X_i(t)\,dt)
\Big)^2\Bigg)>0, \label{eq:thmident1}
\end{align}}
whenever $|\alpha-\alpha^*|>0$,
% or $\mathbb{E}\Big(\big(\int_a^b(\beta(t)-\beta^*(t))X(t)dt\big)^2\Big)>0$,
or $\sup_{r=1,\ldots,S}|\beta_r-\beta^*_r|>0$,
or $\sup_{r=S+1,\ldots,S^*}|\beta^*_r|>0$.\\

Now suppose
{\small
$$\E\Bigg(\Big(
g(\alpha + \sum_{r=1}^{S}\beta_rX_i(\tau_r) )
% + \int_a^b\beta(t)X_i(t)\,dt )
- g(\alpha^* + \sum_{r=1}^{S^*}\beta_r^*X_i(\tau_r) )
% + \int_a^b\beta^*(t)X_i(t)\,dt)
\Big)^2\Bigg) =0.$$
}
It then follows that
$
g(\alpha + \sum_{r=1}^{S}\beta_rX_i(\tau_r) )
% + \int_a^b\beta(t)X_i(t)\,dt)
$
and
$
g(\alpha^* + \sum_{r=1}^{S^*}\beta_r^*X_i(\tau_r) )
% +\int_a^b\beta^*(t)X_i(t)\,dt)
$
must be identical, i.e.,
{\small
$$
P\Bigg(
g(\alpha + \sum_{r=1}^{S}\beta_rX_i(\tau_r) )
% + \int_a^b\beta(t)X_i(t)\,dt)
= g(\alpha^* + \sum_{r=1}^{S^*}\beta_r^*X_i(\tau_r) )
% + \int_a^b\beta^*(t)X_i(t)\,dt)
\Bigg)=1.
$$
}
Since $g$ is invertible we then have
{\small
\begin{align*}
P\Bigg(
g(\alpha + \sum_{r=1}^{S}\beta_rX_i(\tau_r) )
% + \int_a^b\beta(t)X_i(t)\,dt)
= g(\alpha^* + \sum_{r=1}^{S^*}\beta_r^*X_i(\tau_r) )
% +\int_a^b\beta(t)X_i(t)\,dt)
\Bigg)=1
\end{align*}
}
if and only if
{\small
\begin{align*}
P\Bigg((\alpha + \sum_{r=1}^{S}\beta_rX_i(\tau_r) )
% +\int_a^b\beta(t)X_i(t)\,dt)
= (\alpha^* + \sum_{r=1}^{S^*}\beta_r^*X_i(\tau_r) )
% + \int_a^b\beta(t)X_i(t)\,dt)
\Bigg)=1.
\end{align*}
}
But by \eqref{eq:thmident1} we have
{\small
$$\E\Bigg( \Big(
(\alpha + \sum_{r=1}^{S}\beta_rX_i(\tau_r) )
% + \int_a^b\beta(t)X_i(t)\,dt)
-
(\alpha^*  + \sum_{r=1}^{S^*}\beta_r^*X_i(\tau_r) )
% +\int_a^b\beta^*(t)X_i(t)\,dt)
\Big)^2\Bigg)>0,$$
}
whenever $|\alpha-\alpha^*|>0$,
% or $\mathbb{E}((\int_a^b(\beta(t)-\beta^*(t))X(t)dt)^2)>0$,
or $\sup_{r=1,\ldots,S}|\beta_r-\beta^*_r|>0$,
or $\sup_{r=S+1,\ldots,S^*}|\beta^*_r|>0$, implying
{\small
$$
P\Bigg(
(\alpha + \sum_{r=1}^{S}\beta_rX_i(\tau_r) )
% + \int_a^b\beta(t)X_i(t)\,dt)
= (\alpha^* + \sum_{r=1}^{S^*}\beta_r^*X_i(\tau_r) )
% +\int_a^b\beta^*(t)X_i(t)\,dt)
\Bigg)<1,
$$
}
whenever $|\alpha-\alpha^*|>0$,
% or $\mathbb{E}((\int_a^b(\beta(t)-\beta^*(t))X(t)dt)^2)>0$,
or $\sup_{r=1,\ldots,S}|\beta_r-\beta^*_r|>0$,
or $\sup_{r=S+1,\ldots,S^*}|\beta^*_r|>0$, which proves the assertion of the theorem. % $\square$
\end{proof}

\bigskip

The following Propostion~\eqref{lem:score1} is instrumental to derive rates of convergence for the system of estimated score equations $\widehat{\bU}_n$ and their derivatives.

%%%%%%%%%%%%%%%%%%%%%%%%%%%%%%%%%%%%%
\begin{prop}\label{lem:score1}
Let $X_i=(X_i(t):t\in[a,b])$, $i=1,...,n$ be i.i.d.~Gaussian processes with covariance function $\sigma(s,t)$ satisfying Assumption~\ref{assum1}. Let $\E(\varepsilon_i|X_i)=0$ with $\E(\varepsilon_i^p|X_i)\leq M_{\varepsilon} <\infty$ for some even $p$ with $p>\frac{2}{\kappa}$ and let $\widehat{\tau}_r$ enjoy the property given by \eqref{thm3eq1}, i.e. $|\widehat{\tau}_r - \tau_r| = O_P(n^{-\frac{1}{\kappa}})$.
We then have for any
%differentiable
bounded function $f: \mathbb{R} \to \mathbb{R}$ with $|f(x)|\leq M_f <\infty$,
%  with a bounded derivative $|f'(x)| \leq  M_{f} < \infty$,
any $t^* \in[a,b]$, any linear predictor $\eta_i^{*} = \beta_0^* + \sum_{r=1}^{S^*}\beta_r^* X_i(t_r^*)$, where $t_r^* \in [a,b]$, $\beta_r^*\in \mathbf{R}$ and $S^*$ are arbitrary
and any $r=1,\dots,S$:
\begin{align}
\frac{1}{n}\sum_{i=1}^n (X_i(\widehat{\tau}_r)-X_i(\tau_r))^2 & = O_P(n^{-1})\label{lem:res0}\\
\frac{1}{n}\sum_{i=1}^n (X_i(\widehat{\tau}_r)-X_i(\tau_r))f(\eta_i^*) &= O_P(n^{-\min\{ 1,\frac{1}{\kappa}\}}) \label{lem:res4a}\\
\frac{1}{n}\sum_{i=1}^n X_i(t^*)(X_i(\widehat{\tau}_r)-X_i(\tau_r))f(\eta_i^*) &= O_P(n^{-\min\{1,\frac{1}{\kappa}\}}) \label{lem:res3}\\
\frac{1}{n}\sum_{i=1}^n (X_i(\widehat{\tau}_r)-X_i(\tau_r))\varepsilon_i f(\eta_i^*) &= O_P(n^{-1}) \label{lem:res7} \\
\frac{1}{n}\sum_{i=1}^n X_i(t^*)(X_i(\widehat{\tau}_r)-X_i(\tau_r))\varepsilon_i f(\eta_i^*) &= O_P(n^{-1}) \label{lem:res6}\\
\frac{1}{n}\sum_{i=1}^n (X_i(\widehat{\tau}_r)-X_i(\tau_r))^4 &= O_P(n^{-2}) \label{lem:res8}
\end{align}
\end{prop}
\bigskip

The proof of this Proposition shares some arguments used in the proof of Lemma~\ref{lem1}.

\bigskip

\begin{proof}[{\bf Proof of Proposition \ref{lem:score1}.}]

Before the different assertions are proven, note that if $(X_1,X_2)$ are bivariate Gaussian, then $cov\left(X_2-\frac{cov(X_1,X_2)}{\V(X_1)}X_1, X_1\right)=0$ and $\left(X_2-\frac{cov(X_1,X_2)}{\V(X_1)}X_1\right)$ and $X_1$ are independent. Hence we additionally have $cov\left(X_2-\frac{cov(X_1,X_2)}{\V(X_1)}X_1, f(X_1)\right) = 0$ and by Lemma~\ref{lem:1aux}:
\begin{align*}
cov(f(X_1),X_2) = \frac{cov(f(X_1),X_1)}{Var(X_1)} cov(X_1,X_2).
\end{align*}
Furthermore, it follows from Stein's Lemma (\citeappendix{S1981}) that $\frac{cov(f(X_1),X_1)}{Var(X_1)} = \E(f'(X_1))$ provided $f$ is differentiable and $\E(|f'(X_1)|)<\infty$ (cf. Lemma 1 in \citeappendix{B2012a})

\bigskip

%We are now equipped with the tools to proof the different assertions of the proposition. Assertion \eqref{lem:res0} follows from Proposition 2 in \citeappendix{App_KPS2015}.
We are now equipped with the tools to proof the different assertions of the proposition. Assertion \eqref{lem:res0} and \eqref{lem:res8} are proven in Proposition~\ref{prop:1new}. The remaining assertions are proven in a similar manner: In order to proof Assertion \eqref{lem:res4a}, choose any $0<s$ sufficiently small and define for $q_1,q_2 \in [-1,1]$
\begin{align*}
\chi_i(q_1,q_2) & :=  (X_i(\tau_r+sq_1)-X_i(\tau_r))f(\eta_i^*) - (X_i(\tau_r+sq_2)-X_i(\tau_r))f(\eta_i^*) \\
&\qquad - \mathbb{E}\Big((X_i(\tau_r+sq_1)-X_i(\tau_r))f(\eta_i^*) - (X_i(\tau_r+sq_2)-X_i(\tau_r))f(\eta_i^*) \Big)\\
& = (X_i(\tau_r+sq_1) - X_i(\tau_r+sq_2))f(\eta_i^*) - \mathbb{E}((X_i(\tau_r+sq_1) - X_i(\tau_r+sq_2))f(\eta_i^*)).
\end{align*}
We then have $\mathbb{E}(\chi_i(q_1,q_2))=0$ and it follows from some straightforward calculations, since $|f(\eta_i^*)| \leq M_f$, that there exists a constant $L_{1}<\infty$ such that for $m=2,3,\dots$ we have
\begin{align}
\mathbb{E}&(|\frac{1}{s^\frac{\kappa}{2}}\chi_i(q_1,q_2)|^m) \leq \frac{m!}{2} (L_{1} |q_1-q_2|^\frac{\kappa}{2})^m.\label{eq:scpf1}
\end{align}
%Indeed: $|f(\eta_i^*)|\leq M_f$
%\mathbb{E}&(|\frac{1}{s^\frac{\kappa}{2}}\chi_i(q_1,q_2)|^m)\\
					%& \leq s^{-\frac{m\kappa}{2}} 2^m \mathbb{E}\Big(|(X_i(\tau_r+sq_1)-X_i(\tau_r+sq_2))f(\eta_i^*)|^m\Big)\\
					%& \leq s^{-\frac{m\kappa}{2}} 2^m \mathbb{E}\Big(|(X_i(\tau_r+sq_1)-X_i(\tau_r+sq_2))f(\eta_i^*)|^m\Big)\\
					%& \leq s^{-\frac{m\kappa}{2}} 2^m M_f^m \mathbb{E}\Big(|(X_i(\tau_r+sq_1)-X_i(\tau_r+sq_2))|^m\Big)\\
					%& \leq s^{-\frac{m\kappa}{2}} 2^m M_f^m \sigma_{(X_i(\tau_r+sq_1)-X_i(\tau_r+sq_2))}^m 2^\frac{m}{2}\frac{\Gamma(\frac{m+1}{2}}{2})}{\sqrt{\pi}}\\
					%& \leq s^{-\frac{m\kappa}{2}} 2^{m}  M_f^m 2^m \sigma_{(X_i(\tau_r+sq_1)-X_i(\tau_r+sq_2))}^m m!/2\\
					%& \leq s^{-\frac{m\kappa}{2}} 4^m M_f^m s^{\frac{m\kappa}{2}}  L_{1,2}^{m/2}|q_1-q_2|^\frac{m\kappa}{2} m!\\
					%& \leq 4^m M_f^m L_{1,2}^{m/2} m!/2 |q_1-q_2|^\frac{m\kappa}{2}\\
					%& \leq \frac{m!}{2} (L_{1,4} |q_1-q_2|^\frac{\kappa}{2})^m
%\end{align*}
%for some constant $0<L_{1,4}<\infty$.
%\bigbreak
%With $\Psi(x) = \exp\bigg(\frac{n}{6}(\sqrt{1+2\sqrt{\frac{6}{n}}x}-1)^2-1\bigg)$,
Corollary 1 in \citeappendix{vandeGeer2013} now guarantees that there exists a constant $0<L_{2}<\infty$ such that the Orlicz norm of $\frac{1}{\sqrt{n s^\kappa}}\sum_{i=1}^n (\chi_i(q_1,q_2))$ can be bounded, i.e., we have for some $0<L_{2}<\infty$:
\begin{align}\label{eq:sc4}
||\frac{1}{\sqrt{n s^\kappa}}\sum_{i=1}^n \chi_i(q_1,q_2)||_\Psi \leq L_{2}|q_1-q_2|^\frac{\kappa}{2}.
\end{align}
%indeed: L_{1,5} = \sqrt{6}L_{1,4}
By \eqref{eq:sc4} one may apply Theorem 2.2.4 of \citeappendix{vanderVaart1996}. The covering integral in this theorem can easily be seen to be finite and one can thus infer that there exists a constant $0<L_{3}<\infty$ such that
%%%%%%%%%%%%%%%%%%%
$$\mathbb{E}\Bigg( \exp\Big(\sup_{q_1,q_2\in [-1,1]} n/6 \big(\sqrt{1+2\sqrt{\frac{6}{n L_{3}^2}}|\frac{1}{\sqrt{n s^\kappa}}\sum_{i=1}^n\chi_i(q_1,q_2)|}-1 \big)^2 \Big) \Bigg)\leq 2.$$
%%%%%%%%%%%%%%%%%%%
For every $x>0$, the Markov inequality then yields
\begin{align*}
P&\Bigg(\sup_{q_1,q_2\in [-1,1]}| \frac{1}{\sqrt{n s^\kappa}}\sum_{i=1}^n\chi_i(q_1,q_2)| \geq x\frac{L_{3}}{2\sqrt{6}} \Bigg)\leq 2\exp\Big(-\frac{n}{6}(\sqrt{1+x/\sqrt{n}}-1 )^2\Big)\nonumber.
\end{align*}
Improving the readability, it then follows from a Taylor expansion of $\frac{n}{6}(\sqrt{1+x/\sqrt{n}}-1 )^2$ that we may conclude that there exists a constant $0<L_{4}<\infty$ such that for all $0<x\leq \sqrt{n}$ we have
\begin{align}\label{eq:sup1}
P&\Bigg(\sup_{q_1,q_2\in [-1,1]}| \frac{1}{\sqrt{n s^\kappa}}\sum_{i=1}^n\chi_i(q_1,q_2)| < L_{4}x\Bigg) \geq 1-2 \exp(-x^2).
\end{align}
%Indeed:
%\begin{align*}
%P&\Bigg(\sup_{q_1,q_2\in [-1,1]}| \frac{1}{\sqrt{n s^\kappa}}\sum_{i=1}^n\chi_i(q_1,q_2)| \geq x\frac{L_{1,7}}{2\sqrt{6}} \Bigg)\\
%& = P\Bigg(\exp\Big(\sup_{q_1,q_2} \frac{n}{6}(\sqrt{1+ 2\sqrt{\frac{6}{L_{1,7}^2}n}| \frac{1}{\sqrt{ns^\kappa}}\sum_{i=1}^n\chi_i(q_1,q_2)|}-1)^2\Big) \geq \exp\Big(\frac{n}{6}(\sqrt{1+x/\sqrt{n}}-1 )^2\Big)\Bigg)\nonumber\\
%&\leq 2\exp\Big(-\frac{n}{6}(\sqrt{1+x/\sqrt{n}}-1 )^2\Big)\nonumber\\
%&\leq 2\exp\Big(-D_7x^2\Big) \quad \text{(for all $0<x\leq \sqrt{n}$)}\nonumber
%\end{align*}

%Now observe that
%\begin{align*}
%\sup_{q_1,q_2\in [-1,1]} &| \frac{1}{\sqrt{n s^\kappa}}\sum_{i=1}^n\chi_i(q_1,q_2)|\\
%& \geq \sup_{q_1\in [-1,1], q_2=0}| \frac{1}{\sqrt{n s^\kappa}}\sum_{i=1}^n\chi_i(q_1,0)|\\
%& \geq \sup_{q_1\in [-1,1]} |\frac{1}{\sqrt{n s^\kappa}}\sum_{i=1}^n (X_i(\tau_r+sq_1)-X_i(\tau_r))f(\xi_i(u))| \\
%&\qquad -  \sup_{q_1}|\frac{\sqrt{n}}{\sqrt{s^\kappa}} \mathbb{E}((X_i(\tau_r+sq_1)-X_i(\tau_r))f(\xi_i(u)))|
%\end{align*}

\noindent Now, note that it follows from Lemma~\ref{lem:1aux} %the proof of Lemma \ref{lem:1}
that there exists a constant $|c_0|<\infty$, not depending on $t^*$, such that $\mathbb{E}(X(t^*)f(\eta_i^*))=c_0\E(X(t^*)\eta_i^*)$ for all $t^* \in [a,b]$. %\citeappendix{App_S1981}    %there exists a constant $c_{0}$ such that $|\E((X_i(\tau_r + sq_1)-X_i(\tau_r))f(\eta_i^*))| \leq c_0 \textcolor{red}{together with Stein's Lemma ACHTUNG HIER MUSS MAN NOCHMAL VERWEISEN, DASS DIE KONSTANTE DURCH $\E(f')$ dargestellt werden kann IST ABER HIER NICHT NÖTIGT, DIE KONSTANTE IST AUCH DURCH CS UNGLEICHUNG BESCHRAENKT... benoetigen nur boundedness},
Together with \eqref{eq:diffexpect}% \eqref{eq:evalue1}
we can therefore conclude that there exists a constant $0 \leq L_{5}<\infty$ such that for all $q_1\in[-1,1]$:
\begin{align}
|\mathbb{E}((X_i(\tau_r + sq_1)-X_i(\tau_r))f(\eta_i^*))| %= |\sum_{m=2}^{S+1}\beta_{m-1}\mathbb{E}(f'(\eta_i^*))\mathbb{E}(X_i(t_m^*)(X_i(\tau_r + sq_1)-X_i(\tau_r)))|
\leq L_{5}s^{\min\{1,\kappa\}}\label{eq:scpf2}
\end{align}

Using \eqref{eq:scpf2} together with \eqref{eq:sup1} we can conclude that for all $0<x\leq \sqrt{n}$ we have:
\begin{align}
\begin{split}
P\Bigg(\sup_{\tau_r-s\leq  u_r \leq \tau_r +s} |\frac{1}{n} \sum_{i=1}^n (X_i(u_r) - X_i(\tau_r))f(\eta_i^*)| < L_{5}s^{\min\{1,\kappa\}} & + L_{4}\frac{s^\frac{\kappa}{2}}{\sqrt{n}}x\Bigg)\\
& \geq 1-2\exp(-x^2)
\end{split}\label{eq:tatatat}
\end{align}
Assertion \eqref{lem:res4a} then follows immediately from \eqref{thm3eq1}.\\
%scheint alles okay zu sein bis hierhin.

%%%%%%%%%%%%%%%%%%%%%%%%%%%%%%%%%%%%%%%%%%%%%%%%%%%%%%%%%%%%%%%%%%%%%%%
By the boundedness of $f$, the proof of \eqref{lem:res3} proceeds similar, but one now has to bound
$$|\mathbb{E}(X_i(t^*)(X_i(\tau_r + sq_1) - X_i(\tau_r))f(\eta_i^*))|.$$
\noindent For $X_i(t^*)= \eta_i^*$, Lemma~\ref{lem:1aux} %Lemma \ref{lem:1}
together with \eqref{eq:diffexpect} %\eqref{eq:evalue1}
already implies that there exists a constant $L_{6}<\infty$ such that  $|\mathbb{E}(X_i(t^*)(X_i(\tau_r + sq_1) - X_i(\tau_r))f(\eta_i^*))| \leq L_{6} s^{\min\{1,\kappa\}}$. Let $X_i(t^*) \neq \eta_i^*$. Note that $(X_i(t^{*}),(X_i(\tau_r + sq_1) - X_i(\tau_r)), \eta_i^*)$ are multivariate normal. Hence also the conditional distribution of $((X_i(\tau_r + sq_1) - X_i(\tau_r)), \eta_i^*)$ given $X_i(t^*)$ is multivariate normal. To ease the notation set $X_1= \eta_i^*$, $X_2 = (X_i(\tau_r + sq_1) - X_i(\tau_r))$ and $X_3=X_i(t^*)$ and define by $\sigma_{i,j}$, $i,j \in \{1,2,3\}$ their associated covariance and variances.
%We then have:
%$$((X_1, X_2)| X_3) \sim N( \overline{\mu}, \overline{\Sigma})$$
%where
%\begin{align*}
%\overline{\mu} &= \begin{pmatrix} a  + \frac{\sigma_{13}}{\sigma_{33}}X_3\\
																	     %\frac{\sigma_{23}}{\sigma_{33}}X_3
									%\end{pmatrix}
%\intertext{as well as}
%\overline{\Sigma}& = \begin{pmatrix}
			%\sigma_{11}-\frac{\sigma_{13}^2}{\sigma_{33}} & \sigma_{12} - \frac{\sigma_{13}\sigma_{23}}{\sigma_{33}}\\
			%\sigma_{12} - \frac{\sigma_{13}\sigma_{23}}{\sigma_{33}}& \sigma_{22}-\frac{\sigma_{23}^2}{\sigma_{33}}
%\end{pmatrix}
%\end{align*}
We then have by conditional expectation together with an application of Lemma~\ref{lem:1aux} %Lemma \ref{lem:1} (c.f \cite[Lemma 1]{B2012a})
%%, \textcolor{red}{ACHTUNG hier braucht es ein wenig mehr, da wir die covariancen brauchen und nicht mehr $E(XY)$...Lemma 1 aber für zentrierte variablen geschrieben ist...eventuell hier auf billinger verweisen?!} %together with Stein's Lemma \textcolor{red}{HIER NOCHMAL DIE REFERENZ VON UNTER DEM BEWEIS VON LEMMA 1}:
%i.e. \newpage
%\begin{align*}cov(f(X_1),X_2) &= \mathbb{E}(f'(X_1))cov(X_1,X_2) = \frac{cov(f(X_1),X_1)}{Var(X_1)}cov(X_1,X_2) \\
%&= \mathbb{E}(f(X_1)X_2)-\mathbb{E}(f(X_1))\mathbb{E}(X_2).
%\end{align*}
%
\begin{align*}
|\mathbb{E}(X_i(t^*)&(X_i(\tau_r + sq_1)  - X_i(\tau_r))f(\eta_i^*))| =|\mathbb{E}(f(X_1)X_2X_3)|  = |\mathbb{E}(X_3 \mathbb{E}(f(X_1)X_2|X_3))| \\
%& = |\mathbb{E}(X_3 \mathbb{E}(f(X_1)X_2|X_3))| \\
%& = \mathbb{E}(X_3 (cov(f(X_1),X_2|X_3)+\mathbb{E}(f(X_1)|X_3)\mathbb{E}(X_2|X_3)))\\
%& = \mathbb{E}(X_3( \mathbb{E}(f'(X_1)|X_3)cov(X_1,X_2|X_3) +  \mathbb{E}(f(X_1)|X_3)\mathbb{E}(X_2|X_3)))\\
%& = \mathbb{E}(X_3\mathbb{E}(f'(X_1)|X_3)cov(X_1,X_2|X_3)) \\
%&\qquad + \mathbb{E}(X_3 \mathbb{E}(f(X_1)|X_3)\mathbb{E}(X_2|X_3) )\\
%&= \mathbb{E}(\mathbb{E}(f'(X_1)|X_3) X_3 (\sigma_{12}-\frac{\sigma_{13}\sigma_{23}}{\sigma_{33}}))\\
%&\qquad + \mathbb{E}(X_3 \frac{\sigma_{23}}{\sigma_{33}} X_3 \mathbb{E}(f(X_1)|X_3))\\
%&= |(\sigma_{12}-\frac{\sigma_{13}\sigma_{23}}{\sigma_{33}})\mathbb{E}(\mathbb{E}(f'(X_1)|X_3) X_3) \\  %%%%%%%%%HIER REINKOMMENTIEREN, WENN MAN AUCH NOCH f' bounded annehmen würde!!!!
%&\qquad +\frac{\sigma_{23}}{\sigma_{33}} \mathbb{E}(X_3^2  \mathbb{E}(f(X_1)|X_3))|\\                   %%%%%%%%%% dafür unten das weg...
%&= |(\sigma_{12}-\frac{\sigma_{13}\sigma_{23}}{\sigma_{33}})\mathbb{E}(\mathbb{E}(f'(X_1)|X_3) X_3)   %%%%%%%%%HIER REINKOMMENTIEREN, WENN MAN AUCH NOCH f' bounded annehmen würde!!!!
% +\frac{\sigma_{23}}{\sigma_{33}} \mathbb{E}(X_3^2  \mathbb{E}(f(X_1)|X_3))|.                   %%%%%%%%%% dafür unten das weg...
%%%%%%%%%%%%%%%%%%%%%%%%%%%%UPDATET, WENN WIR NICHT |f'|<\infty brauchen
 &= |(\sigma_{12}-\frac{\sigma_{13}\sigma_{23}}{\sigma_{33}})\mathbb{E}(\frac{cov(f(X_1),X_1|X_3)}{\bV(X_1|X_3)} X_3)
 +\frac{\sigma_{23}}{\sigma_{33}} \mathbb{E}(X_3^2  \mathbb{E}(f(X_1)|X_3))|.
\end{align*}
Using \eqref{eq:diffexpect}%\eqref{eq:evalue1}
 it is then easy to see that there exists a constant $0<L_{7}<\infty$ such that $|\sigma_{12}|\leq L_{7}s^{\min\{1,\kappa\}}$, as well as $|\sigma_{23}| \leq L_{7}s^{\min\{1,\kappa\}}$. On the other hand Assumption \ref{assum1} implies that there exists a constant $0<L_{8}<\infty$ such that $|\sigma_{13}|\leq L_{8}$. % Since $f$ and $f'$ are bounded it then follows immediately that there exists a constant $0<L_{1,13}<\infty$ such that for all $u \in [0,1]$, $t^* \in [a,b]$ and all $r=1,\dots, S$
Note that $cov(f(X_1),X_1|X_3) = \E(f(X_1)X_1|X_3) - \E(f(X_1)|X_3)\E(X_1|X_3)$ and  $\bV(X_1|X_3)=\sigma_{11}-\frac{\sigma_{13}^2}{\sigma_{33}}>0.$
%and $\bV(X_1|X_3)=\sigma_{11}-\frac{\sigma_{13}^2}{\sigma_{33}}$
Moreover note that if $f$ is assumed to be differentiable and $\E(|f'(X_1)||X_3)<\infty$, it follows from and Stein's Lemma (\citeappendix{S1981}) that $cov(f(X_1),X_1|X_3)/\bV(X_1|X_3)$ can be substituted by $\E(f'(X_1)|X_3)$.

\noindent Since $f$ is bounded
%and $f'$ are bounded
it then follows immediately that for all linear predictors $\eta_i^*$ and all $t^* \in [a,b]$ there exists a constant $0<L_{9}<\infty$ such that for all $q_1 \in[-1,1]$ and all sufficiently small $s$ and all $r=1,\dots, S$ we have:
\begin{align}\label{eq:tada0}
|\mathbb{E}(X_i(t^{*})(X_i(\tau_r + sq_1) - X_i(\tau_r))f(\eta_i^*))| \leq L_{9} s^{\min\{1,\kappa\}}.
\end{align}
By \eqref{eq:tada0} one can conclude similar to \eqref{eq:tatatat} that for all $0<x\leq \sqrt{n}$ and for some constant $0< L_{10}<\infty$
\begin{align*}
P\Bigg(\sup_{\tau_r-s\leq  u \leq \tau_r +s} |\frac{1}{n} \sum_{i=1}^n X_i(t^*)(X_i(u) - X_i(\tau_r))f(\eta_i^*)| < L_{9}s^{\min\{1,\kappa\}} &+ L_{10}\frac{s^\frac{\kappa}{2}}{\sqrt{n}}x\Bigg)\\
&\geq 1-2\exp(-x^2).
\end{align*}
Assertion \eqref{lem:res3} then follows again immediately from \eqref{thm3eq1}.\\
%%%%%%%%%%%%%%%%%%%%%%%%%%%%%%%%%%%%%%%%%%%%%%%%%%%%%%%%%%%%%%%%%%%%%%%%%%%%%%%%%%%%%%%%%%
%%%%%%%%%%%%%%%%%%%%%%%%%%%%%%%%%%%%%%%%%%%%%%%%%%%%%%%%%%%%%%%%%%%%%%%%%%%%%%%%%%%%%%%%%%
%%%%%%%%%%%%%%%%%%%%%%%%%%%%%%%%%%%%%%%%%%%%%%%%%%%%%%%%%%%%%%%%%%%%%%%%%%%%%%%%%%%%%%%%%%

In order to show assertion $\eqref{lem:res7}$ we make use the Orlicz-norm $||X||_p$.\\
%https://math.stackexchange.com/questions/2480236/convergence-of-higher-absolute-moments-given-convergence-in-distribution
Choose some  $p>\frac{2}{\kappa}=p_{\kappa}$, and let $p$ be even. Note that $\E((X_i(\tau_r+sq_1)-X_i(\tau_r+sq_2))\varepsilon_i f(\eta_i^*)) = 0$. For all sufficiently small $0<s$ and all $q_1, q_2 \in [-1,1]$ it is easy to show that there exists a constant $L_{11}<\infty$ such that
$$\E(|s^{-\kappa/2} \frac{1}{\sqrt{n}}\sum_{i=1}^{n}(X_i(\tau_r+sq_1)-X_i(\tau_r+sq_2))\varepsilon_i f(\eta_i^*)|^p) \leq L_{11}^p |q_1-q_2|^{\frac{p \kappa}{2}}.
$$
%Idee: man addiert viele unkorrelierte terme und wenig terme die etwas bedeuten. von den bedeutungsvollen stellt größte anzahl (insgesamt n(n-1)...(n-p/2)-1) möglichkeiten) das produkt der varianzen $E(y_i^2)$, von $p/2$ terme....
%Mathematical Statistics for Economics and Business, p.271
%https://math.stackexchange.com/questions/2480236/convergence-of-higher-absolute-moments-given-convergence-in-distribution/2480252?noredirect=1#comment5124333_2480252
%(proof habe ich hier vorgekaut bekommen...: der trick ist $$(X_i(\tau_r+sq_1)-X_i(\tau_r+sq_2) = \sigma_{q_1,q_2} (X_i(\tau_r+sq_1)-X_i(\tau_r+sq_2)/\sigma_{q_1,q_2}$$ zu schreiben, letzter ist $N(0,1)$ verteilt. damit kommt dann schonmal die standardabweichung raus. alle momente gebounded, daher gilt dann auch: (achtung, hier ist noch nicht so richtig standardisisiert...aber CS  https://math.stackexchange.com/questions/2480236/convergence-of-higher-absolute-moments-given-convergence-in-distribution)
We may conclude
\begin{align}
||s^{-\kappa/2} \frac{1}{\sqrt{n}}\sum_{i=1}^{n}(X_i(\tau_r+sq_1)-X_i(\tau_r+sq_2))\varepsilon_i f(\eta_i^*)||_p\leq L_{11}|q_1-q_2|^{\frac{\kappa}{2}}.\label{eq:on1}
\end{align}
By assertion \eqref{eq:on1} one may apply Theorem 2.2.4 in \citeappendix{vanderVaart1996}. Our condition on $p$ ensures that the covering integral appearing in this theorem is finite.
%%%%%%%%
%Indeed
%For $x\geq 0$, the inverse of $\Phi(x) = |x|^p=y$ is $\Phi^{-1}(x) =y^{\frac{1}{p}}$ and $[a,b]$ can be covered (w.r.t to our semimetric) with $O(\epsilon^{-\frac{1}{\alpha}})$ balls of radius $\frac{1}{2}\epsilon$\\
%A $\epsilon$-Ball around $t$ w.r.t to the semimetric is $[t-\epsilon^{1/\alpha}, t+\epsilon^{1/\alpha}]$...
%But:
%$$
%\Phi^{-1}(\epsilon^{-\frac{1}{\alpha}})= \epsilon^{- \frac{1}{\alpha p}} = \epsilon^{- \frac{2}{\kappa p}}
%$$
%and
%$$\int_{0}^{b-a} \Phi^{-1}(\epsilon^{-\frac{1}{\alpha}}) \, d\epsilon = \int_{0}^{b-a}  \epsilon^{- \frac{2}{\kappa p}} <\infty$$
%if and only if
%\begin{align*}
%\frac{2}{\kappa p} &<1 \\
%\Leftrightarrow  \quad \frac{2}{\kappa} &< p\\
%\end{align*}
%which holds by assumption. whenever $p> 1\kappa +2$
%%%%%
The maximum inequalities of empirical processes then imply:
$$||\sup_{q_1,q_2\in [-1,1]}|s^{-\kappa/2} \frac{1}{\sqrt{n}}\sum_{i=1}^{n}(X_i(\tau_r+sq_1)-X_i(\tau_r+sq_2))\varepsilon_i f(\eta_i^*)|||_{\Psi_p} \leq L_{12}$$
for some constant $L_{12}<\infty$. At the same time, the Markov inequality implies
\begin{align*}
P\Bigg(&\sup_{\tau_r-s\leq u \leq \tau_r + s} |\frac{1}{n}\sum_{i=1}^n(X_i(u)- X_{i}(\tau_r))\varepsilon_i f(\eta_i^*)|>s^{\kappa/2}\frac{x}{\sqrt{n}}\Bigg) \\
%&= P(\sup_{\tau_r-s\leq u \leq \tau_r + s} |s^{-\frac{\kappa}{2}}\frac{1}{\sqrt{n}}\sum_{i=1}^n X_i(t^*) (X_i(u)- X_{i}(\tau_r))\varepsilon_i f(\eta_i^*)|> x)\\
%&\leq P(\sup_{q_1,q_2} |s^{-\frac{\kappa}{2}}\frac{1}{\sqrt{n}}\sum_{i=1}^n X_i(t^*) (X_i(\tau_r +sq_1)- X_{i}(\tau_r+sq_2))\varepsilon_i f(\eta_i^*)| >x)\\
&\leq P\Bigg( |\sup_{q_1,q_2\in [-1,1]} |s^{-\frac{\kappa}{2}}\frac{1}{\sqrt{n}}\sum_{i=1}^n (X_i(\tau_r +sq_1)- X_{i}(\tau_r+sq_2))\varepsilon_i f(\eta_i^*)||^p > x^p\Bigg)
\leq \frac{L_{12}^p}{x^p}.
\end{align*}
Assertion \eqref{lem:res7} then follows from \eqref{thm3eq1} and our conditions on $p$. Moreover, assertion \eqref{lem:res6} follows from similar steps.

\end{proof}

\bigskip

%%%%%%%%%%%%%%%%%%%%%%%%%%%%%%%%%%%%%%%%%%%%%%%%%%%%%%%%%%%%%%%%%%%%%%%
%%%%%%%%%%%%%%%%%%%%%%%%%%%%%%%%%%%%%%%%%%%%%%%%%%%%%%%%%%%%%%%%%%%%%%%
%%%%%%%%%%%%%%%%%%%%%%%%%%%%%%%%%%%%%%%%%%%%%%%%%%%%%%%%%%%%%%%%%%%%%%%

\bigskip

The rest of the Appendix makes use of the notation introduced in Section \ref{sec:PES}.  In what follows, it will, however, be convenient to set $\bX_i:=\bX_i(\btau) = (1, X_i(\tau_1),\dots, X_i(\tau_S)$ and $\widehat{\bX}_i := \bX_i(\widehat{\btau}) = (1, X_i(\widehat{\tau}_1),\dots, X_i(\widehat{\tau}_S)$. The $j$th element of $\bX_i$  and $\widehat{\bX}_i$ is then denoted by $X_{ij}$ and $\widehat{X}_{ij}$, respectively.

For the following proofs we introduce some additional notation. Let {\small$h(x)=g'(x)/\sigma^2(g(x))$} and note that differentiating the estimation equation $$\frac{1}{n}\widehat{\bU}_n(\bbeta)= \frac{1}{n}\widehat{\bD}_n(\bbeta)\widehat{\bV}_n^{-1}(\bbeta)(y-\bmu(\bbeta)) = \frac{1}{n}\sum_{i=1}^n h(\widehat{\eta}_i(\bbeta))\widehat{\bX}_i(y_i-g(\widehat{\eta}_i(\bbeta)))$$ leads to
\begin{align*}
\frac{1}{n}\widehat{\bH}(\bbeta) &
=   \frac{1}{n} \frac{\partial\widehat{\bU}_n(\bbeta)}{\partial \bbeta}
= - \frac{1}{n} \widehat{\bD}_n(\bbeta)^T\widehat{\bV}_n(\bbeta)^{-1}\widehat{\bD}_n(\bbeta)  + \frac{1}{n}\sum_{i=1}^n h'(\widehat{\eta}_i) \widehat{\bX}_i \widehat{\bX}_i^T(y_i-g(\widehat{\eta}_i(\bbeta)))\\
&= -\frac{1}{n}\widehat{\bF}_n(\bbeta) + \frac{1}{n}\widehat{\bR}_n(\bbeta) \quad \text{say.}
\end{align*}
Similarly, one obtains by replacing the estimates $\widehat{\tau}_r$ with their true counterparts $\tau_r$:
\begin{align*}
\frac{1}{n}\bH(\bbeta) = \frac{1}{n} \frac{\partial\bU_n(\bbeta)}{\partial \bbeta} = - \frac{1}{n}\bF_n(\bbeta) + \frac{1}{n}\bR_n(\bbeta),
\end{align*}
where
$$\frac{1}{n}\bF_n(\bbeta) = \frac{1}{n} \bD_n(\bbeta)^T\bV_n(\bbeta)^{-1}\bD_n(\bbeta),$$
and
$$\frac{1}{n}\bR_n(\bbeta)=\frac{1}{n}\sum_{i=1}^n h'(\eta_i(\bbeta)) \bX_i \bX_i^T(y_i-g(\eta_i(\bbeta))).$$

Now, let $\widehat{\eta}(\bbeta)$, $\widehat{\bX}$ and $y$ be generic copies of $\widehat{\eta}_i(\bbeta)$, $\widehat{\bX}_i$ and $y_i$. We then have
$$\E(\frac{1}{n}\widehat{\bF}_n(\bbeta)) = \E(\frac{g'(\widehat{\eta}(\bbeta))^2}{\sigma^2(g(\widehat{\eta}(\bbeta))} \widehat{\bX} \widehat{\bX}^T) =: \E(\widehat{\bF}(\bbeta)),$$
as well as
$$\frac{1}{n}\widehat{\bR}_n(\bbeta) = \E(h'(\widehat{\eta}(\bbeta)) \widehat{\bX} \widehat{\bX}^T(y-g(\widehat{\eta}(\bbeta)))) =: \E(\widehat{\bR}(\bbeta)).$$
In a similar manner $\E(\bF(\bbeta))=\E(n^{-1}\bF_n(\bbeta))$ and $\E(\bR(\bbeta))=\E(n^{-1}\bR_n(\bbeta))$ are defined.\\

%\noindent Moreover, remember that we write $\bX_i(\widehat{\btau}) = (1, X_i(\widehat{\tau}_1), \dots, X_i(\widehat{\tau}_S))^T$ and denote the $j$th element of $\bX_i(\widehat{\btau})$ as $\widehat{X}_{ij}$, while $X_{ij}$ denotes the $j$th element of $\bX_i(\btau) = (1, X_i(\tau_1), \dots, X_i(\tau_S))^T$.

\bigskip

The next proposition is crucial, as it tells us that the estimated score function and its derivative are sufficiently close to each other. Of particular importance are the facts that $$\frac{1}{n}\widehat{\bU}_n(\bbeta_0) = \frac{1}{n}\bU_n(\bbeta_0) + o_P(n^{-\frac{1}{2}}),$$
and
$$\frac{1}{n}\widehat{\bF}_n(\bbeta_0) = \frac{1}{n}\bF_n(\bbeta_0) + O_P(n^{-\frac{1}{2}}),$$
which follow from this proposition.
%%%%%%%%%%%%%%%%%%%%%%%%%%%%%%%%%%%%%%%%%%%%%%%
%%%%%%%%%%%%%%%%%%%%%%%%%%%%%%%%%%%%%%%%%%%%%%%
%%%%%%%%%%%%%%%%%%%%%%%%%%%%%%%%%%%%%%%%%%%%%%%
%%%%%%%%%%%%%%%%%%%%%%%%%%%%%%%%%%%%%%%%%%%%%%%

\begin{prop}\label{th:paraestQML}
Let $X_i=(X_i(t):t\in[a,b])$, $i=1,...,n$ be i.i.d.~Gaussian processes. Under Assumption \ref{assum3} and under the results of Proposition~\ref{lem:score1}
%\textcolor{red}{hier fehlen noch annahmen. $X$ muss gaussian sein und verschiedene funktionen $h_1(x)$ und $h(x)$ gebounded und auch eine geboundete ableitung hat.}
we have
\begin{align}
\frac{1}{n}\widehat{\bU}_n(\bbeta_0) &= \frac{1}{n}\bU_n(\bbeta_0) + O_P(n^{-\min\{1,1/\kappa\}}).\label{eq:thQML1}
%%%%%%%%%%%%
%\frac{1}{n}\widehat{\bD}_n(\bbeta)^T\widehat{\bV}_n(\bbeta)^{-1}\widehat{\bD}_n(\bbeta)  &= \frac{1}{n}\bD_n(\bbeta)^T\bV_n(\bbeta)^{-1}\bD_n(\bbeta)  + O_P(n^{-1/2})\label{eq:thQML2} \\
 %&= \E(\mathbf{F}(\bbeta)) + O_P(n^{-1/2}) \label{eq:thQML2}
%\frac{1}{n}\widehat{\bD}_n(\bbeta)^T\widehat{\bV}_n(\bbeta)^{-1}\widehat{\bD}_n(\bbeta)  &= \frac{1}{n}\bD_n(\bbeta)^T\bV_n(\bbeta)^{-1}\bD_n(\bbeta)  + O_P(n^{-1/2})\label{eq:thQML2}
\end{align}
Additionally, for all $\bbeta \in \mathbf{R}^{S+1}$:
\begin{align}
\frac{1}{n}\widehat{\bU}_n(\bbeta) &= \frac{1}{n}\bU_n(\bbeta) + O_P(n^{-\frac{1}{2}}),\label{eq:thQML1a}\\
\frac{1}{n}\widehat{\bF}_n(\bbeta) &= \frac{1}{n}\bF_n(\bbeta) + O_P(n^{-1/2}),\label{eq:thQML2} \\
%\end{align}
%Moreover, if additionally $|h'''(x)|\leq M_h<\infty$ then for all $\bbeta \in \mathbf{R}^{S+1}$:
%\begin{align}
\frac{1}{n}\widehat{\bR}_n(\bbeta) &= \frac{1}{n}\bR_n(\bbeta) + O_P(n^{-\frac{1}{2}}) \label{eq:thQML3}.
\end{align}
Moreover, we have as $n\to \infty$
\begin{align}
\E(\frac{1}{n}\widehat{\bU}_n(\bbeta))  \to \E(\frac{1}{n}\bU_n(\bbeta))  \label{eq:thQML3mean0},\\
\E(\frac{1}{n}\widehat{\bF}_n(\bbeta)) \to \E(\frac{1}{n}\bF_n(\bbeta))  \label{eq:thQML3mean1}, \\
\E(\frac{1}{n}\widehat{\bR}_n(\bbeta)) \to \E(\frac{1}{n}\bR_n(\bbeta))  \label{eq:thQML3mean2}.
\end{align}
Particularly,
\begin{align}
\E(\frac{1}{n}\widehat{\bR}_n(\bbeta_0)) \to 0 \label{eq:thQML3mean}\\
 \intertext{and}
\E(\frac{1}{n}\widehat{\bU}_n(\bbeta_0)) \to 0 \label{eq:thQML3meanaa}.
\end{align}
\end{prop}

\bigskip

\begin{proof}[{\bf Proof of Proposition \ref{th:paraestQML}.}]
To ease notation we use $\bbeta_0 = (\beta_0^{(0)},\beta_1^{(0)}, \dots,\beta_S^{(0)})^T$ to denote the true parameter vector. For instance, the intercept is given by $\beta_0^{(0)}$, while $\beta_r^{(0)}$ is the coefficient for the $r$th point of impact. Similar we denote the entries of $\bbeta$ by $(\beta_0,\dots, \beta_S)$. Write
\begin{align}
\frac{1}{n}\widehat{\bU}_n(\bbeta)  = \frac{1}{n}\bU_n(\bbeta) +\mathbf{Rest}_n(\bbeta),\label{eq:QMLSTART0}
\end{align}
then $\mathbf{Rest}_n(\bbeta)$ can be decomposed into two parts:
{\small\begin{align}
\mathbf{Rest}_n(\bbeta)% &= \frac{1}{n}\widehat{\bD}_n^T\widehat{\bV}_n^{-1}(\bY_n-(\bmu_n+(\widehat{\bmu}_n-\bmu_n)))-\frac{1}{n}\bD_n^T\bV_n^{-1}(\bY_n-\bmu_n) \nonumber \\
		 &= \frac{1}{n}(\widehat{\bD}_n^T(\bbeta)\widehat{\bV}_n^{-1}(\bbeta)-\bD_n^T(\bbeta)\bV_n^{-1}(\bbeta))(\bY-\bmu(\bbeta)) - \frac{1}{n}\widehat{\bD}_n^T(\bbeta)\widehat{\bV}_n^{-1}(\bbeta)(\widehat{\bmu}_n(\bbeta)-\bmu_n(\bbeta)) \nonumber \\
			&= \mathbf{Rest}_{1}(\bbeta) + \mathbf{Rest}_{2}(\bbeta), \quad \text{say.} \label{eq:QMLRest1}
\end{align}}
The first summand $\mathbf{Rest}_{1}(\bbeta)$ is given by:
$$\mathbf{Rest}_{1}(\bbeta)=\frac{1}{n}(\widehat{\bD}_n^T(\bbeta)\widehat{\bV}_n^{-1}(\bbeta)-\bD_n^T(\bbeta)\bV_n^{-1}(\bbeta))(\bY_n-\bmu_n(\bbeta)).$$
The $j$th equation of $\mathbf{Rest}_{1}(\bbeta)$ can be written as
\begin{align}
Rest_{j,1}(\bbeta) &= \frac{1}{n}\sum_{i=1}^n (\frac{g'(\widehat{\eta}_i(\bbeta))}{\sigma^2(g(\widehat{\eta}_i(\bbeta)))} \widehat{X}_{ij}
-    \frac{g'(\eta_i(\bbeta))}{\sigma^2(g(\eta_i(\bbeta)))} X_{ij})(y_i-g(\eta_i(\bbeta))) \nonumber \\
				&=\frac{1}{n}\sum_{i=1}^n X_{ij}(\frac{g'(\widehat{\eta}_i(\bbeta))}{\sigma^2(g(\widehat{\eta}_i(\bbeta)))} -  \frac{g'(\eta_i(\bbeta))}{\sigma^2(g(\eta_i(\bbeta)))}) (y_i-g(\eta_i(\bbeta))) \nonumber \\
				&\quad + \frac{1}{n}\sum_{i=1}^n (\widehat{X}_{ij}-X_{ij})(\frac{g'(\widehat{\eta}_i(\bbeta))}{\sigma^2(g(\widehat{\eta}_i(\bbeta)))}) (y_i-g(\eta_i(\bbeta))) \nonumber \\
				&= R_{j,1,a}(\bbeta) + R_{j,1,b}(\bbeta), \quad \text{say.}  \label{eq:thasdda}
\end{align}

%
%\bigskip
%\textcolor{red}{We assume that $h(x) = \frac{g'(\widehat{x})}{\sigma^2(g(\widehat{x}))}$ is bounded with a bounded derivative. i.e. $|h(x)| \leq M_h $ and $|h'(x)|\leq M_{h'}$.\\
%A sufficient condition is $|g'(x)|<\infty$, $|g''(x)|<\infty$ and $\sigma^2(g(x))>M_{\sigma}$}
%\bigskip

%It follows from the mean value theorem that for some $c_i$ between $\widehat{\eta}_i$ and $\eta_i$ we have
%$$h(\widehat{\eta}_i) - h(\eta_i) = (\widehat{\eta}_i- \eta_i) h'(c_i) = h'(c_i) \sum_{r=1}^{S} \beta_r (\widehat{X}_{ir}-X_{ir}) $$
With $h(x) = g'(x)/\sigma^2(g(x))$, a Taylor expansion implies the existence of some some $\xi_{i,1}$ between $\widehat{\eta}_i(\bbeta)$ and $\eta_i(\bbeta)$ such that for $\bbeta = \bbeta_0$
\begin{align*}
R_{j,1,a}(\bbeta_0) & = \frac{1}{n}\sum_{i=1}^n X_{ij}(\frac{g'(\widehat{\eta}_i(\bbeta_0))}{\sigma^2(g(\widehat{\eta}_i(\bbeta_0)))} -  \frac{g'(\eta_i(\bbeta_0))}{\sigma^2(g(\eta_i(\bbeta_0)))}) (y_i-g(\eta_i(\bbeta_0)))\\
%&= \frac{1}{n}\sum_{i=1}^n X_{ij}(h(\widehat{\eta}_i) - h(\eta_i)) \varepsilon_i\\
&= \sum_{r=2}^{S+1} \beta_{r-1}^{(0)} \frac{1}{n}\sum_{i=1}^n X_{ij}(\widehat{X}_{ir}-X_{ir})\varepsilon_i h'(\eta_i(\bbeta_0)) \\
&\quad  + \frac{1}{n}\sum_{i=1}^n X_{ij}\varepsilon_i h''(\xi_{i,1})/2 (\sum_{l=2}^{S+1} \beta_{l-1}^{(0)} (X_{il}-\widehat{X}_{il}))^2.
\end{align*}
%\textcolor{red}{ACHTUNG, HIER GIBT ES am 7.11.2017 eine aenderung. $h''$ wird gebraucht! vorher war taylor expansion erster ordnung, aber da ist unklar, ob der e-wert dieser taylorexpansion $0$ ist...der beweis ist in der hinsicht konservativ! man koennte annehmen, dass $\varepsilon_i$ unabhängig von $X_i(t)$ ist, dann waere immernoch alles in ordnung OHNE $h''$ zu benutzen. dies schließt jedoch die logistische regression aus, allerdings gilt dort eh $h' = 0$, so dass $R_{j,1,a}(\bbeta_0) = 0$. ANMERKUNG MACHEN UNTER ANNAHMEN!!!}
Since $|h'(\cdot)| \leq M_h$ and $|h''(\cdot)| \leq M_h$, $R_{j,1,a}(\bbeta_0) = O_P(n^{-1})$ for $j=1,\dots,S+1$ follows immediately from \eqref{lem:res7}  and  \eqref{lem:res6}
together with the Cauchy-Schwarz inequalitiy and \eqref{lem:res8}.
%\textcolor{red}{ACHTUNG!!! HIER KOENNTE EIN FEHLER SEIN: IST $\E(X_{ij}\varepsilon_i h''(\xi_{i,1})/2(X_{il}-\widehat{X}_{il})^2) = 0$? Sonst muss man eventuell annehmen: entweder $h'' = 0$ ODER $\varepsilon_i$ UNABHAENGIG VON $X_i(t)$ UND $h''$ GEBOUNDED!}
%\textcolor{red}{HIER WEITERMACHEN...\eqref{lem:res8} ist zu schwach, allerdings SOLLTE MAN HIER
%\begin{align*}
 %|\frac{1}{n}\sum_{i=1}^n X_{ij}\varepsilon_i h''(\xi_{i,1})/2 (X_{il}-\widehat{X}_{il})^2|&  \leq \sqrt{\frac{1}{n}\sum_{i=1}^n (\varepsilon_i h''(\xi_{i,1})/2(X_{il}-\widehat{X}_{il}))^2} \sqrt{\frac{1}{n}\sum_{i=1}^n (X_{ij}(X_{il}-\widehat{X}_{il}))^2} \\
%&\leq M_h \sqrt{\sqrt{\frac{1}{n}\sum_{i=1}^n (\varepsilon_i^4)} \sqrt{\frac{1}{n}\sum_{i=1}^n(X_{il}-\widehat{X}_{il})^4}} \sqrt{\frac{1}{n}\sum_{i=1}^n (X_{ij}(X_{il}-\widehat{X}_{il}))^2} \\
%&\leq M_h (\E(\varepsilon_i^4)^{\frac{1}{4}}+o_P(1))O_P(n^{-\frac{1}{4}})O_P(n^{-\frac{1}{2}}) = O_P(n^{-3/4})
%\end{align*}}
At the same time it follows from similar arguments that for all $j=1,\dots,S+1$ we have
%\begin{align*}
 $R_{j,1,b}(\bbeta_0) = \frac{1}{n}\sum_{i=1}^n (\widehat{X}_{ij}-X_{ij})h(\widehat{\eta}_i(\bbeta_0))\varepsilon_i = O_P(n^{-1})$.
%\end{align*}
%Achtung: hier muss man eventuell eine Taylor 2. ordnung von $h(\widehat{\eta}_i(\bbeta_0))$ an der Stelle $\eta_i(\bbeta_0)$ machen (dann hängt $h'$ nicht mehr von den geschätzten werten ab)
The above arguments then imply:
\begin{align}
\mathbf{Rest}_{1}(\bbeta_0) = O_P(n^{-1}).\label{eq:thqml1eq1}
\end{align}
%Moreover, since $|h(x)|\leq M_h$
%\begin{align*}
% R_{j,1,b} &= \frac{1}{n}\sum_{i=1}^n (\widehat{X}_{ij}-X_{ij})(\frac{g'(\widehat{\eta}_i)}{\sigma^2(g(\widehat{\eta}_i))}) (y_i-g(\eta_i)) \\
%					&= \frac{1}{n}\sum_{i=1}^n (\widehat{X}_{ij}-X_{ij})h(\widehat{\eta}_i)\varepsilon_i\\
%\end{align*}
%But by \ref{lem:res6} it follows that
%$$\frac{1}{n}\sum_{i=1}^n (\widehat{X}_{ij}-X_{ij})h(\widehat{\eta}_i)\varepsilon_i = o_P(n^{-1/2})$$
%since $h$ is bounded (and the bernstein conditions on $\varepsilon_i$ hold)....hier schaff ich auch O_P(n^{-1})
\bigskip
The $j$th equation of $\mathbf{Rest}_2(\bbeta)$ can be written as $Rest_{j,2}(\bbeta) = \frac{1}{n} \sum_{i=1}^n h(\widehat{\eta}_i(\bbeta))\widehat{X}_{ij}(g(\eta_i(\bbeta))-g(\widehat{\eta}_i(\bbeta)))$. Using again Taylor expansions together with assertions \eqref{lem:res4a}, \eqref{lem:res3} as well as the Cauchy-Schwarz inequality together with \eqref{lem:res8}, can now be used to conclude that for all $\bbeta$ and $j=1,\dots, S+1$ we have
\begin{align}
Rest_{j,2}(\bbeta) = O_P(n^{-\min\{1,1/\kappa\}})\label{eq:thqml1eq2}.
\end{align}
Assertion  \eqref{eq:thQML1} then follows from  \eqref{eq:thasdda}, \eqref{eq:thqml1eq1} and \eqref{eq:thqml1eq2}. Note that our assumptions in particular imply that $\mathbf{Rest}_1(\bbeta_0)$ and $\mathbf{Rest}_2(\bbeta_0)$ are uniform integrable. Additional to $\eqref{eq:thQML1}$, we thus have $\E(\mathbf{Rest}_n(\bbeta_0)) \to \bzero$ implying \eqref{eq:thQML3meanaa}, $\E(\widehat{\bU}_n(\bbeta_0)/n) \to \bzero$, since $\E(\bU_n(\bbeta_0)/n) = 0$.\\

In order to proof assertion \eqref{eq:thQML1a} suppose $\bbeta \neq \bbeta_0$ and note that we still have \eqref{eq:thqml1eq2}. However, $\mathbf{Rest}_{1}(\bbeta)$ needs a closer investigation. Its $j$th row can be written as
\begin{align*}
Rest_{j,1}(\bbeta) &= \frac{1}{n}\sum_{i=1}^n (\frac{g'(\widehat{\eta}_i(\bbeta))}{\sigma^2(g(\widehat{\eta}_i(\bbeta)))} \widehat{X}_{ij}
-    \frac{g'(\eta_i(\bbeta))}{\sigma^2(g(\eta_i(\bbeta)))} X_{ij})(y_i-g(\eta_i(\bbeta)))\\
			%	&=\frac{1}{n}\sum_{i=1}^n X_{ij}(\frac{g'(\widehat{\eta}_i(\bbeta))}{\sigma^2(g(\widehat{\eta}_i(\bbeta)))} -  \frac{g'(\eta_i(\bbeta))}{\sigma^2(g(\eta_i(\bbeta)))}) (y_i-g(\eta_i(\bbeta)))\\
			%	&\quad + \frac{1}{n}\sum_{i=1}^n (\widehat{X}_{ij}-X_{ij})(\frac{g'(\widehat{\eta}_i(\bbeta))}{\sigma^2(g(\widehat{\eta}_i(\bbeta)))}) (y_i-g(\eta_i(\bbeta)))\\
		%		&=\frac{1}{n}\sum_{i=1}^n X_{ij}(\frac{g'(\widehat{\eta}_i(\bbeta))}{\sigma^2(g(\widehat{\eta}_i(\bbeta)))} -  \frac{g'(\eta_i(\bbeta))}{\sigma^2(g(\eta_i(\bbeta)))}) (y_i-g(\eta_i(\bbeta_0)))\\
	%			&\quad -\frac{1}{n}\sum_{i=1}^n X_{ij}(\frac{g'(\widehat{\eta}_i(\bbeta))}{\sigma^2(g(\widehat{\eta}_i(\bbeta)))} -  \frac{g'(\eta_i(\bbeta))}{\sigma^2(g(\eta_i(\bbeta)))}) (g(\eta_i(\bbeta)) - g(\eta_i(\bbeta_0)))\\
	%			&\quad + \frac{1}{n}\sum_{i=1}^n (\widehat{X}_{ij}-X_{ij})(\frac{g'(\widehat{\eta}_i(\bbeta))}{\sigma^2(g(\widehat{\eta}_i(\bbeta)))}) (y_i-g(\eta_i(\bbeta_0)))\\
%				&\quad - \frac{1}{n}\sum_{i=1}^n (\widehat{X}_{ij}-X_{ij})(\frac{g'(\widehat{\eta}_i(\bbeta))}{\sigma^2(g(\widehat{\eta}_i(\bbeta)))}) (g(\eta_i(\bbeta))-g(\eta_i(\bbeta_0)))\\
				&=\frac{1}{n}\sum_{i=1}^n X_{ij}(h(\widehat{\eta}_i(\bbeta)) - h(\eta_i(\bbeta))) (y_i-g(\eta_i(\bbeta_0)))\\
				&\quad -\frac{1}{n}\sum_{i=1}^n X_{ij}(h(\widehat{\eta}_i(\bbeta)) - h(\eta_i(\bbeta))) (g(\eta_i(\bbeta)) - g(\eta_i(\bbeta_0)))\\
				&\quad + \frac{1}{n}\sum_{i=1}^n (\widehat{X}_{ij}-X_{ij})h(\widehat{\eta}_i(\bbeta)) (y_i-g(\eta_i(\bbeta_0)))\\
				&\quad - \frac{1}{n}\sum_{i=1}^n (\widehat{X}_{ij}-X_{ij})h(\widehat{\eta}_i(\bbeta)) (g(\eta_i(\bbeta))-g(\eta_i(\bbeta_0))).
\end{align*}
To obtain \eqref{eq:thQML1a} it is sufficient to use some rather conservative inequalities of each of the appearing terms. %\textcolor{red}{der nächste abschnitt ist sehr konservativ abgeschätzt...eigentlich müsste man hier auch wesentlich schnellere raten rausbekommen, werden aber nicht benötigt..}
For instance, another Taylor expansion together with the Cauchy-Schwarz inequality and \eqref{lem:res0} now yield
\begin{align}
\frac{1}{n}\sum_{i=1}^n& X_{ij}(h(\widehat{\eta}_i(\bbeta)) - h(\eta_i(\bbeta))) (y_i-g(\eta_i(\bbeta_0)))  = O_P(n^{-\frac{1}{2}}). \label{eq:Uhaha1}
\end{align}
While the Cauchy-Schwarz inequalitiy together with \eqref{lem:res0} yields
\begin{align}
\frac{1}{n}\sum_{i=1}^n (\widehat{X}_{ij}-X_{ij})h(\widehat{\eta}_i(\bbeta))(y_i-g(\eta_i(\bbeta_0))) = O_P(n^{-\frac{1}{2}}). \label{eq:Uhaha2}
\end{align}
It follows from additional Taylor expansions that there exists a $\xi_{i,2}$ between $\widehat{\eta}_i(\bbeta)$ and $\eta_i(\bbeta)$ as well as some $\xi_{i,3}$ between $\eta_i(\bbeta)$ and $\eta_i(\bbeta_0)$ such that:
\begin{align*}
\frac{1}{n}\sum_{i=1}^n & X_{ij}(h(\widehat{\eta}_i(\bbeta)) - h(\eta_i(\bbeta))) (g(\eta_i(\bbeta)) - g(\eta_i(\bbeta_0)))\\
& = \sum_{r=2}^{S+1}\beta_{r-1} \sum_{l=1}^{S+1} (\beta_{l-1}^{(0)}- \beta_{l-1}) \frac{1}{n}\sum_{i=1}  X_{ij}(X_{ir}-\widehat{X}_{ir})X_{il}h'(\xi_{i,2})g'(\xi_{i,3}).
\end{align*}
Again, with the help of the Cauchy-Schwarz inequality together with \eqref{lem:res0} it can immediately seen that
\begin{align}
\frac{1}{n}\sum_{i=1}^n & X_{ij}(h(\widehat{\eta}_i(\bbeta)) - h(\eta_i(\bbeta))) (g(\eta_i(\bbeta)) - g(\eta_i(\bbeta_0))) = O_P(n^{-\frac{1}{2}}).\label{eq:Uhaha3}
\end{align}
Similar one may show that
\begin{align}
\frac{1}{n}\sum_{i=1}^n (\widehat{X}_{ij}-X_{ij})(\frac{g'(\widehat{\eta}_i(\bbeta))}{\sigma^2(g(\widehat{\eta}_i(\bbeta)))}) (g(\eta_i(\bbeta))-g(\eta_i(\bbeta_0)))= O_P(n^{-\frac{1}{2}}).\label{eq:Uhaha4}
\end{align}
Assertion \eqref{eq:thQML1a} then follows from \eqref{eq:thqml1eq2} and \eqref{eq:Uhaha1}--\eqref{eq:Uhaha4}.  \eqref{eq:thQML3mean0} follows again from a closer investigation of the existence and boundedness of moments of the involved remainder terms, leading to \eqref{eq:thqml1eq2}.

In order to proof \eqref{eq:thQML2}, note that the $(s+1) \times (s+1)$ matrix $\widehat{\bF}(\bbeta)=\widehat{\bD}^T(\bbeta)\widehat{\bV}^{-1}(\bbeta)\widehat{\bD}(\bbeta)$ may be written as

$$\frac{1}{n}\widehat{\bF}_n(\bbeta)
% = \frac{1}{n}\widehat{\bD}_n^T(\bbeta)\widehat{\bV}_n^{-1}(\bbeta)\widehat{\bD}_n(\bbeta)
 = \frac{1}{n}\bF_n(\bbeta) +\mathbf{Rest}_n^{(F)}(\bbeta).$$
%where $\bi_\beta$ corresponds to $\widehat{\bi}_\beta$ evaluated at $X_i$ instead of $\widehat{X}_{i}$.
%Since $\widehat{X}_{ij}= X_{ij} + (\widehat{X}_{ij}-X_{ij})$,
$\mathbf{Rest}_n^{(F)}(\bbeta)$ has a typical element $Rest_{jk}^{(F)}(\bbeta)$ which is given by
\begin{align}
Rest_{jk}^{(F)}(\bbeta) & = \frac{1}{n}\sum_{i=1}^n (\frac{g'(\widehat{\eta}_i(\bbeta))^2}{\sigma^2(g(\widehat{\eta}_i(\bbeta)))}\widehat{X}_{ij}\widehat{X}_{ik} - \frac{g'(\eta_i(\bbeta))^2}{\sigma^2(g(\eta_i(\bbeta)))}X_{ij}X_{ik})\nonumber  \\
     & = \frac{1}{n}\sum_{i=1}^n (\frac{g'(\widehat{\eta}_i(\bbeta))^2}{\sigma^2(g(\widehat{\eta}_i(\bbeta)))} - \frac{g'(\eta_i(\bbeta))^2}{\sigma^2(g(\eta_i(\bbeta)))})X_{ij}X_{ik} \label{eq:restQMLa} \\
		 &   \quad + \frac{1}{n}\sum_{i=1}^n \frac{g'(\widehat{\eta}_i(\bbeta))^2}{\sigma^2(g(\widehat{\eta}_i(\bbeta)))}(\widehat{X}_{ij}-X_{ij}) X_{ik}\label{eq:restQMLb}\\
		 &   \quad + \frac{1}{n}\sum_{i=1}^n \frac{g'(\widehat{\eta}_i(\bbeta))^2}{\sigma^2(g(\widehat{\eta}_i(\bbeta)))}(\widetilde{X}_{ik}-X_{ik}) X_{ij}\label{eq:restQMLc}\\
		 &   \quad + \frac{1}{n}\sum_{i=1}^n \frac{g'(\widehat{\eta}_i(\bbeta))^2}{\sigma^2(g(\widehat{\eta}_i(\bbeta)))}(\widehat{X}_{ik}-X_{ik})(\widehat{X}_{ij}-X_{ij}).\label{eq:restQMLd}
\end{align}
$Rest_{jk}^{(F)}(\bbeta)$ consists of the sum of four terms. We begin with \eqref{eq:restQMLa}.\\

\noindent Define $h_1(x) = g'(x)^2/\sigma^2(g(x))$ and note that $|h_1(x)|\leq M_{h_1}$ as well as $|h_1'(x)|\leq M_{h_1}$ for some constant $M_{h_1}<\infty$.
With the help of the Cauchy-Schwarz inequality and $\eqref{lem:res0}$, it follows from another Taylor expansion that there exists a $\xi_{i,4}$ between $\widehat{\eta}_i(\bbeta)$ and $\eta_i(\bbeta)$ such that:
\begin{align*}
%|\frac{1}{n}\sum_{i=1}^n& (\frac{g'(\widehat{\eta}_i(\bbeta))^2}{\sigma^2(g(\widehat{\eta}_i(\bbeta)))} - \frac{g'(\eta_i(\bbeta))^2}{\sigma^2(g(\eta_i(\bbeta)))})X_{ij}X_{ik}|
|\frac{1}{n}\sum_{i=1}^n& (\frac{g'(\widehat{\eta}_i(\bbeta))^2}{\sigma^2(g(\widehat{\eta}_i(\bbeta)))} - \frac{g'(\eta_i(\bbeta))^2}{\sigma^2(g(\eta_i(\bbeta)))})X_{ij}X_{ik}|
= |\sum_{r=2}^{S+1} \beta_{r-1} \frac{1}{n}\sum_{i=1}^n (\widehat{X}_{ir}-X_{ir})X_{ij}X_{ik}h_1'(\xi_{i,4})|\\
&\leq \sum_{r=2}^{S+1} |\beta_{r-1}| \sqrt{\frac{1}{n}\sum_{i=1}^n (\widehat{X}_{ir}-X_{ir})^2}\sqrt{\frac{1}{n}\sum_{i=1}^n (X_{ij}X_{ik}h_1'(\xi_{i,4}))^2} = O_P(n^{-\frac{1}{2}}).
\end{align*}
On the other hand, the Cauchy-Schwarz inequality together with the boundedness $|h_1(x)|$ and \eqref{lem:res0} implies that each of the other terms \eqref{eq:restQMLb}--\eqref{eq:restQMLd} is $O_P(n^{-1/2})$.
%On the other hand, \eqref{lem:res4a} together with \eqref{lem:res3} implies that each of the other terms \eqref{eq:restQMLb}- \eqref{eq:restQMLd} is $O_P(n^{-\min\{1,1/\kappa\}})$, since $|h_1(x)|$ and $|h_1'(x)|$ are bounded.
Assertion \eqref{eq:thQML2} is then an immediate consequence. Moreover, since $h_1(x)$ is bounded, it can immediately be seen that $Rest_{jk}^{(F)}(\bbeta)$ is uniform integrable, providing additionally $\E(Rest_{jk}^{(F)}(\bbeta)) \to 0$. Assertion \eqref{eq:thQML3mean1} follows immediately.\\
%%%%%%%%%%%%%%%%%%%%%%%%%%%%%%%%%%%%%%%%%%%%%%%%%%%%%%%%%%%%%%%%%%
%%%%%%%%%%%%%%%%%%%%%%%%%%%%%%%%%%%%%%%%%%%%%%%%%%%%%%%%%%%%%%%%%%

In order so show \eqref{eq:thQML3}, note that
$\widehat{\bR}_n(\bbeta)/n = {\bR}_n(\bbeta)/n + \mathbf{Rest}_n^{(R)}(\bbeta)$.
 A typical entry $Rest^{(R)}_{jk}(\bbeta)$ of $\mathbf{Rest}_n^{(R)}(\bbeta)$ reads as
\begin{align}
Rest^{(R)}_{jk}(\bbeta) &= \frac{1}{n}\sum_{i=1}^n (h'(\widehat{\eta}_i(\bbeta))-h'(\eta_i(\bbeta)))X_{ij}X_{ik}(y_i-g(\eta_i(\bbeta))) \label{eq:h1}\\
&\qquad + \frac{1}{n}\sum_{i=1}^n h'(\widehat{\eta}_i(\bbeta))X_{ij}(\widehat{X}_{ik}-X_{ik})(y_i-g(\eta_i(\bbeta))) \label{eq:h2} \\
&\qquad + \frac{1}{n}\sum_{i=1}^n h'(\widehat{\eta}_i(\bbeta))(\widehat{X}_{ij}-X_{ij})X_{ik}(y_i-g(\eta_i(\bbeta))) \label{eq:h3} \\
&\qquad + \frac{1}{n}\sum_{i=1}^n h'(\widehat{\eta}_i(\bbeta))(\widehat{X}_{ij}-X_{ij})(\widehat{X}_{ik}-X_{ik})(y_i-g(\eta_i(\bbeta))) \label{eq:h4}\\
&\quad - \frac{1}{n}\sum_{i=1}^n h'(\widehat{\eta}_i(\bbeta)) \widehat{X}_{ij} \widehat{X}_{ik} (g(\widehat{\eta}_i(\bbeta))-g(\eta_i(\bbeta))). \label{eq:h5}
\end{align}
We will first show
\begin{align}
\frac{1}{n}\widehat{\bR}_n(\bbeta_0) &= \frac{1}{n}\bR_n(\bbeta_0) + O_P(n^{-\frac{1}{2}}). \label{eq:thQML3a}
\end{align}
For $\bbeta=\bbeta_0$, since $|h''(\cdot)|\leq M_h$, a Taylor expansion together with the Cauchy-Schwarz inequality and \eqref{lem:res0} yield  $\frac{1}{n}\sum_{i=1}^n (h'(\widehat{\eta}_i(\bbeta_0))-h'(\eta_i(\bbeta_0)))X_{ij}X_{ik}\varepsilon_i  = O_P(n^{-\frac{1}{2}})$.
Similarly each of the assertions \eqref{eq:h2}--\eqref{eq:h4} are $O_P(n^{-\frac{1}{2}})$
%are $O_P(n^{-1})$ by \ref{lem:res7} and \ref{lem:res6}.
At the same time another Taylor expansion of \eqref{eq:h5} yields together with the Cauchy-Schwarz inequality and \eqref{lem:res0} for some $\xi_{i,5}$ between $\widehat{\eta}_i(\bbeta)$ and $\eta_i(\bbeta)$:
\begin{align*}
\frac{1}{n}\sum_{i=1}^n& h'(\widehat{\eta}_i(\bbeta_0)) \widehat{X}_{ij} \widehat{X}_{ik} (g(\widehat{\eta}_i(\bbeta_0))-g(\eta_i(\bbeta_0)))\\
&  = \sum_{r=2}^{S+1} \beta_{r-1} \frac{1}{n}\sum_{i=1}^n h'(\widehat{\eta}_i(\bbeta_0)) g'(\xi_{i,5}) \widehat{X}_{ij} \widehat{X}_{ik} (X_{ir}- \widehat{X}_{ir}) =  O_P(n^{-\frac{1}{2}}).
\end{align*}
We may conclude that
$$\widehat{\bR}_n(\bbeta_0) = \bR_n(\bbeta_0) + O_P(n^{-\frac{1}{2}}).$$
Moreover, our assumptions in particular imply that besides $Rest^{(R)}_{jk}(\bbeta_0)/n = O_P(n^{-\frac{1}{2}})$ we have $\E(Rest^{(R)}_{jk}(\bbeta_0)) \to 0,$ proving assertions \eqref{eq:thQML3a} and \eqref{eq:thQML3mean}.

\noindent Now suppose $\bbeta\neq \bbeta_0$ and take another look at \eqref{eq:h1}:
\begin{align*}
\frac{1}{n}\sum_{i=1}^n &(h'(\widehat{\eta}_i(\bbeta))-h'(\eta_i(\bbeta)))X_{ij}X_{ik}(y_i-g(\eta_i(\bbeta))) \\
&= \frac{1}{n}\sum_{i=1}^n (h'(\widehat{\eta}_i(\bbeta))-h'(\eta_i(\bbeta)))X_{ij}X_{ik}\varepsilon_i \\
 &\quad  -  \frac{1}{n}\sum_{i=1}^n (h'(\widehat{\eta}_i(\bbeta))-h'(\eta_i(\bbeta)))X_{ij}X_{ik} (g(\eta_i(\bbeta)) - g(\eta_i(\bbeta_0)))
\end{align*}
Similar arguments as before, together with $\E(\varepsilon_i^4)<\infty$, can now be used to show that $$\frac{1}{n}\sum_{i=1}^n (h'(\widehat{\eta}_i(\bbeta))-h'(\eta_i(\bbeta)))X_{ij}X_{ik}\varepsilon_i  = O_P(n^{-\frac{1}{2}}).$$
%WICHTIGWICHTIGWICHTIG:
%(\textcolor{red}{4 moment von $\varepsilon$ muss existieren für CS sonst annahme dass das 2. moment des produkts $\E(X_{ij}^2X_{ik}^2 \varepsilon_i^2)$ existiert!!...})
A Taylor expansion of $g(\eta_i(\bbeta))$ leads for some $\xi_{i,6}$ between $\eta_i(\bbeta)$ and $\eta_i(\bbeta_0)$ to
\begin{align*}
\frac{1}{n}\sum_{i=1}^n & (h'(\widehat{\eta}_i(\bbeta))-h'(\eta_i(\bbeta)))X_{ij}X_{ik} (g(\eta_i(\bbeta)) - g(\eta_i(\bbeta_0))) \\
& =\sum_{r=1}^{S+1} (\beta_{r-1}-\beta_{r-1}^{(0)}) \frac{1}{n}\sum_{i=1}^n  (h'(\widehat{\eta}_i(\bbeta))-h'(\eta_i(\bbeta)))X_{ij}X_{ik}X_{ir}g'(\xi_{i,6}).
\end{align*}
Another Taylor expansion of $h'(\widehat{\eta}_i(\bbeta))$ together with the Cauchy-Schwarz inequality and the boundedness of $|g'(x)|$ and $|h''(x)|$ leads for some $\xi_{i,7}$ between $\widehat{\eta}_i(\bbeta)$ and $\eta_i(\bbeta)$ to
\begin{align*}
\sum_{r=1}^{S+1} (\beta_{r-1}-\beta_{r-1}^{(0)}) \sum_{l=2}^{S+1}\beta_{l-1} \frac{1}{n}\sum_{i=1}^n & (X_{il}-\widehat{X}_{il})X_{ij}X_{ik}X_{ir}g'(\xi_{i,6})h''(\xi_{i,7}) = O_P(n^{-\frac{1}{2}}).
\end{align*}
%\textcolor{red}{!!!!!!!!!!!!!!!!!!!!PROVIDED $h''(\xi_{i,8})<\infty$ OR THE EXPECTATION EXISTS!!!!!!!!!!}
With similar arguments $\eqref{eq:h5}$ and \eqref{eq:h2} are, for all $\bbeta$,  $O_P(n^{-\frac{1}{2}})$. \\
%We will discuss \eqref{eq:h2}, note that
%\begin{align*}
%\frac{1}{n}\sum_{i=1}^n& h'(\widehat{\eta}_i(\bbeta))X_{ij}(\widehat{X}_{ik}-X_{ik})(y_i-g(\eta_i(\bbeta)))  \\
%&= \frac{1}{n}\sum_{i=1}^n h'(\widehat{\eta}_i(\bbeta))X_{ij}(\widehat{X}_{ik}-X_{ik})(y_i-g(\eta_i(\bbeta_0))) \\
%&\quad -  \frac{1}{n}\sum_{i=1}^n h'(\widehat{\eta}_i(\bbeta))X_{ij}(\widehat{X}_{ik}-X_{ik})(g(\eta_i(\bbeta))-g(\eta_i(\bbeta_0)))\\
%&= \frac{1}{n}\sum_{i=1}^n h'(\widehat{\eta}_i(\bbeta))X_{ij}(\widehat{X}_{ik}-X_{ik})\varepsilon \\
%&\quad -  \sum_{r=1}^{S+1} (\beta_{r-1}^{(0)}-\beta_{r-1}) \frac{1}{n}\sum_{i=1}^n h'(\widehat{\eta}_i(\bbeta))g'(\xi_{i,9}) X_{ij}X_{ir}(\widehat{X}_{ik}-X_{ik})\\
%\end{align*}
%While $$\frac{1}{n}\sum_{i=1}^n h'(\widehat{\eta}_i(\bbeta))X_{ij}(\widehat{X}_{ik}-X_{ik})\varepsilon  = O_P(n^{-1})$$ by \eqref{lem:res7} and \eqref{lem:res6},
%The Cauchy-Schwarz inequalitiy and $\eqref{lem:res0}$ then imply
%$$\sum_{r=1}^{S+1} (\beta_{r-1}-\beta_{r-1}^{(0)}) \frac{1}{n}\sum_{i=1}^n h'(\widehat{\eta}_i(\bbeta))g'(\xi_{i,9}) X_{ij}X_{ir}(\widehat{X}_{ik}-X_{ik}) = O_P(n^{-\frac{1}{2}})$$
Considerations for \eqref{eq:h3}--\eqref{eq:h4} are parallel to the case $\eqref{eq:h2}$ assertion \eqref{eq:thQML3} follows immediately. \eqref{eq:thQML3mean2} follows again from a closer investigation of the existence and boundedness of the moments of the rest terms used in the derivations \eqref{eq:thQML3}.
%$\square$
\end{proof}
%%%%%%%%%%%%%%%%%%%%%%%%%%%%%%%%%%%%%%%%%%%%%%%%%%%%%%%%%%%%%%%%%%%%

%%%%%%%%%%%%%%%%%%%%%%%%%%%%%%%%%%%%%%%%%%%%%%%%%%%%%%%%%%%%%%%%%%%%
%%%%%%%%%%%%%%%%%%%%%%%%%%%%%%%%%%%%%%%%%%%%%%%%%%%%%%%%%%%%%%%%%%%%

\bigskip

The proof of Theorem~\ref{th:paraestML1} consists roughly of two steps. In a first step asymptotic existence and consistency of our estimator $\widehat{\bbeta}$ is developed. In a second step we can then make use of the usual Taylor expansion of the estimation equation $\widehat{\bU}_n(\bbeta)$. With the help of Proposition \ref{th:paraestQML} asymptotic normality of our estimator will follow.

\bigskip

\begin{proof}[{\bf{Proof of Theorem~\ref{th:paraestML1}.}}]
For a $q_1\times q_2$ matrix $\mathbf{A}$ let $||\mathbf{A}|| = \sqrt{\sum_{i=1}^{q_1}\sum_{j=1}^{q_2} a_{ij}^2}$ its Frobenius norm. Moreover we denote by $\mathbf{A}^{1/2}$ ($\mathbf{A}^{T/2}$) the left (the corresponding right) square root of a positive definite matrix $\mathbf{A}$.

The proof generalizes the arguments used in Corollary 3 and Theorem 1 in \citeappendix{FK1985}.
For $\delta_1>0$ define the neighborhoods $$N_n(\delta_1) = \{\bbeta:||\widehat{\bF}_n^{1/2}(\bbeta_0)(\bbeta - \bbeta_0)|| \leq \delta_1 \},$$ and remember that with $h_1(x)=g'(x)^2/\sigma^2(g(x))$ we have:
$$\frac{1}{n}\widehat{\bF}_n(\bbeta) = \frac{1}{n}\sum_{i=1}^n h_1(\widehat{\eta_i}(\bbeta)) \widehat{\bX}_{i}\widehat{\bX}_{i}^T.$$
The $(j,k)$-element of this random matrix is given by $1/n\sum_{i=1}^nh_1(\widehat{\eta_i}(\bbeta)) \widehat{X}_{ij}\widehat{X}_{ik}$
and constitutes a triangular array of row-wise independent and identical distributed random variables.
Let $\widehat{\eta}(\bbeta)$, $\widehat{\bX}$ and $\varepsilon$ be generic copies of $\widehat{\eta_i}(\bbeta)$, $\widehat{\bX}_i$ and $\varepsilon_i$.
Since $h_1$ is bounded it is then easy to see that for any compact neighborhood $N$ around $\bbeta_0$ we have for all $p\geq1$:
\begin{align}
\E(\max_{\bbeta\in N}|h_1(\widehat{\eta}(\bbeta)) \widehat{X}_{j}\widehat{X}_{k}|^p)\leq M_{1}\label{eq:thm2pf1}
\end{align}
for some constant $M_{1}<\infty$, not depending on $n$. On the other hand the $(j,k)$-element  of $\widehat{\bR}_n(\bbeta)/n$ can be written as
\begin{align*}
 \frac{1}{n}\sum_{i=1}^n  h'(\widehat{\eta_i}(\bbeta))\widehat{X}_{ij}\widehat{X}_{ik}(g(\eta_i(\bbeta_0))-g(\widehat{\eta}_i(\bbeta))) + \frac{1}{n}\sum_{i=1}^n h'(\widehat{\eta_i}(\bbeta))\widehat{X}_{ij}\widehat{X}_{ik}\varepsilon_i.
\end{align*}
Using the boundedness of $g'$ and $h'$ it follows from a Taylor expansion that for all $p\geq 1$:
\begin{align}
\E(\max_{\bbeta \in N} |h'(\widehat{\eta}(\bbeta))\widehat{X}_{j}\widehat{X}_{k}(g(\eta(\bbeta_0))-g(\widehat{\eta}(\bbeta)))|^p)\leq  M_{2} \label{eq:thm2pf2}
\end{align}
for some constant $M_{2}<\infty$, not depending on $n$. While the Cauchy-Schwarz inequality together with the assumption $\E(\varepsilon^4)<\infty$ implies that for $1\leq p\leq 2$:
\begin{align}
\E(\max_{\bbeta \in N} |h'(\widehat{\eta}(\bbeta))\widehat{X}_{j}\widehat{X}_{k}\varepsilon|^p) \leq M_{3} \label{eq:thm2pf3}
\end{align}
for some constant $M_{3}<\infty$, not depending on $n$. By \eqref{eq:thm2pf1}, \eqref{eq:thm2pf2} and \eqref{eq:thm2pf3} a uniform law of large numbers for triangular arrays leads to
\begin{align}
\max_{\bbeta \in N}||\frac{1}{n}\widehat{\bF}_n(\bbeta)- \E(\widehat{\bF}(\bbeta))|| &\stackrel{p}{\to} 0, \label{eq:exist0}
\intertext{as well as}
\max_{\bbeta \in N}||\frac{1}{n}\widehat{\bR}_n(\bbeta)- \E(\widehat{\bR}(\bbeta))|| &\stackrel{p}{\to} 0. \label{eq:exist0a}
\end{align}
%%%%%%%%%%
Moreover, by \eqref{eq:thm2pf1}, $\widehat{\bF}_n(\bbeta_0)/n$ converges a.s.~to
%%%%
%brauchen P=4, siehe lecture notes von terrence tao zu den strong laws of large numbers for triangular arrays
%%%%%%%%%
$\E(\widehat{\bF}(\bbeta_0))$, implying $\lambda_{min}\widehat{\bF}_n(\bbeta_0) \to \infty$ a.s., where $\lambda_{min} \mathbf{A}$ denotes the smallest eigenvalue of a matrix $\mathbf{A}$. Note that as a direct consequence the neighborhoods $N_n(\delta_1)$ shrink (a.s.) to $\bbeta_0$ for all $\delta_1>0$.
On the other hand, since by \eqref{eq:thQML3mean}, $\E(\widehat{\bR}(\bbeta_0)) \to 0$ and $\E(\widehat{\bR}(\bbeta))$ is continuous in $\bbeta$ we have for all $\epsilon>0$, with probability converging to $1$,
\begin{align*}
||\frac{1}{n}\widehat{\bR}_n(\bbeta) ||  \leq ||\frac{1}{n}\widehat{\bR}_n(\bbeta) - \E(\widehat{\bR}(\bbeta))|| + ||\E(\widehat{\bR}(\bbeta)) - \E(\widehat{\bR}(\bbeta_0))|| +  ||\E(\widehat{\bR}(\bbeta_0))|| \leq \epsilon
\end{align*}
if $\bbeta$ is sufficiently close to $\bbeta_0$.

The usual decomposition then yields for all $\epsilon>0$, with probability converging to $1$:
\begin{align*}
|| & - \frac{1}{n}\widehat{\bH}_n(\bbeta) - \frac{1}{n} \widehat{\bF}_n(\bbeta_0)|| \leq ||\frac{1}{n}\widehat{\bF}_n(\bbeta) - \frac{1}{n} \widehat{\bF}_n(\bbeta_0)|| +  ||\frac{1}{n}\widehat{\bR}_n(\bbeta)||  \\
& \leq  || \frac{1}{n}\widehat{\bF}_n(\bbeta) - \E(\widehat{\bF}(\bbeta))|| + ||\E(\widehat{\bF}(\bbeta_0)) - \frac{1}{n} \widehat{\bF}_n(\bbeta_0)||  + ||\E(\widehat{\bF}(\bbeta)) - \E(\widehat{\bF}(\bbeta_0))|| \\
& + ||\frac{1}{n}\widehat{\bR}_n(\bbeta)|| \leq \epsilon,
\end{align*}
if $\bbeta$ is sufficiently close to $\bbeta_0$. Similar to the proof of Corollary 3 in \citeappendix{FK1985} we may infer from this inequality that for all $\delta_1>0$ we have
%%%%%%%%%%%%%%%%%
$$\max_{\bbeta \in N_n(\delta_1)}|| \widehat{\boldsymbol{\mathcal{V}}}_n(\bbeta) - \bI_{S+1}|| \stackrel{p}{\to} 0,$$
%%%%%%%%%%%%%%%%%
where  $\widehat{\boldsymbol{\mathcal{V}}}_n(\bbeta) = -\widehat{\bF}_n^{-1/2}(\bbeta_0) \widehat{\bH}_n(\bbeta)\widehat{\bF}_n^{-T/2}(\bbeta_0)$ and $\bI_p$ denotes the $p\times p$ identity matrix. Again, following the arguments in \citeappendix[cf. Section 4.1]{FK1985}, this in particular implies that for all $\delta_1>0$ we have
\begin{align}
P(-\widehat{\bH}_n(\bbeta) - c \widehat{\bF}_n(\bbeta_0) \text{ positive semidefinite for all $\bbeta \in N_n(\delta_1)$}) \to 1 \label{eq:QMLexistscstar}
\end{align}
for some constant $c>0$, $c$ independent of $\delta_1$.\\
Let $\widehat{Q}_n(\bbeta)$ be the quasi-likelihood function evaluated at the points of impact estimates $\widehat{\tau}_r$. We aim to show that for any $\zeta >0$ there exists a $\delta_1>0$ such that
\begin{align}
P(\widehat{Q}_n(\bbeta) - \widehat{Q}_n(\bbeta_0)< 0 \text{ for all } \bbeta \in \partial N_n(\delta_1)) \geq 1-\zeta \label{eq:QMLEexist}
\end{align}
for all sufficiently large $n$. Note that the event $\widehat{Q}_n(\bbeta) - \widehat{Q}_n(\bbeta_0)< 0 \text{ for all } \bbeta \in \partial N_n(\delta_1)$ implies that the there is a maximum inside of $N_n(\delta_1)$. Moreover, since $\widehat{\bR}_n(\bbeta)/n$ is asymptotical negligible in a neighborhood around $\bbeta_0$, and at the same time $\widehat{\bF}_n(\bbeta)/n$ converges in probability to a positive definite matrix, the maximum will, with probability converging to 1,  be uniquely determined as a zero of the score function $\widehat{\bU}_n(\bbeta)$.  \eqref{eq:QMLEexist} then in particular implies that $P(\widehat{\bU}_n(\widehat{\bbeta}) = 0) \to 1$ and, together with the observation that $N_n(\delta_1)$ shrink (a.s.) to $\bbeta_0$, it implies consistency of our estimator, i.e. $\widehat{\bbeta} \stackrel{p}{\to} \bbeta_0$. \\
A Taylor expansion yields, with $\boldsymbol{\lambda}= \widehat{\bF}_n^{T/2}(\bbeta_0)(\bbeta-\bbeta_0)/\delta_1$, for  some $\widetilde{\bbeta}$ on the line segment between $\bbeta$ and $\bbeta_0$:
$$\widehat{Q}_n(\bbeta) -\widehat{Q}_n(\bbeta_0) = \delta_1 \boldsymbol{\lambda}'\widehat{\bF}_n^{-1/2}(\bbeta_0)\widehat{\bU}_n(\bbeta_0) - \delta_1^2 \boldsymbol{\lambda}'\widehat{\boldsymbol{\mathcal{V}}}_n(\widetilde{\bbeta})\boldsymbol{\lambda}/2, \quad \boldsymbol{\lambda}'\boldsymbol{\lambda}=1.$$

Using for the next few lines the spectral norm one may argue similarly to (3.9) in \citeappendix{FK1985}, that it suffices to show that for any $\zeta>0$ we have
$$P(||\widehat{\bF}_n^{-1/2}(\bbeta_0)\widehat{\bU}_n(\bbeta_0)||<\delta_1^2\lambda_{\min}^2\widehat{\boldsymbol{\mathcal{V}}}_n(\widetilde{\bbeta})/4) \geq 1-\zeta.$$
Note that \eqref{eq:QMLexistscstar} implies that with probability converging to one we have
$$\lambda_{\min}^2\widehat{\boldsymbol{\mathcal{V}}}_n(\widetilde{\bbeta})\geq c^2.$$
Hence, with probability converging to one:
$$P(||\widehat{\bF}_n^{-1/2}(\bbeta_0)\widehat{\bU}_n(\bbeta_0)||^2<\delta_1^2\lambda_{\min}^2\widehat{\boldsymbol{\mathcal{V}}}_n(\widetilde{\bbeta})/4)
\geq P(||\widehat{\bF}_n^{-1/2}(\bbeta_0)\widehat{\bU}_n(\bbeta_0)||^2<(\delta_1 c)^2/4).$$
At the same time \eqref{eq:thQML1} and \eqref{eq:thQML2} can be used to derive
$$\widehat{\bF}_n^{-1/2}(\bbeta_0)\widehat{\bU}_n(\bbeta_0) = (\frac{1}{n}\widehat{\bF}_n)^{-1/2}(\bbeta_0)\frac{1}{\sqrt{n}}\widehat{\bU}_n(\bbeta_0) = \bF_n^{-1/2}(\bbeta_0)\bU_n(\bbeta_0) + o_P(1).$$
By the continuous mapping theorem  we then have for all $\epsilon>0$ with probability converging to~$1$
\begin{align}
||\widehat{\bF}_n^{-1/2}(\bbeta_0)\widehat{\bU}_n(\bbeta_0)||^2 \leq  ||\bF_n^{-1/2}(\bbeta_0)\bU_n(\bbeta_0)||^2 + \epsilon.\label{eq:thmparablablabla}
\end{align}
Since $\E(||\bF_n^{-1/2}(\bbeta_0)\bU_n(\bbeta_0)||^2)=p$, we may conclude from \eqref{eq:thmparablablabla} that with probability converging to $1$ we have for all sufficiently large $n$:
\begin{align*}
P(||\widehat{\bF}_n^{-1/2}(\bbeta_0)\widehat{\bU}_n(\bbeta_0)||^2<\delta_1^2\lambda_{\min}^2\widehat{\boldsymbol{\mathcal{V}}}_n(\widetilde{\bbeta})/4)
&\geq P(||\widehat{\bF}_n^{-1/2}(\bbeta_0)\widehat{\bU}_n(\bbeta_0)||^2<(\delta_1 c)^2/4)\\
&\geq P(||\bF_n^{-1/2}(\bbeta_0)\bU_n(\bbeta_0)||^2<(\delta_1 c)^2/8) \\
&\geq 1- 8p/(\delta_1 c)^2 =  1 - \zeta, \end{align*}
yielding \eqref{eq:QMLEexist} for $\delta_1^2 = 8p/(c^2\zeta)$. Asymptotic existence and consistency of our estimator are immediate consequences.\bigskip

Remember that we have
\begin{align*}\frac{1}{n}\widehat{\bH}_n(\bbeta)& = \frac{1}{n} \frac{\partial\widehat{\bU}_n(\bbeta)}{\partial \bbeta} \\
& =  -\frac{1}{n}\widehat{\bD}_n(\bbeta)^T\widehat{\bV}_n(\bbeta)^{-1}\widehat{\bD}_n(\bbeta)  + \frac{1}{n}\sum_{i=1}^n h'(\widehat{\eta}_i) \widehat{\bX}_i \widehat{\bX}_i^T(y_i-g(\widehat{\eta}_i(\bbeta)))\\
&= -\frac{1}{n}\widehat{\bF}_n(\bbeta) + \frac{1}{n}\widehat{\bR}_n(\bbeta).
\end{align*}
Now, a Taylor expansion of $\widehat{\bU}_n(\widehat{\bbeta})$ around $\bbeta_0$ yields for some $\widetilde{\bbeta}$ between $\widehat{\bbeta}$ and $\bbeta_0$ (note that $\widetilde{\bbeta}$ obviously differs from element to element):
{\small \begin{align*}
 \widehat{\bU}_n(\bbeta_0) & =\widehat{\bU}_n(\widehat{\bbeta}) - \widehat{\bH}_n(\widetilde{\bbeta})(\widehat{\bbeta} - \bbeta_0) = - \widehat{\bH}_n(\widetilde{\bbeta})(\widehat{\bbeta} - \bbeta_0)\\
& = -\Big(- \widehat{\bF}_n(\bbeta_0)(\widehat{\bbeta}-\bbeta_0) + (\widehat{\bH}_n(\widetilde{\bbeta})-\widehat{\bH}_n(\bbeta_0))(\widehat{\bbeta} - \bbeta_0) + (\widehat{\bH}_n(\bbeta_0) +\widehat{\bF}_n(\bbeta_0))(\widehat{\bbeta} - \bbeta_0) \Big).
\end{align*}}
With some straightforward calculations this leads to
\begin{align}
\begin{split}\sqrt{n}(\widehat{\bbeta}-\bbeta_0) & = \Bigg(\mathbf{I}_{S+1} - \Big(\frac{1}{n}\widehat{\bF}_n(\bbeta_0)\Big)^{-1}\Big(\frac{\widehat{\bH}_n(\widetilde{\bbeta})-\widehat{\bH}_n(\bbeta_0)}{n}\Big)  \\
& \quad - \Big(\frac{1}{n}\widehat{\bF}_n(\bbeta_0)\big)^{-1}\Big(\frac{\widehat{\bH}_n(\bbeta_0)+\widehat{\bF}_n(\bbeta_0)}{n}\Big) \Bigg)^{-1} \Big(\frac{\widehat{\bF}_n(\bbeta_0)}{n}\Big)^{-1}\frac{\widehat{\bU}_n(\bbeta_0)}{\sqrt{n}}. \label{eq:thm111}
\end{split}
\end{align}

By \eqref{eq:thQML2} and \eqref{eq:thQML3} in Proposition \ref{th:paraestQML} we have
$$\frac{\widehat{\bH}_n(\bbeta_0)+\widehat{\bF}_n(\bbeta_0)}{n} = \frac{\bH_n(\bbeta_0)+\bF_n(\bbeta_0)}{n} + o_P(1).$$
But since $h'$ is bounded we have for all $\bbeta \in \mathbf{R}^{S+1}$
\begin{align*}
\E( ||\frac{\bH_n(\bbeta)+\bF_n(\bbeta)}{n}||_2^2) &= \sum_{j=1}^{S+1}\sum_{k=1}^{S+1} \E\Big( (\frac{1}{n}\sum_{i=1}^n \varepsilon_i X_{ij}X_{ik}h'(\eta_i(\bbeta)))^2\Big)=O(n^{-1}),
\end{align*}
implying $(\bH_n(\bbeta_0)+\bF_n(\bbeta_0))/n = O_P(n^{-\frac{1}{2}})$ and hence also
$$||\Big(\frac{1}{n}\widehat{\bF}_n(\bbeta_0)\Big)^{-1}\Big(\frac{\widehat{\bH}(\bbeta_0)+\widehat{\bF}(\bbeta_0)}{n}\Big)||_2 = o_P(1).$$
By using \eqref{eq:exist0} and \eqref{eq:exist0a} we can conclude that for any compact neighborhood $N$ around $\bbeta_0$:
\begin{align}
\max_{\bbeta \in N} ||\frac{1}{n}\widehat{\bH}_n(\bbeta)- \E(\widehat{\bH}(\bbeta))|| &\stackrel{p}{\to} 0. \label{eq:exist0bb}
\end{align}
Obviously, $\widetilde{\bbeta}$ is consistent for $\bbeta_0$, since $\widehat{\bbeta}$ is consistent for $\bbeta_0$. We may conclude that $\widetilde{\bbeta}$ will be in some compact neighborhood $N$ around $\bbeta_0$ with probability converging to $1$. Moreover, since $\E(\widehat{\bH}(\bbeta))$ is continuous in $\bbeta$, \eqref{eq:exist0bb} then implies that additionally we have
\begin{align}
\max_{\widetilde{\bbeta} \in N} ||\frac{1}{n}\widehat{\bH}_n(\widetilde{\bbeta})- \E(\widehat{\bH}(\bbeta_0))|| = o_P(1).\label{eq:exists0bbb}
\end{align}

%%%%%%%%%%%%%%%%%%%%%$$\frac{1}{n}\widehat{\bR}_n(\widetilde{\bbeta}) \stackrel{p} $\E(\bR(\bbeta_0)) = $\E(\bF(\bbeta_0)) + o_P(1)$$
%\frac{1}{n}\widehat{\bR}_n(\widetilde{\bbeta}) \stackrel{p} \E(\bF(\bbeta_0))
%%%%%%%%%
%ich _denke_ es gilt:
%X_n \stackrel{p}{\to}E(F_n)
%E(F_n) \to 0
%=> X_n \stackrel{p}{\to} 0
%%%%%%%
%Hier müsste man dann anders argumentiert werden:
%\sqrt(\widehat{bbeta} -\bbeta_0) = %-(\frac{1}{n}\sum_{i=1}^n\widehat{\bR}(\widetilde{\bbeta}))^{-1} \frac{1}{\sqrt{n}}\widehat{U}(\bbeta_0)
%aber \frac{1}{n}\sum_{i=1}^n\widehat{\bR}(\widetilde{\bbeta}) \to \E(\bF(\bbeta_0))
%FERTIG
%Problem: man muss zeigen, dass \E(\frac{1}{n}\widehat{\bR}_n(\widetilde{\bbeta_0})  = \E(\bF(\bbeta_0)) + o(1)
%%%%%%%%%ANDERE ARGUMENTATION:
% and the fact that $\frac{1}{n}\E((\frac{(g'(\widehat{eta}_i(\bbeta_0)))^2}{\sigma^2(g(\widehat{\eta}_i(\bbeta_0)))}\widehat{X}_{ij}\widehat{X}_{ik} + h'(\widehat{\eta}_i)\widehat{X}_{ij}\widehat{X}_{ik}(y_i-g(\widehat{eta}_i)))^2) \to 0$  we may apply a weak law of large numbers for triangular arrays to get
The above arguments can then be used to show that
\begin{align*}
||\Big(\frac{\widehat{\bH}_n(\widetilde{\bbeta})-\widehat{\bH}_n(\bbeta_0)}{n}\Big)|| \leq
||\frac{\widehat{\bH}_n(\widetilde{\bbeta})}{n} - \E(\widehat{\bH}(\bbeta_0))|| + ||\frac{\widehat{\bH}_n(\bbeta_0)}{n} - \E(\widehat{\bH}(\bbeta_0))|| = o_P(1).
\end{align*}
Hence it also holds that
$$||\Big(\frac{1}{n}\widehat{\bF}_n(\bbeta_0)\Big)^{-1}\Big(\frac{\widehat{\bH}_n(\widetilde{\bbeta})-\widehat{\bH}_n(\bbeta_0)}{n}\Big)|| = o_P(1).$$
The asymptotic prevailing term in \eqref{eq:thm111} can then be seen as
\begin{align}
\sqrt{n}(\widehat{\bbeta}-\bbeta) \sim \Bigg(\frac{\widehat{\bF}_n(\bbeta_0)}{n}\Bigg)^{-1}\frac{\widehat{\bU}_n(\bbeta_0)}{\sqrt{n}}.\label{eq:finalctdwn}
\end{align}
It is easy to see that our assumptions on $h(x)=g'(x)/\sigma^2(g(x))$ imply that {\small$\E(||\bF_n(\bbeta_0)/n||^2) = O(\frac{1}{n})$}. Together with \eqref{eq:thQML2} we thus have $\widehat{\bF}_n(\bbeta_0)/n = \bF_n(\bbeta_0)/n + O_P(n^{-\frac{1}{2}}) = \E(\bF(\bbeta_0)) + O_P(n^{-\frac{1}{2}})$ as well as $(\widehat{\bF}_n(\bbeta_0)/n)^{-1} = (\E(\bF(\bbeta_0)))^{-1}+ O_P(n^{-\frac{1}{2}})$.

On the other hand, the Lindeberg-L\'{e}vy central limit theorem implies that $\frac{1}{\sqrt{n}}U(\bbeta_0) \stackrel{d}{\to} N(\bzero, \E(\bF(\bbeta_0)))$. Together with \eqref{eq:thQML1} we then obtain
\begin{align*}
\Bigg(\frac{\widehat{\bF}(\bbeta_0)}{n}\Bigg)^{-1}\frac{\widehat{\bU}(\bbeta_0)}{\sqrt{n}}
%&= (\E(\bF(\bbeta_0))^{-1}+ o_P(1))(\frac{\bU(\bbeta_0)}{\sqrt{n}}  +  o_P(1))\\
%&= (\E(\bF(\bbeta_0))^{-1}\frac{\bU(\bbeta_0)}{\sqrt{n}} + o_P(1) \\
&\stackrel{d}{\to} N(\bzero, (\E(\bF(\bbeta_0))^{-1}),
\end{align*}
which proves the theorem.
%$\square$
\end{proof}

\bigskip

\begin{corollary}\label{cor:uniform}
Under the assumptions of Section~\ref{sec:PES}. For any compact neighborhood $N$ around $\bbeta_0$ we have
\begin{align}
\max_{\bbeta \in N}||\frac{1}{n}\widehat{\bU}_n(\bbeta) - \frac{1}{n}\bU_n(\bbeta)|| &= o_P(1), \label{eq:unif1}\\
\max_{\bbeta \in N}||\frac{1}{n}\widehat{\bF}_n(\bbeta) - \frac{1}{n}\bF_n(\bbeta)|| &= o_P(1), \label{eq:unif2}\\
\max_{\bbeta \in N}||\frac{1}{n}\widehat{\bR}_n(\bbeta) - \frac{1}{n}\bR_n(\bbeta)|| &= o_P(1), \label{eq:unif3}\\
\intertext{as well as}
\max_{\bbeta \in N}||\frac{1}{n}\widehat{\bH}_n(\bbeta) - \frac{1}{n}\bH_n(\bbeta)|| &= o_P(1). \label{eq:unif4}
\end{align}
\end{corollary}

\bigskip

\begin{proof}[{\textbf{Proof of Corollary~\ref{cor:uniform}:}}]
%Note that Proposition \ref{th:paraestQML} implies that $G_n(\bbeta) = \frac{1}{n}\widehat{\bU}(\bbeta)\frac{1}{n}\bU(\bbeta)- \frac{1}{n}\widehat{\bU}(\bbeta)\frac{1}{n}\bU(\bbeta)$ converges in probability pointwise to $\bzero$ for all $\bbeta \in \mathbf{R}^{S+1}$. On compact subsets, pointwise convergence can be extended to uniform convergence provided $G_n(\bbeta)$ is stochastically equicontinuous.
%For a compact neighborhood $N$ arround $\bbeta_0$
The proofs of Assertions \ref{eq:unif1}-\ref{eq:unif3} are very similar. We begin with the proof of Assertion~\ref{eq:unif1}.
Using again generic copies of $\widehat{\eta}_i$, $\widehat{\bX}_i$ and $y_i$ we have with $h(x) = g'(x)/\sigma^2(g(x))$:
$$\E(n^{-1}\widehat{\bU}_n(\bbeta))= \E(\widehat{\bU}(\bbeta)) = \E(h(\widehat{\eta})\widehat{\bX}(y-g(\widehat{\eta}(\bbeta)))).$$
The $j$-th equation of $\widehat{\bU}(\bbeta)$ can be rewritten as
\begin{align*}
h(\widehat{\eta}(\bbeta))\widehat{X}_{j}(y-g(\widehat{\eta}(\bbeta))) & = h(\widehat{\eta}(\bbeta))\widehat{X}_{j}\epsilon  +  h(\widehat{\eta}(\bbeta))\widehat{X}_{j}(g(\eta(\bbeta_0))-g(\widehat{\eta}(\bbeta))).
\end{align*}
Choose an arbitrary compact neighborhood $N$ around $\bbeta_0$. Since $|h(\cdot)|\leq M_h$, $\E(\epsilon^4)<\infty$ and $|g'(\cdot)|<M_g$, it follows from a Taylor expansion that for $1\leq p \leq 2$ we have
\begin{align}
\E(\max_{\bbeta \in N}| h(\widehat{\eta}(\bbeta))\widehat{X}_{j}(y-g(\widehat{\eta}(\bbeta)))|^p) \leq M_{1,1}\label{eq:coruni1}
\end{align}
for a constant $0\leq M_{1,1}<\infty$ not depending on $n$. By \eqref{eq:coruni1} we can apply a uniform law of large numbers for triangular arrays to conclude that
\begin{align}
\max_{\bbeta \in N} || \frac{1}{n}\widehat{\bU}_n(\bbeta) - \E(\widehat{\bU}(\bbeta))|| = o_P(1). \label{eq:coruni2}
\end{align}
Similar considerations lead to
\begin{align}
\max_{\bbeta \in N} || \frac{1}{n}\bU_n(\bbeta) - \E(\bU(\bbeta))|| = o_P(1). \label{eq:coruni3}
\end{align}
By the usual decomposition we have
\begin{align}
\begin{split}
||\frac{1}{n} \widehat{\bU}_n(\bbeta) - \frac{1}{n}\bU_n(\bbeta)|| &\leq ||\frac{1}{n}\widehat{\bU}_n(\bbeta) - \E(\widehat{\bU}(\bbeta))|| \\
&\quad + ||\frac{1}{n}\bU_n(\bbeta) - \E(\bU(\bbeta))|| + ||\E(\widehat{\bU}(\bbeta) ) - \E(\bU(\bbeta))||.
\end{split}\label{eq:coruni3a}
\end{align}
Assertion \eqref{eq:unif1} then follows immediately from \eqref{eq:coruni2}, \eqref{eq:coruni3}, if we can show that $\E(\widehat{\bU}(\bbeta))$ converges uniformly to $\E(\bU(\bbeta))$ and not only pointwise as given in \eqref{eq:thQML3mean0}.\\

It is well known that pointwise convergence of a sequence of functions $f_n$  on a compact set $N$ can be extended to uniform convergence over $N$, if $f_n$ is an equicontinuous sequence. Remember that a sufficient condition for equicontinuity is that there exists a common Lipschitz constant. We aim to show that there exists a constant $L<\infty$, where $L$ does not depend on $n$, such that for all $\bbeta$ and $\widetilde{\bbeta}$ in $N$ we have $||\E(\widehat{\bU}(\bbeta)) - \E(\widehat{\bU}(\widetilde{\bbeta}))|| \leq L ||\bbeta-\widetilde{\bbeta}||.$
%Define $f_n(\bbeta)=\E(\widehat{\bU}(\bbeta))$.
Remember that the $j$th equation of $\E(\widehat{\bU}(\bbeta))$ is given by
%$f_n(\bbeta)$ is given by $f_{n,j}(\bbeta)=
$\E(h(\widehat{\eta}(\bbeta))\widehat{X}_{j}(y-g(\widehat{\eta}(\bbeta))))$.
Note that
\begin{align}
\begin{split}h&(\widehat{\eta}(\bbeta)) \widehat{X}_{j}(y-g(\widehat{\eta}(\bbeta))) -
h(\widehat{\eta}(\widetilde{\bbeta}))\widehat{X}_{j}(y-g(\widehat{\eta}(\widetilde{\bbeta}))))\\
%& = \widehat{X}_j y_i (h(\widehat{\eta}(\bbeta))-h(\widehat{\eta}(\widetilde{\bbeta})))\\
%& \quad + h(\widehat{\eta}(\widetilde{\bbeta}))g(\widehat{\eta}(\widetilde{\bbeta}))\widehat{X}_{j} -  h(\widehat{\eta}(\bbeta))\widehat{X}_{j}g(\widehat{\eta}(\bbeta))\\
& = \widehat{X}_j y (h(\widehat{\eta}(\bbeta))-h(\widehat{\eta}(\widetilde{\bbeta})))\\
& \quad + \widehat{X}_{j}h(\widehat{\eta}(\widetilde{\bbeta}))(g(\widehat{\eta}(\widetilde{\bbeta}))-g(\widehat{\eta}(\bbeta)) )\\
& \quad - \widehat{X}_{j}(h(\widehat{\eta}(\widetilde{\bbeta})) -  h(\widehat{\eta}(\bbeta)))g(\widehat{\eta}(\bbeta)).
\end{split}\label{eq:mvmcc}
\end{align}
Since for a $J\times K$ matrix  $A$ we have $||A||= \sqrt{\sum_{j,k}a_{jk}^2} \leq \sum_{j,k} |a_{jk}|$ and since $h$, $h'$ and $g'$ are bounded and $N$ is compact, our assumptions on $X$ then in particularly imply together with \eqref{eq:mvmcc} that there exists a constant $L<\infty$, which is in particular independent of $n$ such that for all $\bbeta$ and $\widetilde{\bbeta} \in N$
%\widetilde{\bbeta} kann man immer mit dem maximum abschätzen...die erwartungswerte mit CS und der gaussannahme. da der process über kompakten träger definiert ist, hat er auch eine beschränkte varianz. alle varianzen von $X_j$ und $\widehat{X}_j$ kann man also abschätzen ($\widehat{X}_j$ gegeben $\widehat{\tau}_j$ ist ohnehin wieder normalverteilt...)
%Man benötigt hier also nur noch die dreiecksungleichung $|E(diff)|\leq E(|diff|)$
%und die diff lässt sich in in 3 teile zerlegen. einmal noch g an der stelle 0 taylorn und dann ausnutzen, dass $|h(x)-h(y)|\leq M_h |x-y| gilt$...
\begin{align}
||\E(\widehat{\bU}(\bbeta)) - \E(\widehat{\bU}(\widetilde{\bbeta}))|| \leq L ||\bbeta-\widetilde{\bbeta}||.\label{eq:splittige1}
\end{align}
Assertion \eqref{eq:unif1} then follows from \eqref{eq:coruni3a} together with \eqref{eq:coruni2}, \eqref{eq:coruni3}, \eqref{eq:thQML3mean0} and \eqref{eq:splittige1}.

In order to proof Assertion \eqref{eq:unif2} we can use the decomposition
\begin{align*}
||\frac{1}{n}\widehat{\bF}_n(\bbeta) - \frac{1}{n}\bF_n(\bbeta)|| &\leq ||\frac{1}{n}\widehat{\bF}_n(\bbeta) - \E(\widehat{\bF}(\bbeta))|| \\
&\quad + ||\frac{1}{n}\bF_n(\bbeta) - \E(\bF(\bbeta))|| + ||\E(\widehat{\bF}(\bbeta) ) - \E(\bF(\bbeta))||.
\end{align*}
Let $h_1(x)= g'(x)^2/\sigma^2(g(x))$ and remember that $|h_1(\cdot)|\leq M_{h_1}$ for some constant $M_{h_1}<\infty$.
It immediately follows that for any compact neighborhood $N$ arround $\bbeta_0$ we have
\begin{align}
\E(\max_{\bbeta\in N}||h_1(\eta_i(\bbeta))\bX_i\bX_i^T||) <\infty. \label{eq:lalalalala}
\end{align}
By $\eqref{eq:lalalalala}$ we can apply a uniform law of large numbers to derive
\begin{align}
\max_{\bbeta \in N}|| \frac{1}{n}\bF_n(\bbeta) - \E(\bF(\bbeta))|| = o_P(1). \label{eq:lalalalala3}
\end{align}
Assertion \ref{eq:unif2} then follows immediately from \eqref{eq:lalalalala3}, \eqref{eq:exist0} and
\eqref{eq:thQML3mean1}, which holds uniformly on $N$ by similar steps as above by noting that a typical element %$F_{j,k}(\bbeta,\widetilde{\bbeta})$
of $\widehat{\bF}(\bbeta)-\widehat{\bF}(\widetilde{\bbeta})$ can be written as $\widehat{X}_j\widehat{X}_k(h_1(\widehat{\eta}(\bbeta))- h_1(\widehat{\eta}(\widetilde{\bbeta})))$. \\
Assertion \ref{eq:unif3} can be proved in a similar manner using \eqref{eq:thQML3mean0},  \eqref{eq:exist0bb} and the uniform convergence of $||\bR_n(\bbeta)/n- \E(\bR(\bbeta))||$, which is easy to establish using  $h_1(\eta(\bbeta))\bX \bX^T(y-g(\eta(\bbeta))) = h_1(\eta(\bbeta)) \bX \bX^T\varepsilon + h_1(\eta(\bbeta)) \bX \bX^T(g(\eta(\bbeta_0))-g(\eta(\bbeta)))$ and the assumption that $|h_1'(\cdot)|\leq M_{h_1}$.\\

To proof assertion \eqref{eq:unif4}, remember that $\widehat{\bH}(\bbeta)/n = -\widehat{\bF}_n(\bbeta)/n + \widehat{\bR}_n(\bbeta)/n$, assertion \eqref{eq:unif4} follows then immediately from \eqref{eq:unif2} and \eqref{eq:unif3}.
%$\square$\\
\end{proof}
\bibliographystyleappendix{Chicago}
\bibliographyappendix{bibfileappendix}
%%%%%%%%%%%%%%%%%%%%%%%%%%%%%%%%%%%%%%%%

\end{document}